\pdfoutput=1
\RequirePackage{ifpdf}
\ifpdf 
\documentclass[pdftex]{sigma}
\else
\documentclass{sigma}
\fi

\numberwithin{equation}{section}

\newtheorem{Theorem}{Theorem}[section]
\newtheorem*{Theorem*}{Theorem}
\newtheorem{Corollary}[Theorem]{Corollary}
\newtheorem{Lemma}[Theorem]{Lemma}
\newtheorem{prop}[Theorem]{Proposition}
 { \theoremstyle{definition}
\newtheorem{Definition}[Theorem]{Definition}

\newtheorem{Remark}[Theorem]{Remark} }
\newtheorem{rhp}[Theorem]{Riemann--Hilbert Problem}

\usepackage[mathscr]{euscript}
\usepackage{commath, enumitem, tikz}

\newcommand*\lw{\ensuremath{\vcenter{\hbox{\includegraphics[width=.6em]{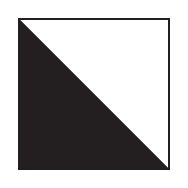}}}}}

\newcommand{\ee}{\mathrm{e}}
\newcommand{\dd}{\mathrm{d}}
\newcommand{\ii}{\mathrm{i}}
\newcommand{\para}[1]{\breve{#1}}

\newcommand{\R}{{\mathbb{R}}}
\newcommand{\N}{{\mathbb{N}}}
 \newcommand{\C}{{\mathbb{C}}}

\newcommand{\Z}{{\mathbb{Z}}}
\newcommand\lb{\left(}
\newcommand\rb{\right)}

\newcommand{\re}{\operatorname{Re}}
\newcommand{\im}{\operatorname{Im}}
\newcommand{\Tr}{\operatorname{Tr}}

\renewcommand{\arg}{\operatorname{Arg}}

\newcommand{\sgn}{\operatorname{sgn}}

\begin{document}

\renewcommand{\thefootnote}{}

\newcommand{\arXivNumber}{2307.11217}

\renewcommand{\PaperNumber}{019}

\FirstPageHeading

\ShortArticleName{Painlev\'e-III Monodromy Maps Under the $D_6\to D_8$ Confluence}

\ArticleName{Painlev\'e-III Monodromy Maps Under the $\boldsymbol{D_6\to D_8}$\\ Confluence and Applications to the Large-Parameter\\ Asymptotics of Rational Solutions\footnote{This paper is a~contribution to the Special Issue on Evolution Equations, Exactly Solvable Models and Random Matrices in honor of Alexander Its' 70th birthday. The~full collection is available at \href{https://www.emis.de/journals/SIGMA/Its.html}{https://www.emis.de/journals/SIGMA/Its.html}}}

\Author{Ahmad BARHOUMI~$^{\rm ab}$, Oleg LISOVYY~$^{\rm c}$, Peter D.~MILLER~$^{\rm a}$ and Andrei PROKHOROV~$^{\rm ad}$}

\AuthorNameForHeading{A.~Barhoumi, O.~Lisovyy, P.D.~Miller and A.~Prokhorov}

\Address{$^{\rm a)}$~Department of Mathematics, University of Michigan,\\
\hphantom{$^{\rm a)}$}~East Hall, 530 Church St., Ann Arbor, MI 48109, USA}
\EmailD{\href{mailto:millerpd@umich.edu}{millerpd@umich.edu}, \href{mailto:andreip@umich.edu}{andreip@umich.edu}}

\Address{$^{\rm b)}$~Department of Mathematics, KTH Royal Institute of Technology,\\
\hphantom{$^{\rm b)}$}~Lindstedtsv\"agen 25, 114 28, Stockholm, Sweden}
\EmailD{\href{mailto:ahmadba@kth.se}{ahmadba@kth.se}}

\Address{$^{\rm c)}$~Institut Denis-Poisson, Universit\'e de Tours, CNRS, Parc de Grandmont, 37200 Tours, France}
\EmailD{\href{mailto:lisovyi@lmpt.univ-tours.fr}{lisovyi@lmpt.univ-tours.fr}}

\Address{$^{\rm d)}$~St. Petersburg State University, Universitetskaya emb.~7/9, 199034 St.~Petersburg, Russia}

\ArticleDates{Received July 24, 2023, in final form January 23, 2024; Published online March 09, 2024}

\Abstract{The third Painlev\'e equation in its generic form, often referred to as Painlev\'e-III($D_6$), is given by
 \[
 \dod[2]{u}{x}=\dfrac{1}{u}\left( \dod{u}{x} \right)^2-\dfrac{1}{x} \dod{u}{x} + \dfrac{\alpha u^2 + \beta}{x}+4u^3-\frac{4}{u}, \qquad \alpha,\beta \in \C.
 \]
 Starting from a generic initial solution $u_0(x)$ corresponding to parameters $\alpha$, $\beta$, denoted as the triple $(u_0(x), \alpha, \beta)$, we apply an explicit B\"acklund transformation to generate a family of solutions $(u_n(x), \alpha + 4n, \beta + 4n)$ indexed by $n \in \N$. We study the large $n$ behavior of the solutions $(u_n(x), \alpha + 4n, \beta + 4n)$ under the scaling $x = z/n$ in two different ways: (a)~analyzing the convergence properties of series solutions to the equation, and (b)~using a Riemann--Hilbert representation of the solution $u_n(z/n)$. Our main result is a proof that the limit of solutions $u_n(z/n)$ exists and is given by a solution of the degenerate Painlev\'e-III equation, known as Painlev\'e-III($D_8$),
 \[
 \dod[2]{U}{z}=\dfrac{1}{U}\left( \dod{U}{z}\right)^2-\dfrac{1}{z} \dod{U}{z} + \dfrac{4U^2 + 4}{z}.
 \]
 A notable application of our result is to rational solutions of Painlev\'e-III($D_6$), which are constructed using the seed solution $(1, 4m, -4m)$ where $m \in \C \setminus \big(\Z + \frac{1}{2}\big)$ and can be written as a particular ratio of Umemura polynomials. We identify the limiting solution in terms of both its initial condition at $z = 0$ when it is well defined, and by its monodromy data in the general case. Furthermore, as a consequence of our analysis, we deduce the asymptotic behavior of generic solutions of Painlev\'e-III, both $D_6$ and $D_8$ at $z = 0$. We also deduce the large $n$ behavior of the Umemura polynomials in a neighborhood of $z = 0$.}


\Classification{34M55; 34E05; 34M50; 34M56; 33E17}

\renewcommand{\thefootnote}{\arabic{footnote}}
\setcounter{footnote}{0}

\section{Introduction}
This paper is a study of the confluence of solutions of the generic Painlev\'e-III equation to solutions of its parameter-free degeneration. The six Painlev\'e equations and their solutions, often referred to as \emph{Painlev\'e transcendents}, have been the subject of intense study. This is largely motivated by the fact that Painlev\'e transcendents are generically transcendental, and yet appear in various applications in integrable systems, integrable probability, and random matrix theory to name a few.

\subsection{B\"acklund transformations and rational solutions of Painlev\'e-III}
All Painlev\'e equations but the first are actually families of differential equations indexed by complex parameters appearing as coefficients. However, certain solutions corresponding to different parameters can be related via \emph{B\"acklund transformations}. For example, consider our main object of study, the generic Painlev\'e III equation, known as PIII($D_6$):
\begin{equation}
\dod[2]{u}{x}=\dfrac{1}{u}\left( \dod{u}{x} \right)^2-\dfrac{1}{x} \dod{u}{x} + \dfrac{\alpha u^2 + \beta}{x}+4u^3-\frac{4}{u}, \qquad \alpha,\beta \in \C.
\label{eq:PIII-$D_6$}
\end{equation}
In \cite{Gromak}, Gromak discovered that the transformation
\begin{equation}
 u(x) \mapsto \hat{u}(x) := \dfrac{2xu'(x) + 4xu(x)^2 + 4x -\beta u(x) - 2u(x)}{u(x) ( 2xu'(x) + 4x u(x)^2 +4x + \alpha u(x) + 2u(x) )}
\label{eq:Gromak-transformation}
\end{equation}
mapped solutions of \eqref{eq:PIII-$D_6$} with parameters $(\alpha, \beta)$ to solutions of \eqref{eq:PIII-$D_6$} with parameters $(\alpha + 4,\allowbreak \beta+4)$. With this one can construct from a given seed solution $(u_0, \alpha, \beta)$ a family of solutions $(u_n, \alpha + 4n, \beta + 4n)$ by iterating transformation \eqref{eq:Gromak-transformation}. The paper \cite{MCB} contains a survey of families of solutions of \eqref{eq:PIII-$D_6$} constructed using this and other B\"acklund transformations. A notable family of solutions constructed in this manner is a sequence of rational solutions $u=u_n(x;m)$ obtained from the seed function $u_0(x) \equiv 1$ and parameters $\alpha = -\beta = 4m$. This family of solutions has been numerically and analytically explored in \cite{BMS18}, and many conjectures were formulated there. While some of these were later resolved in the sequel \cite{BM20}, some conjectures remained open, see \cite[Conjectures 4 and 5]{BMS18}. Conjecture 5 is concerned with the behavior of $u_n(x;m)$ near the singular point $x = 0$. As was done in \cite{BMS18}, writing \[z = nx,\qquad U_n(z;m):= u_n(x;m)\] and considering large $n$ for fixed $m$ yields the differential equation
\[
\dod[2]{U_n}{z} = \dfrac{1}{U_n} \left( \dod{U_n}{z}\right)^2 - \dfrac{1}{z} \dod{U_n}{z} + \dfrac{4U_n^2 + 4}{z} + \mathcal{O}\big(n^{-1}\big).
\]
Formally taking the limit and denoting the limiting function $U(z;m)$ yields the parameter-free Painlev\'e-III equation, referred to as PIII($D_8$),
\begin{equation}
\dod[2]{U}{z}=\dfrac{1}{U}\left( \dod{U}{z}\right)^2-\dfrac{1}{z} \dod{U}{z} + \dfrac{4U^2 + 4}{z}.
\label{eq:PIII-$D_8$}
\end{equation}
The content of Conjecture 5 is that this convergence holds at the level of solutions, not just equations.

\subsection{Results}
To begin with, we prove Conjecture 5 from \cite{BMS18} in this work; to be more precise we establish the following theorem.
\begin{Theorem}\label{thm:limit-of-rational-solutions}
Fix $m \in \C \setminus \big(\Z + \frac{1}{2}\big)$ and let $u_n(x;m)$ be the family of rational solutions described above. There exists a unique solution $U(z)=U(z;m)$ of the Painlev\'e-{\rm III}$(D_8)$ equation \eqref{eq:PIII-$D_8$} analytic at the origin with $U(0;m)=\tan\big(\frac{\pi}{2}\big(m+\frac{1}{2}\big)\big)\neq 0$ such that
\begin{gather}
\lim_{j \to \infty} u_{2j}\left(\frac{z}{2j}; m\right) = U(z;m),\nonumber\\
\lim_{j \to \infty} u_{2j+1}\left(\frac{z}{2j+1}; m\right) = -1/U(z;m)\label{eq:even-odd-limit}
\end{gather}
for $z\notin \Sigma(m)$, where $\Sigma(m)$ denotes the union of all poles and zeros of $z\mapsto U(z;m)$. The convergence is uniform on compact subsets of $\C \setminus \Sigma(m)$.
\end{Theorem}
\begin{figure}[t]
 \centering
\includegraphics[scale=0.8]{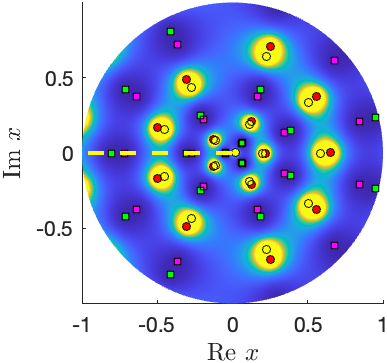}\qquad
\includegraphics[scale=0.8]{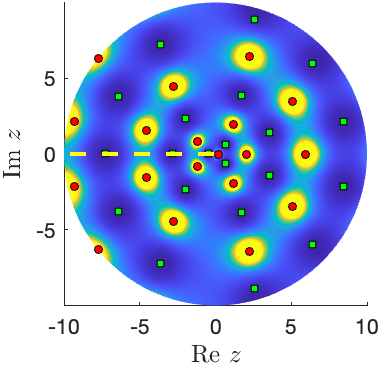}
 \caption{Left: the rational solution $u_{10}(x;0.25)$. Right: the limiting solution $U(z;0.25)$, where we recall the notation $z=nx$ for $n=10$. All poles of $u_{10}(x;0.25)$ are simple with residue $\frac{1}{2}$/$-\frac{1}{2}$, indicated in the plot with red/yellow circles. Likewise, all zeros of $u_{10}(x;0.25)$ are simple with derivative $2$/$-2$, indicated in the plot with pink/green squares. On the other hand, all poles and zeros of $U(z;0.25)$ have multiplicity $2$ and are marked with red circles and green squares respectively. }
 \label{fig:rational}
\end{figure}
We illustrate this theorem in Figure \ref{fig:rational}. The pictures are made using the code from \cite{FFW}, which was generously provided by the authors.

In Section \ref{sec:Frobenius-method}, we study the Maclaurin series solutions of \eqref{eq:PIII-$D_6$}; this characterizes the limiting solution of \eqref{eq:PIII-$D_8$} via its initial conditions and produces a local version of Theorem \ref{thm:limit-of-rational-solutions}, see Theorem \ref{thm:$D_6$-$D_8$-series} and Corollary \ref{cor:local-limit-of-rational-solutions} below.

The rational solutions $u_n(x;m)$ are related to the so-called Umemura polynomials $s_n(x;m)$ by the formula
\begin{equation}
u_n(x;m)=\dfrac{s_n(x;m-1)s_{n-1}(x;m)}{s_n(x;m)s_{n-1}(x;m-1)}.
\label{eq:un-polys}
\end{equation}
Indeed, a sequence of rational functions $x\mapsto s_n(x;m)$ is determined by the recurrence relation
\begin{gather}
s_{n+1}(x;m)\nonumber\\
\quad=\dfrac{(4x+2m+1)s_n(x;m)^2-s_n'(x;m)s_n(x;m)-x\big(s_n''(x;m)s_n(x;m)-s_n'(x;m)^2\big)}{2s_{n-1}(x;m)}
\label{eq:umemura-recurrence}
\end{gather}
with initial conditions
\begin{equation}\label{eq:umemura-recurrence_init}
s_{-1}(x;m)=1,\qquad s_{0}(x;m)=1.
\end{equation}
It was shown in \cite{clarkson2018constructive,Umemura} that the rational functions $s_n(x;m)$ are actually polynomials. In Section~\ref{sec:Umemura}, we use Corollary \ref{cor:local-limit-of-rational-solutions} to deduce asymptotics of the Umemura polynomials themselves. To formulate our result, we need to introduce a certain Fredholm determinant; more precisely, let $K_r\colon L^2[0, r] \to L^2[0, r]$ denote the integral operator with the continuous Bessel kernel
 	\begin{align*}
 	&K(x,y)=\frac{\sqrt{x}J_1\big( \sqrt x\big) J_0\big(\sqrt y\big)-
 	 J_0\big(\sqrt x\big) \sqrt{y}J_1\big( \sqrt y\big)}{2( x-y)}.
 	\end{align*}
For any $\lambda\in\mathbb C$, let $D_\lambda( r)$ be the Fredholm determinant
\[
 	D_\lambda( r) :=\det ( \mathbf 1-\lambda K_r).
\]
It is well known (see, e.g., \cite[Chapter 24]{Lax}) that the Fredholm determinant $D_\lambda( r)$ is an entire function of $\lambda$. Since $K_r$ is a trace-class integral operator, one of several equivalent ways to define~$D_\lambda\lb r\rb$ is via the Plemelj--Smithies formula
 \begin{gather}
 D_\lambda( r)=\exp\biggl(-\sum_{\ell=1}^\infty \Tr K_r^{\ell} \frac{\lambda^\ell}{\ell}\biggr).
 \label{eq:plemelj-smithies}
 \end{gather}
The traces in \eqref{eq:plemelj-smithies} have explicit expressions as iterated integrals
\[
\Tr K_r^\ell := \int_0^r K^{(\ell)}(t, t) \dd t,
\]
where
\[
K^{(1)}(x, y) = K(x, y) \qquad \text{and} \qquad K^{(\ell)}(x, y) = \int_0^r K(x, t) K^{(\ell - 1)}(t, y) \dd t.
\]
By re-scaling the integrals to bring the $r$-dependence to the integrand and observing that $J_0 \big( \sqrt{xy} \big)$ and $\sqrt{xy}J_1\big(\sqrt{xy}\big)$ are both entire functions with respect to both $x$ and $y$, we see that $\Tr K_r^{\ell}$ and~$D_\lambda ( r)$ can be extended to analytic functions of $r$ in a neighborhood of $r=0$, and in fact~$\Tr K_r^{\ell}=\mathcal{O}\big(r^\ell\big)$ as $r\to 0$, from which we obtain $D_\lambda(0)=1$. We are now ready to state our second theorem.
\begin{Theorem}\label{thm:umemura_asym}
Fix $m \in \C \setminus \big(\Z + \frac{1}{2}\big)$. Then, there exists a small enough neighborhood of the origin, $\mathcal{G}$, such that the Umemura polynomials admit the following limits along the even and odd subsequences:
\begin{equation}
 \lim_{j\to\infty}\frac{s_{2j}\big( \frac{z}{2j+1};m \big)}{s_{2j}(0;m)}=\ee^{2\ii z}\left(\frac{U(z;m)}{U(0;m)}\right)^{-\frac{1}{4}} \sqrt{D_{\lambda(m)}\lb 32\ii z\rb},
\label{eq:umemura-asymtotics-1-bessel}
\end{equation}
and
\begin{equation}
 \lim_{j\to\infty}\frac{s_{2j-1} \big(\frac{z}{2j} ;m\big)}{s_{2j-1}(0;m)}=\ee^{2\ii z}\left(\frac{U(z;m)}{U(0;m)}\right)^{\frac{1}{4}} \sqrt{D_{\lambda(m)}\lb 32\ii z\rb},
\label{eq:umemura-asymtotics-2-bessel}
\end{equation}
where $\lambda(m)=1/\big(1+\ee^{2\pi \ii m}\big)$, the square root and fractional powers denote the principal branches taking the value $1$ at $z=0$, and the convergence is uniform for $z\in\mathcal{G}$. Furthermore,
the values of the Umemura polynomials at the origin have the leading asymptotics
\begin{gather}
 s_{2j}(0;m)\sim\sqrt{2\pi}\ee^{4\zeta'(-1)}\frac{j^{2j^2+j+\frac{m^2}{2}+\frac{m}{2}+\frac{1}{24}}\ee^{-3j^2-j}2^{2j^2+2j}(-\cos(\pi m))^j}{G\big(\frac{5}{4}+\frac{m}{2}\big)G\big(\frac{5}{4}-\frac{m}{2}\big)G\big(\frac{7}{4}+\frac{m}{2}\big)G\big(\frac{3}{4}-\frac{m}{2}\big)},\qquad j\to\infty,\label{eq:umemura-zero-asym-even}
\\
 s_{2j-1}(0;m)\sim\frac{\ee^{4\zeta'(-1)}}{\sqrt{2\pi}}\frac{j^{2j^2-j+\frac{m^2}{2}+\frac{m}{2}+\frac{1}{24}}\ee^{-3j^2+j}2^{2j^2}(\cos(\pi m))^j}{G\big(\frac{3}{4}+\frac{m}{2}\big)G\big(\frac{3}{4}-\frac{m}{2}\big)G\big(\frac{5}{4}+\frac{m}{2}\big)G\big(\frac{1}{4}-\frac{m}{2}\big)},\qquad j\to\infty,\label{eq:umemura-zero-asym-odd}
\end{gather}
in which $G$ denotes the Barnes $G$-function and $\zeta$ denotes the Riemann zeta function.
\end{Theorem}
In fact, one can check that the expressions on the right-hand side of \eqref{eq:umemura-asymtotics-1-bessel} and \eqref{eq:umemura-asymtotics-2-bessel} admit analytic continuation from a neighborhood of $z=0$ to the whole $z$-plane. Although it does not follow from our proof, this suggests that the neighborhood $\mathcal{G}$ can be taken to be an arbitrary bounded set.

Our analysis of series solutions in Section \ref{sec:Frobenius-method} points to a more general statement about the coalescence of solutions of \eqref{eq:PIII-$D_6$} to solutions of \eqref{eq:PIII-$D_8$}. The technical result leading to Theorem~\ref{thm:limit-of-rational-solutions} by Maclaurin series (see Theorem \ref{thm:$D_6$-$D_8$-series} below) applies not only to rational solutions, but to all sequences of solutions with initial conditions converging to finite, nonvanishing limits. This, however, is a serious limitation since $x = 0$ is a singular point of Painlev\'e-III, and generic solutions of \eqref{eq:PIII-$D_6$} will be singular at this point and behave like $u(x)\simeq ax^p$, $|{\re}(p)|< 1$. More specifically, based on symbolic computation we expect the asymptotic expansion for solutions of \eqref{eq:PIII-$D_6$} in the form
\[
u(x)\sim \sum_{k=0}^{\infty}\sum_{l=0}^{k+1}\big(b_{kl}x^{2k+(2l-1)p}+c_{kl}x^{2k+1+2lp}\big) \qquad \text{as}\quad x\to 0.
\]
To tackle this issue, we develop a second approach that avoids series expansions and instead relies on the isomonodromy representation of the Painlev\'e transcendents. It was first discovered by Garnier \cite{MR1509146} and further explicated by Jimbo and Miwa in \cite{MR625446} that Painlev\'e equations can be formulated as monodromy-preserving, or isomonodromic, deformations of corresponding~${2 \times 2}$ first-order systems of differential equations. This allows one to characterize solutions of a~given Painlev\'e equation in terms of a $2\times 2$ Riemann--Hilbert problem. Such a monodromy representation was obtained for rational solutions of Painlev\'e-III($D_6$) in \cite{BMS18}. From this point of view, one can show that for fixed $\alpha, \beta \in \C$, the solutions of \eqref{eq:PIII-$D_6$} are parametrized by triples~$(x_1, x_2, x_3) \in \C^3$ on the cubic surface, known as the \emph{monodromy manifold}, given by
\begin{equation}
 \label{eq:cubic-D6} x_1x_2x_3+x_1^2+x_2^2+x_2 \bigl(\ee^{-{\ii\pi \alpha}/{4}}-\ee^{-{\ii\pi \beta}/{4}} \bigr)+x_1 \bigl(1-\ee^{-{\ii\pi(\alpha+ \beta)}/{4}}\bigr)-\ee^{-{\ii\pi(\alpha+ \beta)}/{4}}=0.
\end{equation}
The exponential constants appearing as coefficients in \eqref{eq:cubic-D6} will appear in multiple equations, making it convenient to introduce the notation
\begin{equation}
 e_0:=\ee^{{\ii\pi\alpha}/{8}}\neq 0 \qquad \text{and} \qquad e_\infty:=\ii\ee^{-{\ii\pi\beta}/{8}}\neq 0.
 \label{eq:e0-einfty-alpha-beta}
\end{equation}
In Section \ref{sec:monodromy-rep-$D_6$}, we reproduce the derivation of the cubic surface \eqref{eq:cubic-D6} carried out in \cite{PS} and connect the quantities $x_i$ with other invariant quantities that appear in the Riemann--Hilbert Problem~\ref{rhp:initial} associated with PIII($D_6$). In Section \ref{sec:monodromy-rep-$D_8$}, we present an analogous parametrization of solutions of the $D_8$ degeneration \eqref{eq:PIII-$D_8$} of PIII in terms of triples $(y_1, y_2, y_3) \in \C^3$ appearing in the Riemann--Hilbert Problem~\ref{rhp:D8} and satisfying
\begin{equation}
\label{eq:cubic-D8}
y_1 y_2 y_3 + y_1^2 + y_2^2 + 1 = 0.
\end{equation}
Away from its singular points, we parametrize points $(x_1,x_2,x_3)$ on the cubic surface \eqref{eq:cubic-D6} using parameters $e_1$, $e_2$ appearing naturally from the point of view of the Riemann--Hilbert problem. In fact, $e_1^2$, $e_1^{-2}$ are eigenvalues of a certain monodromy matrix for a circuit about the origin for a linear system, see \eqref{eq:generic-system}. The parameter $e_2$ appears in the connection matrix for the same system, see \eqref{eq:C-zero-infty-diagonalization}.
We call $(e_1,e_2)$ monodromy parameters.
\begin{Definition}[see Section \ref{sec:monodromy-rep-$D_6$} for details]\label{def:generic}\samepage
We say the monodromy parameters $(e_1, e_2)$ are \emph{generic} if
 \begin{enumerate}\itemsep=0pt
 \item[(i)] $e_1^4 \neq 1$,
 \item[(ii)] $e_1 e_2 \neq 0$,
 \item[(iii)] $e_1^2 \neq e_\infty^{\pm 2}$ and $e_1^2 \neq e_0^{\pm 2}$.
 \end{enumerate}
\end{Definition}
Before moving on, we pause to make a few observations.
\begin{itemize}\itemsep=0pt
 \item Condition (ii) implies that generic monodromy parameters are nonvanishing, hence we may write
 \begin{gather}
 e_1=\ee^{\ii\pi\mu}\qquad \text{and}\qquad e_2=\ee^{\ii\pi\eta}. \label{eq:e1-e2-mu-eta}
 \end{gather}
 \item Since $e_1^2$, $e_1^{-2}$ are the essential quantities related to the complex parameter $e_1$ in our parametrization, and the former are insensitive to a change in sign of $e_1$, we may take~$e_1$ to be in the right half plane; in view of \eqref{eq:e1-e2-mu-eta}, this corresponds to $-\frac12 < \re(\mu) \leq \frac12$. Note that the choice of including the upper versus the lower endpoint of this range is arbitrary.
 \item Due to $e_1^2$, $e_1^{-2}$ being eigenvalues of the same matrix (see \eqref{eq:Stokes-products-eigenvectors} below), the parameters~$\mu$,~$-\mu$ correspond to the same solution of \eqref{eq:PIII-$D_6$}. While we could restrict to parameters where $\re(\mu) >0$ (say), we choose not to and in turn arrive at slightly simpler formul\ae; see Remark \ref{remark:e2-mu-minus-mu-change} below.
 \item It turns out that the parameter $e_2$ is determined up to a sign as well, see Remark \ref{remark:change-sign-e2} below. As such, we take it to be in the right half plane as well, or $-\frac12 < \re(\eta) \leq \frac12$.
\end{itemize}
With this in mind, we can now state a more general theorem.
\begin{Theorem}\label{thm:general}
 Let $u_0$ be the solution of \eqref{eq:PIII-$D_6$} corresponding to monodromy data $(\alpha, \beta, x_1, x_2, x_3)$ parametrized by generic monodromy parameters $(e_1,e_2)$ using formul\ae\ $($consistent with \eqref{eq:cubic-D6}$)$
 \begin{gather}
 x_1=
 \frac{e_1^2 \big(e_0^2 e_2^2e_\infty^2 \big(e_1^2 e_\infty^2-1\big)+e_0^2 e_1^2-1\big)\big(\big(e_0^2 e_1^2-1\big)^2 + e_0^2e_2^2e_\infty^2 \big(e_0^2-e_1^2\big) \big(e_\infty^2-e_1^2\big)\big)}{e_0^4 e_2^2e_\infty^2\big(1-e_0^2 e_1^2\big)\big(e_1^4-1\big)^2 },\label{eq:x1}
\\
 x_2= \frac{\big(e_0^2 e_2^2 e_1^2e_\infty^2 \big(e_1^2-e_\infty^2\big)+1-e_0^2 e_1^2\big) \big(\big(e_0^2 e_1^2-1\big)^2\!+e_0^2 e_1^2e_2^2 e_\infty^2 \big(e_0^2-e_1^2\big) \big(e_1^2 e_\infty^2\!-1\big)\big)}{e_0^4e_2^2e_\infty^2 \big(1-e_0^2 e_1^2\big) \big(e_1^4-1\big)^2 }, \!\!\!\!
\\
 x_3=
 e_1^2+\frac{1}{e_1^2},\label{eq:x3}
\end{gather}
and let $U(z)=U(z;y_1,y_2,y_3)$ denote the solution of \eqref{eq:PIII-$D_8$} with monodromy data $($consistent with \eqref{eq:cubic-D8}$)$
 \begin{gather}
 y_1 = \ii\sqrt{\frac{e_\infty^2-e_1^2}{1-e_0^2 e_1^2}} \cdot\frac{ \big(1-e_0^2 e_1^2+e_0^2e_1^6e_2^2e_
 \infty^2\big(e_1^2-e_\infty^2\big)\big)}{e_0 e_1 e_2 e_\infty\big( e_\infty^2-e_1^2\big) \big(e_1^4-1\big) },
 \label{eq:y1}
\\
 y_2 = \ii\sqrt{\frac{e_\infty^2-e_1^2}{1-e_0^2 e_1^2}} \cdot\frac{e_1 \big(1-e_0^2 e_1^2+e_0^2e_1^2e_2^2e_
 \infty^2\big(e_1^2-e_\infty^2\big)\big)}{e_0e_2 e_\infty \big( e_\infty^2-e_1^2\big) \big(e_1^4-1\big)},
 \label{eq:y2}
 \\
 y_3 = -e_1^2-\frac{1}{e_1^2}.
 \label{eq:y3}
 \end{gather}
If $u_n$ is the $n$th iterate of $u_0$ under transformation \eqref{eq:Gromak-transformation}, then for $z\notin \Sigma(y_1,y_2,y_3)$
 \begin{gather*}
 \lim_{j \to \infty} u_{2j}(z/2j) = U(z;y_1,y_2,y_3),\\ \lim_{j \to \infty} u_{2j+1}(z/(2j+1)) = -1/U(z;y_1,y_2,y_3),
 \end{gather*}
where convergence is uniform on compact subsets of $\C\setminus \Sigma(y_1, y_2, y_3)$ slit along $\arg(z)= \pm \pi$ and~$\Sigma(y_1,y_2,y_3)$ is the union of all poles and zeros of $z\mapsto U(z;y_1,y_2,y_3)$.
\end{Theorem}
There is nothing fundamental about the exclusion of $\arg(z) = \pm \pi$; in fact, the Riemann--Hilbert analysis below can be continued onto the universal cover of $\C \setminus \{0\}$ with a suitable extension of the set $\Sigma(y_1, y_2, y_3)$. Similar observations apply to Proposition \ref{prop:un-zero-asymptotics} and Theorem~\ref{thm:D8-asymptotics-zero} below.

We illustrate Theorem \ref{thm:general} for solutions that are not single-valued near the origin in Figure \ref{fig:random}. Note that while the point $(x_1,x_2,x_3)\in\mathbb{C}^3$ on the monodromy manifold \eqref{eq:cubic-D6} only depends on the squares of $e_1$, $e_2$, the point $(y_1,y_2,y_3)\in\mathbb{C}^3$ on the monodromy manifold \eqref{eq:cubic-D8} of the limiting solution $U(z;y_1,y_2,y_3)$ of \eqref{eq:PIII-$D_8$} has a sign ambiguity in the coordinates $y_1$ and $y_2$. However, if either $e_1$ or $e_2$ changes sign, then the signs of $y_1$ and $y_2$ change together, and it turns out that the triples $(y_1,y_2,y_3)$ and $(-y_1,-y_2,y_3)$ both lie on the surface \eqref{eq:cubic-D8} together and correspond to the same solution of \eqref{eq:PIII-$D_8$}; see Remark~\ref{rem:minus-y1-y2} below. Similarly, there is no need for us to specify the sign of the square roots in \eqref{eq:y1}--\eqref{eq:y2} provided they are both taken to be the same. One might expect a similar ambiguity to arise from the replacement of $e_1^2 \mapsto e_1^{-2}$, since both are eigenvalues of the same matrix, but it turns out that $(x_1, x_2, x_3)$ is invariant under this change provided~$e_2$ is appropriately modified, and $(y_1, y_2, y_3)$ remains invariant up to the sign ambiguity described above, see Remark~\ref{remark:e2-mu-minus-mu-change} below.

\begin{figure}[t]
 \centering
\includegraphics[scale=0.8]{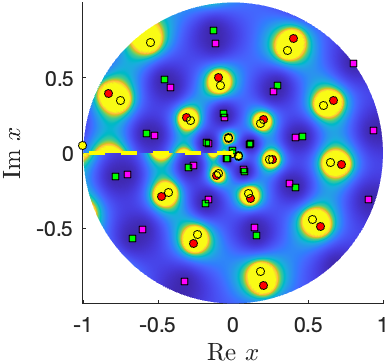}\qquad
\includegraphics[scale=0.8]{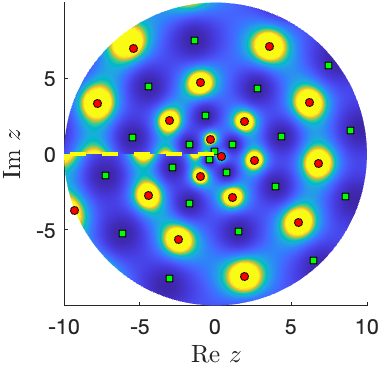}
 \caption{Left: the solution $u_{10}(x)$ of \eqref{eq:PIII-$D_6$} generated by ten iterations of \eqref{eq:Gromak-transformation} with seed $u_0(x)$ corresponding to monodromy data $\mu=0.23+0.39\ii$ (see \eqref{eq:e1-e2-mu-eta}), $e_2=-0.45-0.96\ii$ and $\alpha=40.5+0.63\ii$, $\beta=40.98+0.59\ii$. Right: the limiting solution $U(z)$ of \eqref{eq:PIII-$D_8$}. The labeling of poles and zeros is the same as in Figure \ref{fig:rational}. Note that both $u_{10}(x)$ and $U(z)$ are branched at the origin.}
 \label{fig:random}
\end{figure}

The proof of Theorem \ref{thm:general} is given in Section \ref{sec:proof}, and relies on Riemann--Hilbert analysis. The idea of the proof is to use parametrices constructed out of confluent hypergeometric functions near zero and infinity to reduce the setup to a Riemann--Hilbert Problem~\ref{rhp:Q} on the circle. After some additional transformations, the problem allows taking a large $n$ limit which gives us a~Riemann--Hilbert Problem~\ref{rhp:Rhat-even/odd} with a~jump on the circle in terms of Bessel functions. Further transformations using parametrices constructed out of Bessel functions simplify the jump and we arrive at a Riemann--Hilbert Problem~\ref{rhp:D8} for Painlev\'e-III$(D_8)$. In Section~\ref{sec:suleimanov-solution-connection}, we transform this into another Riemann--Hilbert Problem~\ref{rhp:S-hat} for \eqref{eq:PIII-$D_8$} already known in the literature. It is worth pointing out that even in the case of rational solutions, the Painlev\'e-III$(D_6)$ Riemann--Hilbert Problem~\ref{rhp:initial} exhibits Stokes phenomenon near both singular points \cite{BMS18} and hence requires the use of confluent hypergeometric parametrices to desingularize the problem before passing to the limit.

While the formul\ae\ for $y_i$ are daunting, they drastically simplify in the case of the rational solutions, where $u_0$ has monodromy data parametrized by
\begin{gather}
 \alpha = -\beta = 4m, \qquad e_0^2 = -e_\infty^2 = \ee^{\ii \pi m},\qquad
 e_1^2 = \ii, \qquad
 e_2 = \sqrt{\ee^{-2\pi \ii m}\dfrac{1 - \ii \ee^{\pi \ii m}}{1 + \ii \ee^{\pi \ii m}}},
\label{eq:monodromy-data-rational}
\end{gather}
see Section \ref{sec:rational-solutions-parameters} for details. With parameters chosen as in \eqref{eq:monodromy-data-rational}, the genericity conditions in Definition \ref{def:generic} imply $m \in \C \setminus \big(\Z + \frac{1}{2}\big)$. Then, we have
\[
y_1 = \frac{\ii \ee^{\ii\pi m}}{\sqrt{1+ \ee^{2\pi \ii m}}}, \qquad y_2 = \frac{\ii }{\sqrt{1+ \ee^{2\pi \ii m}}}, \qquad y_3 = 0.
\]
One can check that with these choices of $\alpha$, $\beta$, $e_1,$ and $e_2$ we have $U(z;y_1,y_2,y_3) = U(z;m)$, cf.\ Theorems \ref{thm:limit-of-rational-solutions} and \ref{thm:general}.

By further specializing $U(z;m)$ to $m \in \ii\R+ \Z$, we arrive at highly symmetric solutions of~PIII($D_8$) which have appeared in various works in nonlinear optics \cite{suleimanov} and as a limiting object of various families of solutions to the focusing nonlinear Schr\"odinger equation in different regimes \cite{MR4007631, BLM20}. Furthermore, these solutions can be identified with pure imaginary solutions of the radial reduction of sine-Gordon equation, see, e.g., \cite[Chapter 13]{FIKN}. It is interesting that they are related to another limiting object appearing in the random matrix theory -- the Bessel kernel determinant. The explicit relation is described in Corollary \ref{cor:bessel-relation} below.

A consequence of the analysis in Section \ref{sec:proof} below is a description of the behavior near the origin of solutions $u(x)$ of \eqref{eq:PIII-$D_6$} corresponding to generic monodromy parameters $(e_1, e_2)$.
\begin{prop}
\label{prop:un-zero-asymptotics}
 Let $u(x)$ be the solution of Painlev\'e-{\rm III}$(D_6)$ equation associated to $\mu, \eta \in \C$ via the generic monodromy parameters given in \eqref{eq:e1-e2-mu-eta} with $-\frac{1}{2}<\re(\eta)\leq\frac{1}{2}$. If $0<|{\re}(\mu)|<\frac{1}{2}$, then it holds that
 \begin{align}
 u(x) ={}& -\frac{ \Gamma (1 - 2\epsilon \mu)^2 \Gamma \left( \epsilon\mu -\frac{\alpha}{8}\right) \Gamma \big( \epsilon \mu+\frac{\beta}{8} + \frac{1}{2} \big)}{\Gamma (2 \epsilon \mu )^2 \Gamma \bigl(-\epsilon \mu-\frac{\alpha}{8}+1\bigr) \Gamma \bigl( - \epsilon \mu +\frac{\beta}{8} + \frac{1}{2} \bigr)} \nonumber\\
 &\times \Bigg(\dfrac{e_{0}^2 e_2^2 e_{\infty}^2\big(e_{0}^2-e_{1}^2\big) \big(e_{1}^2-e_{\infty}^2\big)}{\big(e_{0}^2 e_{1}^2-1\big)^2}\Bigg)^\epsilon
 x^{4 \epsilon \mu - 1} \big(1 + \mathcal{O}(x^\delta) \big)
 \label{eq:u-0-leading}
 \end{align}
 as $x\to 0$ with $|{\arg}(x)|<\pi$ where $\delta=\min(1,2-4\re(\mu))$ and $\epsilon = \sgn (\re (\mu))$.
\end{prop}
Proposition \ref{prop:un-zero-asymptotics} appeared in \cite[Theorem 3.2]{Jimbo} and its derivation is given in \cite{Kitaev87} for an equivalent, degenerate Painlev\'e-V equation. We present its proof using a Riemann--Hilbert approach in Section \ref{sec:prop-1-proof}, which follows the steps of the proof of Theorem \ref{thm:general}. The case $\re(\mu) = 0$ can be handled similarly, but we exclude it here because two distinct terms arise at the same leading order resulting in a more complicated formula. From this formula one can see that if $\re(\mu)=0$ the solution can exhibit sinusoidal oscillations with frequency diverging as $x^{-1}$ consistent with an essential singularity at the origin.

To apply Proposition \ref{prop:un-zero-asymptotics} to the rational solutions \eqref{eq:un-polys}, or more generally to the sequence of B\"acklund iterates starting from any seed solution of \eqref{eq:PIII-$D_6$}, requires knowledge of the corresponding sequence of monodromy data. This is the content of the following proposition, which we prove in Section \ref{sec:schlesinger}.
\begin{prop}\label{prop:schlesinger}
Let $u_0(x)$ be the solution of \eqref{eq:PIII-$D_6$} with parameters $(\alpha,\beta)$ and monodromy data $(\mu,\eta)$ $($see \eqref{eq:e1-e2-mu-eta}$)$ with $-\frac{1}{2}<\re (\mu), \re(\eta)\le\frac{1}{2}$. Then, the B\"acklund iterates $u_n(x)$ are parametrized by the following monodromy data
\[
e^2_{1, n} = \ee^{2\pi \ii \mu_n} ,\qquad e_{2, n} = e_2, \qquad e_{0,n}=\ee^{{\ii\pi(\alpha+4n)}/{8}}, \qquad e_{\infty,n}=\ii\ee^{-{\ii\pi(\beta+4n)}/{8}}, \]
where\footnote{To ensure $-1/2<\re(\mu_n)\leq 1/2$, we set $\epsilon = -1$ in the case where $\re(\mu) = 0$.}
\[
\mu_n = \begin{cases} \mu, & n \in 2\Z, \\ \mu - \frac{\epsilon}{2}, & n + 1 \in 2\Z, \end{cases} \qquad \text{and}\qquad \epsilon = \sgn (\re (\mu)).
\]
\end{prop}

One notable application of Propositions \ref{prop:un-zero-asymptotics} and \ref{prop:schlesinger} is the case corresponding to the rational solutions of PIII($D_6$) described above. In \cite{clarkson2018constructive}, the authors found a product formula for $u_n(0;m)$ (see \eqref{eq:u-n-zero-rational} in Section \ref{sec:Frobenius-method}). Applying Propositions~\ref{prop:un-zero-asymptotics} and \ref{prop:schlesinger} to this case yields the closed-form formula
\begin{equation}\label{eq:u-n-rat-leading}
u_n(0;m)=\frac{\Gamma\big(\frac{1}{4}-\frac{m}{2}-\frac{n}{2}\big)}{\Gamma \big(\frac{1}{4}-\frac{m}{2}+\frac{n}{2} \big)}\frac{\Gamma \big(\frac{3}{4}-\frac{m}{2}+\frac{n}{2} \big)}{\Gamma \big(\frac{3}{4}-\frac{m}{2}-\frac{n}{2} \big)}.
\end{equation}
Another observation is that the expression on the right-hand side of \eqref{eq:u-0-leading} in Proposition \ref{prop:un-zero-asymptotics} evaluated at the $n$-dependent monodromy data from Proposition \ref{prop:schlesinger} and at argument $x=\frac{z}{n}$ has a finite limit along even and odd subsequences of $n$. The limiting expressions relate to the behavior of $U(z;y_1,y_2,y_3)$, which we can take from the literature:
\begin{Theorem}[\cite{FIKN,IN,N}] \label{thm:D8-asymptotics-zero}
 Let $U(z;y_1,y_2,y_3)$ be the solution of the Painlev\'e-{\rm III}$(D_8)$ equation~\eqref{eq:PIII-$D_8$} associated to $(y_1,y_2,y_3) \in \C^3$ parametrized by generic monodromy parameters $(e_1,e_2)$ using formul\ae\ \eqref{eq:y1}--\eqref{eq:y3}. Then, it holds that
 \begin{equation*}
 U(z) = -\frac{ \Gamma (1 - 2\epsilon \mu)^2 }{\Gamma (2\epsilon \mu )^2 2^{4\varepsilon\mu-1} } z^{4 \epsilon \mu - 1} (1 + \mathcal{O}(z) )\Bigg(\dfrac{e_{0}^2 e_2^2 e_{\infty}^2 \big(e_{\infty}^2-e_{1}^2\big)}{
 \big(e_{0}^2 e_{1}^2-1\big)}\Bigg)^{\epsilon}
 \end{equation*}
 as $z\to 0$ with $|{\arg}(z)|<\pi$.
\end{Theorem}
We pause to note that coalescence between Painlev\'e equations has long been in the literature; a coalescence diagram of all six Painlev\'e equations already appeared in Okamoto's work \cite{O1}, and was later expanded on in \cite{OO}. Later, a geometric interpretation of the coalescence was given in \cite{CMR}. That being said, the above degenerations are carried out on the level of the differential equation, so that given a solution of a Painlev\'e equation, one does not have a characterization of the solution one arrives at under the coalescence procedure. Confluence on the level of the solutions of the differential equation has also appeared in the literature; one of the most interesting examples is the merging of regular singularities and corresponding creation of an irregular singularity. This process was studied in the works \cite{Glutsuk,Kitaev06}. In the PhD thesis \cite{Horrobin} the confluence was studied in more detail in the cases Painlev\'e VI $\to$ Painlev\'e V and Painlev\'e V $\to$ Painlev\'e III$(D_6)$. In the works \cite{Kapaev94,Kapaev_Kitaev} the authors considered a transition from Painlev\'e II $\to$ Painlev\'e I that is different in nature.

\subsection{Overview of the paper} In Section \ref{sec:Frobenius-method}, we describe the coalescence map $u \mapsto U$ in terms of initial conditions and prove Theorem \ref{thm:limit-of-rational-solutions} using Maclaurin series of these solutions. We apply it to Umemura polynomials in Section \ref{sec:Umemura}. In Sections \ref{sec:monodromy-rep-$D_6$} and \ref{sec:monodromy-rep-$D_8$}, we describe the monodromy representations of PIII($D_6$) and PIII($D_8$), respectively. In Section~\ref{sec:schlesinger}, we explain the Schlesinger transformations underlying Gromak's B\"acklund transformation \eqref{eq:Gromak-transformation} and prove Proposition~\ref{prop:schlesinger}. In Section \ref{sec:proof}, we~prove Theorem~\ref{thm:general} by Riemann--Hilbert methods. We recycle the same methodology to prove Proposition~\ref{prop:un-zero-asymptotics} in Section \ref{sec:prop-1-proof}. In Section \ref{sec:suleimanov-solution-connection}, we perform a Fabry-type\footnote{Named after Eug\`ene Fabry for his work in \cite{fabry1885}, see also \cite[Chapter 17.53]{MR0010757}.} transformation to the Painlev\'e-III$(D_8)$ Riemann--Hilbert problem naturally arising from our limit process to put it in more canonical form and justify its solvability.

\section[Identifying the solution of the limiting Painlev\'e-III(D\_8) equation using Maclaurin series]{Identifying the solution of the limiting Painlev\'e-III($\boldsymbol{D_8}$)\\ equation using Maclaurin series}
\label{sec:Frobenius-method}

The Painlev\'e-III($D_6$) equation \eqref{eq:PIII-$D_6$} for $u_n(x;\alpha, \beta)$ implies the following equivalent differential equation for $U_n(z;\alpha, \beta):=u_n(z/n;\alpha, \beta)$:
\begin{equation}
U_n''=\frac{(U_n')^2}{U_n}-\frac{U_n'}{z} +\frac{\alpha_nU_n^2}{z}+\frac{\beta_n}{z}+\gamma_nU_n^3+\frac{\delta_n}{U_n},
\label{eq:v-ODE}
\end{equation}
where
\begin{equation}
\alpha_n:=4+\frac{\alpha}{n},\qquad\beta_n:=4 + \frac{\beta}{n},\qquad\gamma_n:=\frac{4}{n^2},\qquad\delta_n:=-\frac{4}{n^2}.
\label{eq:$D_6$-parameters}
\end{equation}
Note that for arbitrary $\alpha\in\mathbb{C}$ and $\beta\in\mathbb{C}$ fixed and $n>0$ sufficiently large we have the following crude inequalities:
\begin{equation}
 |\alpha_n|\le 5,\qquad|\beta_n|\le 5,\qquad |\gamma_n|\le 1,\qquad |\delta_n|\le 1.
 \label{eq:$D_6$-parameters-inequalities}
\end{equation}
We construct solutions of \eqref{eq:v-ODE} analytic at $z=0$ as follows. First multiply \eqref{eq:v-ODE} through by~$zU_n(z)$ to obtain
\begin{equation}
 -zU_nU_n'' + z(U_n')^2-U_nU_n'+\alpha_nU_n^3+\beta_nU_n+\gamma_nzU_n^4 + \delta_nz=0.
 \label{eq:v-ODE-II}
\end{equation}
We substitute into \eqref{eq:v-ODE-II} a power series
\begin{equation}
U_n(z)=\sum_{k=0}^\infty \upsilon_kz^k,
\label{eq:v-series}
\end{equation}
and express all products through the Cauchy product formula. The left-hand side of \eqref{eq:v-ODE-II} is then a formal power series in $z$, and assuming that $\upsilon_0\neq 0$, the coefficient of $z^0$ yields
\begin{equation}
\upsilon_1=\beta_n+\alpha_n\upsilon_0^2,
\label{eq:c1-exact}
\end{equation}
the coefficient of $z^1$ yields
\begin{gather}
 \upsilon_2 = \frac{1}{4\upsilon_0}\big[3\alpha_n\upsilon_0^2\upsilon_1+\beta_n\upsilon_1+\gamma_n\upsilon_0^4 + \delta_n\big],
\end{gather}
and for $k\ge 2$, the coefficient of $z^k$ yields
\begin{align}
 & \upsilon_{k+1}= \frac{1}{\upsilon_0(k+1)^2}\Bigg[\sum_{a=0}^ka(k+1-2a)\upsilon_a\upsilon_{k+1-a} +
 \alpha_n\sum_{a=0}^{k}\sum_{b=0}^{k-a}\upsilon_a\upsilon_b \upsilon_{k-a-b}\nonumber\\
&\hphantom{\upsilon_{k+1}=}{}
+\beta_n\upsilon_k +
 \gamma_n\sum_{a=0}^{k-1}\sum_{b=0}^{k-1-a}\sum_{c=0}^{k-1-a-b}\upsilon_a\upsilon_b \upsilon_c \upsilon_{k-1-a-b-c}\Bigg],\qquad k\ge 2.
 \label{eq:xi-general-exact}
\end{align}
We may omit the term with $a=0$ from the first sum on the right-hand side.
Using
\[
k+1\ge 1\qquad\text{and}\qquad \frac{a|k+1-2a|}{(k+1)^2}\le 1\qquad\text{for}\quad a=0,\dots,k,
\]
along with the inequalities \eqref{eq:$D_6$-parameters-inequalities}, the coefficients in the series \eqref{eq:v-series} are subject to the inequalities
\begin{gather}
 |\upsilon_1|\le 5\big(1+|\upsilon_0|^2\big),\qquad
 |\upsilon_2|\le \frac{1}{4|\upsilon_0|}\big[15|\upsilon_0|^2|\upsilon_1| + 5|\upsilon_1|+|\upsilon_0|^4 + 1\big],\nonumber\\
 |\upsilon_{k+1}|\le \frac{1}{|\upsilon_0|}\Bigg[\sum_{a=1}^k|\upsilon_a||\upsilon_{k+1-a}| +
 5\sum_{a=0}^k\sum_{b=0}^{k-a}|\upsilon_a||\upsilon_b||\upsilon_{k-a-b}|\nonumber\\
 \phantom{|\upsilon_{k+1}|\le}{}+5|\upsilon_k| + \sum_{a=0}^{k-1}\sum_{b=0}^{k-1-a}\sum_{c=0}^{k-1-a-b}
 |\upsilon_a||\upsilon_b||\upsilon_c||\upsilon_{k-1-a-b-c}|\Bigg],\qquad k\ge 2.
 \label{eq:xi-inequalities}
\end{gather}
Now we define a sequence of positive numbers $\{\Upsilon_k\}_{k=0}^\infty$ by taking $\Upsilon_0>0$ arbitrary and setting
\begin{gather}
 \Upsilon_1= 5\big(1+\Upsilon_0^2\big), \label{eq:V-coeffsU1}\\
 \Upsilon_2=\frac{1}{4\Upsilon_0^2}\big[15\Upsilon_0^2\Upsilon_1+5\Upsilon_1 + \Upsilon_0^4+1\big]=\frac{1}{4\Upsilon_0^2}\big[76\Upsilon_0^4 + 100\Upsilon_0^2 + 26\big], \label{eq:V-coeffsU2}\\
 \Upsilon_{k+1}=\frac{1}{\Upsilon_0}\Bigg[\sum_{a=1}^k\Upsilon_a\Upsilon_{k+1-a} + 5\sum_{a=0}^k\sum_{b=0}^{k-a}\Upsilon_a\Upsilon_b\Upsilon_{k-a-b}\nonumber\\
 \phantom{\Upsilon_{k+1}=}{}+5\Upsilon_k + \sum_{a=0}^{k-1}\sum_{b=0}^{k-1-a}\sum_{c=0}^{k-1-a-b}\Upsilon_a\Upsilon_b\Upsilon_c\Upsilon_{k-1-a-b-c}\Bigg],\qquad k\ge 2.
 \label{eq:V-coeffs}
\end{gather}
Following \cite[Proposition 1.1.1, p.\ 261]{Iwasaki_Kimura_Shimomura}, we construct an \emph{algebraic} equation formally satisfied by the power series
\begin{equation}
\mathcal{U}(z)=\sum_{k=0}^\infty\Upsilon_kz^k.
\label{eq:V-series}
\end{equation}
We first rewrite the generic $k\ge 2$ equation in \eqref{eq:V-coeffs} in the equivalent form
\begin{gather}
-3\Upsilon_0\Upsilon_{k+1} + \sum_{a=0}^{k+1}\Upsilon_a\Upsilon_{k+1-a} +
5\sum_{a=0}^k\sum_{b=0}^{k-a}\Upsilon_a\Upsilon_b\Upsilon_{k-a-b}\nonumber\\
\phantom{-3\Upsilon_0\Upsilon_{k+1} }{}+ 5\Upsilon_k +\sum_{a=0}^{k-1}\sum_{b=0}^{k-1-a}\sum_{c=0}^{k-1-a-b}\Upsilon_a\Upsilon_b\Upsilon_c\Upsilon_{k-1-a-b-c}=0,\qquad k\ge 2.
\label{eq:V-coeff-kge2}
\end{gather}
Comparing with \eqref{eq:V-series}, this is the coefficient of $z^k$ in the power series expansion about $z=0$ of the equation
\[
-\frac{3\Upsilon_0}{z}\mathcal{U} +\frac{1}{z}\mathcal{U}^2 +5\mathcal{U}^3+ 5\mathcal{U} +z\mathcal{U}^4=0.
\]
More generally, since $k\ge 2$ holds in \eqref{eq:V-coeff-kge2}, these relations are consistent also with the equation
\begin{gather}
-\frac{3\Upsilon_0}{z}\mathcal{U} +\frac{1}{z}\mathcal{U}^2 +5\mathcal{U}^3+ 5\mathcal{U} +z\mathcal{U}^4= \frac{A}{z}+B+Cz.
\label{eq:generic-V-algebraic-equation}
\end{gather}
We now pick the constants $A$, $B$, $C$ so that \eqref{eq:generic-V-algebraic-equation} is also consistent with \eqref{eq:V-coeffsU1}--\eqref{eq:V-coeffsU2} and $\mathcal{U}(0)=\Upsilon_0$ in the series \eqref{eq:V-series}. Indeed, $\mathcal{U}(0)=\Upsilon_0$ is equivalent to the following equation obtained from the coefficient of $z^{-1}$ in \eqref{eq:generic-V-algebraic-equation}:
\[
-3\Upsilon_0^2 + \Upsilon_0^2 = A\implies A=-2\Upsilon_0^2.
\]
Then taking $\Upsilon_1$ from \eqref{eq:V-coeffsU1}, the constant term in \eqref{eq:generic-V-algebraic-equation} gives the equation
\[
-3\Upsilon_0\Upsilon_1 + 2\Upsilon_0\Upsilon_1 + 5\Upsilon_0^3+5\Upsilon_0 =B\implies B=
-5\Upsilon_0(1+\Upsilon_0^2) + 5\Upsilon_0^3+5\Upsilon_0=0.
\]
Finally, obtaining also $\Upsilon_2$ from \eqref{eq:V-coeffsU2}, the coefficient of $z^1$ in \eqref{eq:generic-V-algebraic-equation} is
\begin{gather*}
-3\Upsilon_0\Upsilon_2 + 2\Upsilon_0\Upsilon_2+\Upsilon_1^2 + 15\Upsilon_0^2\Upsilon_1 + 5\Upsilon_1 + \Upsilon_0^4 = C\implies\\
C=-\Upsilon_0\Upsilon_2 +\Upsilon_1^2+15\Upsilon_0^2\Upsilon_1+5\Upsilon_1+\Upsilon_0^4\\
\phantom{C}{}=-\frac{1}{4}\big[76\Upsilon_0^4+100\Upsilon_0^2+26\big]+25\big(1+\Upsilon_0^2\big)^2 + 75\Upsilon_0^2\big(1+\Upsilon_0^2\big) +25\big(1+\Upsilon_0^2\big) + \Upsilon_0^4\\
\phantom{C}{}=82\Upsilon_0^4 +125\Upsilon_0^2 + \frac{87}{2}.
\end{gather*}
The formal series \eqref{eq:V-series} with the recurrence relations \eqref{eq:V-coeffsU1}--\eqref{eq:V-coeffs} is therefore consistent with the algebraic equation (rewriting \eqref{eq:generic-V-algebraic-equation} with the above expressions for $A$, $B$, $C$):
\begin{equation}
 -3\Upsilon_0 \mathcal{U} +\mathcal{U}^2+2\Upsilon_0^2 = z\bigg[\left(82\Upsilon_0^4 + 125\Upsilon_0^2 +\frac{87}{2}\right)z-5\mathcal{U}^3 - 5\mathcal{U} - z\mathcal{U}^4\bigg].
 \label{eq:V-algebraic-equation}
\end{equation}
However, it is a straightforward application of the implicit function theorem to observe that equa\-tion~\eqref{eq:V-algebraic-equation} has a unique solution $\mathcal{U}=\mathcal{U}(z)$ analytic at $z=0$ with $\mathcal{U}(0)=\Upsilon_0>0$ (this condition guarantees that the root $\mathcal{U}=\Upsilon_0$ of the quadratic on the left-hand side of \eqref{eq:V-algebraic-equation} is simple). This proves that the formal series \eqref{eq:V-series} with coefficients determined from \eqref{eq:V-coeffsU1}--\eqref{eq:V-coeffs} has a~positive radius of convergence for each given value $\Upsilon_0> 0$.
\begin{Theorem}\label{thm:$D_6$-$D_8$-series}
Fix $\alpha\in\mathbb{C}$ and $\beta\in\mathbb{C}$ and let $\{U_n(z;\alpha,\beta)\}_{n=1}^\infty$ be a sequence of solutions of~\eqref{eq:v-ODE} that are analytic at the origin $z=0$ and suppose that
\[\lim_{n\to\infty}U_n(0;\alpha,\beta)=\upsilon_{\infty,0}=\upsilon_{\infty,0}(\alpha,\beta)\neq 0.\]
 Then there exists a radius $\rho>0$ such that for all $n$ sufficiently large $U_n(z;\alpha,\beta)$ is analytic for $|z|<\rho$ and such that $U_n(z;\alpha,\beta)\to U_\infty(z;\alpha,\beta)$ as $n\to\infty$ uniformly for $|z|<\rho$,
where~${U(z)=U_\infty(z;\alpha,\beta)}$ is the unique solution of the Painlev\'e-{\rm III}$(D_8)$ equation \eqref{eq:PIII-$D_8$}
that is analytic at the origin with $U_\infty(0;\alpha,\beta)=\upsilon_{\infty,0}$.
\end{Theorem}
\begin{proof}
Let $\{\upsilon_{n,k}\}_{k=0}^\infty$ denote the power series coefficients of $U_n(z;\alpha,\beta)$ as in \eqref{eq:v-series}. Define~$\Upsilon_0$ by~$\Upsilon_0>2|\upsilon_{\infty,0}|$ (say), and obtain the subsequent coefficients $\{\Upsilon_k\}_{k=1}^\infty$ via \eqref{eq:V-coeffsU1}--\eqref{eq:V-coeffs}. Comparing \eqref{eq:xi-inequalities}--\eqref{eq:V-coeffs} then shows that for all $n$ sufficiently large,
$|\upsilon_{n,k}|\le \Upsilon_k$ holds for all \linebreak ${k=0,1,2,\dots}$. For each fixed $k=0,1,2,\dots$, the recurrence relations \eqref{eq:c1-exact}--\eqref{eq:xi-general-exact} together with the limit ${\upsilon_{n,0}\to\upsilon_{\infty,0}}$ show that $\upsilon_{n,k}$ tends to a limiting value $\upsilon_{\infty,k}$ as $n\to\infty$, with $|\upsilon_{\infty,k}|\le\Upsilon_k$. The convergence of $U_n(z;\alpha,\beta)$ to a limiting analytic function $U_\infty(z;\alpha,\beta)$ with $U_\infty(0;\alpha,\beta)=\upsilon_{\infty,0}(\alpha,\beta)$ then follows by dominated convergence. That the limiting analytic function $U_\infty(z;\alpha,\beta)$ is a~solution of~\eqref{eq:PIII-$D_8$} follows from passing to the limit in each term of \eqref{eq:v-ODE} using \eqref{eq:$D_6$-parameters}. That this solution is the unique analytic solution of \eqref{eq:PIII-$D_8$} with the specified value at $z=0$ then follows from passing to the limit in the recurrence relations \eqref{eq:c1-exact}--\eqref{eq:xi-general-exact}.
\end{proof}

Now we apply this result to the rational solutions $u_n(x;m)$ of equation~\eqref{eq:PIII-$D_6$}, corresponding to ${\alpha = -\beta= 4m}$. To this end, we point out that in \cite{clarkson2018constructive}, the authors studied the Umemura polynomials $s_n(x;m)$ at $x=0$, and we begin by recalling one of their results.
\begin{Lemma}[\cite{clarkson2018constructive}]\label{lem:umemura-at-zero}
Set $y:=m+\tfrac{1}{2}$ and write $\phi_n(y):=s_n(0;m)$. If $n=2k$ is even, then
\begin{equation*}
\phi_{2k}(y)=y^k\big(y^2-1\big)^k\prod_{j=1}^{k-1}\big(y^2-(2j)^2\big)^{k-j}\big(y^2-(2j+1)^2\big)^{k-j},\qquad k=1,2,3,\dots,
\end{equation*}
while if $n=2k-1$ is odd, then
\begin{equation}
\phi_{2k-1}(y)=\phi_{2k}(y)\prod_{j=1}^{k}\big(y^2-(2j-1)^2\big)^{-1},\qquad k=1,2,3,\dots.
\label{eq:phi-shift-down}
\end{equation}
\end{Lemma}

It follows from two identities above that also
\begin{equation}
\phi_{2k+1}(y)=\phi_{2k}(y)\cdot y\cdot \prod_{j=1}^{k}\big(y^2-(2j)^2\big),\qquad k=1,2,3,\dots.
\label{eq:phi-shift-up}
\end{equation}
Using \eqref{eq:un-polys} and $s_n(0;m-1)=\phi_n\big(m-\frac{1}{2}\big)$ one has
\begin{equation}
\label{eq:u-n-zero-rational}
 u_n(0;m)=\frac{\phi_n\big(m-\tfrac{1}{2}\big)\phi_{n-1}\big(m+\tfrac{1}{2}\big)}{\phi_n\big(m+\tfrac{1}{2}\big)\phi_{n-1}\big(m-\tfrac{1}{2}\big)}.
\end{equation}
Therefore, from \eqref{eq:phi-shift-down} one gets that
\begin{gather}
u_{2k}(0;m)=\prod_{j=1}^k\frac{\big(m-\tfrac{1}{2}\big)^2-(2j-1)^2}{\big(m+\tfrac{1}{2}\big)^2-(2j-1)^2} = \frac{\displaystyle \prod_{j=1}^k \Bigg(1-\bigg(\frac{m-\tfrac{1}{2}}{2j-1}\bigg)^2\Bigg)}{\displaystyle\prod_{j=1}^k\Bigg(1-\bigg(\frac{m+\tfrac{1}{2}}{2j-1}\bigg)^2\Bigg)}.
\label{eq:u-at-zero-even}
\end{gather}
Similarly, from \eqref{eq:phi-shift-up} one gets
\begin{gather}
u_{2k+1}(0;m)=\frac{m-\tfrac{1}{2}}{m+\tfrac{1}{2}}\prod_{j=1}^k\frac{\big(m-\tfrac{1}{2}\big)^2-(2j)^2}{\big(m+\tfrac{1}{2}\big)^2-(2j)^2} = \frac{m-\tfrac{1}{2}}{m+\tfrac{1}{2}}\cdot\frac{\displaystyle\prod_{j=1}^k\Bigg(1-\bigg(\frac{m-\tfrac{1}{2}}{2j}\bigg)^2\Bigg)}{\displaystyle\prod_{j=1}^k\Bigg(1-\bigg(\frac{m+\tfrac{1}{2}}{2j}\bigg)^2\Bigg)}.
\label{eq:u-at-zero-odd}
\end{gather}
Using the infinite product formul\ae\ (see \cite[equations~(4.22.1)--(4.22.2)]{DLMF})
\[
\sin(x)=x\prod_{j=1}^\infty\left(1-\frac{x^2}{\pi^2j^2}\right),\qquad \cos(x)=\prod_{j=1}^\infty\left(1-\frac{4x^2}{\pi^2(2j-1)^2}\right),
\]
we get the following result.
\begin{Lemma} \label{cor:1} Assume that $m\in\mathbb{C}\setminus\big(\mathbb{Z}+\frac{1}{2}\big)$. Then
\begin{gather*}
\lim_{k\to\infty}u_{2k}(0;m)=
\tan\bigg(\frac{\pi}{2}\left(m+\frac{1}{2}\right)\bigg),\qquad
\lim_{k\to\infty}u_{2k+1}(0;m)=
-\cot\bigg(\frac{\pi}{2}\left(m+\frac{1}{2}\right)\bigg).
\end{gather*}
\end{Lemma}
We then apply Theorem~\ref{thm:$D_6$-$D_8$-series} and obtain the following Corollary, which completes the local proof of Theorem~\ref{thm:limit-of-rational-solutions}.
\begin{Corollary}\label{cor:local-limit-of-rational-solutions}
Let $m\in\mathbb{C}\setminus\big(\mathbb{Z}+\tfrac{1}{2}\big)$, and let $U=U(z;m)$ denote the unique solution of the Painlev\'e-III$(D_8)$ equation \eqref{eq:PIII-$D_8$} that is analytic at the origin with $U(0;m)=\tan\big(\frac{\pi}{2}\big(m+\frac{1}{2}\big)\big)$. Then for all $z$ in a neighborhood of the origin, we have
\[
\lim_{k\to\infty} u_{2k}\left(\dfrac{z}{2k};m\right)=U(z;m),
\qquad
\lim_{k\to\infty} u_{2k+1}\left(\dfrac{z}{2k+1};m\right)=-\frac{1}{U(z;m)}.
\]
\end{Corollary}

Note that the equation \eqref{eq:PIII-$D_8$} is invariant under the $\mathbb{Z}_2$-B\"acklund transformation $U(z)\mapsto -U(z)^{-1}$, so the even/odd subsequences of rational solutions both tend to related solutions of the same equation.

\section{Asymptotic behavior of Umemura polynomials}
\label{sec:Umemura}
In this section, we obtain asymptotic results about the Umemura polynomials $s_n(x;m)$ and, as a consequence, particular $2j - k$ determinants (see \eqref{eq:lag} below).

\subsection[Painlev\'e-III tau functions, the Toda lattice, and expressing s\_n(x;m) in terms of u\_\{n+1\}(x;m)]{Painlev\'e-III tau functions, the Toda lattice,\\ and expressing $\boldsymbol{s_n(x;m)}$ in terms of $\boldsymbol{u_{n+1}(x;m)}$}

As a first step, we would like to obtain an expression of the Umemura polynomials $s_n(x;m)$ in terms of the rational Painlev\'e-III solutions themselves. We follow closely the works \cite{FW03,O87}. We introduce the Hamiltonian $H_n \equiv H_n(x;m)$ via the equation
\begin{gather}
 xH_n(x;m) = 2p_n(x;m)^2u_n(x;m)^2+p_n(x;m)\big(2x-2xu_n(x;m)^2 \nonumber\\
\hphantom{xH_n(x;m) =}{}
+(1+2m-2n)u_n(x;m)\big) -(2m+1)xu_n(x;m),
 \label{eq:$D_6$-Hamiltonian}
\end{gather}
where the momentum, $p_n \equiv p_n(x;m)$, is given by
\begin{equation}
 p_n=\frac{x}{4u_n^2}\dod{u_n}{x}+\frac{x}{2}-\frac{x}{2u_n^2}-\frac{2m-2n+1}{4u_n}.
 \label{eq:Hamiltonian-v}
\end{equation}
In other words, the canonical system
\[
\frac{\dd u_n}{\dd x}=\frac{\partial H_n}{\partial p_n}\qquad\text{and}\qquad\frac{\dd p_n}{\dd x}=-\frac{\partial H_n}{\partial u_n}
\]
is equivalent to the definition \eqref{eq:Hamiltonian-v} of $p_n$ and the PIII($D_6$) equation \eqref{eq:PIII-$D_6$} for $u=u_n(x;m)$ where $\alpha=4(n+m)$ and $\beta=4(n-m)$.

The tau function $\tau_n(x;m)$ can be defined up to a constant of integration by the relation
\begin{equation}
\dod{}{x}\ln(\tau_n(x;m))=H_n(x;m)+\frac{1}{x}u_n(x;m)p_n(x;m).
\label{eq:tau-function-log-derivative}
\end{equation}
We would like to fix the constant in this definition by choosing a path of integration going to~$x=\infty$ in the sector $|{\arg}(x)|<\pi$. To this end, it was shown in \cite{BMS18} that the rational functions~$u_n(x;m)$ behave at infinity as $u_n(x;m)=1+\mathcal{O}\big(x^{-1}\big)$. In fact, using this in the Painlev\'e-III$(D_6)$ equation \eqref{eq:PIII-$D_6$} with $\alpha=4(n+m)$ and $\beta=4(n-m)$ gives the more refined asymptotics
\begin{equation*}
 u_n(x;m)=1-\frac{n}{2x}+\frac{n(2m+n)}{8x^2} -\frac{n\big(4m^2+4mn+1\big)}{32x^3}+\mathcal{O}\big(x^{-4}\big) \qquad \text{as} \quad x\to \infty,
\end{equation*}
which, together with \eqref{eq:Hamiltonian-v} implies that the right-hand side of \eqref{eq:tau-function-log-derivative} satisfies
\begin{gather}
H_n(x;m)+\frac{1}{x}u_n(x;m)p_n(x;m) \nonumber\\
\qquad=-2 m-1-\frac{(2 m+1) (2 m-4 n+3)}{8 x}+\frac{(1+2m)(1-n)n}{8x^2}\nonumber\\
\phantom{\qquad=}{}+\frac{(1+2m)^2(n-1)n}{32x^3}+\mathcal{O}\big(x^{-4}\big) \qquad \text{as} \quad x\to \infty.
\label{eq:tau-function-log-derivative-expansion}
\end{gather}
Now, every pole $x_0\neq 0$ of $u_n(x;m)$ is simple with residue $\pm\frac{1}{2}$, and moreover
directly from \eqref{eq:PIII-$D_6$}, we find that
\[
u_n(x;m)=\pm\frac{1}{2}(x-x_0)^{-1} -\frac{1}{2x_0}\left(n+m\pm\frac{1}{2}\right)+\mathcal{O}(x-x_0),\qquad x\to x_0.
\]
Similarly, all zeros $x_0\neq 0$ of $u_n(x;m)$ are simple, with $u_n'(x_0;m)=\pm 2$, and again from \eqref{eq:PIII-$D_6$} we have
\[
u_n(x;m)=\pm 2(x-x_0)+\frac{2}{x_0}\lb \pm \frac{1}{2} +m-n\rb (x-x_0)^2 + \mathcal{O}\big((x-x_0)^3\big),\qquad x\to x_0.
\]
These expansions can be differentiated with respect to $x$ to obtain corresponding expansions of~$p_n(x;m)$ via \eqref{eq:Hamiltonian-v} and then of~$H_n(x;m)$ via \eqref{eq:$D_6$-Hamiltonian}. These expansions show that the only possible singularities $x_0\neq 0$ of the right-hand side of \eqref{eq:tau-function-log-derivative} are simple poles of residue $1$ that occur at simple zeros of $u_n(x;m)$ with $u_n'(x_0;m)=-2$. Furthermore, if $m\not\in\mathbb{Z}+\frac{1}{2}$, then $u_n(x;m)$ is analytic and nonzero at $x=0$, and it follows that the right-hand side of \eqref{eq:tau-function-log-derivative} has a simple pole at the origin with residue $-\frac{1}{8}\big(4(m-n+1)^2-1\big)$.
Therefore, arbitrarily fixing an integration constant, the tau function $\tau_n(x;m)$ then can be defined for $m\not\in\mathbb{Z}+\frac{1}{2}$ and $|{\arg}(x)|<\pi$ by
\begin{align}
\tau_n(x;m)={}&\ee^{- (2 m+1)x}x^{-\frac{(2 m+1) (2 m-4 n+3)}{8 }}\nonumber\\
&\times \exp\biggl(-\int_x^\infty\biggl( H_n(y;m)+\frac{u_n(y;m)p_n(y;m)}{y}\nonumber\\
&\phantom{\times}{}+2m+1+\frac{(2 m+1) (2 m-4 n+3)}{8 y}\biggr) \dd y\biggr),
\label{eq:tau-def-infinity}
\end{align}
where the power function denotes the principal branch, the path of integration lies in the sector~$|{\arg}(y)|<\pi$ avoiding all poles of the meromorphic integrand, and then the integral is independent of path modulo $2\pi\ii$. It then follows from \eqref{eq:tau-function-log-derivative-expansion} that $\tau_n(x;m)$ admits the expansion
\begin{gather}
\tau_n(x;m)=\ee^{- (2 m+1)x}x^{-\frac{(2 m+1) (2 m-4 n+3)}{8 }}\nonumber\\
 \phantom{\tau_n(x;m)=}{}\times \left(1+\frac{(2m+1)(n-1)n}{8x}+\frac{(2m+1)^2\big(n^2-1\big)(n-2)n}{128x^2}+\mathcal{O}\big(x^{-3}\big)\right),\nonumber \\
 \text{as} \quad x\to \infty,\qquad|{\arg}(x)|<\pi,\label{eq:tau_asym}
\end{gather}
and that $\tau_n(x;m)x^{(4(m-n+1)^2-1)/8}$ extends to a neighborhood of $x=0$ as an analytic nonvanishing function.
From the point of view of the function $\tau_n(x;m)$ the recurrence \eqref{eq:umemura-recurrence}, which defines the Umemura polynomials, is equivalent to the Toda equation. More precisely, if we define the function \begin{equation}\label{eq:hn-def}
h_n(x;m)=H_n(x;m)+\frac{u_n(x;m)p_n(x;m)}{x}-2x+\frac{n^2}{x},\end{equation}
then using Gromak's B\"acklund transformation \eqref{eq:Gromak-transformation} with $u=u_n(x;m)$, $\hat{u}=u_{n+1}(x;m)$, and~${\alpha=4(n+m)}$, $\beta=4(n-m)$, we can check that $h_n$ satisfies the identity
\begin{equation}\label{eq:first-hn}
 h_{n+1}(x;m)-h_n(x;m)=-\frac{2u_n(x;m)p_n(x;m)}{x}+\frac{2n+1}{x}.
\end{equation}
Similarly, using the inverse of Gromak's transformation \eqref{eq:Gromak-transformation}:
\begin{equation*}
 \hat{u}(x) \mapsto u(x) =\dfrac{2x\hat{u}'(x) -4x\hat{u}(x)^2 -4x +(\beta+4) \hat{u}(x) - 2\hat{u}(x)}{\hat{u}(x) \cdot \big( 2x\hat{u}'(x) - 4x \hat{u}(x)^2 -4x - (\alpha+4) \hat{u}(x) + 2\hat{u}(x)\big)},
\end{equation*}
in which $u(x)$ solves \eqref{eq:PIII-$D_6$} and $\hat{u}(x)$ solves the same equation with parameters $(\alpha,\beta)$ replaced by $(\alpha+4,\beta+4)$, one can check the identity
 \begin{equation} \label{eq:second-hn}
h_{n-1}(x;m)-h_n(x;m)=
-\frac{2u_n(x;m)p_n(x;m)}{x}-\frac{2m+1}{x}+\frac{1-2n}{x-p_n(x;m)}.
 \end{equation}
 Combining \eqref{eq:first-hn} and \eqref{eq:second-hn}, we get
 \begin{gather*}
h_{n+1}(x;m)+h_{n-1}(x;m)-2h_{n}(x;m)=
-\frac{4u_n(x;m)p_n(x;m)}{x}+\frac{2(n-m)}{x}+\frac{1-2n}{x\!-\!p_n(x;m)}.
 \end{gather*}
Differentiating \eqref{eq:hn-def}, we can notice that
\begin{equation*} \dod{}{x}\ln\left(x \dod{}{x}\left( xh_n(x;m)\right)\right)={h_{n+1}(x;m)+h_{n-1}(x;m)-2h_n(x;m)}. \end{equation*}
Given any $K_n(m) \in \C$, if we now define the function
\begin{equation*}\hat{\tau}_n(x;m)= K_n(m)\ee^{-x^2}x^{n^2}{\tau}_n(x;m), \end{equation*}
then $h_n(x;m)=\dod{}{x}\ln \hat{\tau}_n(x;m)$ and we see that
$\hat{\tau}_n(x;m)$ satisfies the Toda equation
 \begin{equation}\label{eq:toda}
 x \dod{}{x} x\dod{}{x}\ln\hat{\tau}_n(x;m)=C_n(m)\frac{\hat{\tau}_{n+1}(x;m)\hat{\tau}_{n-1}(x;m)}{\hat{\tau}_n(x;m)^2} \end{equation}
 with some constants $\{C_n(m)\}_{n=0}^\infty$ depending on the $\{K_n(m)\}_{n=0}^\infty$. We now choose the constants~$K_n(m)$ so as to have $C_n(m) = 1$. To this end, using the detailed asymptotics \eqref{eq:tau_asym}, one can check that the leading term of both sides of \eqref{eq:toda} as $x\to\infty$ is proportional to $x^2$ and equating those coefficients under the assumption that $C_n(m)=1$ yields the equation
 \[
 -4 = \dfrac{K_{n + 1}(m)K_{n-1}(m)}{K_n(m)^2},
 \]
 of which we choose a particular solution $K_n(m) = (2\ii)^{n^2}$, which yields the expression
 \begin{gather}
 \hat{\tau}_n(x;m)= (2\ii)^{n^2}\ee^{-x^2}x^{n^2}{\tau}_n(x;m).
 \label{eq:hattau-tau}
 \end{gather}
 Now if we put
\[
s_n(x;m)=\ii^{-{(n+1)^2}{}}2^{-\frac{(n+1)(n+2)}{2}} \ee^{(2 m+1)x+x^2} x^{\frac{4m^2-4n^2-8mn-16n-9}{8}}\hat{\tau}_{n+1}(x;m),
\]
then it follows from \eqref{eq:toda} with $C_n(m)=1$ that $s_n(x;m)$ satisfies the Umemura recurrence relation \eqref{eq:umemura-recurrence}. Moreover, using
\[
u_0(x;m)=1 \qquad\text{and}\qquad u_1(x;m)=\frac{8x+4m-2}{8x+4m+2}
\]
shows that the integrand in the exponent of $\tau_0(x;m)$ and $\tau_1(x;m)$ vanishes identically, from which it follows that the initial conditions \eqref{eq:umemura-recurrence_init} are satisfied as well. Since the recurrence relation and initial conditions together have a unique solution,
using \eqref{eq:tau-def-infinity} and \eqref{eq:hattau-tau}, the Umemura polynomials are given by
\begin{align}
&s_n(x;m)= 2^{\frac{n(n+1)}{2}} \ee^{ (2 m+1)x} x^{\frac{4(m-n)^2-1}{8}}{\tau}_{n+1}(x;m)\nonumber\\
 &\hphantom{s_n(x;m)}{}=(2x)^{\frac{n(n+1)}{2}} \exp\biggr(-\int_x^\infty\bigg(H_{n+1}(y;m)+\frac{u_{n+1}(y;m)p_{n+1}(y;m)}{y}\nonumber\\
&\hphantom{s_n(x;m)=}{}+2m+1+\frac{(2m+1)(2m-4n-1)}{8y}\bigg) \dd y\biggl).
 \label{eq:umemura-tau}
\end{align}
This formula achieves the goal of explicitly expressing $s_n(x;m)$ in terms of $u_{n+1}(x;m)$.
Since
\begin{align*}
\begin{split}
\frac{\dd}{\dd x}\ln(\tau_n(x;m))={}&H_n(x;m)+\frac{1}{x}u_n(x;m)p_n(x;m) \\
={}&-\frac{4(m-n+1)^2-1}{8 x}+\mathcal{O}(1) \qquad \text{as} \quad x\to 0,
\end{split}
\end{align*}
we can also use analyticity of $s_n(x;m)$ at the origin and the first line of \eqref{eq:umemura-tau} to write the alternative expression
\begin{align}
 &s_n(x;m)=s_n(0;m) \ee^{ (2 m+1)x}\exp\bigg(\int_0^x\bigg( H_{n+1}(y;m)+\frac{u_{n+1}(y;m)p_{n+1}(y;m)}{y}\nonumber\\
&\hphantom{s_n(x;m)=}{} +\frac{4(m-n)^2-1}{8 y}\bigg) \dd y\bigg).
\label{eq:umemura-tau_at_zero}
\end{align}
We could equally well have used \eqref{eq:umemura-tau_at_zero} to derive the Toda equation instead of \eqref{eq:umemura-tau}, but it is nice to have two different formul\ae\ for Umemura polynomials.

\subsection[The ratio s\_n(x;m)/s\_n(0;m) for large n and small x]{The ratio $\boldsymbol{s_n(x;m)/s_n(0;m)}$ for large $\boldsymbol{n}$ and small $\boldsymbol{x}$}

The representation \eqref{eq:umemura-tau_at_zero} can be combined with Theorem~\ref{thm:limit-of-rational-solutions} to obtain a limiting formula for~$s_n(x;m)/s_n(0;m)$ as $n\to\infty$ and $x\to 0$ at related rates. First, we note that with the notation $U_n(z;m):=u_n(z/n;m)$, from \eqref{eq:Hamiltonian-v} we obtain
\begin{align*}
 &p_n\left(\frac{z}{n};m\right)= \frac{n}{2U_n(z;m)}\\
 &\hphantom{p_n\left(\frac{z}{n};m\right)=}{} \times\bigg[1+\left(\frac{zU_n'(z;m)}{2U_n(z;m)}-m-\frac{1}{2}\right)\frac{1}{n}+\left(zU_n(z;m)-\frac{z}{U_n(z;m)}\right)\frac{1}{n^2}\bigg].
\end{align*}
Next, note that by Theorem \ref{thm:$D_6$-$D_8$-series} we can differentiate the limit in Theorem \ref{thm:limit-of-rational-solutions} for $z$ near the origin, and hence for small $z$ and $n$ even we have $U_n(z;m)\to U(z;m)$ and $U_n'(z;m)\to U'(z;m)$, while for $n$ odd we have instead $U_n(z;m)\to -U(z;m)^{-1}$ and $U_n'(z;m)\to U(z;m)^{-2}U'(z;m)$. Therefore, we have the following limit:
\begin{gather*}
\lim_{n\to\infty}\frac{1}{n}\bigg( H_n(x;m)+\frac{2}{x}u_n(x;m)p_n(x;m) +\frac{4(m-n+1)^2-1}{8x}\bigg|_{x=z/n}\bigg)\\
\qquad=\frac{zU'(z;m)^2}{8U(z;m)^2}\pm\frac{U'(z;m)}{4U(z;m)}-U(z;m)+\frac{1}{U(z;m)},
\end{gather*}
where we take the plus sign for $n$ even and the minus sign for $n$ odd, and the convergence is uniform for $|z|$ sufficiently small. It follows that if $x=z/(n+1)$ in \eqref{eq:umemura-tau_at_zero}, by the corresponding substitution $y\mapsto y/(n+1)$
\begin{equation}
\lim_{j\to\infty}\frac{s_{2j-1} \big(\frac{z}{2j} ;m\big)}{s_{2j-1}(0;m)}= \exp \bigg(\int_0^{z}\bigg(\frac{y U'(y;m)^2}{8 U(y;m)^2}+\frac{U'(y;m)}{4 U(y;m)}-U(y;m)+\frac{1}{U(y;m)}\bigg) \dd y\bigg),
\label{eq:umemura-asymtotics-2}
\end{equation}
and
\begin{equation}
 \lim_{j\to\infty}\frac{s_{2j}\big( \frac{z}{2j+1};m \big)}{s_{2j}(0;m)}=\exp\bigg(\displaystyle\int_0^z\bigg(\frac{y U'(y;m)^2}{8 U(y;m)^2}-\frac{U'(y;m)}{4 U(y;m)}-U(y;m)+\frac{1}{U(y;m)}\bigg) \dd y\bigg),
\label{eq:umemura-asymtotics-1}
\end{equation}
with the limits being uniform for $|z|$ sufficiently small.
To reduce the right-hand side in each case to the corresponding formula presented in Theorem \ref{thm:umemura_asym} we refer to Section~\ref{sec:Bessel-determinant} below.

\subsection[Asymptotic behavior of s\_n(0;m) for large n]{Asymptotic behavior of $\boldsymbol{s_n(0;m)}$ for large $\boldsymbol{n}$}
We now compute the large $n$ asymptotics of $s_n(0;m)=\phi_n\big(m+\frac{1}{2}\big)$. First we write the formula for $\phi_n(y)$ from Lemma \ref{lem:umemura-at-zero} in terms of Gamma functions
\begin{gather*}
\phi_{n}(y)=\begin{cases}\displaystyle
 \mathop{\prod}\limits_{j=1}^{\frac{n}{2}}\frac{\Gamma(y+2j)}{\Gamma(y+1-2j)},& \text{$n$ even},\\
 \displaystyle \mathop{\prod}\limits_{j=1}^{\frac{n+1}{2}}\frac{\Gamma(y+2j-1)}{\Gamma(y+2-2j)},& \text{$n$ odd.}
\end{cases}
\end{gather*}
Since we are interested in asymptotics for large $n$, we need to use the reflection formula for the Gamma function \cite[equation~(5.5.3)]{DLMF} in the denominator:
\begin{gather*}
\phi_{n}(y)=\begin{cases} \displaystyle
 \left(-\frac{\sin(\pi y)}{\pi}\right)^\frac{n}{2} \mathop{\prod}\limits_{j=1}^{\frac{n}{2}}{\Gamma(2j+y)\Gamma(2j-y)},& \text{$n$ even},\\
 \displaystyle \left(\frac{\sin(\pi y)}{\pi}\right)^\frac{n+1}{2}\mathop{\prod}\limits_{j=1}^{\frac{n+1}{2}}{\Gamma(2j-1+y)\Gamma(2j-1-y)},& \text{$n$ odd.}\\
\end{cases}
\end{gather*}
Next, we use the Gamma duplication formula \cite[equation~(5.5.5)]{DLMF} and get
\begin{gather*}
\phi_{n}(y)=
\begin{cases} \displaystyle
 \left(-\frac{\sin(\pi y)}{\pi^2}\right)^\frac{n}{2} 2^\frac{n^2}{2}
\\
 \qquad \displaystyle{}\times\prod\limits_{j=1}^{\frac{n}{2}}\Gamma\left(j+\frac{y}{2}\right)
 \Gamma\left(j-\frac{y}{2}\right)\Gamma\left(j+\frac12+\frac{y}{2}\right)\Gamma\left(j+\frac12-\frac{y}{2}\right),& \text{$n$ even,}\\
 \displaystyle \left(\frac{\sin(\pi y)}{\pi^2}\right)^\frac{n+1}{2}2^\frac{n^2-1}{2}
\\
 \qquad \displaystyle{} \times\prod \limits_{j=1}^{\frac{n+1}{2}}\Gamma\left(j+\frac{y}{2}\right)
 \Gamma\left(j-\frac{y}{2}\right)\Gamma\left(j-\frac12+\frac{y}{2}\right)\Gamma\left(j-\frac12-\frac{y}{2}\right),& \text{$n$ odd.}
 \end{cases}
\end{gather*}
Now we can rewrite $\phi_n(y)$ in terms of the Barnes $G$-function:
\begin{gather*}
\phi_{n}(y)=\begin{cases} \displaystyle
 \left(-\frac{\sin(\pi y)}{\pi^2}\right)^\frac{n}{2} 2^\frac{n^2}{2}\vspace{1mm}\\
 \qquad\displaystyle{}\times\frac{G\big(\frac{n + 2 + y}{2}\big)G\big(\frac{n + 2 - y}{2}\big)G\big(\frac{n + 3 + y}{2}\big)G\big(\frac{n + 3 - y}{2}\big)}{G\big(1+\frac{y}{2}\big)G\big(1-\frac{y}{2}\big)G\big(\frac32+\frac{y}{2}\big)G\big(\frac32-\frac{y}{2}\big)},
 &\text{$n$ even}, \vspace{1mm}\\
 \displaystyle \left(\frac{\sin(\pi y)}{\pi^2}\right)^\frac{n+1}{2}2^\frac{n^2-1}{2} \vspace{1mm}\\
 \qquad\displaystyle{}\times{\frac{G\big(\frac{n+3 + y}{2}\big)G\big(\frac{n+3 - y}{2}\big)G\big(\frac{n+2 + y}{2}\big)G\big(\frac{n+2 - y}{2}\big)}{G\big(1+\frac{y}{2}\big)G\big(1-\frac{y}{2}\big)G\big(\frac12+\frac{y}{2}\big)G\big(\frac12-\frac{y}{2}\big)}}, &{\text{$n$ odd.}}
\end{cases}
\end{gather*}
Using the large argument asymptotics of the Barnes G-function \cite[equation~(5.17.5)]{DLMF}, we get
\begin{gather*}
\phi_{n}(y)\sim \begin{cases} \displaystyle
 \frac{n^{\frac{n^2+n}{2}}\ee^{-\frac{3n^2}{4}-\frac{n}{2}}(-\sin(\pi y))^\frac{n}{2}2^{\frac{n}{2}}n^{-\frac{1}{12}+\frac{y^2}{2}}\sqrt{\pi}\ee^{\frac{1}{3}}2^{\frac{7}{12}-\frac{y ^2}{2}}}{A^4G\big(1+\frac{y}{2}\big)G\big(1-\frac{y}{2}\big)G\big(\frac32+\frac{y}{2}\big)G\big(\frac32-\frac{y}{2}\big)},& \text{$n$ even}, \vspace{1mm}\\
 \displaystyle \frac{n^{\frac{n^2+n}{2}}\ee^{-\frac{3n^2}{4}-\frac{n}{2}}(\sin(\pi y))^\frac{n+1}{2}2^{\frac{n}{2}}n^{-\frac{1}{12}+\frac{y^2}{2}}\ee^{\frac{1}{3}}2^{\frac{1}{12}-\frac{y^2}{2}}} {\sqrt{\pi}A^4G\big(1+\frac{y}{2}\big)G\big(1-\frac{y}{2}\big)G\big(\frac12+\frac{y}{2}\big)G\big(\frac12-\frac{y}{2}\big)},& \text{$n$ odd}
\end{cases}
\end{gather*}
as $n\to\infty$, where $A=\ee^{\frac{1}{12}-\zeta'(-1)}$ is Glaisher's constant \cite[equation~(5.17.6)]{DLMF}. Recalling $y=m+\frac{1}{2}$, we complete the proof of the formul\ae\ \eqref{eq:umemura-zero-asym-even}--\eqref{eq:umemura-zero-asym-odd}.

\subsection{Connection with the Fredholm determinant of the Bessel kernel}
\label{sec:Bessel-determinant}
 We have already seen how the PIII($D_8$) equation \eqref{eq:PIII-$D_8$} can be obtained from the PIII($D_6$) equation \eqref{eq:PIII-$D_6$} by confluence. There exists another, less known relation between the two equations -- namely, a quadratic transformation mapping the solutions of PIII($D_8$) to solutions of PIII($D_6$) with special parameter values. Moreover, for precisely this parameter choice the relevant PIII($D_6$) admits a family of transcendental analytic solutions that can be expressed in terms of Fredholm determinants of the continuous Bessel kernel. Under quadratic transformations, they are mapped to solutions of PIII($D_8$) analytic at $z=0$. This allows one to give yet another characterization of the PIII($D_8$) transcendent describing the large-order asymptotics of the rational PIII($D_6$)
 solutions.

 Indeed, let $U(z)$ be an arbitrary solution of the PIII($D_8$) equation \eqref{eq:PIII-$D_8$}. It is then straightforward to check that the function $\sigma(r)$ defined by
 	\begin{gather}
 	\sigma(r):=\frac{z^2U'(z)^2}{4U(z)^2}-
 2z\lb U(z)-\frac1{U(z)}\rb-4\ii z, \qquad r = 32\ii z, \label{eq:v-to-sigma}
 	\end{gather}
 satisfies the $\sigma$-form of a particular PIII($D_6$) equation, namely,\footnote{Observe that our definition of $\sigma$ differs from that of \cite{TWBessel} by a negative sign.}
 	\begin{gather}\label{sigmapiii}
 	\lb r\sigma''(r)\rb^2=\sigma'(r)(4\sigma'(r)+1)(\sigma(r)-r\sigma'(r)).
 	\end{gather}
 Indeed, letting
$
 {\varsigma}(t) := \sigma(4t) + t
$
 transforms equation \eqref{sigmapiii} to
 \begin{equation*}
 (t\varsigma''(t) )^2 = 4\varsigma'(t)(\varsigma'(t) - 1)(\varsigma(t) - t \varsigma'(t)).
 \end{equation*}
 The latter appears in \cite[equation (3.13)]{Jimbo} and \cite[equation~($E_{\mathrm{III}'}$)]{Okamoto_1980}. These relate to \eqref{eq:PIII-$D_6$} via the following transformations; letting
 \[
 q(t):= -\dfrac{t \varsigma''(t)}{2\varsigma'(t)(\varsigma'(t)-1)}
 \]
 yields (a special case of) the so-called ``prime" version of Painlev\'e-III
 \[
 \dod[2]{q}{t} = \frac{1}{q} \left(\dod{q}{t}\right)^2 - \frac1t \dod{q}{t}+ \frac{1}{t^2}q^3 + \frac1t - \frac{1}{q}.
 \]
 Next, letting $t = x^2$ and $q(t) = xu(x)$ yields \eqref{eq:PIII-$D_6$} with parameters $\alpha = 0$ and $\beta = 4$. Combining the transformations $U(z)\mapsto\sigma(r)\mapsto\varsigma(t)\mapsto q(t)\mapsto u(x)$ yields an explicit formula for $u(x)$ in terms of $U(z)$:
 \[
 u(x)=\frac{1}{2}\frac{\dd}{\dd x}\log\left(\frac{1+\ii U\big(\frac{x^2}{8\ii}\big)}{1-\ii U\big(\frac{x^2}{8\ii}\big)}\right).
 \]
 Correcting for a typo,\footnote{The relevant equation in \cite{BLM20} should be corrected to read
 \[
 u(x):=-\frac{8}{x}\Bigg[\frac{\dd}{\dd X}\log\left(X\frac{W'(X)}{W(X)}\right)\bigg|_{X=-\frac{1}{8}x^2}\Bigg]^{-1}.
 \]
 Here $W(X)$ is related to a solution $U(z)$ of \eqref{eq:PIII-$D_8$} by $W(X)=\ii U(\ii X)$.} this is equivalent to \cite[equation~(112)]{BLM20}. Note that if $U(z)$ is a solution of~\eqref{eq:PIII-$D_8$} that is analytic at $z=0$ with $U(0)\neq 0$, and hence also from \eqref{eq:PIII-$D_8$} $U'(0)=4\big(1+U(0)^2\big)$, then \eqref{eq:v-to-sigma} implies that $\sigma(r)$ is analytic at $r=0$ with $\sigma(0)=0$, and, in fact,
 \begin{gather}
 \sigma(r)=\bigg(\frac{\ii}{16}\left(U(0)-\frac{1}{U(0)}\right)-\frac{1}{8}\bigg)r + \frac{1}{256}\left(U(0)+\frac{1}{U(0)}\right)^2r^2 + \mathcal{O}\big(r^3\big),\qquad \! r\to 0.\!\!\!
 \label{eq:sigma-prime-at-zero}
 \end{gather}
 Also, differentiating \eqref{eq:v-to-sigma} and using \eqref{eq:PIII-$D_8$} to eliminate $U''(z)$ yields the relation
\begin{gather}
 -16\ii \sigma'(r)=U(z)-\frac1{U(z)}+2\ii,\qquad r=32\ii z,\label{eq:sigma-prime}
\end{gather}
 which can be regarded as an algebraic equation expressing $U(z)$ in terms of $\sigma'(32\ii z)$.

 Conversely, any solution $\sigma(r)$ of \eqref{sigmapiii} different from an affine function $a r+b$ can be mapped to a pair of solutions $U(z)$ of PIII($D_8$) related by the $\mathbb Z_2$ B\"acklund transformation ${U(z)\mapsto -1/U(z)}$ with the help of the formula \eqref{eq:sigma-prime}. To see this, one first uses \eqref{sigmapiii} to explicitly express $\sigma(r)$ in terms of its derivatives and $r$, and then differentiates the resulting expression with respect to $r$. Each term of the resulting equation has a common factor of $r\sigma''(r)$. Hence if~$\sigma(r)$ is non-affine, one may cancel this factor, and then $\sigma'(r)$, $\sigma''(r)$, and $\sigma'''(r)$ can be eliminated from the reduced equation using \eqref{eq:sigma-prime} and its derivatives. This implies that either~$U(z)^2+1=0$ or $U(z)$ is a solution of \eqref{eq:PIII-$D_8$}, and the latter admits precisely the constant solutions $U(z)=\pm\ii$ so we may conclude that any meromorphic function $U(z)$ obtained from a~non-affine solution of \eqref{sigmapiii} via \eqref{eq:sigma-prime} is a solution of \eqref{eq:PIII-$D_8$}.

 Now recall a classical result of Tracy and Widom \cite[equation (1.21) with $\alpha = 0$]{TWBessel}.
 \begin{prop} 	
\label{prop:bessel}
 The logarithmic derivative
\begin{gather}
 \sigma(r)=r \dod{}{r}\ln D_\lambda( r) \label{eq:sigma-Fredholm-det}
\end{gather}
 satisfies the $\sigma$-{\rm PIII}$(D_6)$ equation \eqref{sigmapiii}.
 \end{prop}
The Bessel kernel can be equivalently written as
 	 	\begin{align}
 \label{kernelconvolution}	K( x,y)=& \frac14\int_{0}^1
 	J_0\big(\sqrt{xz}\big) J_0\big(\sqrt{yz}\big) \dd z = \sum_{m=0}^\infty\sum_{n=0}^\infty
 	\frac{\lb -1\rb^{m+n}2^{-2\lb m+n+1\rb}}{\lb m!\rb^2\lb n!\rb^2 \lb m+n+1\rb} x^my^n.
 	\end{align}
The first of these identities follows from the easily verified differentiation formula
\begin{gather*}
\dod{}{z}( z K( xz,yz))=\frac14 J_0\big( \sqrt{xz}\big) J_0\big( \sqrt{yz}\big),
\end{gather*}
whereas the second one is obtained by substituting into the integral expression the standard series representation of $J_0(\cdot)$ \cite[equation~(10.2.2)]{DLMF}. Using \eqref{eq:plemelj-smithies} along with representation \eqref{kernelconvolution}, then enables one to compute the traces of powers of $K_r$ in the form of a series in $r$. It yields
\begin{gather}\label{seriesdet}
 	\ln D_\lambda( r)=-\sum_{\ell=1}^\infty\frac{\lambda^\ell r^\ell}{2^{2\ell}\ell}\sum_{n_1=0}^{\infty}\cdots\sum_{n_{2\ell}=0}^\infty \frac{\lb-r/4\rb^{\sum_{k=1}^{2\ell} n_k}}{\prod_{k=1}^{2\ell}\lb n_k!\rb^2\lb n_k+n_{k+1}+1\rb},\qquad n_{2\ell+1}=n_1.
\end{gather}
Expansions of such form are known for Fredholm determinants appearing in random matrix theory, see \cite[Section 20.5]{Mehta}.
 Let us record explicitly the few first terms of \eqref{seriesdet}:
 	\begin{align*}\ln D_\lambda( r)={}& -\frac{\lambda r}{4}\lb 1-\frac{r}{8}+\frac{r^2}{96}-\frac{5r^3}{9216}\rb-\frac{\lambda^2r^2}{32}\lb 1-\frac{r}{4}+\frac{41r^2}{1152}\rb\\
 & -
 	\frac{\lambda^3r^3}{192}\lb 1-\frac{3r}{8}\rb-
 	\frac{\lambda^4r^4}{1024}+\mathcal{O}\big( r^5\big),\qquad r\to 0,
 \end{align*}
 which implies that the Bessel determinant solution of \eqref{sigmapiii} guaranteed by Proposition~\ref{prop:bessel} has the asymptotics
 	\begin{align} \sigma(r)=r\frac{\dd}{\dd r}\ln D_\lambda(r)=& -\frac{\lambda r}{4}+\frac{1}{16} \big(\lambda -\lambda ^2\big) r^2+\frac{1}{128} \bigl(-2 \lambda ^3+3 \lambda ^2-\lambda \bigr) r^3\nonumber\\& +\frac{\bigl(-36 \lambda ^4+72 \lambda ^3-41 \lambda ^2+5 \lambda \bigr) r^4}{9216}+\mathcal{O}\big( r^5\big),\qquad r\to 0. \label{eq:sigma-Bessel-expand}
 \end{align}
 This expression is of course consistent with the differential equation \eqref{sigmapiii}.

 On the other hand, if $U(z)=U(z;m)$ is the particular solution of \eqref{eq:PIII-$D_8$} relevant to Theorem~\ref{thm:limit-of-rational-solutions}, which for $m\in\mathbb{C}\setminus\big(\mathbb{Z}+\frac{1}{2}\big)$ is analytic at the origin with
 \[U(0;m)=\tan\left(\frac{\pi}{2}\left(m+\frac{1}{2}\right)\right)\in\mathbb{C},\]
 then according to \eqref{eq:sigma-prime-at-zero}, the corresponding solution of \eqref{sigmapiii} analytic at the origin satisfies
 \begin{equation}
 \sigma(r)=-\frac{\lambda(m)}{4}r + \frac{r^2}{64\cos^2(\pi m)} + \mathcal{O}\big(r^3\big),\qquad r\to 0,\qquad \lambda(m):=\frac{1}{1+\ee^{2\pi\ii m}}.
 \label{eq:sigma-v-expand}
 \end{equation}
 Note that $\lambda(m)$ is necessarily finite for $m\in\mathbb{C}\setminus\big(\mathbb{Z}+\frac{1}{2}\big)$ and there are only two values it never takes for any $m$: $\lambda(m)\neq 0,1$. Also, the coefficient of $r^2$ cannot vanish for any $m\in\mathbb{C}$. Now we need the following result.

 \begin{prop}
 \label{prop:sigma-unique}
 Let $\sigma_1(r)$ and $\sigma_2(r)$ denote two non-affine solutions of \eqref{sigmapiii} both analytic at the origin and both satisfying $\sigma_j(r)=-\frac{1}{4}\lambda r + \mathcal{O}\big(r^2\big)$ as $r\to 0$ with $\lambda\neq 0,1$. Then $\sigma_1(r)=\sigma_2(r)$ on a neighborhood of $r=0$.
 \end{prop}
 \begin{proof}
 If $\sigma(r)=\sigma_{1,2}(r)$ is a solution of \eqref{sigmapiii} analytic at the origin with $\sigma'(0)=-\frac{1}{4}\lambda$, then it has a locally-convergent Taylor series
 \[
 \sigma(r)=-\frac{\lambda}{4} r + \sum_{k=2}^\infty s_kr^k,\qquad |r|<\rho
 \]
 for some $\rho>0$. Using this in the differential equation \eqref{sigmapiii}, from the coefficient of $r^2$ one obtains
 \begin{equation}
 4 s_2^2 -\frac{1}{4}\lambda(1-\lambda)s_2 = 0,
 \label{eq:sigma-first}
 \end{equation}
 whereas from the coefficient of $r^k$ for $k\ge 3$,
 \begin{gather}
 \left(4k s_2 - \frac{1}{4}\lambda (1 -\lambda)\right)s_k \nonumber\\
 \qquad = \dfrac{1}{k - 1} \Bigg( \sum_{\ell = 2}^{k - 1} \ell(k - \ell) s_{\ell } s_{k - \ell+1} - \sum_{\ell = 3}^{k - 1}(\ell - 1)\ell(k - \ell + 1)(k - \ell+2) s_{\ell } s_{k - \ell+2} \nonumber\\
 \phantom{\qquad =}{}+ \sum_{\ell = 2}^{k - 1} \sum_{j = 0}^{\ell-1}(j + 1)(\ell - j ) (k-\ell)s_{j+1} s_{\ell - j} s_{k - \ell+1}\Bigg) , \qquad k\geq 3,
 \label{eq:sigma-recurrence}
 \end{gather}
 where on the right-hand side, $s_1:=-\frac{1}{4}\lambda$. Now, \eqref{eq:sigma-first} implies that either $s_2=0$ or $s_2=\frac{1}{16}\lambda(1-\lambda)\neq 0$.

 Suppose first that $s_2=0$. Then setting $k=3$ in \eqref{eq:sigma-recurrence} gives
 \begin{equation}
 -\frac{1}{4}\lambda(1-\lambda)s_3=0\implies s_3=0,
 \label{eq:sigma-induction-base}
 \end{equation}
 since $\lambda\neq 0,1$. We now use \eqref{eq:sigma-induction-base} as the base case for an inductive argument. Suppose ${s_2=s_3=\cdots=s_k=0}$. Using \eqref{eq:sigma-recurrence} for the coefficient of $r^{k+1}$, we obtain
 \[
 -\frac{1}{4}\lambda(1-\lambda)s_{k+1}=0\implies s_{k+1}=0,
 \]
 from which it follows that $\sigma(r)=-\frac{1}{4}\lambda r$ exactly. This is a contradiction, because $\sigma(r)$ is not affine. Therefore, $s_2\neq 0$.

 Taking $s_2=\frac{1}{16}\lambda(1-\lambda)\neq 0$ as necessary, we note that in \eqref{eq:sigma-recurrence} for $k\ge 3$, $s_k$ appears only on the left-hand side with coefficient
 \[
 4ks_2 -\frac{1}{4}\lambda(1-\lambda) = \frac{1}{4}(k-1)\lambda(1-\lambda)\neq 0,
 \]
 while the right-hand side only involves $s_1,\dots,s_{k-1}$. Therefore, all subsequent coefficients $s_k$,~${k\!\ge \!3}$ are uniquely determined by the recurrence, implying that $\sigma_1(r)=\sigma_2(r)$.
 \end{proof}

\begin{Remark}
 It is worth noting that $D_1(r)=\ee^{-r/4}$, for which $\sigma(r)$ defined by \eqref{eq:sigma-Fredholm-det} is an affine function. See \cite{TWBessel}.
\end{Remark}

 Since the analytic solutions with expansions given in \eqref{eq:sigma-Bessel-expand} and \eqref{eq:sigma-v-expand} have the same leading term if $\lambda=\lambda(m)\neq 0,1$, and neither solution is an affine function, they coincide for small $|r|$. Because the function $U(z;m)$ is then determined up to the involution $U\mapsto -U^{-1}$ by \eqref{eq:sigma-prime}, we have proved the following result.

 \begin{Corollary} \label{cor:bessel-relation}
 Let $m\in\mathbb{C}\setminus\big(\mathbb{Z}+\tfrac{1}{2}\big)$. The function $U(z;m)$ appearing in the asymptotics \eqref{eq:even-odd-limit} of Theorem~{\rm \ref{thm:limit-of-rational-solutions}} is related to the continuous Bessel kernel determinant by
 \begin{gather*}
 U( z;m)-\frac1{U( z;m)}=-2\ii-\frac12\dod{}{z}z\dod{}{z}\ln D_{\lambda( m)}( 32\ii z),
 \end{gather*}
 with $\lambda( m)={1}/\big( 1+\ee^{2\pi \ii m}\big)$. In particular, the expansion of $U( z;m)$ in powers of $z$ can be read off from the series representation \eqref{seriesdet}.
 \end{Corollary}
Furthermore, using \eqref{sigmapiii} and \eqref{eq:sigma-Fredholm-det} shows that the integrand in \eqref{eq:umemura-asymtotics-2}, \eqref{eq:umemura-asymtotics-1} is given by
\begin{gather}
\frac{z U'(z;m)^2}{8 U(z;m)^2} \pm \frac{U'(z;m)}{4 U(z;m)}-U(z;m)+\frac{1}{U(z;m)} \nonumber\\
\quad= \dfrac{1}{2z}\sigma(32\ii z) + 2\ii \pm \dfrac{U'(z; m)}{4U(z; m)}=\frac{1}{2}\frac{\dd}{\dd z}\ln D_{\lambda(m)}(32\ii z) + 2\ii \pm\frac{U'(z;m)}{4U(z;m)}\label{eq:relation-to-bessel}
\end{gather}
and we get Theorem \ref{thm:umemura_asym}.

\subsection[Connection with 2j - k determinants]{Connection with $\boldsymbol{2j - k}$ determinants}
On the other hand, the Umemura polynomials admit the following Wronskian determinant representation \cite{Kajiwara_Masuda}:
\[ s_{n}(x;m)=\prod_{k=1}^n(2k-1)!!\det\big(L_{2i-j}^{(m+1/2-2i+j)}(-2x)\big)_{i,j=1}^{n},\]
where for parameter $\alpha\in\mathbb{C}$ and index $k\in\mathbb{Z}$, $L_{k}^{(\alpha)}(x)$ are generalized Laguerre polynomials for~$k\geq 0$, while $L_{k}^{(\alpha)}(x)=0$ for $k<0$.

Expressions like that on the right-hand side are called Wronskian Appell polynomials in~\cite{BHSS}; similar formul\ae\ hold for rational solutions of other Painlev\'e equations as well. Wronskian determinants of generalized Laguerre polynomials were also studied in \cite{CD}.

The generalized Laguerre polynomials admit the following integral representation \cite[equation~(18.10.8)]{DLMF}:
\begin{equation}\label{eq:contlaguerre}
 L_k^{(\alpha)}(x)=\frac{\ee^x x^{-\alpha}}{2\pi \ii}\int_{|z-x|=\varepsilon}z^k\frac{z^\alpha \ee^{-z}}{(z-x)^{k+1}} \dd z
\end{equation}
for $|{\arg}(x)|<\pi$ (analytically continuable to $x\in\mathbb{C}$) where $\varepsilon >0$ is small enough so that the branch cut $z\le 0$ is outside the contour of integration. Making the transformation \[z=x+\frac{yx}{2}\] in \eqref{eq:contlaguerre} yields
\begin{equation*}
 L_k^{(\alpha)}(x)=\frac{1}{2^{\alpha}}\int_{|y|=\varepsilon}y^{-k}{(y+2)^{k+\alpha} \ee^{-\frac{yx}{2}}}\frac{\dd y}{2\pi \ii y}.
\end{equation*}
We denote
\[w(y;x,m):=(y+2)^{m+\frac{1}{2}} \ee^{{yx}{}} \qquad \text{and} \qquad w_k(x,m):=\int_{|y|=1}y^{-k}w(y;x,m)\frac{\dd y}{2\pi \ii y}.\]
Using this notation, we obtain
\begin{equation}\label{eq:lag}
 s_{n}(x;m)=\Bigg[\prod_{\ell=1}^n(2\ell-1)!!\Bigg]2^{mn-\frac{n^2}{2}}\det(w_{2j-k}(x,m))_{j,k=1}^{n}.
\end{equation}
Similar ``$2j - k$" determinants have appeared in various works in the literature, see, e.g., \cite{gharakhloo_witte}. Denoting
\begin{equation*}
\boldsymbol{\mathscr{D}}_n(x;m)=\det\left(w_{2j-k}(x,m)\right)_{j,k=1}^{n},
\end{equation*}
it immediately follows from \eqref{eq:umemura-asymtotics-1}, \eqref{eq:umemura-asymtotics-2} that in the limit $j\to\infty$,
\begin{align*}
\boldsymbol{\mathscr{D}}_{2j}\left(\frac{z}{2j+1};m\right)\sim{}&\frac{s_{2j}(0;m)2^{2{j^2}{}-2mj}}{\prod_{\ell=1}^{2j}(2\ell-1)!!}\\
&\times \exp\Bigg(\int_0^z\left(\frac{y U'(y;m)^2}{8 U(y;m)^2}-\frac{U'(y;m)}{4 U(y;m)}-U(y;m)+\frac{1}{U(y;m)}\right) \dd y\Bigg)
\end{align*}
and
\begin{align*}
\boldsymbol{\mathscr{D}}_{2j-1}\left(\frac{z}{2j};m\right) \sim{}&\frac{s_{2j-1}(0;m)2^{2j^2-2(m+1)j+m+\frac{1}{2}}}{\prod_{\ell=1}^{2j-1}(2\ell-1)!!}\\
&\times \exp\Bigg(\int_0^{z}\left(\frac{y U'(y;m)^2}{8 U(y;m)^2}+\frac{U'(y;m)}{4 U(y;m)}-U(y;m)+\frac{1}{U(y;m)}\right) \dd y\Bigg).
\end{align*}
To write the analog of Theorem \ref{thm:umemura_asym} for $\boldsymbol{\mathscr{D}}_n(x;m)$ we need to compute the asymptotic behavior of $\prod_{k=1}^{n}(2k-1)!!$. We use \cite[equation~(5.4.2)]{DLMF} to get
\begin{align*}
\prod_{\ell=1}^{2n-1}(2\ell-1)!!&= \pi^{-\frac{n}{2}}2^{\frac{n(n+1)}{2}}\prod_{\ell=1}^n\Gamma\left(\ell+\frac12\right)=\pi^{-\frac{n}{2}}2^{\frac{n(n+1)}{2}}\frac{G\big(n+\frac32\big)}{G\big(\frac32\big)}\\
&\sim n^{\frac{n^2}{2}+\frac{n}{2}}\ee^{-\frac{3n^2}{4}-\frac{n}{2}}2^{\frac{n^2}{2}+n}n^\frac{1}{24}2^{\frac{5}{24}}\ee^{-\frac{\zeta'(-1)}{2}},\qquad n\to\infty.
\end{align*}
Combining this with formul\ae\ \eqref{eq:umemura-zero-asym-even}, \eqref{eq:umemura-zero-asym-odd} and using \eqref{eq:relation-to-bessel}, we get
\begin{gather*}
\boldsymbol{\mathscr{D}}_{2j}\left(\frac{z}{2j+1};m\right) \\
\qquad\sim\frac{2^{-j(2m+1)}(-\cos(\pi m))^jj^{\frac{m^2}{2}+\frac{m}{2}}2^{\frac{1}{4}}\sqrt{\pi}\ee^{\frac{9\zeta'(-1)}{2}}}{G\big(\frac{3}{4}-\frac{m}{2}\big)G\big(\frac{5}{4}-\frac{m}{2}\big)G\big(\frac{5}{4}+\frac{m}{2}\big)G\big(\frac{7}{4}+\frac{m}{2}\big)}\ee^{2\ii z}\left(\frac{U(z;m)}{U(0;m)}\right)^{-\frac{1}{4}} \sqrt{D_{\lambda(m)}\lb 32\ii z\rb},\\
\boldsymbol{\mathscr{D}}_{2j-1}\left(\frac{z}{2j};m\right) \\
\qquad\sim\frac{2^{-j(2m+1)}(\cos(\pi m))^jj^{\frac{m^2}{2}+\frac{m}{2}}2^{\frac{1}{4}+m}\ee^{\frac{9\zeta'(-1)}{2}}}{\sqrt{\pi}G\big(\frac{1}{4}-\frac{m}{2}\big)G\big(\frac{3}{4}-\frac{m}{2}\big)G\big(\frac{3}{4}+\frac{m}{2}\big)G\big(\frac{5}{4}+\frac{m}{2}\big)}\ee^{2\ii z}\left(\frac{U(z;m)}{U(0;m)}\right)^{\frac{1}{4}} \sqrt{D_{\lambda(m)}\lb 32\ii z\rb},
\end{gather*}
both in the limit $j\to\infty$.
\section[General monodromy data: Painleve-III(D\_6)]{General monodromy data: Painlev\'e-III($\boldsymbol{D_6}$)}
\label{sec:monodromy-rep-$D_6$}

\subsection[Lax pair for Painleve-III(D\_6)]{Lax pair for Painlev\'e-III($\boldsymbol{D_6}$)}
\label{sec:lax-pair-$D_6$}

Following Jimbo and Miwa \cite{MR625446}, we use the fact that each Painlev\'e equation can be recast as an isomonodromic deformation condition for a $2\times 2$ system of linear ODEs with rational coefficients. The case of Painlev\'e-III ($D_6$) corresponds to the situation where the coefficient matrix for the equation in the spectral variable, $\lambda$, has exactly two poles on the Riemann sphere leading to irregular singularities at $\lambda = 0$ and $\lambda = \infty$, at each of which the leading term is diagonalizable. After some normalization, the differential equation can be written in the form\footnote{Henceforth, we use bold capital letters to denote matrices, with the only exceptions being the identity matrix, denoted $\mathbb{I}$ and the Pauli matrices, denoted $\sigma_k$, $k = 1, 2, 3$.}
\begin{equation}
\label{eq:generic-system}
\frac{\partial\mathbf{\Psi}}{\partial\lambda} = \mathbf{\Lambda}^{(6)}(\lambda, x) \mathbf{\Psi}(\lambda, x),
\end{equation}
where
\begin{equation}
\label{eq:lambda-coefficient-D6}
\mathbf{\Lambda}^{(6)}(\lambda, x) = \frac{\ii x}{2}\sigma_3 + \dfrac{1}{2\lambda} \begin{bmatrix} -\Theta_\infty & 2y \\ 2v & \Theta_\infty \end{bmatrix} + \dfrac{1}{2\lambda^2} \begin{bmatrix} \ii x - 2\ii st & 2\ii s \\ -2\ii t(st - x) & -\ii x + 2\ii st\end{bmatrix},
\end{equation}
In this case, the deformation equation is
\begin{equation}
 \frac{\partial\mathbf{\Psi}}{\partial x} = \mathbf{X}(\lambda, x) \mathbf{\Psi}(\lambda, x),
 \label{eq:generic-deformation}
\end{equation}
where
\begin{equation*}
\mathbf{X}(\lambda, x) = \dfrac{\ii \lambda}{2} \sigma_3 + \dfrac{1}{x} \begin{bmatrix} 0 & y \\ v & 0 \end{bmatrix} - \dfrac{1}{2\lambda x} \begin{bmatrix} \ii x - 2\ii st & 2\ii s \\ -2\ii t(st - x) & -\ii x + 2\ii st\end{bmatrix}.
\end{equation*}
In the expressions for $\mathbf{\Lambda}^{(6)}(\lambda, x)$ and $\mathbf{X}(\lambda, x)$, $\Theta_\infty$ is a complex parameter and $s = s(x)$, $t = t(x)$, $ v = v(x)$, $y = y(x)$. The equations \eqref{eq:generic-system} and \eqref{eq:generic-deformation} constitute an over-determined system with compatibility condition
\[
\dpd{\mathbf{\Lambda}^{(6)}}{x}(\lambda, x) - \dpd{\mathbf X}{\lambda}(\lambda, x) + \big[\mathbf{\Lambda}^{(6)}(\lambda, x), \mathbf X(\lambda, x)\big] = \mathbf{0},
\]
where $\big[\mathbf{\Lambda}^{(6)}, \mathbf{X}\big]$ is the commutator. This boils down to the scalar equations
\begin{alignat}{3}
& x \dod{y}{x} = -2xs + \Theta_\infty y, \qquad&&
x \dod{v}{x} = -2xt(st - x) - \Theta_\infty,&\nonumber \\
&x \dod{s}{x} = (1 - \Theta_\infty)s - 2xy + 4yst, \qquad&&
x \dod{t}{x} = \Theta_\infty t - 2yt^2 + 2v.&
\label{eq:scalar-equations}
\end{alignat}
If we let
\begin{equation}
u(x) := - \dfrac{y(x)}{s(x)},
\label{eq:u-from-y-s}
\end{equation}
then it follows from \eqref{eq:scalar-equations} that
\[
x \dod{u}{x} = 2x - (1 - 2\Theta_\infty) u + 4stu^2 - 2xu^2,
\]
which can be seen to be equivalent to \eqref{eq:PIII-$D_6$} by taking another $x$-derivative and using \eqref{eq:scalar-equations} again, after which the quantity
\[
I(x) := \dfrac{2\Theta_\infty }{x} st - \Theta_\infty - \dfrac{2yt}{x}(st - x) + \dfrac{2sv}{x},
\]
appears. However, from \eqref{eq:scalar-equations} it follows that $I'(x) = 0$, so denoting the constant value of $I$ by~$\Theta_0$, we arrive at \eqref{eq:PIII-$D_6$} with parameters
\begin{equation}
 \Theta_0=\frac{\alpha}{4}, \qquad \Theta_\infty=1-\frac{\beta}{4}.
 \label{eq:thetas-alpha-beta}
\end{equation}

The constants $\Theta_0$, $\Theta_\infty$ can be naturally interpreted on the level of the $2 \times 2$ system \eqref{eq:generic-system}, which we now explore. For all the calculations that follow, we assume for simplicity that $x > 0$. The system \eqref{eq:generic-system} admits formal solutions near the singular points\footnote{Here, we use the standard notation $f^{\sigma_3} := \mathrm{diag}\big(f, f^{-1}\big)$.}
\begin{equation}
\mathbf{\Psi}_{\text{formal}}^{(\infty)}(\lambda, x) = \big( \mathbb{I} +\mathbf{\Xi}^{(6)}(x)\lambda^{-1}+ \mathcal{O}\big(\lambda^{-2}\big) \big) \ee^{\ii x \lambda \sigma_3 /2} \lambda^{-\Theta_\infty \sigma_3/2} \qquad \text{as} \quad \lambda \to \infty,
\label{eq:formal-infty}
\end{equation}
and
\begin{equation}
\mathbf{\Psi}^{(0)}_{\text{formal}}(\lambda, x) = \big( \mathbf{\Delta}^{(6)}(x) + \mathcal{O}(\lambda) \big) \ee^{-\ii x \lambda^{-1} \sigma_3 /2} \lambda^{\Theta_0 \sigma_3/2} \qquad \text{as} \quad \lambda \to 0.
\label{eq:formal-zero}
\end{equation}
Here $\mathbf{\Delta}^{(6)}(x)$ is an (invertible) eigenvector matrix of the coefficient of $\lambda^{-2}$ in \eqref{eq:lambda-coefficient-D6}, so the leading term of $\mathbf{\Delta}^{(6)}(x)^{-1}\mathbf{\Lambda}^{(6)}(\lambda,x)\mathbf{\Delta}^{(6)}(x)$ at $\lambda=0$ is diagonal.

For $k = 1, 2,3$, we define the Stokes sectors,
\begin{gather*}
S_k^{(\infty)} = \bigl\{ \lambda \in \C \colon |\lambda| > R, \ k\pi - 2\pi < \arg (\lambda) < k\pi \bigr \},\\
S_k^{(0)} = \bigl \{ \lambda \in \C \colon |\lambda| < r, \ k\pi - 2\pi < \arg (\lambda) < k\pi \bigr\}.
\end{gather*}
 It follows from the classical theory of linear systems that there exist \emph{canonical solutions} \smash{$\mathbf{\Psi}_k^{(\infty)}$} and~\smash{$\mathbf{\Psi}_k^{(0)}$} analytic for \smash{$\lambda\in S_k^{(\infty)}$} and \smash{$\lambda\in S_k^{(0)}$} respectively and determined uniquely by the asymptotic condition
\begin{equation}
\mathbf{\Psi}_k^{(\nu)}(\lambda, x) = \mathbf{\Psi}^{(\nu)}_{\text{formal}}(\lambda, x), \qquad \lambda \in S_k^{(\nu)}, \qquad \nu \in \{0, \infty\}, \qquad k=1,2,3.
\label{eq:psi-asymptotic-condition}
\end{equation}
In these asymptotic conditions, the meaning of the power functions in \eqref{eq:formal-infty} and \eqref{eq:formal-zero} is determined from the range of $\arg(\lambda)$ in the definition of \smash{$S_k^{(\nu)}$}.
The canonical solutions in consecutive Stokes sectors are related to one another by multiplication on the right with \emph{Stokes matrices}, i.e.,
\begin{align}
\label{eq:stokes-condition-1}\mathbf{\Psi}_2^{(\infty, 0)}(\lambda, x) &= \mathbf{\Psi}_1^{(\infty, 0)}(\lambda, x) \mathbf{S}^{\infty, 0}_{1}, \qquad
\mathbf{\Psi}_3^{(\infty, 0)}(\lambda, x) = \mathbf{\Psi}_2^{(\infty, 0)}( \lambda, x) \mathbf{S}^{\infty, 0}_{2},
\end{align}
where for some \emph{Stokes multipliers} $s_j^{\infty,0}\in\mathbb{C}$, $j=1,2$,
\begin{equation}
\mathbf{S}^{\infty, 0}_{1}= \begin{bmatrix}1&s^{\infty, 0}_1\\0&1\end{bmatrix},\qquad \mathbf{S}^{\infty, 0}_{{2}}=
\begin{bmatrix}1&0\\{s^{\infty, 0}_2}&1\end{bmatrix}.
 \label{eq:Stokes-infty}
\end{equation}
Likewise, by uniqueness and the different interpretation of the multi-valued powers in the formal solutions on the otherwise identical sectors $S^{(\infty,0)}_1$ and $S^{(\infty,0)}_3$, we have the identities
\begin{align}
\mathbf{\Psi}_3^{(\infty)}(\lambda, x) &= \mathbf{\Psi}_1^{(\infty)}\big(\ee^{-2\pi \ii} \lambda, x\big) e_\infty^{-2\sigma_3},
\qquad
\mathbf{\Psi}_3^{(0)}(\lambda, x) = \mathbf{\Psi}_1^{(0)}\big(\ee^{-2\pi \ii} \lambda, x\big) e_0^{2\sigma_3},\label{eq:stokes-condition-4}
\end{align}
where, combining \eqref{eq:e0-einfty-alpha-beta} with \eqref{eq:thetas-alpha-beta} gives
\begin{equation}
 e_0 = \ee^{\ii \pi \Theta_0/2},\qquad e_\infty = \ee^{\ii \pi \Theta_\infty/2}.
 \label{eq:e0-einfty}
\end{equation}

Canonical solutions in, say $S_k^{(\infty)}$ admit analytic continuation into $S_k^{(0)}$ and since both canonical solutions solve \eqref{eq:generic-system} in the same domain, they must be related by multiplication on the right by a constant \emph{connection matrix}, which we define using
\begin{align}
\label{eq:connection-matrix-1}\mathbf{\Psi}_1^{(0)}(\lambda, x) &= \mathbf{\Psi}_1^{(\infty)}(\lambda, x) \mathbf{C}_{0\infty}^+, \\
\label{eq:connection-matrix-2}\mathbf{\Psi}_2^{(0)}(\lambda, x) &= \mathbf{\Psi}_2^{(\infty)}(\lambda, x) \mathbf{C}_{0\infty}^-.
\end{align}

The condition that the coefficients $y$, $v$, $s$, $t$ in the matrix $\mathbf{\Lambda}^{(6)}(\lambda,x)$ depend on $x$ as a solution of \eqref{eq:scalar-equations} implies simultaneous solvability of \eqref{eq:generic-system} and \eqref{eq:generic-deformation}, and the latter system implies that the Stokes matrices and connection matrices are, like $\Theta_0$ and $\Theta_\infty$, independent of $x$. We show below in Sections~\ref{subsec:cyclic} and \ref{subsec:monodromy-D6} that the four Stokes multipliers and the elements of the two connection matrices are determined from just two essential \emph{monodromy parameters} that we denote by $e_1$ and $e_2$.

\subsection[Riemann--Hilbert problem for Painlev\'e-III(D\_6)]{Riemann--Hilbert problem for Painlev\'e-III($\boldsymbol{D_6}$)} \label{sec:initial-rhp}
Using the canonical solutions, we define the following sectionally-analytic function
\[
\mathbf{\Psi}(\lambda, x) = \begin{cases}
\mathbf \Psi ^{(\infty)}_1(\lambda, x), & |\lambda| >1 \ \text{and} \ \re (\lambda) > 0, \\
\mathbf \Psi ^{(\infty)}_2(\lambda, x), & |\lambda| >1 \ \text{and} \ \re (\lambda) < 0, \\
\mathbf \Psi ^{(0)}_1(\lambda, x), & |\lambda| <1 \ \text{and} \ \re (\lambda) > 0, \\
\mathbf \Psi ^{(0)}_2(\lambda, x), & |\lambda| <1 \ \text{and} \ \re (\lambda) < 0. \\
\end{cases}
\]
Then, it follows from the asymptotic conditions \eqref{eq:psi-asymptotic-condition} and the relations \eqref{eq:stokes-condition-1}--\eqref{eq:stokes-condition-4} and \eqref{eq:connection-matrix-1}--\eqref{eq:connection-matrix-2} that $\mathbf{\Psi}$ solves the following $2 \times 2$ Riemann--Hilbert problem.
Let $\lambda_{\lw}^p$ denote the branch of the power function analytic in $\C\setminus \ii \R_-$ with argument chosen so that
\begin{equation}
-\dfrac{\pi}{2} < \mathrm{arg}_{\lw}(\lambda) < \dfrac{3\pi}{2}.
\label{eq:lw-branch-x-real}
\end{equation}
The notation reminds us that the branch cut of these functions is the contour carrying lower triangular Stokes matrices.
\begin{rhp} \label{rhp:initial}
Fix generic monodromy parameters $( e_1, e_2 )$ determining the Stokes and connection matrices, and $x>0$. We seek a $2 \times 2$ matrix function $\lambda\mapsto\mathbf \Psi(\lambda, x)$ satisfying:
\begin{itemize}\itemsep=0pt
\item Analyticity: $\mathbf \Psi (\lambda, x)$ is analytic in $\C \setminus L^{(6)}$, where $L^{(6)} = \{ | \lambda| = 1\} \cup \ii \R$ is the jump contour shown in Figure {\rm\ref{fig:1}}.
\item  Jump condition: $\mathbf \Psi (\lambda, x)$ has continuous boundary values on $L^{(6)} \setminus \{0\}$ from each component of $\mathbb{C}\setminus L^{(6)}$, which satisfy
$\mathbf \Psi_+ (\lambda, x) = \mathbf \Psi_- (\lambda, x) \mathbf J_{\mathbf \Psi}(\lambda)$, where $\mathbf J_{\mathbf \Psi}(\lambda)$ is as shown in Figure {\rm\ref{fig:1}} and where the $+$ $($resp., $-)$ subscript denotes a boundary value taken from the left $($resp., right$)$ of an arc of $L^{(6)}$.
\item  Normalization: $\mathbf \Psi$ satisfies the asymptotic conditions
\begin{equation}
\label{eq:Psi-asymptotic-infinity}
\boldsymbol \Psi(\lambda, x) = \big( \mathbb{I} +\mathbf{\Xi}^{(6)}(x)\lambda^{-1}+ \mathcal{O}\big(\lambda^{-2}\big) \big) \ee^{\ii x \lambda \sigma_3 /2} \lambda_{\lw}^{-\Theta_\infty \sigma_3/2} \qquad \text{as} \quad \lambda \to \infty,
\end{equation}
and
\begin{equation}
\label{eq:Psi-asymptotic-zero}
\boldsymbol \Psi(\lambda, x) = \big( \mathbf{\Delta}^{(6)}(x) + \mathcal{O}(\lambda) \big) \ee^{-\ii x \lambda^{-1} \sigma_3 /2} \lambda_{\lw}^{\Theta_0 \sigma_3/2} \qquad \text{as} \quad \lambda \to 0,
\end{equation}
where $\mathbf{\Delta}^{(6)}(x)$ is a matrix determined from $\boldsymbol\Psi(\lambda,x)$ having unit determinant.
\end{itemize}
\end{rhp}
\begin{figure}[t]\centering
\includegraphics[scale = 1]{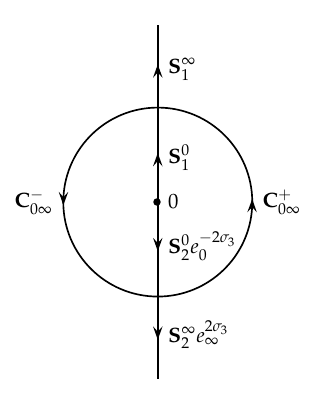}
\caption{The jump contour $L^{(6)}$ for $\mathbf{\Psi}(\lambda, x)$ and definition of $\mathbf J_{\mathbf \Psi}(\lambda)$ when $x>0$.}
\label{fig:1}
\end{figure}

Observe that if $\mathbf \Psi$ solves Riemann--Hilbert Problem~\ref{rhp:initial}, then the following limit exists:
\begin{equation}
\label{eq:T-definition}
 \mathbf{\Xi}^{(6)}(x):=\lim_{\lambda\to\infty}\lambda\big[\mathbf{\Psi}(\lambda,x)\ee^{-\ii x\lambda\sigma_3/2}\lambda_{\lw}^{\Theta_\infty\sigma_3/2}-\mathbb{I}\big].
\end{equation}
 Existence of a solution $\mathbf{\Psi}(\lambda, x)$ to Riemann--Hilbert Problem \ref{rhp:initial} which is meromorphic in $x$ on a covering of the plane is well established; see, e.g., \cite[Theorem 5.4]{FIKN}. Furthermore, it follows from the Riemann--Hilbert problem that $\det (\mathbf{\Psi}(\lambda; x)) = 1$; hence $\mathbf{\Xi}^{(6)}(x)$ defined by \eqref{eq:T-definition} has zero trace. The solution of the Painlev\'e-III($D_6$) equation for the initial data that generated the matrices for the inverse monodromy problem is given by
 \begin{equation}
 u(x)=\frac{-\ii \Xi^{(6)}_{12}(x)}{{\Delta}^{(6)}_{11}(x){\Delta}^{(6)}_{12}(x)},
 \label{eq:u-recover-general}
 \end{equation}
where $\mathbf{\Delta}^{(6)}$, $\mathbf{\Xi}^{(6)}$ are as in \eqref{eq:Psi-asymptotic-zero}, \eqref{eq:T-definition}, respectively.

To study the direct monodromy problem and obtain the jump matrices given just the values of $u$ and $u'$ at an initial point $x_0$, it is necessary to introduce artificial initial values of the auxiliary functions $s$, $t$, $v$, $y$ at $x_0$ in way consistent with the definition \eqref{eq:u-from-y-s} of $u(x)$. Different consistent choices lead to different jump matrices, but the jump matrices determine the same function $u(x)$ via \eqref{eq:u-recover-general}. This symmetry is reflected at the level of $\mathbf{\Psi}(\lambda,x)$ by the conjugation
 $\mathbf{\Psi}(\lambda,x)\mapsto \delta^{-\sigma_3}\mathbf{\Psi}(\lambda,x)\delta^{\sigma_3}$
 for any $\delta\neq 0$. Another symmetry that also leaves $u(x)$ invariant but changes the jump matrices $\mathbf{C}_{0\infty}^\pm$ is multiplication of $\mathbf{\Psi}(\lambda,x)$ on the right for $|\lambda|<1$ only by a~unit-determinant diagonal matrix. Therefore, having obtained the jump matrices for the inverse monodromy problem via a direct monodromy calculation, after the fact we may introduce an arbitrary transformation of $\mathbf{\Psi}(\lambda,x)$ of the form
\begin{equation}\label{eq:transf}
 \boldsymbol \Psi(\lambda,x)\mapsto \widetilde{\mathbf{\Psi}}(\lambda,x):=\begin{cases}
 \delta^{-\sigma_3} \boldsymbol \Psi(\lambda,x)\gamma^{\sigma_3},&|\lambda|<1,\\
 \delta^{-\sigma_3}\boldsymbol \Psi(\lambda,x)\delta^{\sigma_3},&|\lambda|>1
 \end{cases}
\end{equation}
without changing $u(x)$. This transformation modifies the Stokes matrices as follows:
 \begin{equation}\label{eq:stokes-transformation}
 \mathbf{S}^\infty_{1,2}\mapsto \widetilde{\mathbf{S}}^\infty_{1,2}:=\delta^{-\sigma_3}\mathbf{S}^\infty_{1,2}\delta^{\sigma_3}\qquad\text{and}\qquad
 \mathbf{S}^0_{1,2}\mapsto \widetilde{\mathbf{S}}^0_{1,2}:=\gamma^{-\sigma_3}\mathbf{S}^0_{1,2}\gamma^{\sigma_3}
 \end{equation}
 and it modifies the connection matrices as
 \begin{equation}
 \mathbf{C}_{0\infty}^\pm\mapsto\widetilde{\mathbf{C}}_{0\infty}^\pm:=\delta^{-\sigma_3}\mathbf{C}_{0\infty}^\pm\gamma^{\sigma_3}.
 \label{eq:connection-modify}
 \end{equation}

\subsection[Monodromy parameters (e\_1, e\_2)]{Monodromy parameters $\boldsymbol{(e_1, e_2)}$}\label{subsec:cyclic}
The cyclic products of the jump matrices for the inverse monodromy problem about the two non-singular self-intersection points of the jump contour $\lambda=\pm\ii$ read
 \begin{gather}
 \text{about } \lambda=+\ii \colon\ (\mathbf{C}_{0\infty}^-)^{-1}(\mathbf{S}^\infty_{1})^{-1}\mathbf{C}_{0\infty}^+\mathbf{S}^0_{1}=\mathbb{I},\nonumber\\
 \text{about } \lambda=-\ii\colon\
 \mathbf{S}_{2}^\infty e_\infty^{2\sigma_3}\mathbf{C}_{0\infty}^+e_0^{2\sigma_3}\big(\mathbf{S}^0_{2}\big)^{-1}(\mathbf{C}_{0\infty}^-)^{-1}=\mathbb{I}.
 \label{eq:cyclic-condition}
 \end{gather}
 We can use the second relation to explicitly write $\mathbf{C}_{0\infty}^+$ in terms of two Stokes matrices and the other connection matrix:
 \begin{equation} \mathbf{C}_{0\infty}^+=e_\infty^{-2\sigma_3}(\mathbf{S}_{2}^\infty )^{-1}\mathbf{C}_{0\infty}^-\mathbf{S}^0_{2}e_0^{-2\sigma_3}.
 \label{eq:right-connection-formula}
 \end{equation}
This identity is an analog of \cite[equation~(3.17)]{Jimbo}.
 Under the condition that $\det \mathbf{C}_{0\infty}^- = 1$, we immediately get that $\det \mathbf{C}_{0\infty}^+ = 1$. Furthermore, using \eqref{eq:right-connection-formula} we eliminate $\mathbf{C}_{0\infty}^+$ from the first equation of \eqref{eq:cyclic-condition} to obtain the identity
 \begin{equation}
 \big(\mathbf{S}^0_{1}\big)^{-1}e_0^{2\sigma_3}\big(\mathbf{S}_{2}^0 \big)^{-1} = \big(\mathbf{C}_{0\infty}^-\big)^{-1}\big(\mathbf{S}^\infty_{1}\big)^{-1}e_\infty^{-2\sigma_3}(\mathbf{S}_{2}^\infty )^{-1}\mathbf{C}_{0\infty}^-.
 \label{eq:left-connection-identity}
 \end{equation}

 In other words, $ (\mathbf{S}^\infty_{1})^{-1}e_\infty^{-2\sigma_3}(\mathbf{S}_{2}^\infty )^{-1}$ and $\big(\mathbf{S}^0_{1}\big)^{-1}e_0^{2\sigma_3}\big(\mathbf{S}_{2}^0 \big)^{-1}$ are similar unit-determinant matrices. Note that this is merely reflective of the fact that both products are monodromy matrices, possibly expressed in terms of different bases of fundamental solutions, for a simple circuit about the origin for solutions of the system \eqref{eq:generic-system}. Let us assume that they have distinct eigenvalues that we will denote $e_1^{\pm 2}$. Then, both products are diagonalizable, so there exist unit-determinant eigenvector matrices $\mathbf{E}^\infty$ and $\mathbf{E}^0$ such that
 \begin{equation}
 (\mathbf{S}^\infty_{1})^{-1}e_\infty^{-2\sigma_3}(\mathbf{S}_{2}^\infty )^{-1}\mathbf{E}^\infty=\mathbf{E}^\infty e_1^{2\sigma_3}\qquad\text{and}\qquad
 \big(\mathbf{S}^0_{1}\big)^{-1}e_0^{2\sigma_3}\big(\mathbf{S}_{2}^0 \big)^{-1}\mathbf{E}^0=\mathbf{E}^0 e_1^{2\sigma_3}.
 \label{eq:Stokes-products-eigenvectors}
 \end{equation}
To specify the eigenvector matrices $\mathbf{E}^\infty$, $\mathbf{E}^0$ uniquely, we agree that their (2,2) entries are both equal to $1$.

 Using \eqref{eq:Stokes-products-eigenvectors} in \eqref{eq:left-connection-identity} gives a homogeneous linear equation on $\mathbf{C}_{0\infty}^-$ that can be written in commutator form as
 \begin{equation*}
 \big[e_1^{2\sigma_3},(\mathbf{E}^\infty)^{-1}\mathbf{C}_{0\infty}^-\mathbf{E}^0\big]=\mathbf{0}.
 \end{equation*}
 The diagonal matrix $e_1^{2\sigma_3}$ can be written in the form
 \begin{equation*}
 e_1^{2\sigma_3}=f\mathbb{I}+g\sigma_3,\qquad f:= \frac{1}{2}\big(e_1^2+e_1^{-2}\big),\qquad g:=\frac{1}{2}\big(e_1^2-e_1^{-2}\big).
 \end{equation*}
 Under the assumption $e_1^4\neq 1$ we already invoked to obtain diagonalizability, $g\neq 0$ so the commutator equation implies that $(\mathbf{E}^\infty)^{-1}\mathbf{C}_{0\infty}^-\mathbf{E}^0$ is a diagonal unit-determinant matrix that we may write in the form $e_2^{\sigma_3}$. Thus we have the identity
 \begin{equation}
 \label{eq:C-zero-infty-diagonalization}
 \mathbf{C}_{0\infty}^-=\mathbf{E}^\infty e_2^{\sigma_3}\big(\mathbf{E}^0\big)^{-1}.
 \end{equation}
 \begin{Remark}\label{remark:change-sign-e2}
 Changing the sign of $e_2$ changes the sign of the connection matrix. This corresponds to multiplication of the solution of Riemann--Hilbert Problem \ref{rhp:initial} by $-1$ inside of unit disc. Looking at formula \eqref{eq:u-recover-general}, we see that solution $u(x)$ does not change after such transformation. Therefore, we can assume $-\frac{\pi}{2}<\arg(e_2)\leq\frac{\pi}{2}$.
 \end{Remark}

 Using \eqref{eq:C-zero-infty-diagonalization} in \eqref{eq:right-connection-formula} then gives the equivalent representations
 \begin{equation}
 \mathbf{C}_{0\infty}^+
 =e_\infty^{-2\sigma_3}(\mathbf{S}_{2}^\infty )^{-1}\mathbf{E}^\infty e_2^{\sigma_3}\big(\mathbf{E}^0\big)^{-1}\mathbf{S}^0_{2}e_0^{-2\sigma_3}
 =\mathbf{S}_{1}^\infty \mathbf{E}^\infty e_2^{\sigma_3}\big(\mathbf{E}^0\big)^{-1}\big(\mathbf{S}^0_{1}\big)^{-1}.
 \label{eq:C-zero-infty-plus-diagonalization}
 \end{equation}

 \subsection{Parametrization of Stokes multipliers and connection matrix}
 Taking the trace of \eqref{eq:left-connection-identity}, we get
\begin{equation*}
e_1^2+\frac{1}{e_1^2}=e_\infty^2+\frac1{e_\infty^2}+{s^\infty_1 s^\infty_2}{e_\infty^2}=e_0^2+\frac1{e_0^2}+\frac{s^0_1s^0_2}{e_0^2}.
\end{equation*}
It is clear that one can solve for the products $s^\infty_1s^\infty_2$ and $s_1^0s_2^0$ in terms of $\big(e_1^2,e_\infty^2\big)$ and $\big(e_1^2,e_0^2\big)$ respectively. Using the transformation \eqref{eq:stokes-transformation}, we can take a particular solution of this relation and hence obtain the Stokes multipliers:
\begin{equation}
\label{eq:stokes-parameters}
 s^\infty_1=\frac{e_\infty^2-e_1^2}{e_1^2e_\infty^4},\qquad s^\infty_2=1-e_1^2e_\infty^2,\qquad s^0_1=\frac{e_1^2-e_0^2}{e_1^2},\qquad s^0_2={e_1^2e_0^2 - 1}.
\end{equation}
With the Stokes matrices specified in this way, the eigenvector matrices $\mathbf{E}^\infty$ and $\mathbf{E}^0$ are uniquely specified as mentioned earlier by taking the $(2,2)$ entry to be $1$ in each case, which yields
\begin{gather}
\mathbf{E}^{\infty}=\left[
\begin{matrix}
 \dfrac{e_1^2 \big(e_1^2-e_\infty^2\big)}{e_1^4-1} & -\dfrac{1}{e_1^2e_\infty^2} \vspace{1mm} \\
 \dfrac{e_1^2 e_\infty^2\big(e_1^2 e_\infty^2-1\big)}{e_1^4-1} & 1
\end{matrix}
\right],\label{eq:E-infinity-formula}\\
\mathbf{E}^{0}=\left[
\begin{matrix}
 \dfrac{e_1^2 \big(e_0^2 e_1^2-1\big)}{e_0^2 \big(e_1^4-1\big)} & \dfrac{e_1^2-e_0^2}{e_1^2 \big(e_0^2 e_1^2-1\big)} \vspace{1mm}\\
 \dfrac{e_1^2 \big(1 - e_0^2 e_1^2\big)}{e_0^2 \big(e_1^4-1\big)} & 1
\end{matrix}
\right].\label{eq:E-zero-formula}
\end{gather}
After making such choices, we obtain the formul\ae\ \eqref{eq:x1}--\eqref{eq:x3}. At this point, it can be directly checked that our choices are consistent with the full equation \eqref{eq:left-connection-identity} with $\mathbf C_{0\infty}^-$ given by \eqref{eq:C-zero-infty-diagonalization}.

One can think of fixing the $(2,2)$ entry in the following way: the eigenvector matrices $\mathbf{E}^\infty$ and~$\mathbf{E}^0$ represent ``internal degrees of freedom'' that have an additional symmetry, namely, arbitrary scalings of the eigenvectors that preserve determinants. In other words, while \eqref{eq:transf} induces a conjugation symmetry on the eigenvector matrices, there is an additional symmetry for each involving multiplication on the right by an arbitrary unit-determinant diagonal matrix. Thus, the matrices $\mathbf{E}^\infty$ and $\mathbf{E}^0$ undergo the transformations
 \begin{equation*}
 \mathbf{E}^{\infty}\mapsto \widetilde{\mathbf{E}}^\infty:=\delta^{-\sigma_3}\mathbf{E}^\infty\delta^{\sigma_3}\epsilon_\infty^{\sigma_3}\qquad\text{and}\qquad
 \mathbf{E}^{0}\mapsto \widetilde{\mathbf{E}}^0:=\gamma^{-\sigma_3}\mathbf{E}^0\gamma^{\sigma_3}\epsilon_0^{\sigma_3}
 \end{equation*}
 for some arbitrary nonzero quantities $\epsilon_\infty$, $\epsilon_0$. Note that these transformations along with~\eqref{eq:connection-modify} and $e_2^{\sigma_3}=(\mathbf{E}^\infty)^{-1}\mathbf{C}_{0\infty}^-\mathbf{E}^0$ imply that
 \begin{equation*}
 e_2\mapsto\widetilde{e}_2:=\frac{\gamma\epsilon_0}{\delta\epsilon_\infty}e_2.
 \end{equation*}
 By contrast, $e_1^2\mapsto\widetilde{e}_1^2:=e_1^2$ is a symmetry invariant.

\begin{Remark} \label{rmk:general-arg-x}
 In the case where one is interested in values of $x \in \C$ with $|{\arg} (x)| < \pi$, the analogue of Figure \ref{fig:1} is shown in Figure \ref{fig:rotated-contour}, where the nonsingular self-intersection points are at $\lambda=\pm \ii \ee^{\pm \ii \arg(x)}$ (independent $\pm$ signs).
 \begin{figure}
 \centering
 \includegraphics{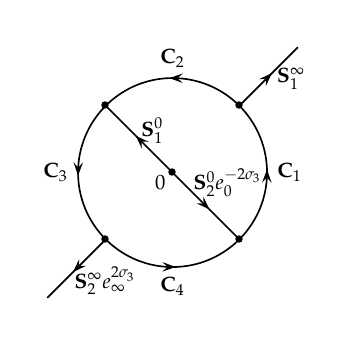}
 \caption{The analogue of the contour $L^{(6)}$ in Figure \ref{fig:1} when $|{\arg}(x)| \neq 0$.}
 \label{fig:rotated-contour}
 \end{figure}
 The angles of the rays in the contour $L^{(6)}$ are chosen so that $\ii \lambda x \in \R$ on the rays extending to $\lambda = \infty$, and $\ii \lambda^{-1} x \in \R$ on the rays extending to $\lambda = 0$. Similar to Section~\ref{sec:initial-rhp}, one can formulate a Riemann--Hilbert problem for a sectionally analytic function~$\lambda\mapsto\mathbf \Psi(\lambda, x)$ off of the contour $L^{(6)}$ and one finds \emph{four} connection matrices instead of two, denoted $\mathbf{C}_1$ through~$\mathbf{C}_4$, defined on corresponding arcs of the unit circle as shown in Figure~\ref{fig:rotated-contour}. These satisfy cyclic conditions similar to~\eqref{eq:cyclic-condition}, namely,
 \begin{alignat*}{3}
 & \text{about}\ \lambda=\ii\ee^{-\ii \arg(x)}\colon\ &&\mathbf{C}_1^{-1}\mathbf{S}_1^\infty \mathbf{C}_2= \mathbb{I}, &\\
 & \text{about} \ \lambda=-\ii\ee^{\ii \arg(x)}\colon\ &&\mathbf{C}_1\big(\mathbf{S}_2^0e_0^{-2\sigma_3}\big)^{-1} \mathbf{C}_4^{-1} = \mathbb{I}, &\\
 & \text{about} \ \lambda=-\ii\ee^{-\ii \arg(x)}\colon\ &&\mathbf{C}_3^{-1} \mathbf{S}_2^\infty e_\infty^{2\sigma_3} \mathbf{C}_4 = \mathbb{I},&\\
 & \text{about}\ \lambda=\ii\ee^{\ii \arg(x)}\colon\ &&\mathbf{C}_3 \big(\mathbf S_1^0\big)^{-1} \mathbf C_2^{-1} = \mathbb{I}.&
 \end{alignat*}
 Eliminating all but $\mathbf{C}_3$ from the above identities yields the analog of \eqref{eq:left-connection-identity}, namely,
 \[
 \big(\mathbf S_1^0\big)^{-1} e_0^{2\sigma_3} \big(\mathbf S_2^0\big)^{-1} = \mathbf C_3^{-1} (\mathbf S_1^\infty)^{-1} e_\infty^{-2\sigma_3} (\mathbf S_2^\infty)^{-1} \mathbf C_3.
 \]
 Reasoning similar to that of Section \ref{subsec:cyclic} yields
 \begin{equation}
 \mathbf{C}_3 = \mathbf{E}^\infty e_2^{\sigma_3} \big(\mathbf{E}^0\big)^{-1},
 \label{eq:general-connection-factorization-1}
 \end{equation}
 which, in turn, yields
 \begin{gather}
 \mathbf{C}_1 = \mathbf S_1^\infty \mathbf{E}^\infty e_2^{\sigma_3} \big(\mathbf{E}^0\big)^{-1} \big(\mathbf S_1^0\big)^{-1}, \qquad
 \mathbf{C}_2 = \mathbf{E}^\infty e_2^{\sigma_3} \big(\mathbf{E}^0\big)^{-1} \big(\mathbf S_1^0\big)^{-1},\nonumber \\
 \mathbf{C}_4 = e_\infty^{-2\sigma_3}\big(\mathbf{S}_2^\infty\big)^{-1}\mathbf{E}^\infty e_2^{\sigma_3} \big(\mathbf{E}^0\big)^{-1} = \mathbf S_1^\infty \mathbf E^\infty e_1^{2\sigma_3} e_2^{\sigma_3} (\mathbf{E}^0)^{-1}.
 \label{eq:general-connection-factorization-2}
 \end{gather}
 In this setting, we must adjust our choice of the branch $\lambda\mapsto \mathrm{arg}_{\lw}(\lambda)$, and we choose a branch which satisfies (cf.\ \eqref{eq:lw-branch-x-real} when $\arg(x) = 0$)
 \begin{equation*} 
 -\dfrac{\pi}{2} - \arg(x) < \mathrm{arg}_{\lw} (\lambda) < \dfrac{3\pi}{2} - \arg(x), \qquad |\lambda| \to \infty,
 \end{equation*}
 and
 \begin{equation*} 
 -\dfrac{\pi}{2} + \arg(x) < \mathrm{arg}_{\lw} (\lambda) < \dfrac{3\pi}{2} + \arg(x), \qquad |\lambda| \to 0.
 \end{equation*}
 A concrete branch cut is chosen later, see Remark \ref{rmk:rotated-lenses} below.
\end{Remark}

\subsection[Example: rational solutions of Painlev\'e-III(D\_6)]{Example: rational solutions of Painlev\'e-III($\boldsymbol{D_6}$)}\label{sec:rational-solutions-parameters}
One can check that the Painlev\'e-III($D_6$) equation with parameters related by $\Theta_0=\Theta_\infty-1$ admits the constant solution $u(x)\equiv 1$. Its monodromy data was calculated in \cite[Section 4]{BMS18} by taking advantage of the fact that the compatible $x$-equation \eqref{eq:generic-deformation} in the Lax pair has simple coefficients. Denoting $m=\Theta_0=\Theta_\infty-1$ gives $e_0^{-2}=\ee^{-\ii\pi m}$ and $e_\infty^2=-\ee^{\ii\pi m}$. Choosing $\gamma$, $\delta$ in~\eqref{eq:transf} satisfying
\[ 
\delta^2 = \ee^{-2\pi \ii m} \big(1 - \ii \ee^{\pi \ii m}\big)\dfrac{\Gamma\big(\frac{1}{2} - m\big)}{\sqrt{2\pi}} \qquad \text{and} \qquad \gamma^2 = \big(1 + \ii \ee^{\pi \ii m}\big)\dfrac{\Gamma\big(\frac{1}{2} - m\big)}{\sqrt{2\pi}},
\]
one obtains
 \begin{alignat*}{3}
& s^0_1=\frac{\sqrt{2\pi}}{\Gamma\big(\tfrac{1}{2}-m\big)},\qquad&&
 s^0_2=-\ee^{\ii\pi m}\frac{\sqrt{2\pi}}{\Gamma\big(\tfrac{1}{2}+m\big)},&\\
& s^\infty_1=-\frac{\sqrt{2\pi}}{\Gamma\big(\tfrac{1}{2}-m\big)},\qquad&&
 s^\infty_2=\ee^{-\ii\pi m}\frac{\sqrt{2\pi}}{\Gamma\big(\tfrac{1}{2}+m\big)}.&
 \end{alignat*}
 With this choice of $\gamma$, $\delta$ and $\mathbf{E}^\infty$, $\mathbf{E}^0$ chosen as in \eqref{eq:E-infinity-formula}, \eqref{eq:E-zero-formula} (that is, we insist that the $(2,2)$ entry of $\mathbf{E}^\infty$, $\mathbf{E}^0$ is 1 by setting $\epsilon_0 = \epsilon_\infty = 1$), the connection matrices are
 \begin{equation*}
 \mathbf{C}_{0\infty}^+=\begin{bmatrix}1&-\dfrac{\sqrt{2\pi}}{\Gamma\big(\tfrac{1}{2}-m\big)} \\0&1\end{bmatrix},\qquad\mathbf{C}_{0\infty}^-=\begin{bmatrix}1&\dfrac{\sqrt{2\pi}}{\Gamma\big(\tfrac{1}{2}-m\big)} \\0&1\end{bmatrix},
 \end{equation*}
 and
 \begin{equation*}
 e_1^2=\ii \qquad \text{and} \qquad e_2 = 1.
 \end{equation*}
\begin{Remark}
 The above gauge is only needed to match our setup with that of \cite{BMS18}; in the sequel we will be working with $\gamma = \delta = 1$. Formula \eqref{eq:C-zero-infty-diagonalization} then implies
 \[
 e_2^2 = \ee^{-2\pi \ii m}\dfrac{1 - \ii \ee^{\pi \ii m}}{1 + \ii \ee^{\pi \ii m}}.
 \]
 This is important to note when, for example, one tries to verify that \eqref{eq:u-n-leading} below reduces to~\eqref{eq:u-n-rat-leading}.
\end{Remark}

Before beginning to study the large $n$ behavior of $u_n$, we must first establish a similar monodromy representation of the limiting solution of Painlev\'e-III($D_8$), which we do in Section~\ref{sec:monodromy-rep-$D_8$} below.

\subsection{Monodromy manifold}
\label{subsec:monodromy-D6}
It is known that the monodromy manifold for Painlev\'e-III($D_6$) can be given by a cubic equation (see, e.g., \cite{PS}), which can be recovered from our point of view as follows. Denote
\begin{equation*}
 \mathbf{C}_{0\infty}^-=\begin{bmatrix}
 \ell_1&\ell_2\\\ell_3&\ell_4
 \end{bmatrix}
\end{equation*}
and
\begin{equation*}
 \mathbf{S}^0_{2}e_0^{-2\sigma_3}\mathbf{S}_{1}^0 =e_0^{-2} \begin{bmatrix}
 1& s^0_1\\ s^0_2& \left(e_0^4+s^0_1s^0_2 \right)
 \end{bmatrix}=\begin{bmatrix}
 m_1&m_2\\m_3&m_4
 \end{bmatrix}.
\end{equation*}
Then, the inverse of the cyclic relation \eqref{eq:left-connection-identity} allows us to solve for $s_1^\infty$, $s_2^\infty$ in terms of parameters $m_i$, $\ell_i$, and imposes the constraint
\begin{equation}\label{eq:beta}
 e_\infty^2=\ell_1\ell_4m_1-\ell_1\ell_3m_2+\ell_2\ell_4m_3-\ell_2\ell_3m_4.
\end{equation}
Hence, we are left with these eight parameters subject to the constraint \eqref{eq:beta} and the unit-determinant conditions
\begin{equation}\label{eq:det}
 \ell_1\ell_4-\ell_2\ell_3=1,\qquad m_1m_4-m_2m_3=1.
\end{equation}
We may define coordinates which are invariant under the transformation \eqref{eq:transf}:
\begin{equation*}
 I_1:= \ell_1\ell_4,\qquad I_2 := m_2\ell_1\ell_3,\qquad I_3 :=m_3\ell_2\ell_4,\qquad I_4:= m_4, \qquad I_5:= m_1.
 \end{equation*}
Equations \eqref{eq:beta}, \eqref{eq:det} imply
\begin{equation*}
 e_\infty^2 =e_0^{-2} I_1-I_2+I_3-I_4(I_1-1),\qquad I_2I_3-I_1\big(e_0^{-2} I_4-1\big)(I_1-1)=0.
\end{equation*}
We eliminate $I_3$ and get
\begin{gather*}
 -I_1+e_0^{-2} I_1 I_4+I_1^2-e_0^{-2} I_4 I_1^2+e_\infty^2 I_2- I_2I_4 -e_0^{-2}I_1I_2+I_1I_2I_4+I_2^2=0.
\end{gather*}
Introducing new variables
\begin{equation}
x_1:=I_1-1,\qquad x_2:=-e_0^{-2} I_1+I_2,\qquad x_3:=I_4+e_0^{-2}
\label{eq:x_i-def}
\end{equation}
yields the following equation, which defines the \emph{monodromy manifold} for the problem
\begin{equation}
 \label{eq:cubic} x_1x_2x_3+x_1^2+x_2^2+x_2\big(e_0^{-2}+e_\infty^2\big)+x_1\big(1+e^{-2}_0 e^{2}_\infty\big)+e^{-2}_0e^{2}_\infty=0.
\end{equation}
This matches \eqref{eq:cubic-D6} upon using \eqref{eq:thetas-alpha-beta} and \eqref{eq:e0-einfty}. Using \eqref{eq:x_i-def}, we obtain formul\ae\ \eqref{eq:x1}--\eqref{eq:x3}.

To find the singularities of \eqref{eq:cubic}, we adjoin to \eqref{eq:cubic} the three equations obtained by setting to zero the components of the gradient vector of the left-hand side of \eqref{eq:cubic} with respect to $(x_1,x_2,x_3)$. There is therefore at most one singularity:
\begin{alignat}{3}\label{eq:crit1}
 &\text{for}\ e^{-2}_0=e^2_\infty\colon\ &&(x_1,x_2,x_3)=\big(0,-e^{-2}_0,e^{2}_0+e_0^{-2}\big),&\\
 &\text{for}\ e^{2}_0=e_\infty^{2}\colon\ &&(x_1,x_2,x_3)=\bigl(-1,0,e^{2}_0+e_0^{-2}\bigr).&\label{eq:crit2}
\end{alignat}
In particular, if neither $e^{-2}_0=e^2_\infty$ nor $e^{2}_0=e^2_\infty$, then the monodromy manifold is a smooth curve with no singular points.
Notice that we can use $(x_1,x_2)$ as parameters for the generic collection of points on monodromy manifold \eqref{eq:cubic} for which $x_1x_2\neq 0$, because $x_3$ can be explicitly expressed in terms of the other coordinates. The points satisfying \eqref{eq:cubic} with $x_1=0$ form a 1-dimensional variety consisting in general of two distinct lines:
\begin{equation*}
 (x_1,x_2,x_3)=\big(0,-e^{-2}_0,x_3\big)\qquad\text{or}\qquad (x_1,x_2,x_3)=\big(0,-e^2_\infty,x_3\big)
\end{equation*}
each parametrized by $x_3\in\mathbb{C}$. If $e_0^{-2}=e_\infty^2$, the two lines coincide and pass through the critical point \eqref{eq:crit1} of \eqref{eq:cubic}. Likewise there are generally two lines on \eqref{eq:cubic} along which $x_2=0$ each parametrized by $x_3\in\mathbb{C}$:
\begin{equation*}
 (x_1,x_2,x_3)=(-1,0,x_3)\qquad\text{or}\qquad (x_1,x_2,x_3)=\bigl(-e^{-2}_0e^2_\infty,0,x_3\bigr)
\end{equation*}
and if $e^{2}_0=e_\infty^{2}$, the two lines again coincide and pass through the critical point \eqref{eq:crit2} of \eqref{eq:cubic}.

\section[General monodromy data: Painlev\'e-III(D\_8)]{General monodromy data: Painlev\'e-III($\boldsymbol{D_8}$)} \label{sec:monodromy-rep-$D_8$}

\subsection[Lax pair for Painlev\'e-III(D\_8)]{Lax pair for Painlev\'e-III($\boldsymbol{D_8}$)}
\label{sec:lax-pair-$D_8$}

The Painlev\'e-III($D_8$) equation \eqref{eq:PIII-$D_8$} can also be formulated as an isomonodromic deformation of a linear system. In this case we need two ramified irregular singularities at $\lambda = 0$ and $\lambda = \infty$, i.e., we consider the system
\begin{align}\label{eq:lax_pair_D8}
\frac{\partial\mathbf{\Omega}}{\partial\lambda} (\lambda, z) &= {\mathbf{\Lambda}}^{(8)}(\lambda, z) \mathbf{\Omega}(\lambda, z), \\
\frac{\partial\mathbf{\Omega}}{\partial z} (\lambda, z) &= {\mathbf{Z}}(\lambda, z) \mathbf{\Omega}(\lambda, z),
\label{eq:lax_pair_D8_2}
\end{align}
where
\begin{gather*}
{\mathbf{\Lambda}}^{(8)}(\lambda, z) = \begin{bmatrix}0&{\ii z}{} \\0&0\end{bmatrix} +\dfrac{1}{4\lambda}\begin{bmatrix}V(z) & {W}(z)\\2 & -V(z)\end{bmatrix} + \dfrac{1}{\lambda^2}\begin{bmatrix}X(z) & -2\ii X(z)^2U(z)\\-\ii/(2U(z)) & -X(z)\end{bmatrix},
\\
{\mathbf{Z}}(\lambda, z) = \lambda\begin{bmatrix}0& {\ii} \\0&0\end{bmatrix} + \frac{1}{4z}\begin{bmatrix}V(z) & W(z)\\2 & -V(z)\end{bmatrix} - \frac{1}{z\lambda}\begin{bmatrix}X(z) & -2\ii X(z)^2U(z)\\-\ii/(2U(z)) & -X(z)\end{bmatrix},
\end{gather*}
and functions $U(z)$, $V(z)$, $W(z)$, $X(z)$ satisfy the identities
\begin{gather}
 W(z) +4zU(z) +4\ii U(z)V(z)X(z) + 8 U(z )^2X(z)^2 =0,
 \label{eq:UVWXidentity1}\\
 U(z) V(z)^2-4 U(z) V(z)+2 U(z) W(z)+3 U(z)+8 z=0.\label{eq:UVWidentity1}
\end{gather}
Note the characteristic feature that the leading terms of $\mathbf{\Lambda}^{(8)}(\lambda,z)$ and of $\mathbf{Z}(\lambda,x)$ at the singular points $\lambda=0,\infty$ are singular and nondiagonalizable matrices.

Since $\mathbf{\Omega}(\lambda,z)$ is a simultaneous fundamental solution matrix of the Lax system \eqref{eq:lax_pair_D8}--\eqref{eq:lax_pair_D8_2}, the zero-curvature compatibility condition for that system is therefore satisfied:
\begin{gather*}
 \frac{\partial\mathbf{\Lambda}^{(8)}}{\partial z}(\lambda,z)-\frac{\partial\mathbf{Z}}{\partial\lambda}(\lambda,z)+\big[\mathbf{\Lambda}^{(8)}(\lambda,z),\mathbf{Z}(\lambda,z)\big]=\mathbf{0}.
\end{gather*}
Equating to zero the coefficients of different powers of $\lambda$ on the left-hand side gives a first-order system of four differential equations on the four functions $U(z)$, $V(z)$, $W(z)$, and $X(z)$:
\begin{gather}
 zU'(z)=V(z)U(z)-U(z)-4\ii X(z)U(z)^2,\qquad
 V'(z)=\frac{4}{U(z)},\nonumber\\
 W'(z)=-16\ii X(z),\qquad
 zX'(z)=X(z)+2\ii U(z)X(z)^2-\frac{\ii W(z)}{4U(z)}.
 \label{eq:PossiblyD8system}
\end{gather}
It is possible to express the functions $W(z)$, $X(z)$, and $V(z)$ in terms of $U(z)$ and $U'(z)$ using \eqref{eq:UVWXidentity1}, \eqref{eq:UVWidentity1}, and \eqref{eq:PossiblyD8system}, but since we do not use these formul\ae, we do not present them here. Using \eqref{eq:PossiblyD8system} to repeatedly eliminate all derivatives, it is straightforward to obtain the following identity:
\begin{gather*}
 U''(z)-\frac{U'(z)^2}{U(z)}+\frac{U'(z)}{z} -\frac{4U(z)^2+4}{z}\\
 \qquad=-\frac{U(z)}{z^2}\big[ W(z)+4 z U(z) +4 \ii U(z) V(z) X(z) + 8 U(z)^2 X(z)^2\big].
\end{gather*}
Of course the right-hand side vanishes as a result of the identity \eqref{eq:UVWXidentity1}. Hence $U(z)$ is a solution of \eqref{eq:PIII-$D_8$}, the Painlev\'e-III($D_8$) equation.

For all the calculations that follow, we assume for simplicity that $z > 0$. The system \eqref{eq:lax_pair_D8} admits formal solutions near the singular points
\begin{equation}\label{eq:Omega-expand-infty-early}
\mathbf{\Omega}_{\text{formal}}^{(\infty)}(\lambda,z) = \bigg( \mathbb{I} + \frac{\mathbf{\Xi}^{(8)}(z)}{\lambda}+\mathcal{O}\big(\lambda^{-2}\big) \bigg) \rho_\infty^{\sigma_3/2}
 {\dfrac{1}{\sqrt{2}}} \begin{bmatrix} \ii & -1\\ 1 &-\ii \end{bmatrix} \ee^{\ii \rho_\infty \sigma_3} \qquad \text{as} \quad \lambda \to \infty,
\end{equation}
and
\begin{align}
\mathbf{\Omega}^{(0)}_{\text{formal}}(\lambda, z) ={}& {\mathbf{\Delta}}^{(8)}(z)\big( \mathbb{I} +\mathbf{\Pi}(z)\lambda+ \mathcal{O}\big(\lambda^2\big) \big)\nonumber\\
&\times\rho_0^{\sigma_3/2} \ee^{-\pi \ii \sigma_3/4}{\dfrac{1}{\sqrt{2}}} {}{} \begin{bmatrix} \ii & -1 \\ 1 &-\ii \end{bmatrix} \ee^{ \rho_0 \sigma_3} \qquad \text{as} \quad \lambda \to 0,\label{eq:Omega-expand-zero-early}
\end{align}
where
\begin{equation*}
\rho_\infty=\sqrt{-2 \ii z\lambda},\qquad \rho_0= \sqrt{2 \ii z\lambda^{-1}}
\end{equation*}
 and the square roots denote principal branches.
The function ${\mathbf{\Delta}}^{(8)}(z)$ satisfies the identity
\begin{equation}
{\mathbf{\Delta}}^{(8)}(z)\begin{bmatrix}0&-\ii z\\0&0\end{bmatrix}{\mathbf{\Delta}}^{(8)}(z)^{-1}=\begin{bmatrix}X(z) & -2\ii X(z)^2U(z)\\-\ii/(2U(z)) & -X(z)\end{bmatrix},
\label{eq:Delta-U-X-identity}
\end{equation}
and hence the solution $U(z)$ can be expressed as
\begin{equation}\label{eq:un-U-early}
U(z):=-\frac{1}{2z\Delta^{(8)}_{21}(z)^2}.
\end{equation}
For $k = 1, 2$, we define the Stokes sectors,
\begin{align*}
&\mathcal S_k^{(\infty)}= \left \{ \lambda \in \C \colon  |\lambda| > R, \, 2\pi k - \frac{7\pi}{2} < \arg (\lambda) < 2\pi k+\frac{\pi}{2} \right \},\\
&\mathcal S_k^{(0)}= \left \{ \lambda \in \C \colon  |\lambda| < r, \, 2\pi k - \frac{5\pi}{2} < \arg (\lambda) < 2\pi k+\frac{3\pi}{2}\right \}.
\end{align*}

It follows from the classical theory of linear systems that there exist canonical solutions \smash{$\mathbf{\Omega}_k^{(\infty)}$},~\smash{$\mathbf{\Omega}_k^{(0)}$} determined uniquely by the asymptotic condition
\begin{equation}
\mathbf{\Omega}_k^{(\nu)}(\lambda, z) = \mathbf{\Omega}^{(\nu)}_{\text{formal}}(\lambda, z), \qquad \lambda \in \mathcal S_k^{(\nu)}, \quad \nu \in \{0, \infty\}.
\label{eq:psi-asymptotic-condition-D8}
\end{equation}
The canonical solutions in consecutive Stokes sectors at $\lambda=0,\infty$ are related to one another by multiplications on the right with Stokes matrices, i.e.,
\begin{gather}
\label{eq:stokes-condition-D8-1}\mathbf{\Omega}_2^{(\infty)}(\lambda, z) = \mathbf{\Omega}_1^{(\infty)}(\lambda,z) \mathbf{S}^{\infty}_{1},\qquad \lambda\in \mathcal S_1^{(\infty)}\cap \mathcal S_2^{(\infty)},\\
\mathbf{\Omega}_1^{(0)}(\lambda, z) = \mathbf{\Omega}_0^{(0)}(\lambda, z) \mathbf{S}^{0}_{0}, \qquad \lambda\in \mathcal S_0^{(0)}\cap \mathcal S_1^{(0)}, \\
\mathbf{\Omega}_{2}^{(\infty)}(\lambda, z) = \mathbf{\Omega}_1^{(\infty)}\big(\ee^{-2\pi \ii} \lambda, z\big) (-\ii\sigma_2), \\
\label{eq:stokes-condition-D8-4}\mathbf{\Omega}_{1}^{(0)}(\lambda, z) = \mathbf{\Omega}_0^{(0)}\big(\ee^{-2\pi \ii} \lambda, z\big) (\ii\sigma_2),
\end{gather}
where
\begin{equation}
\mathbf{S}^\infty_{1}= \begin{bmatrix}1&t^\infty_1\\0&1\end{bmatrix},\qquad \mathbf{S}^0_{{0}}=
 \begin{bmatrix}1&0\\{t^0_0}&1\end{bmatrix} .
 \label{eq:Stokes-zero}
\end{equation}
Canonical solutions in, say, $\mathcal S_k^{(\infty)}$ admit analytic continuation into $\mathcal S_k^{(0)}$ and since both canonical solutions solve \eqref{eq:generic-system} in the same domain, they must be related by multiplication on the right by a constant connection matrix, which we define using
\begin{align}
\mathbf{\Omega}_0^{(0)}(\lambda, z) &= \mathbf{\Omega}_1^{(\infty)}(\lambda, z) \mathbf{C}_{0\infty}.
\label{eq:connection-matrix-D8}
\end{align}

\subsection[Riemann--Hilbert problem for Painlev\'e-III(D\_8)]{Riemann--Hilbert problem for Painlev\'e-III($\boldsymbol{D_8}$)} In a fashion similar to Section \ref{sec:initial-rhp}, we now formulate a $2\times 2$ Riemann--Hilbert problem for a~sectionally-analytic function ${\mathbf{\Omega}}$ defined by
\[
{\mathbf{\Omega}}(\lambda, z) =
 \begin{cases}
\mathbf \Omega_1^{(\infty)}(\lambda, z) , & |\lambda| >1 \ \text{and} \ -\dfrac{\pi}{2} < \arg (\lambda) < \dfrac{3\pi}{2}, \vspace{1mm}\\
\mathbf \Omega_0^{(0)}(\lambda, z) , & |\lambda| <1 \ \text{and} \ -\dfrac{\pi}{2} < \arg (\lambda) < \dfrac{3\pi}{2}.
\end{cases}
\]
Then, it follows from the asymptotic conditions \eqref{eq:psi-asymptotic-condition-D8} and the relations \eqref{eq:stokes-condition-D8-1}--\eqref{eq:stokes-condition-D8-4} and \eqref{eq:connection-matrix-D8} that $\mathbf{\Omega}$ solves the following $2 \times 2$ Riemann--Hilbert problem.
\begin{rhp} \label{rhp:D8}
Fix monodromy data $\big( t_0^0, t_1^\infty \big)$ and $z>0$. We seek a $2 \times 2$ matrix function $\lambda\mapsto{\mathbf{\Omega}}(\lambda, z)$ satisfying:
\begin{itemize}\itemsep=0pt
\item Analyticity: ${\mathbf{\Omega}} (\lambda, z)$ is analytic in $\C \setminus L^{(8)}$, where $L^{(8)} = \{ | \lambda| = 1\} \cup \ii \R_-$ is the jump contour shown in Figure {\rm\ref{fig:D8}}.
\item Jump condition: ${\mathbf{\Omega}} (\lambda, z)$ has continuous boundary values on $L^{(8)} \setminus \{0\}$ from each component of $\mathbb{C}\setminus L^{(8)}$, which satisfy
\[{\mathbf{\Omega}}_+ (\lambda, z) = {\mathbf{\Omega}}_- (\lambda, z) \mathbf J_{{\mathbf{\Omega}}}(\lambda),\] where $\mathbf J_{{\mathbf{\Omega}}}(\lambda)$ is as shown in Figure {\rm \ref{fig:D8}}.
\item Normalization: ${\mathbf{\Omega}}(\lambda,z)$ satisfies the asymptotic conditions
\begin{equation*}
{\mathbf{\Omega}}(\lambda, z) = \big( \mathbb{I} + \mathcal{O}\big(\lambda^{-1}\big) \big) \rho_\infty^{\sigma_3/2}
 {\dfrac{1}{\sqrt{2}}} \begin{bmatrix} \ii & -1\\ 1 &-\ii \end{bmatrix} \ee^{\ii \rho_\infty \sigma_3} \qquad \text{as} \quad \lambda \to \infty,
\end{equation*}
and
\begin{equation*}
{\mathbf{\Omega}}(\lambda, z) = \big( {\mathbf{\Delta}}^{(8)}(z) + \mathcal{O}(\lambda) \big)\rho_0^{\sigma_3/2} {\dfrac{\ee^{-\frac{\ii\pi \sigma_3} {4}}}{\sqrt{2}}} {}{} \begin{bmatrix} \ii & -1\\ 1 &-\ii \end{bmatrix} \ee^{ \rho_0 \sigma_3} \qquad \text{as} \quad \lambda \to 0,
\end{equation*}
where ${\mathbf{\Delta}}^{(8)}(z)$ is a matrix determined from $\mathbf{\Omega}(\lambda,z)$ having unit determinant.
\end{itemize}
\end{rhp}
Solvability of Riemann--Hilbert Problem \ref{rhp:D8} is discussed in Section~\ref{sec:suleimanov-solution-connection}.

\begin{figure}[t]\centering
\includegraphics[scale = 1]{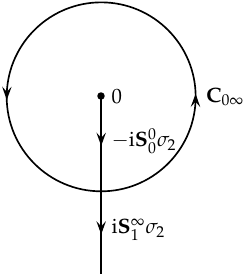}
\caption{The jump contour $L^{(8)}$ and definition of $\mathbf J_{{\mathbf{\Omega}}}(\lambda)$ when $z>0$.}\label{fig:D8}
\end{figure}

 \subsection[Lax pair equations for Omega (lambda,z)]{Lax pair equations for $\boldsymbol{\Omega(\lambda,z)}$}\label{sec:Lax-pair}

Since the jump matrices depend on neither $\lambda$ nor $z$, the matrices
\begin{equation}
 \mathbf{\Lambda}^{(8)}(\lambda,z):=\frac{\partial\mathbf{\Omega}}{\partial\lambda}(\lambda,z)\mathbf{\Omega}(\lambda,z)^{-1}\qquad\text{and}\qquad\mathbf{Z}(\lambda,z):=\frac{\partial\mathbf{\Omega}}{\partial z}(\lambda,z)\mathbf{\Omega}(\lambda,z)^{-1}
 \label{eq:Lambda-Z-def}
\end{equation}
are both analytic functions of $\lambda$ in the domain $\mathbb{C}\setminus\{0\}$.
We determine these analytic functions by computing sufficiently many terms in their asymptotic expansions as $\lambda\to\infty$ and $\lambda\to 0$ using \eqref{eq:Omega-expand-infty-early}--\eqref{eq:Omega-expand-zero-early}.
We will use the identities
\begin{equation}
 \frac{\partial\rho_\infty}{\partial\lambda} = -\ii z\rho_\infty^{-1}\qquad\text{and}\qquad\frac{\partial\rho_\infty}{\partial z} = -\ii\lambda\rho_\infty^{-1}
 \label{eq:rhoinfty-derivs}
\end{equation}
and
\begin{equation}
 \frac{\partial\rho_0}{\partial\lambda}=-\ii z\lambda^{-2}\rho_0^{-1}\qquad\text{and}\qquad\frac{\partial\rho_0}{\partial z} = \ii\lambda^{-1}\rho_0^{-1}.
 \label{eq:rhozero-derivs}
\end{equation}
Using \eqref{eq:rhoinfty-derivs} and \eqref{eq:Omega-expand-infty-early} gives, in the limit $\lambda\to\infty$, the expansions
\begin{gather}
 \mathbf{\Lambda}^{(8)}(\lambda,z)= \begin{bmatrix}0&{\ii z}{} \\0&0\end{bmatrix}
 +\frac{1}{4\lambda}\begin{bmatrix}1-{4\ii z}\Xi_{21}^{(8)}(z) & {4\ii z}{}(\Xi_{11}^{(8)}(z)-\Xi_{22}^{(8)}(z)) \vspace{1mm}\\ 2 &-1 +{4\ii z}\Xi_{21}^{(8)}(z)\end{bmatrix}+\mathcal{O}\big(\lambda^{-2}\big),\\
 \mathbf{Z}(\lambda,z)=
 \lambda\begin{bmatrix}0 & {\ii}{} \\0&0\end{bmatrix}
 +\frac{1}{4z}\begin{bmatrix}
 1-{4\ii z}\Xi_{21}^{(8)}(z) &{4\ii z}{}(\Xi_{11}^{(8)}(z)-\Xi_{22}^{(8)}(z)) \vspace{1mm}\\2 & -1+{4\ii z}\Xi_{21}^{(8)}(z)
 \end{bmatrix}
 +\mathcal{O}\big(\lambda^{-1}\big).
\label{eq:Lax-coefficients-infty}
\end{gather}
Actually, we can also go to higher order and compute the coefficient of $\lambda^{-2}$ in the matrix element~$\Lambda_{21}(\lambda,z)$,
in the limit $\lambda\to\infty$:
\begin{align}
\Lambda_{21}^{(8)}(\lambda,z)=\frac{1}{2\lambda} +\frac{1}{2\lambda^2}\bigl(-\Xi_{21}^{(8)}(z)-{2\ii z}{}\Xi_{21}^{(8)}(z)^2-\Xi_{11}^{(8)}(z)+\Xi_{22}^{(8)}(z)\bigr)
+\mathcal{O}\big(\lambda^{-3}\big).
 \label{eq:Lambda21-higher-order}
\end{align}
Likewise, using \eqref{eq:rhozero-derivs} and \eqref{eq:Omega-expand-zero-early} gives that as $\lambda \to 0$
\begin{gather}
 \mathbf{\Lambda}^{(8)}(\lambda,z)=\mathbf{\Delta}^{(8)}(z)\Bigg(\frac{1}{\lambda^2}\begin{bmatrix}0&{-\ii z}{} \\0&0\end{bmatrix}\nonumber \\ \phantom{\mathbf{\Lambda}^{(8)}(\lambda,z)=}{}-\frac{1}{4\lambda}\begin{bmatrix}1-{4\ii z}\Pi_{21}(z)&{4\ii z}{}(\Pi_{11}(z)-\Pi_{22}(z))\\2&-1+{4\ii z}\Pi_{21}(z)\end{bmatrix}\Bigg)\mathbf{\Delta}^{(8)}(z)^{-1}
+ \mathcal{O}(1),\nonumber
 \\
 \mathbf{Z}(\lambda,z)=\frac{1}{\lambda}\mathbf{\Delta}^{(8)}(z)\begin{bmatrix}0&{\ii}{} \\0&0\end{bmatrix}\mathbf{\Delta}^{(8)}(z)^{-1} + \mathcal{O}(1).
\label{eq:Lax-coefficients-zero}
\end{gather}
Applying Liouville's theorem yields the exact expressions
\begin{align}
 & \mathbf{\Lambda}^{(8)}(\lambda,z)= \begin{bmatrix}0&{\ii z}{} \\0&0\end{bmatrix} +\frac{1}{4\lambda}\begin{bmatrix}1-{4\ii z}\Xi^{(8)}_{21}(z) & {4\ii z}{}\big(\Xi^{(8)}_{11}(z)-\Xi^{(8)}_{22}(z)\big) \\ 2 &-1 +{4\ii z}{}\Xi^{(8)}_{21}(z)\end{bmatrix}\nonumber
 \\
 &\hphantom{\mathbf{\Lambda}^{(8)}(\lambda,z)= }{}
 +\frac{\ii z}{\lambda^2}\begin{bmatrix}\Delta^{(8)}_{11}(z)\Delta^{(8)}_{21}(z) & -\Delta^{(8)}_{11}(z)^2 \\
 \Delta^{(8)}_{21}(z)^2 & -\Delta^{(8)}_{11}(z)\Delta^{(8)}_{21}(z)\end{bmatrix},
\label{eq:Lambda-exact}
\end{align}
and
\begin{align*}
 & \mathbf{Z}(\lambda,z)= \lambda\begin{bmatrix}0 & {\ii}{} \\0&0\end{bmatrix}+\frac{1}{4z}\begin{bmatrix}
 1-{4\ii z}{}\Xi^{(8)}_{21}(z) &{4\ii z}{}\big(\Xi^{(8)}_{11}(z)-\Xi^{(8)}_{22}(z)\big) \\2 & -1+{4\ii z}{}\Xi^{(8)}_{21}(z)
 \end{bmatrix} \\
 &\hphantom{\mathbf{Z}(\lambda,z)=}{} -
 \frac{\ii }{ \lambda}\begin{bmatrix}\Delta^{(8)}_{11}(z)\Delta^{(8)}_{21}(z) & -\Delta^{(8)}_{11}(z)^2 \\
 \Delta^{(8)}_{21}(z)^2 & -\Delta^{(8)}_{11}(z)\Delta^{(8)}_{21}(z)\end{bmatrix}.
\end{align*}
Using the notation \eqref{eq:un-U-early} and noting the structure of the coefficients of the different powers of $\lambda$, it is convenient to reparametrize the coefficients as follows:
\begin{equation*}
 \mathbf{\Lambda}^{(8)}(\lambda,z)=\begin{bmatrix}0&{\ii z}{} \\0&0\end{bmatrix} +\frac{1}{4\lambda}\begin{bmatrix}V(z) & W(z)\\2 & -V(z)\end{bmatrix} +\frac{1}{\lambda^2}\begin{bmatrix} X(z) & -2 \ii X(z)^2U(z)\\-\ii/(2U(z)) & -X(z)\end{bmatrix}
\end{equation*}
and
\begin{equation*} \mathbf{Z}(\lambda,z)=\lambda\begin{bmatrix}0&{\ii}{} \\0&0\end{bmatrix} + \frac{1}{4z}\begin{bmatrix}V(z) & W(z)\\2 & -V(z)\end{bmatrix}-\frac{1}{z\lambda}\begin{bmatrix}X(z) & -2\ii X(z)^2U(z)\\-\ii/(2U(z)) & -X(z)\end{bmatrix}.
\end{equation*}
The quantities $U(z)$, $V(z)$, $W(z)$, and $X(z)$ are not independent; comparing the $21$-element of the coefficient of $\lambda^{-1}$ in the expansion of $\mathbf{\Delta}^{(8)}(z)^{-1}\mathbf{\Lambda}^{(8)}(\lambda,z)\mathbf{\Delta}^{(8)}(z)$ computed using \eqref{eq:Lax-coefficients-infty} and~\eqref{eq:Lax-coefficients-zero} gives the identity \eqref{eq:UVWXidentity1}.
At the same time from formula \eqref{eq:Lambda21-higher-order} we get identity \eqref{eq:UVWidentity1}.

Since \eqref{eq:Lambda-Z-def} holds for the same matrix function $\mathbf{\Omega}(\lambda,z)$, the latter satisfies the equations of a compatible Lax system
\begin{equation}
 \frac{\partial\mathbf{\Omega}}{\partial\lambda}(\lambda,z)=\mathbf{\Lambda}^{(8)}(\lambda,z)\mathbf{\Omega}(\lambda,z)\qquad\text{and}\qquad\frac{\partial\mathbf{\Omega}}{\partial z}(\lambda,z)=\mathbf{Z}(\lambda,z)\mathbf{\Omega}(\lambda,z),
\label{eq:Lax-system}
\end{equation}
which coincides with the system \eqref{eq:lax_pair_D8}--\eqref{eq:lax_pair_D8_2}.

\subsection{Monodromy manifold}
\label{sec:monodromy-rep-$D_8$-manifold}
Introducing notation for the connection matrix elements
\begin{equation*}
 \mathbf{C}_{0\infty}=\begin{bmatrix}
 n_1&n_2\\n_3&n_4
 \end{bmatrix},\qquad \det(\mathbf{C}_{0\infty})=1,
\end{equation*}
we have the cyclic relation around the unique nonsingular point of self-intersection of $L^{(8)}$
 \begin{equation*}
\mathbf{S}^\infty_{{1}}\ii\sigma_2= \mathbf{C}_{0\infty}\mathbf{S}^0_{{0}}(-\ii\sigma_2)(\mathbf{C}_{0\infty})^{-1},
\end{equation*}
which implies
\[
n_1=-n_4,\qquad t_1^{\infty}=t_{0}^0,\qquad n_3=n_2-n_4t_1^{\infty}.
\]
Denoting
\[
y_1=n_3,\qquad y_2=n_4,\qquad y_3=t_1^{\infty},
\]
the condition $\det(\mathbf{C}_{0\infty})=1$ implies that the coordinates $(y_1,y_2,y_3)$ are related by the cubic equation \eqref{eq:cubic-D8}.

\begin{Remark}
 If the solution $\mathbf{\Omega}(\lambda,z)$ is multiplied by the scalar $-1$ for $|\lambda|<1$ and left unchanged for $|\lambda|>1$, then the elements of the connection matrix $\mathbf{C}_{0\infty}$ change sign while the Stokes multiplier $t_1^\infty$ is invariant. Therefore, this transformation changes $(y_1,y_2,y_3)$ to $(-y_1,-y_2,y_3)$, yielding a different point on the cubic \eqref{eq:cubic-D8}. The matrix coefficient $\mathbf{\Delta}^{(8)}(z)$ also changes sign, however \smash{$\Delta_{21}^{(8)}(z)^2$} is invariant, so the solution $U(z)$ of the Painlev\'e-III($D_8$) equation \eqref{eq:PIII-$D_8$} is the same for both points.
 \label{rem:minus-y1-y2}
\end{Remark}

\section{Schlesinger transformation and proof of Proposition~\ref{prop:schlesinger}}\label{sec:schlesinger}

Fix generic monodromy parameters $(e_1, e_2)$. In view of the parametrization of the Stokes multipliers in \eqref{eq:stokes-parameters} and the eigenvector matrices in \eqref{eq:E-infinity-formula}, \eqref{eq:E-zero-formula}, this data determines from Riemann--Hilbert Problem~\ref{rhp:initial} a matrix $\mathbf{\Psi}(\lambda, x)$ which is meromorphic in $x$ and satisfies asymptotic conditions \eqref{eq:Psi-asymptotic-infinity} and \eqref{eq:Psi-asymptotic-zero}, which we write in the form\footnote{The coefficients $\mathbf{\Psi}_j^{\infty}$, $\mathbf{\Psi}_j^{0}$ should not be confused with the fundamental solutions discussed in the previous sections. The reader can rest assured that this notation will only appear in this section.}
\begin{align*}
&\mathbf{\Psi}(\lambda, x) \lambda_{\lw}^{\Theta_\infty \sigma_3 /2} \ee^{-\ii x \lambda \sigma_3 /2} = \mathbb{I} + \mathbf{\Psi}_1^{\infty}(x) \lambda^{-1} + \mathcal{O}\big(\lambda^{-2}\big) , \qquad \lambda \to \infty, \\
&\mathbf{\Psi}(\lambda, x) \lambda_{\lw}^{-\Theta_0 \sigma_3 /2} \ee^{\ii x \lambda^{-1} \sigma_3 /2} = \mathbf{\Psi}_0^{0}(x) + \mathbf{\Psi}_1^{0}(x) \lambda + \mathcal{O}\big(\lambda^{2}\big) , \qquad \lambda \to 0.
\end{align*}
Define the matrices
\begin{equation*}
\sigma_+ = \begin{bmatrix} 1 & 0 \\ 0 & 0 \end{bmatrix} \qquad \text{and} \qquad \sigma_- = \begin{bmatrix} 0 & 0 \\ 0 & 1 \end{bmatrix}.
\end{equation*}
Following \cite{BMS18}, assuming the $(1, 1)$ entry of $\mathbf{\Psi}_0^{0}(x)$, denoted ${\Psi}_{0, 11}^{0}(x)$, is not identically zero, we consider the Schlesinger transformation
\begin{equation*}
\hat{\mathbf{\Psi}}(\lambda, x) := \big( \sigma_+ \lambda_{\lw}^{1/2} + \hat{\mathbf{S}}(x) \lambda_{\lw}^{-1/2}\big) \mathbf{\Psi}(\lambda, x),
\end{equation*}
where
\begin{equation*}
\hat{\mathbf{S}} (x):= \begin{bmatrix}
\Psi_{0, 21}^0(x) \Psi_{1, 12}^\infty(x)/\Psi^0_{0,11}(x) & -\Psi_{1, 12}^\infty(x) \vspace{1mm}\\ -\Psi_{0, 21}^0(x)/\Psi_{0, 11}^0(x) & 1
\end{bmatrix}.
\end{equation*}
Since $\lambda_{\lw}^{\pm 1/2}$ has its branch cut along part of the curve $L^{(6)}$, we see that $\hat{\mathbf{\Psi}}(\lambda, x)$ is analytic in $\C \setminus L^{(6)}$ and, by direct calculation, has the jumps on \smash{$L^{(6)}$} summarized by Figure \ref{fig:1}, with the exception of the sign changes
\[
\mathbf{S}_2^0 e_0^{-2\sigma_3} \mapsto -\mathbf{S}_2^0 e_0^{-2\sigma_3} \qquad \text{and} \qquad \mathbf{S}_2^\infty e_\infty^{2\sigma_3} \mapsto -\mathbf{S}_2^\infty e_\infty^{2\sigma_3}.
\]
Furthermore, one can verify using the definition of $\hat{\mathbf{\Psi}}(\lambda, x)$ that
\begin{equation}
\label{eq:psi-hat-asymptotics-infty}
\hat{\mathbf{\Psi}}(\lambda, x) \lambda_{\lw}^{(\Theta_\infty - 1)\sigma_3/2} \ee^{-\ii x \lambda \sigma_3/2}= \mathbb{I} + \hat{\mathbf{\Psi}}^\infty_1(x) \lambda^{-1} + \mathcal{O}\big(\lambda^{-2}\big), \qquad \lambda \to \infty,
\end{equation}
where
\[
\hat{\mathbf{\Psi}}_1^\infty(x) := \sigma_+ \mathbf{\Psi}^\infty_1(x) \sigma_+ + \sigma_+ \mathbf{\Psi}_2^\infty(x) \sigma_- + \hat{\mathbf{S}}(x) \sigma_+ + \hat{\mathbf{S}}(x) \mathbf{\Psi}_1^\infty (x) \sigma_-.
\]
Similarly, one can check that
\begin{equation}
\label{eq:psi-hat-asymptotics-zero}
\hat{\mathbf{\Psi}}(\lambda, x) \lambda_{\lw}^{-(\Theta_0 + 1)\sigma_3/2} \ee^{\ii x \lambda^{-1} \sigma_3/2}= \hat{\mathbf{\Psi}}^0_0(x) + \hat{\mathbf{\Psi}}^0_1(x) \lambda + \mathcal{O}\big(\lambda^{2}\big), \qquad \lambda \to 0,
\end{equation}
where
\[
\hat{\mathbf{\Psi}}^0_0(x) = \hat{\mathbf{S}}(x) \mathbf{\Psi}^0_0(x) \sigma_- + \hat{\mathbf{S}}(x) \mathbf{\Psi}^0_1(x) \sigma_+ + \sigma_+ \mathbf{\Psi}^0_0(x) \sigma_+.
\]
The transformation $\mathbf{\Psi} \mapsto \hat{\mathbf{\Psi}}$ is invertible so long as $\hat{\Psi}_{0, 22}^0(x)$ does not identically vanish, and its inverse is given by
\[
\mathbf{\Psi}(x, \lambda) \mapsto \check{\mathbf{\Psi}}(x, \lambda) := \big( \sigma_- \lambda^{1/2}_{\lw} + \check{\mathbf{S}}(x) \lambda^{-1/2}_{\lw} \big) \ \mathbf{\Psi}(x, \lambda),
\]
where
\[
\check{\mathbf{S}}(x) := \begin{bmatrix}
1 & - \Psi_{0, 12}^0(x) /\Psi_{0, 22}^0(x) \vspace{1mm}\\ - \Psi_{1, 21}^\infty(x) & \Psi_{0, 12}^0(x) \Psi_{1, 21}^\infty(x) /\Psi_{0, 22}^0(x)
\end{bmatrix}.
\]
It follows that $\check{\mathbf{\Psi}}$ satisfies conditions similar to \eqref{eq:psi-hat-asymptotics-infty} and \eqref{eq:psi-hat-asymptotics-zero} as $\lambda$ approaches $\infty, 0$, respectively. That these operations are inverses of one another is the content of \cite[Lemma~1]{BMS18}.

In this way, starting with $\mathbf{\Psi}$ and iterating the map $\mathbf{\Psi \mapsto \hat{\mathbf{\Psi}}}$ (assuming $\Psi^0_{0, 11}(x)$, $\Psi^0_{0, 22}(x)$ do not identically vanish after each step), we may define the $n$th iterate of this Schlesinger transformation, which we denote $\mathbf{\Psi}_n$. This matrix, if it exists, satisfies the following Riemann--Hilbert problem.
\begin{rhp} \label{rhp:initial-with-n}
Fix generic monodromy parameters $( e_1, e_2 )$, $n \in \Z$, and $x>0$. We seek a $2 \times 2$ matrix function $\lambda\mapsto\mathbf \Psi(\lambda, x)$ satisfying:
\begin{itemize}\itemsep=0pt
\item  Analyticity: $\mathbf \Psi_n (\lambda, x)$ is analytic in $\C \setminus L^{(6)}$, where $L^{(6)} = \{ | \lambda| = 1\} \cup \ii \R$ is the jump contour shown in Figure {\rm \ref{fig:1}}.
\item  Jump condition: $\mathbf \Psi_n (\lambda, x)$ has continuous boundary values on $L^{(6)} \setminus \{0\}$ from each component of $\mathbb{C}\setminus L^{(6)}$, which satisfy
\[\mathbf \Psi_{n, +} (\lambda, x) = \mathbf \Psi_{n, -} (\lambda, x) \mathbf J_{\mathbf \Psi_n}(\lambda),\] where $\mathbf J_{\mathbf \Psi_n}(\lambda)$ is as shown in Figure {\rm \ref{fig:1}} but with the modification
\[
\mathbf{S}_2^0 e_0^{-2\sigma_3} \mapsto (-1)^n\mathbf{S}_2^0 e_0^{-2\sigma_3} \qquad \text{and} \qquad \mathbf{S}_2^\infty e_\infty^{2\sigma_3} \mapsto (-1)^n \mathbf{S}_2^\infty e_\infty^{2\sigma_3}.
\]
\item  Normalization: $\mathbf \Psi_n(\lambda,x)$ satisfies the asymptotic conditions
\begin{equation}
\label{eq:Psi-n-asymptotic-infinity}
\boldsymbol \Psi_n(\lambda, x) = \big( \mathbb{I}+\mathbf{\Xi}^{(6)}_n(x) + \mathcal{O}\big(\lambda^{-2}\big) \big) \ee^{\ii x \lambda \sigma_3 /2} \lambda_{\lw}^{(n - \Theta_\infty) \sigma_3/2} \qquad \text{as} \quad \lambda \to \infty,
\end{equation}
and
\begin{equation}
\label{eq:Psi-n-asymptotic-zero}
\boldsymbol \Psi_n(\lambda, x) = \big( \mathbf{\Delta}^{(6)}_n(x) + \mathcal{O}(\lambda) \big) \ee^{-\ii x \lambda^{-1} \sigma_3 /2} \lambda_{\lw}^{(\Theta_0 + n) \sigma_3/2} \qquad \text{as} \quad \lambda \to 0,
\end{equation}
where $\mathbf{\Delta}^{(6)}_n(x)$ is a matrix determined from $\mathbf{\Psi}_n(\lambda,x)$ having unit determinant.
\end{itemize}
\end{rhp}
That $\mathbf{\Psi}_n$ solves the above Riemann--Hilbert problem implies the existence of the limit
\begin{equation}
\label{eq:T-n-definition}
 \mathbf{\Xi}^{(6)}_n(x):=\lim_{\lambda\to\infty}\lambda\big[\mathbf{\Psi}_n(\lambda,x)\ee^{-\ii x\lambda\sigma_3/2}\lambda_{\lw}^{\Theta_\infty\sigma_3/2}-\mathbb{I}\big].
\end{equation}
It follows that the function
\begin{equation}
\label{eq:u-n-recover}
u_n(x) = \dfrac{-\ii \Xi^{(6)}_{n, 12}(x)}{\Delta^{(6)}_{n, 11}(x) \Delta^{(6)}_{n, 12}(x)}
\end{equation}
satisfies PIII($D_6$) in the form
\begin{equation*}\label{eq:PIII-D6-n}
u_n''=\dfrac{(u_n')^2}{u_n}-\dfrac{u_n'}{x}+\dfrac{4(n+\Theta_0)u_n^2}{x}+\dfrac{4(1 + n-\Theta_\infty)}{x}+4u_n^3-\frac{4}{u_n}.
\end{equation*}
It was shown in \cite[Lemma 2]{BMS18} that if for some $n\in\mathbb{Z}$ the inverse monodromy problem is solvable for a given $x \in D$, where $D$ is a domain in $\C \setminus \{0\}$, then $\mathbf{\Psi}_n$ satisfies the Lax pair
\begin{align*}
\dpd{\mathbf{\Psi}_n}{\lambda}(\lambda, x) &= \Bigg( \frac{\ii x}{2}\sigma_3 + \frac{1}{2\lambda}\begin{bmatrix} n-\Theta_\infty & 2y \\ 2v & \Theta_\infty-n \end{bmatrix} + \frac{1}{2\lambda^2} \begin{bmatrix} \ii x - 2\ii st & 2\ii s \\ 2\ii t(x - st) & 2\ii st-\ii x\end{bmatrix} \Bigg) \mathbf{\Psi}_n(\lambda, x), \\
\dpd{\mathbf{\Psi}_n}{x}(\lambda, x) &= \Bigg(\dfrac{\ii \lambda}{2} \sigma_3 + \dfrac{1}{x} \begin{bmatrix} 0 & y \\ v & 0\end{bmatrix} - \dfrac{1}{2\lambda x} \begin{bmatrix} \ii x - 2\ii st & 2\ii s \\ 2\ii t(x -st ) & 2\ii st-\ii x\end{bmatrix}\Bigg)\mathbf{\Psi}_n(\lambda, x),
\end{align*}
where potentials $s$, $t$, $u$, $v$, $y$ all depend on $x$ and $n$. Furthermore, in this domain, the functions $\Psi_{0, 11}^{0}(x)$, $\Psi_{0, 22}^{0}(x)$ extracted from $\mathbf{\Psi}_n(\lambda,x)$ are not identically zero.\footnote{Lemma~2 in~\cite{BMS18} was stated for parameters corresponding to rational solutions of Painlev\'e-III, but the proof is almost exactly the same in this case.}

One can check that if a solution to Riemann--Hilbert Problem~\ref{rhp:initial-with-n} exists, it must be unique, and we attempt to identify this solution as a solution of Riemann--Hilbert Problem~\ref{rhp:initial} with possibly different monodromy data. The diagonal elements of $\mathbf{S}_2^0 e_0^{-2\sigma_3}$, $\mathbf{S}_2^\infty e_\infty^{2\sigma_3}$ alternate signs which implies the change
\[
e_0^2 \mapsto (-1)^n e_0^2, \qquad e_\infty^2 \mapsto (-1)^n e_\infty^2.
\]
Furthermore, in view of \eqref{eq:stokes-parameters}, we can write
\[
(-1)^n s_2^\infty e_\infty^2 = (-1)^n\big(1 - e_1^2 e_\infty^2\big) e_\infty^2 = \big(1 - (-1)^ne_1^2(-1)^ne_\infty^2\big)(-1)^ne_\infty^2,
\]
and
\[
(-1)^n s_2^0 e_0^{-2} = (-1)^n\big(e_1^2 e_0^2 - 1\big) e_0^{-2} = \big((-1)^ne_1^2(-1)^ne_0^2 -1\big)(-1)^ne_0^{-2}.
\]
Combining the above with the fact that $\mathbf{C}_{0 \infty}^\pm$ remain invariant under the iterated Schlesinger transformations implies the change in monodromy data
\begin{gather} \label{eq:monodromy-n-dependence}
e_1^2 \mapsto (-1)^ne_1^2 \qquad \text{and} \qquad e_2 \mapsto e_2.
\end{gather}
Since $e_1$, $e_2$ are assumed to be nonvanishing, we may write them in the form \eqref{eq:e1-e2-mu-eta}
for some ${\mu, \eta \in \C}$ with $-1 < \re(\mu), \re(\eta) \leq 1$. Moreover, since the transformations $e_1\mapsto-e_1$, ${e_2\mapsto -e_2}$ preserve the monodromy data, we can assume $-\frac{1}{2}<\re(\eta)\le\frac{1}{2}$ and $-\frac{1}{2}<\re(\mu) \leq \frac{1}{2}$. Equation~\eqref{eq:monodromy-n-dependence} implies in turn that $\eta$ does not depend on $n \in \Z$, while $\mu$ is replaced with
\begin{equation}
\label{eq:mu-n}
\mu \mapsto \mu_n := \begin{cases} \mu, & n \in 2\Z, \\ \mu-\frac{1}{2}, & n +1 \in 2\Z, \end{cases}
\end{equation}
This proves Proposition \ref{prop:schlesinger}. We end this section with two important remarks.

\begin{Remark}
\label{remark:e2-mu-minus-mu-change}
It was noted in the introduction that one could restrict $0 < \re(\mu_n) \leq 1/2$, in which case, the above iterations interchange the roles of $e_1^2$, $e_1^{-2}$ and we have to perform the transformation $\mu\to -\mu$, which corresponds to the replacements
\begin{gather*}
\mathbf{E}^\infty\to \left(\sqrt{\frac{e_\infty^2-e_1^2}{e_1^2\big(e_1^2e_\infty^2-1\big)}}\right)^{\sigma_3}\mathbf{E}^\infty\left(\sqrt{\frac{e_\infty^2-e_1^2}{e_1^2\big(e_1^2e_\infty^2-1\big)}}\right)^{-\sigma_3}\sigma_1\left(\frac{e_1^2 e_\infty^2 \big(1-e_1^2 e_\infty^2\big)}{ \big(e_1^4-1\big)}\right)^{\sigma_3},\\
\mathbf{E}^0\to \left(\sqrt{\frac{e_0^2-e_1^2}{e_1^2\big(e_1^2e_0^2-1\big)}}\right)^{\sigma_3}\mathbf{E}^0\left(\sqrt{\frac{e_0^2-e_1^2}{e_1^2\big(e_1^2e_0^2-1\big)}}\right)^{-\sigma_3}\sigma_1\left(\frac{e_1^2 \big(e_0^2 e_1^2-1\big)}{ e_0^2 \big(e_1^4-1\big)}\right)^{\sigma_3}.
\end{gather*}
This gauge transformation then allows us to identify the monodromy parameter pairs
\begin{gather}
(e_1,e_2)\sim \left(\frac{1}{e_1},\frac{1}{e_2e_0^2e_\infty^2}\sqrt{\frac{\big(e_1^2-e_0^2\big)\big(1-e_0^2e_1^2\big)}{\big(e_1^2-e_\infty^2\big)\big(1-e_1^2e_\infty^2\big)}}\right).
\label{eq:e1-to-1/e1}
\end{gather}
Therefore, we alternatively can write the monodromy data for the Schlesinger transformation as
\begin{gather*}
\mu_n = \begin{cases} \mu, & n \in 2\Z, \\ \frac{1}{2}-\mu, & n +1 \in 2\Z, \end{cases} \\
 e_{2,n}=\begin{cases} e_2, & n \in 2\Z, \\ \dfrac{1}{e_2e_0^2e_\infty^2}\sqrt{\dfrac{\big(e_1^2-e_0^2\big)\big(1-e_0^2e_1^2\big)}{\big(e_1^2-e_\infty^2\big)\big(1-e_1^2e_\infty^2\big)}}, & n +1 \in 2\Z. \end{cases}
\end{gather*}
Furthermore, one can check that $(x_1, x_2, x_3)$ in \eqref{eq:x1}--\eqref{eq:x3} remain invariant under the map described in \eqref{eq:e1-to-1/e1}, whereas $(y_1, y_2, y_3) \mapsto (\pm y_1, \pm y_2, y_3)$ where the sign depends on the choice of the square root in \eqref{eq:e1-to-1/e1} and \eqref{eq:V-root} below. In both cases, the corresponding solution of \eqref{eq:PIII-$D_8$} remains invariant, see Remark \ref{rem:minus-y1-y2}.
\end{Remark}

\begin{Remark}
 \label{rem:n-dependent-notation}
 Moving forward, we will slightly abuse notation by suppressing the $n$-de\-pen\-den\-ce in the parameters
\begin{equation}
\label{eq:e-definitions}
e_\infty = \ee^{\pi \ii (\Theta_\infty - n)/2}, \qquad e_0 = \ee^{\pi \ii (\Theta_0 + n)/2}, \qquad e_1 = \ee^{\pi \ii \mu_n}, \qquad e_2 = \ee^{\pi \ii \eta }.
\end{equation}
\end{Remark}

\section[Asymptotics for large n and small x and proof of Theorem 1.4]{Asymptotics for large $\boldsymbol{n}$ and small $\boldsymbol{x}$\\ and proof of Theorem \ref{thm:general}}
\label{sec:proof}
Let $(e_1, e_2)$ be generic monodromy parameters, see Definition \ref{def:generic}. At this point, we can see more clearly the meaning of the genericity conditions formulated there:
\begin{enumerate}\itemsep=0pt
\item[(i)] $e_1^4 \neq 1$; this is to guarantee diagonalizability in \eqref{eq:Stokes-products-eigenvectors},
\item[(ii)] $e_1 e_2 \neq 0$; this is to guarantee the unit-determinant condition in \eqref{eq:Stokes-products-eigenvectors} and \eqref{eq:C-zero-infty-diagonalization},
\item[(iii)] $e_1^2 \neq e_\infty^{\pm 2}$ and $e_1^2 \neq e_0^{\pm 2}$; this, in particular, implies that the Stokes multipliers \eqref{eq:stokes-parameters} are nonvanishing.
\end{enumerate}

\subsection{Opening the lenses} First, we define a new unknown matrix by $\mathbf \Phi_n(\lambda,x):=\boldsymbol \Psi_n (\lambda,x)\mathbf{L}$ where $\mathbf{L}$ is the piece-wise constant matrix shown in the left-hand panel of Figure~\ref{2}. It follows from \eqref{eq:C-zero-infty-diagonalization}, \eqref{eq:C-zero-infty-plus-diagonalization} that the resulting jump conditions satisfied by $\mathbf \Phi_n(\lambda,x)$ are as shown in the right-hand panel of Figure~\ref{2}.

 \begin{figure}[t]
 \centering
 \raisebox{7mm}{\includegraphics[width=0.43\linewidth]{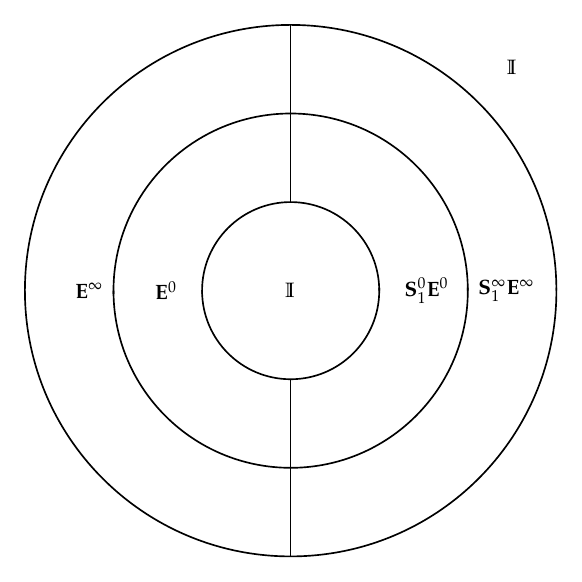}}\quad
 \includegraphics[width=0.53\linewidth]{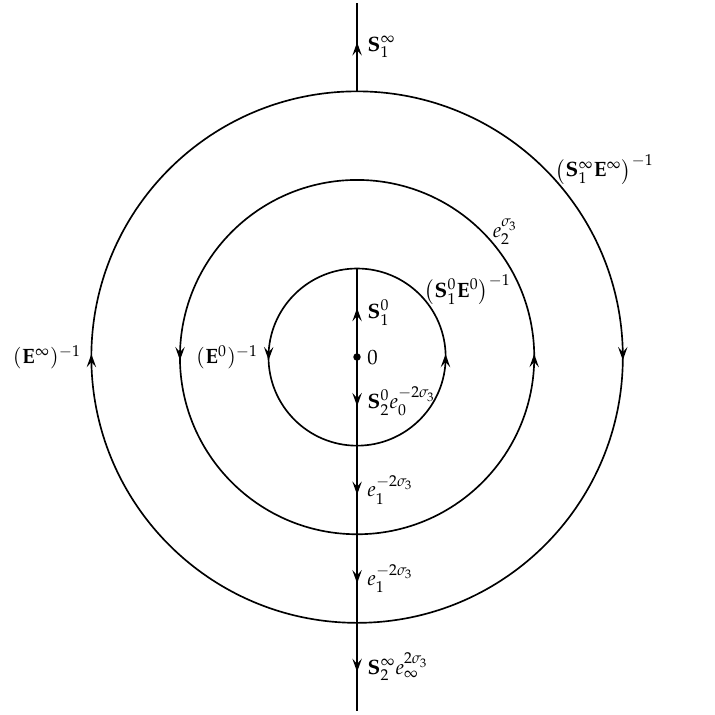}
\caption{Left panel: the definition of the matrix $\mathbf L$; the circles are centered at the origin and have radii $\frac{1}{2}$, $1$, and $2$. Right panel: the jump contour $\Gamma$ and jump conditions for $\mathbf{\Phi}_n(\lambda,x)$.}\label{2}
\end{figure}

\begin{Remark}\label{rmk:rotated-lenses}
In the general case $|{\arg}(x)|<\pi$, the lenses shown in Figure \ref{2} must be rotated in the manner shown in the left panel of Figure \ref{fig:2-rotated}. The resulting jumps follow from the identities~\eqref{eq:general-connection-factorization-1}--\eqref{eq:general-connection-factorization-2} and are shown in the right panel of Figure \ref{fig:2-rotated}.
\begin{figure}[t]
 \centering
\raisebox{-0.5\height}{\includegraphics[width=0.435\linewidth]{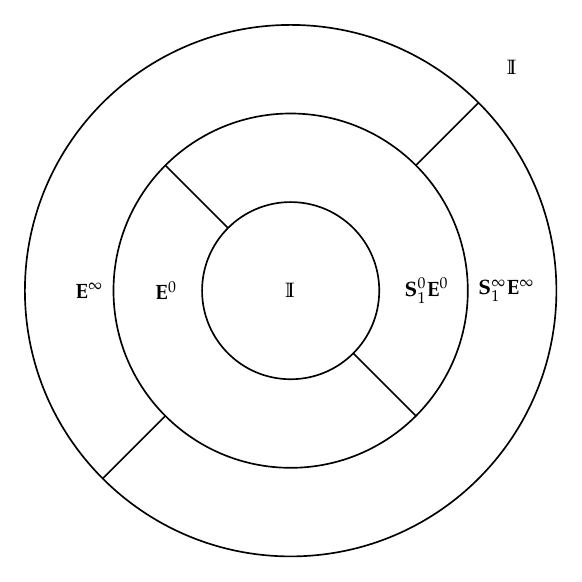}}\hfill%
\raisebox{-0.5\height}{\includegraphics[width=0.535\linewidth]{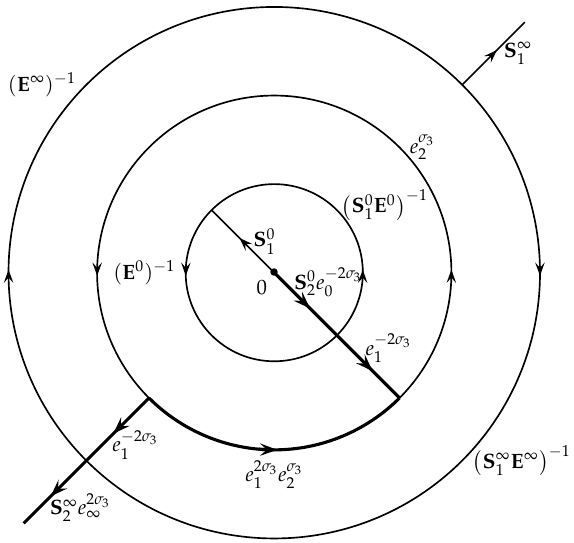}}
\caption{Analogue of Figure \ref{2} when $\arg(x) \neq 0$. The thick line represents the branch cut for the argument chosen as in Remark \ref{rmk:general-arg-x}.}
\label{fig:2-rotated}
\end{figure}
\end{Remark}

 \begin{figure}[h!]
 \centering
\includegraphics[width=0.5\columnwidth]{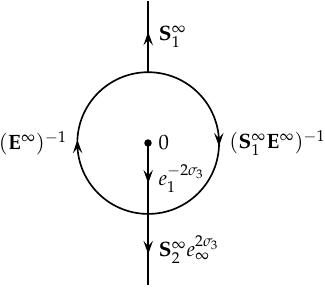}
\caption{The contour $\Gamma^{(\infty)}$ includes the circle $|\lambda|=2$.}\label{fig:6}
\end{figure}

\subsection[Parametrix for Phi\_n(lambda,x) near lambda=infty]{Parametrix for $\boldsymbol{{\Phi}_n(\lambda,x)}$ near $\boldsymbol{\lambda=\infty}$}
\label{sec:Parametrix-infinity}
By definition, the parametrix
$\para{\boldsymbol \Phi}_n^{(\infty)}(\lambda,x)$
satisfies the following Riemann--Hilbert problem.
\begin{rhp} \label{rhp:infinity_parametrix}
Fix generic monodromy parameters $( e_1, e_2 )$ determining the Stokes and connection matrices, $n\in\mathbb{Z}$, and $x>0$. We seek a $2 \times 2$ matrix function $\lambda\mapsto\para{\boldsymbol \Phi}_n^{(\infty)}(\lambda,x)$ satisfying:
\begin{itemize}\itemsep=0pt
\item Analyticity: $\para{\boldsymbol \Phi}_n^{(\infty)}(\lambda,x)$ is analytic in $\C \setminus \Gamma^{(\infty)}$, where
\[\Gamma^{(\infty)} = \{ | \lambda| = 2\} \cup (\ii \R\cap \{ |{\im}\,\lambda-1| > 1\})\]
 is the jump contour shown in Figure {\rm\ref{fig:6}}.
\item  Jump condition: $\para{\boldsymbol \Phi}_n^{(\infty)}(\lambda,x)$ has continuous boundary values on $\Gamma^{(\infty)} \setminus \{0\}$ from each component of $\mathbb{C}\setminus \Gamma^{(\infty)}$, which satisfy
\[
\para{\boldsymbol \Phi}_{n,+}^{(\infty)}(\lambda,x) = \para{\boldsymbol \Phi}_{n,-}^{(\infty)}(\lambda,x)\mathbf J_{\para{\boldsymbol \Phi}_n^{(\infty)}}(\lambda),
\]
 where \smash{$\mathbf J_{\para{\boldsymbol \Phi}_n^{(\infty)}}(\lambda)$} is as shown in Figure {\rm\ref{fig:6}} and where the $+$ $($resp., $-)$ subscript denotes a~boundary value taken from the left $($resp., right$)$ of an arc of $\Gamma^{(\infty)}$.
\item Normalization: \smash{$\para{\boldsymbol \Phi}_n^{(\infty)}(\lambda,x)$} satisfies the asymptotic conditions{\samepage
\begin{gather}
\label{eq:Phi-asymptotic-infinity}
\para{\boldsymbol \Phi}_n^{(\infty)}(\lambda,x) = \bigg(\mathbb I + \dfrac{\mathbf{A}_{n}(x)}{\lambda} + \mathcal{O}\left( \dfrac{1}{\lambda^2}\right) \bigg) \ee^{\ii x \lambda \sigma_3/2} \lambda_{\lw}^{(n - \Theta_\infty )\sigma_3/2} \qquad \text{as} \quad \lambda \to \infty,
\\
\label{eq:Phi-asymptotic-zero}
\para{\boldsymbol \Phi}_n^{(\infty)}(\lambda,x) = ( \mathbf{B}_{n}(x) + \mathcal{O}(\lambda) ) \lambda_{\lw}^{\mu_n \sigma_3} \qquad \text{as} \quad \lambda \to 0,
\end{gather}
where $\mathbf{A}_n(x)$ has zero trace and $ \mathbf{B}_{n}(x)$ has unit determinant.}
\end{itemize}
\end{rhp}

It is easy to see that \smash{$\para{\mathbf{\Phi}}_n^{(\infty)}(\lambda,x)$} necessarily has unit determinant. Furthermore, note that the jump matrix being \smash{$e_1^{-2\sigma_3}$} across the arc terminating at the origin implies
$e_1 = \ee^{\pi \ii \mu_n}$, which is consistent with \eqref{eq:e-definitions}.

\subsubsection[Dependence on lambda]{Dependence on $\boldsymbol{\lambda}$}
It follows from assuming differentiability of the asymptotics in \eqref{eq:Phi-asymptotic-infinity}--\eqref{eq:Phi-asymptotic-zero} that
\begin{gather}
 \frac{\partial\para{\boldsymbol \Phi}_{n}^{(\infty)}}{\partial\lambda}(\lambda,x) \para{\boldsymbol \Phi}^{(\infty)}_n(\lambda,x)^{-1}\nonumber\\
\qquad= \left(\mathbb I + \dfrac{\mathbf{A}_{n}(x)}{\lambda} + \mathcal{O}\big( \lambda^{-2}\big) \right)\left( \dfrac{\ii x}{2} + \dfrac{n - \Theta_\infty}{2 \lambda} \right) \sigma_3 \left(\mathbb I + \dfrac{\mathbf{A}_{n}(x)}{\lambda} + \mathcal{O}\big( \lambda^{-2}\big) \right)^{-1} + \mathcal{O}\big(\lambda^{-2}\big)\nonumber\\
 \qquad= \dfrac{\ii x}{2} \sigma_3 + \left( \dfrac{\ii x}{2} [\mathbf{A}_{n}(x), \sigma_3] + \dfrac{n - \Theta_\infty}{2} \sigma_3 \right) \dfrac{1}{\lambda} + \mathcal{O}\big( \lambda^{-2} \big) \qquad \text{as} \quad \lambda \to \infty,
\label{psi-infty-expansion}
\end{gather}
and
\begin{align}
 \frac{\partial\para{\boldsymbol \Phi}_{n}^{(\infty)}}{\partial\lambda}(\lambda,x)\para{\boldsymbol \Phi}^{(\infty)}_n(\lambda,x)^{-1} &= \left( \mathbf{B}_{n}(x) + \mathcal{O}(\lambda) \right) \left( \dfrac{\mu_n}{\lambda} \sigma_3 \right) \left( \mathbf{B}_{n}(x) + \mathcal{O}(\lambda) \right)^{-1} + \mathcal{O}(1)
 \nonumber\\
 &= \dfrac{\mu_n}{\lambda} \mathbf{B}_{n}(x) \sigma_3 \mathbf{B}_{n}(x)^{-1} + \mathcal{O}(1) \qquad \text{as} \quad \lambda \to 0.
\label{eq:Lambda-expansion-zero}
\end{align}
Since the quantity on the left-hand side of \eqref{psi-infty-expansion} and \eqref{eq:Lambda-expansion-zero} is otherwise an analytic function of~$\lambda$, it follows from Liouville's Theorem that
\begin{gather}
 \frac{\partial\para{\boldsymbol \Phi}_{n}^{(\infty)}}{\partial\lambda}(\lambda,x) \para{\boldsymbol \Phi}_n^{(\infty)}(\lambda,x)^{-1} = \dfrac{\ii x}{2} \sigma_3 + \dfrac{\mu_n}{\lambda} \mathbf{B}_{n}(x) \sigma_3 \mathbf{B}_{n}(x)^{-1} \implies\nonumber\\ \frac{\partial\para{\boldsymbol \Phi}_{n}^{(\infty)}}{\partial\lambda}(\lambda,x)= \left( \dfrac{\ii x}{2} \sigma_3 + \dfrac{\mu_n}{\lambda} \mathbf{B}_{n}(x) \sigma_3 \mathbf{B}_{n}(x)^{-1} \right) \para{\boldsymbol \Phi}_n^{(\infty)} (\lambda,x).
 \label{Lax-system-lambda}
\end{gather}
Noting that $\Tr\big(\mathbf{B}_{n}(x)\sigma_3\mathbf{B}_{n}(x)^{-1}\big)=0$ and $\det \big(\mathbf{B}_{n}(x) \sigma_3 \mathbf{B}_{n}(x)^{-1}\big) = -1$, we may write
\begin{equation}
\mathbf{B}_{n}(x) \sigma_3 \mathbf{B}_{n}(x)^{-1} = \begin{bmatrix} a_n(x) & b_n(x) \\ c_n(x) & -a_n(x) \end{bmatrix} \qquad \text{subject to}\quad a_n(x)^2 + b_n(x)c_n(x) = 1
\label{eq:zero-coefficient-matrix}
\end{equation}
and use this form in \eqref{Lax-system-lambda} to write a coupled scalar system of differential equations satisfied by the elements $\phi_1(\lambda,x)$ and $\phi_2(\lambda,x)$ of the first and second rows, respectively, of any column of~$\para{\boldsymbol \Phi}_n^{(\infty)}(\lambda,x)$:
\begin{align}
\label{psi1-psi2-sys1} \frac{\partial\phi_{1}}{\partial\lambda}(\lambda,x) &= \left(\dfrac{\ii x}{2} + \dfrac{\mu_n a_n(x)}{\lambda}\right) \phi_1(\lambda,x) + \dfrac{\mu_n b_n(x)}{\lambda} \phi_2(\lambda,x), \\
\label{psi1-psi2-sys2} \frac{\partial\phi_{2}}{\partial\lambda}(\lambda,x) &= \dfrac{\mu_n c_n(x)}{\lambda} \phi_1(\lambda,x)-\left(\dfrac{\ii x}{2}+ \dfrac{\mu_n a_n(x)}{\lambda}\right) \phi_2(\lambda,x).
\end{align}
Before beginning to solve this system, observe that equating the coefficients of $\lambda^{-1}$ in \eqref{psi-infty-expansion} and~\eqref{eq:Lambda-expansion-zero} yields the identity
\begin{equation}
\label{C-A relation}
\mu_n \mathbf{B}_{n}(x) \sigma_3 \mathbf{B}_{n}(x)^{-1} = \dfrac{\ii x}{2} [\mathbf{A}_{n}(x), \sigma_3] + \dfrac{n - \Theta_\infty}{2} \sigma_3.
\end{equation}
Since $[\mathbf{A}_{n}(x), \sigma_3]$ is off-diagonal, we arrive at
\begin{equation}
\mu_n a_n(x) = \frac{n - \Theta_\infty}{2}.
\label{eq:little-an}
\end{equation}
Since $\mu_n$ and $n$ are constants, this equation implies that $a_n(x)$ is independent of $x$, so we will simply write $a_n$ going forward.
Now, solving for $\phi_1(\lambda,x)$ in \eqref{psi1-psi2-sys2} and eliminating it from \eqref{psi1-psi2-sys1} yields (assuming $c_n(x) \neq 0$ and using $b_n(x)c_n(x) = 1 - a_n^2$)
\begin{align*}
 \lambda \frac{\partial^2\phi_{2}}{\partial\lambda^2}(\lambda,x) + \frac{\partial\phi_{2}}{\partial\lambda}(\lambda,x) + \left[ \dfrac{\ii x}{2} + \dfrac{x^2}{4} \lambda - \ii x\mu_n a_n - \dfrac{\mu_n^2}{\lambda} \right]\phi_2(\lambda,x) = 0.
\end{align*}
It is easy to see that the first-order derivative term is removed by the substitution $\phi_2(\lambda,x)=\lambda^{-1/2}w(\lambda,x)$.
Indeed, $w(\lambda,x)$ satisfies
\begin{align*}
 \frac{\partial^2w}{\partial\lambda^2}(\lambda,x) + \bigg[ \dfrac{x^2}{4} + \ii x\left( \dfrac{1}{2} - \mu_n a_n \right) \dfrac{1}{\lambda} + \left( \dfrac{1}{4} - \mu_n^2 \right) \dfrac{1}{\lambda^2} \bigg]w(\lambda,x) &= 0.
\end{align*}
Finally, the explicit $x$-dependence in the coefficients can be removed by setting $Z:= \ii x \lambda$ and writing $w(\lambda,x)=W(Z)$. Note that the notation $W(Z)$ here is not related to $W(z)$ appearing in Section~\ref{sec:monodromy-rep-$D_8$}. In this case, $W(Z)$ satisfies the ordinary differential equation
\begin{equation}
 W''(Z) + \bigg[ - \dfrac{1}{4} + \left( \dfrac{1}{2} - \mu_n a_n \right) \dfrac{1}{Z} + \left( \dfrac{1}{4} - \mu_n^2 \right) \dfrac{1}{Z^2} \bigg] W(Z) = 0,
 \label{Whittaker}
\end{equation}
which is Whittaker's equation (see \cite[Chapter~13]{DLMF}) with parameter
\begin{equation}\label{kappa-mu}
 \kappa = \kappa_n:=\frac{1}{2} - \mu_n a_n = \dfrac{1+\Theta_\infty - n }{2}.
\end{equation}
Given $\phi_2(\lambda,x)=\lambda^{-1/2}W(Z)$ for $Z=\ii x\lambda$, and a solution $W(Z)$ of \eqref{Whittaker}, it follows from \eqref{psi1-psi2-sys2} that the corresponding first-row entry is
\begin{equation}
 \phi_1(\lambda,x)=\frac{\ii x\lambda^{1/2}}{\mu_n c_n(x)}\bigg(W'(Z)+\left(\frac{1}{2}-\frac{\kappa_n}{Z}\right)W(Z)\bigg).
 \label{first-row}
\end{equation}
A fundamental pair of solutions of \eqref{Whittaker} is given by $W(Z)=W_{\pm \kappa_n, \mu_n} (\pm Z)$, $\mathrm{arg}(\pm Z) \in (-\pi, \pi) $.
If we take the particular solution $\phi_2(\lambda,x)=\lambda^{-1/2}W_{\kappa_n,\mu_n}(Z)$, then
using the identity
\begin{equation*}
 W'_{\kappa,\mu_n}(Z)=\left(\dfrac{\kappa}{Z}-\dfrac{1}{2}\right)W_{\kappa,\mu_n}(Z)+\bigg(\left(\dfrac{1}{2}-\kappa\right)^2-\mu_n^2\bigg)\frac{1}{Z}W_{\kappa-1,\mu_n}(Z)
\end{equation*}
(see \cite[equation~(13.15.23)]{DLMF}) in \eqref{first-row} gives
\begin{gather*}
 \phi_2(\lambda,x)=\lambda^{-1/2}W_{\kappa_n,\mu_n}(Z)\\
\implies \phi_1(\lambda,x)=\frac{\ii x\lambda^{1/2}}{\mu_n c_n(x)}\bigg(\left(\frac{1}{2}-\kappa_n\right)^2-\mu_n^2\bigg)Z^{-1}W_{\kappa_n-1,\mu_n}(Z).
\end{gather*}
Likewise, if we take the particular solution $\phi_2(\lambda,x)=\lambda^{-1/2}W_{-\kappa_n,\mu_n}(-Z)$, then using the identity
\begin{equation*}
 W'_{\kappa, \mu_n}(Z) = \left( \dfrac{1}{2} - \dfrac{\kappa}{Z} \right)W_{\kappa, \mu_n}(Z) - \dfrac{1}{Z} W_{\kappa+1, \mu_n}(Z)
\end{equation*}
(see \cite[equation~(13.15.26)]{DLMF}) in \eqref{first-row} yields
\begin{equation*}
 \phi_2(\lambda,x)=\lambda^{-1/2}W_{-\kappa_n,\mu_n}(-Z)\implies \phi_1(\lambda,x)=-\frac{\ii x\lambda^{1/2}}{\mu_n c_n(x)}Z^{-1}W_{1-\kappa_n,\mu_n}(-Z).
\end{equation*}
Taking linear combinations with coefficients depending generally on the parameter $x$, the general solution matrix for the system \eqref{Lax-system-lambda} can be written in the form
\begin{equation}
 \para{\boldsymbol\Phi}_n^{(\infty)}(\lambda,x)=\widetilde{\boldsymbol\Phi}^{(\infty)}_n(\lambda,x)\mathbf{K}(x),\qquad\widetilde{\boldsymbol\Phi}_n^{(\infty)}(\lambda,x):=\mathbf{H}(\lambda,x)\mathbf{W}(\ii\lambda x;\kappa_n,\mu_n),
 \label{Psi0-define}
\end{equation}
where $\widetilde{\mathbf{\Phi}}_n^{(\infty)}(\lambda,x)$ is a specific fundamental solution matrix of \eqref{Lax-system-lambda} constructed from
\begin{equation}
 \mathbf{H}(\lambda,x):=\lambda^{\sigma_3/2}\begin{bmatrix}\displaystyle\frac{\ii x}{\mu_n c_n(x)} & 0\\0 & 1\end{bmatrix}
 \label{H-define}
\end{equation}
and
\begin{gather}
 \mathbf{W}(Z;\kappa,\mu_n):=
 \begin{bmatrix}
 \alpha_{\kappa,\mu_n}Z^{-1}W_{\kappa-1,\mu_n}(Z) & -Z^{-1}W_{1-\kappa,\mu_n}(-Z)\\
 W_{\kappa,\mu_n}(Z) & W_{-\kappa,\mu_n}(-Z)
 \end{bmatrix},\nonumber\\ \alpha_{\kappa,\mu_n}:=\left(\frac{1}{2}-\kappa\right)^2-\mu_n^2,
 \label{M-alpha-define}
\end{gather}
in which $\kappa=\kappa_n$ and $\mu_n$ are given by \eqref{kappa-mu} and $\mathbf{K}(x)$ is a matrix of free coefficients.

\subsubsection[Dependence on x]{Dependence on $\boldsymbol{x}$}
Going back to \eqref{eq:Phi-asymptotic-infinity}--\eqref{eq:Phi-asymptotic-zero} and now assuming that the asymptotics are differentiable with respect to $x$,
\begin{gather*}
 \frac{\partial\para{\boldsymbol \Phi}^{(\infty)}_n}{\partial x}(\lambda,x) \para{\boldsymbol \Phi}_n^{(\infty)} (\lambda,x)^{-1}\\
 \qquad= \left(\mathbb I + \dfrac{\mathbf{A}_{n}(x)}{\lambda} + \mathcal{O}\big( \lambda^{-2}\big) \right) \left( \dfrac{\ii \lambda}{2}\sigma_3 \right)\left(\mathbb I + \dfrac{\mathbf{A}_{n}(x)}{\lambda} + \mathcal{O}\big( \lambda^{-2}\big)\right)^{-1} + \mathcal{O}\big( \lambda^{-1} \big)\\
 \qquad = \dfrac{\ii \lambda}{2} \sigma_3 + \dfrac{\ii}{2} [\mathbf{A}_{n}(x), \sigma_3] + \mathcal{O}\big( \lambda^{-1} \big) \qquad \text{as} \quad \lambda \to \infty,
\end{gather*}
and
\begin{align*}
 \frac{\partial\para{\boldsymbol \Phi}^{(\infty)}_n}{\partial x}(\lambda,x) \para{\boldsymbol \Phi}_n^{(\infty)}(\lambda,x)^{-1} &= \mathbf{B}_n'(x) \mathbf{B}_n(x)^{-1} + \mathcal{O}(\lambda) \qquad \text{as} \quad \lambda \to 0.
\end{align*}
So, applying Liouville's theorem yields
\begin{equation}
 \frac{\partial\para{\boldsymbol \Phi}^{(\infty)}_{n}}{\partial x} (\lambda,x) = \left( \dfrac{\ii \lambda}{2} \sigma_3 + \dfrac{\ii}{2} [\mathbf{A}_n(x), \sigma_3] \right) \para{\boldsymbol \Phi}_n^{(\infty)}(\lambda,x),
 \label{Lax-system-x}
\end{equation}
and it follows from \eqref{C-A relation} that
\begin{align}
\dfrac{\ii}{2}[\mathbf{A}_n(x), \sigma_3] &= \dfrac{1}{x} \left( \mu_n\mathbf{B}_n(x) \sigma_3 \mathbf{B}_n(x)^{-1} +\left(\kappa_n-\frac{1}{2}\right) \sigma_3 \right)\nonumber\\
& = \frac{1}{x}\begin{bmatrix} 0 & \dfrac{\mu_n\big(1 - a_n^2\big)}{c_n(x)} \\ \mu_n c_n(x) & 0 \end{bmatrix},
\label{eq:Ainfty-identity}
\end{align}
where $\mu_n$, $a_n$ are independent of $\lambda$, $x$. To determine the $x$-dependence of $c_n(x)$, we use \eqref{eq:Ainfty-identity} to assemble \eqref{Lax-system-lambda} (using also \eqref{C-A relation} and \eqref{kappa-mu}) and \eqref{Lax-system-x} to give the Lax system
\begin{gather}
 \frac{\partial\para{\boldsymbol \Phi}_{n}^{(\infty)}}{\partial\lambda}(\lambda,x) =\para{\boldsymbol\Lambda}(\lambda,x) \para{\boldsymbol \Phi}_n^{(\infty)}(\lambda,x),\nonumber\\
 \para{\mathbf{\Lambda}}(\lambda,x):=\dfrac{\ii x}{2} \sigma_3 + \dfrac{1}{\lambda}\begin{bmatrix} -\frac{1}{2}(\Theta_\infty -n) & \displaystyle \frac{\mu_n\big(1-a_n^2\big)}{c_n(x)} \\ \mu_n c_n(x) & \frac{1}{2}(\Theta_\infty-n) \end{bmatrix}, \label{lax-lambda} \\
 \frac{\partial\para{\boldsymbol \Phi}_{n}^{(\infty)}}{\partial x}(\lambda,x) =\para{\mathbf X}(\lambda,x) \para{\boldsymbol \Phi}_n^{(\infty)}(\lambda,x),\nonumber\\
 \para{\mathbf{X}}(\lambda,x):= \dfrac{\ii\lambda}{2} \sigma_3 + \dfrac{1}{x}\begin{bmatrix} 0 & \displaystyle \frac{\mu_n\big(1-a_n^2\big)}{c_n(x)} \\ \mu_n c_n(x) & 0 \end{bmatrix}. \label{lax-x}
\end{gather}
Since $\para{\boldsymbol\Phi}_n^{(\infty)}(\lambda,x)$ is a simultaneous fundamental solution matrix for these equations, the Lax system is compatible. The compatibility condition reads
\begin{equation*}
 \para{\boldsymbol\Lambda}_{x}(\lambda,x) - \para{\mathbf X}_{\lambda}(\lambda,x) + \big[\para{\boldsymbol{\Lambda}}(\lambda,x), \para{\mathbf X}(\lambda,x)\big] = \mathbf{0},
\end{equation*}
which is equivalent to
\begin{equation}
 xc_n'(x) = (1-2\kappa_n) c_n(x) \implies c_n(x) = \gamma_n x^{1-2\kappa_n},
 \label{c-of-x}
\end{equation}
for some constant $\gamma_n\neq 0$. Thus, the coefficient $c_n(x)$ is determined up to the choice of the constant $\gamma_n$. Note also that the coefficient matrices $\para{\boldsymbol{\Lambda}}(\lambda,x)$ and $\para{\mathbf{X}}(\lambda,x)$ are obviously related by the simple identity
\begin{equation}
 \para{\mathbf{X}}(\lambda,x)-\frac{\lambda}{x}\para{\boldsymbol\Lambda}(\lambda,x)=\frac{1}{x}\left(\kappa_n-\frac{1}{2}\right)\sigma_3.
 \label{XLambda-relation}
\end{equation}
\sloppy Since the fundamental matrix $\widetilde{\boldsymbol \Phi}_n^{(\infty)}(\lambda,x)$ defined by \eqref{Psi0-define} satisfies \eqref{lax-lambda}, then so does \[\para{\boldsymbol \Phi}_n^{(\infty)}(\lambda,x)=\widetilde{\boldsymbol\Phi}_n^{(\infty)}(\lambda,x) \mathbf K(x),\] and $\mathbf K(x)$ must now be chosen so that \eqref{lax-x} is satisfied.
Substituting into \eqref{lax-x}, we obtain an ordinary differential equation on $\mathbf{K}(x)$:
\begin{equation}
 \mathbf{K}'(x)=\bigg(\widetilde{\boldsymbol\Phi}_n^{(\infty)}(\lambda,x)^{-1}\mathbf{X}(\lambda,x)\widetilde{\boldsymbol\Phi}_n^{(\infty)}(\lambda,x)-
 \widetilde{\boldsymbol\Phi}_n^{(\infty)}(\lambda,x)^{-1}\frac{\partial\widetilde{\boldsymbol\Phi}_{n}^{(\infty)}}{\partial x}(\lambda,x)\bigg)\mathbf{K}(x).
 \label{K-ODE}
\end{equation}
Now, from the form of $\widetilde{\boldsymbol\Phi}_n^{(\infty)}(\lambda,x)$ written in \eqref{Psi0-define}, we have both
\begin{gather*}
 \widetilde{\boldsymbol{\Phi}}_n^{(\infty)}(\lambda,x)^{-1}\frac{\partial\widetilde{\boldsymbol\Phi}_{n}^{(\infty)}}{\partial x}(\lambda,x)=\widetilde{\boldsymbol{\Phi}}_n^{(\infty)}(\lambda,x)^{-1}\frac{\partial\mathbf{H}}{\partial x}(\lambda,x)\mathbf{H}(\lambda,x)^{-1}\widetilde{\boldsymbol\Phi}_n^{(\infty)}(\lambda,x) \\
\phantom{ \widetilde{\boldsymbol{\Phi}}_n^{(\infty)}(\lambda,x)^{-1}\frac{\partial\widetilde{\boldsymbol\Phi}_{n}^{(\infty)}}{\partial x}(\lambda,x)=}{}+ \ii\lambda\widetilde{\boldsymbol\Phi}_n^{(\infty)}(\lambda,x)^{-1}\mathbf{H}(\lambda,x)\mathbf{W}'(Z;\kappa_n,\mu_n),\\
 \widetilde{\boldsymbol{\Phi}}_n^{(\infty)}(\lambda,x)^{-1}\frac{\partial\widetilde{\boldsymbol\Phi}_{n}^{(\infty)}}{\partial\lambda}(\lambda,x)=\widetilde{\boldsymbol{\Phi}}_n^{(\infty)}(\lambda,x)^{-1}\frac{\partial\mathbf{H}}{\partial\lambda}(\lambda,x)\mathbf{H}(\lambda,x)^{-1}\widetilde{\boldsymbol\Phi}_n^{(\infty)}(\lambda,x) \\
\phantom{ \widetilde{\boldsymbol{\Phi}}_n^{(\infty)}(\lambda,x)^{-1}\frac{\partial\widetilde{\boldsymbol\Phi}_{n}^{(\infty)}}{\partial\lambda}(\lambda,x)=}{}+ \ii x\widetilde{\boldsymbol\Phi}_n^{(\infty)}(\lambda,x)^{-1}\mathbf{H}(\lambda,x)\mathbf{W}'(Z;\kappa_n,\mu_n),
\end{gather*}
so it follows that
\begin{gather*}
 \widetilde{\boldsymbol{\Phi}}_n^{(\infty)}(\lambda,x)^{-1}\frac{\partial\widetilde{\boldsymbol\Phi}_{n}^{(\infty)}}{\partial x}(\lambda,x)\\
\quad=\frac{\lambda}{x}\widetilde{\boldsymbol{\Phi}}_n^{(\infty)}(\lambda,x)^{-1}\frac{\partial\widetilde{\boldsymbol{\Phi}}_{n}^{(\infty)}}{\partial\lambda}(\lambda,x)\\
 \quad\phantom{=}{}+ \widetilde{\boldsymbol{\Phi}}_n^{(\infty)}(\lambda,x)^{-1}\left[\frac{\partial\mathbf{H}}{\partial x}(\lambda,x)\mathbf{H}(\lambda,x)^{-1}-\frac{\lambda}{x}\frac{\partial\mathbf{H}}{\partial\lambda}(\lambda,x)\mathbf{H}(\lambda,x)^{-1}\right]\widetilde{\boldsymbol{\Phi}}_n^{(\infty)}(\lambda,x)\\
 \quad=\widetilde{\boldsymbol{\Phi}}_n^{(\infty)}(\lambda,x)^{-1}\left[\frac{\lambda}{x}\breve{\boldsymbol\Lambda}(\lambda,x) + \frac{\partial\mathbf{H}}{\partial x}(\lambda,x)\mathbf{H}(\lambda,x)^{-1}-\frac{\lambda}{x}\frac{\partial\mathbf{H}}{\partial\lambda}(\lambda,x)\mathbf{H}(\lambda,x)^{-1}\right]\widetilde{\boldsymbol{\Phi}}_n^{(\infty)}(\lambda,x),
\end{gather*}
where we also used \eqref{lax-lambda}.
Using this in \eqref{K-ODE} along with the explicit definition \eqref{H-define} of $\mathbf{H}(\lambda,x)$ and the identities \eqref{c-of-x} and \eqref{XLambda-relation} gives
\begin{gather*}
 \mathbf{K}'(x)\mathbf{K}(x)^{-1}\\
\quad =\widetilde{\boldsymbol{\Phi}}_n^{(\infty)}(\lambda,x)^{-1}\left[\breve{\mathbf{X}}(\lambda,x)-\frac{\lambda}{x}\breve{\boldsymbol{\Lambda}}(\lambda,x)
\right.\\
\left.
\hphantom{\quad =\widetilde{\boldsymbol{\Phi}}_n^{(\infty)}(\lambda,x)^{-1}}{}
+\frac{\lambda}{x}\frac{\partial\mathbf{H}}{\partial\lambda}(\lambda,x)\mathbf{H}(\lambda,x)^{-1}-\frac{\partial\mathbf{H}}{\partial x}(\lambda,x)\mathbf{H}(\lambda,x)^{-1}\right]\widetilde{\boldsymbol{\Phi}}_n^{(\infty)}(\lambda,x)\\
 \quad =\widetilde{\boldsymbol\Phi}_n^{(\infty)}(\lambda,x)^{-1}\Bigg[\frac{1}{x}\left(\kappa_n-\frac{1}{2}\right)\sigma_3 +\frac{1}{2x}\sigma_3+\frac{\dd}{\dd x}\log\left(\frac{c_n(x)}{x}\right)\begin{bmatrix}1&0\\0&0\end{bmatrix}\Bigg]\widetilde{\boldsymbol{\Phi}}_n^{(\infty)}(\lambda,x)\\
 \quad=\frac{\kappa_n}{x}\widetilde{\boldsymbol{\Phi}}_n^{(\infty)}(\lambda,x)^{-1}\Bigg[\sigma_3-2
 \begin{bmatrix}1&0\\0&0\end{bmatrix}\Bigg]\widetilde{\boldsymbol{\Phi}}_n^{(\infty)}(\lambda,x)
=-\frac{\kappa_n}{x}\mathbb{I}.
\end{gather*}
Therefore, the $x$-dependence of the matrix $\mathbf{K}(x)$ is explicitly given by
\begin{equation*}
 \mathbf{K}(x)=x^{-\kappa_n}\mathbf{K},
\end{equation*}
where $\mathbf{K}$ is now independent of both $\lambda$ and $x$. However, as the domain of analyticity of $\para{\mathbf{\Phi}}_n^{(\infty)}(\lambda,x)$ in the $\lambda$-plane consists of three disjoint regions, we expect to have to specify a different matrix~$\mathbf{K}$ for each. Note also that the constant $\gamma_n$ remains to be determined.

\subsubsection[The parametrix Phi\_n\^infty(lambda,x) on the two regions with |lambda|>2]{The parametrix $\boldsymbol{\para{\mathbf{\Phi}}_n^{(\infty)}(\lambda,x)}$ on the two regions with $\boldsymbol{|\lambda|>2}$}
To fully specify the parametrix $\para{\boldsymbol{\Phi}}_n^{(\infty)}(\lambda,x)$ for $|\lambda|>2$, we concretely take the jump contours for $|\lambda|>2$ to lie along the real axis in the $Z$-plane, corresponding to $\R_+$ and $\R_-$, respectively. Thus, the part of the domain of analyticity of $\para{\boldsymbol\Phi}_n^{(\infty)}(\lambda,x)$ with $|\lambda|>2$ has two components, corresponding to the upper and lower half $Z$-planes. To properly define $\para{\boldsymbol{\Phi}}_n^{(\infty)}(\lambda,x)$ in these two exterior domains, we firstly take the matrix factor $\mathbf{H}(\lambda,x)$ defined in \eqref{H-define} in the precise form
\begin{equation}
 \mathbf{H}(\lambda,x)=\lambda_{\lw}^{\sigma_3/2}\begin{bmatrix}\displaystyle \frac{\ii x^{2\kappa_n}}{\mu_n \gamma_n} & 0\\0 & 1
 \end{bmatrix} = x^{\kappa_n}\lambda_{\lw}^{\sigma_3/2}x^{\kappa_n\sigma_3}\mathbf{D}_n,\qquad\mathbf{D}_n:=\begin{bmatrix}\displaystyle\frac{\ii}{\mu_n\gamma_n} & 0\\0 & 1\end{bmatrix}.
 \label{H-define-2}
\end{equation}
Then, we assume different constant matrices $\mathbf{K}=\mathbf{K}_n^\pm$ in the two domains by writing the parametrix for $|\lambda|>2$ and $\pm\im(Z)> 0$ as
\begin{align}
 &\para{\boldsymbol{\Phi}}_n^{(\infty)}(\lambda,x)=\para{\boldsymbol{\Phi}}_n^{(\infty)\pm}(\lambda,x)=x^{-\kappa_n}\mathbf{H}(\lambda,x)\mathbf{W}(\ii x\lambda;\kappa_n,\mu_n)\mathbf{K}_n^\pm \nonumber\\&\phantom{\para{\boldsymbol{\Phi}}_n^{(\infty)}(\lambda,x)}{}= \lambda_{\lw}^{\sigma_3/2}x^{\kappa_n\sigma_3}\mathbf{D}_n\mathbf{W}(\ii x\lambda;\kappa_n,\mu_n)\mathbf{K}_n^\pm.
 \label{infinity-parametrix-form}
\end{align}
We now express the matrices $\mathbf{K}_n^\pm$ in terms of the remaining constants $\mu_n$ and $\gamma_n$ by enforcing the asymptotic condition \eqref{eq:Phi-asymptotic-infinity} in each of the two sectors with $|\lambda|>2$. According to \cite[equation~(13.19.3)]{DLMF},
\begin{equation*}
 W_{\kappa,\mu_n}(Z)=\ee^{-Z/2}Z^\kappa\big(1+\mathcal{O}\big(Z^{-1}\big)\big),\qquad Z\to\infty,\qquad |\mathrm{arg}(Z)|\le\frac{3\pi}{2}-\delta
\end{equation*}
holds for each $\delta>0$. Hence also
\begin{gather*} 
\begin{split}
& \mathbf{W}(Z;\kappa,\mu_n)=\begin{bmatrix}
 \alpha_{\kappa,\mu_n}Z^{\kappa-2}\big(1+\mathcal{O}\big(Z^{-1}\big)\big) & (-Z)^{-\kappa}\big(1+\mathcal{O}\big(Z^{-1}\big)\big)\vspace{1mm}\\
 Z^\kappa\big(1+\mathcal{O}\big(Z^{-1}\big)\big) & (-Z)^{-\kappa}\big(1+\mathcal{O}\big(Z^{-1}\big)\big)
 \end{bmatrix}\ee^{-Z\sigma_3/2},\\
 & Z\to\infty,\qquad |\mathrm{arg}(Z)|\le\frac{3\pi}{2}-\delta.
\end{split}
\end{gather*}
Under the condition given on $\mathrm{arg}(Z)$, we have
\begin{equation*}
 (-Z)^{-\kappa}=Z^{-\kappa}\begin{cases} \ee^{\ii\pi\kappa},& \text{$\im(Z)>0$ \ (i.e., $0<\arg(Z)<\pi$),}\\
 \ee^{-\ii\pi\kappa},& \text{$\im(Z)<0$ \ (i.e., $-\pi<\arg(Z)<0$).}
 \end{cases}
\end{equation*}
To calculate $Z^{\pm \kappa}$, we recall $Z=\ii x\lambda$ and use \cite[equation~(49)]{BMS18}:
\[
-\dfrac{\pi}{2} - \arg(x) < \mathrm{arg}_{\lw} (\lambda) < \dfrac{3\pi}{2} - \arg(x), \qquad |\lambda| \to \infty.
\]
Next,
\begin{align*}
\im(Z)>0 &\implies 0<\arg(Z) <\pi\!\!\!\! &&\implies\! -\frac{\pi}{2} - \arg(x)< \mathrm{arg}_{\lw}(\lambda) < \frac{\pi}{2} - \arg(x),\!& \\
\im(Z)<0 &\implies -\pi<\arg(Z) < 0\!\!\!\!\! &&\implies\! -\frac{3\pi}{2} - \arg(x) < \mathrm{arg}_{\lw}(\lambda) - 2\pi <-\dfrac{\pi}{2} - \arg(x)\!&
\end{align*}
and hence, for any $x \in \C \setminus \{0\}$ such that $|{\arg}(x)| < \pi$,
\[
Z^{\pm \kappa} = x^{\pm \kappa} \lambda_{\lw}^{\pm \kappa}\begin{cases} \ee^{\pm \ii \pi \kappa/2}, &\im(Z) > 0, \\ \ee^{\mp 3\ii \pi \kappa/2}, & \im(Z) < 0.
\end{cases}
\]
Therefore, for $\lambda$ large such that $\im(Z)>0$,
\begin{gather*}
 \lambda_{\lw}^{\sigma_3/2}x^{\kappa_n\sigma_3}\mathbf{D}_n\mathbf{W}(\ii x\lambda;\kappa_n,\mu_n)\\
 \qquad=
 \mathbf{D}_n
 \begin{bmatrix}
 \mathcal{O}(\lambda^{-1}) & \ee^{\ii\pi\kappa_n/2}\big(1+\mathcal{O}\big(\lambda^{-1}\big)\big)\\
 \ee^{\ii\pi\kappa_n/2}\big(1+\mathcal{O}\big(\lambda^{-1}\big)\big) & \mathcal{O}\big(\lambda^{-1}\big)
 \end{bmatrix}\lambda_{\lw}^{(\kappa_n -\frac{1}{2})\sigma_3}\ee^{-\ii x\lambda\sigma_3/2},
\end{gather*}
so choosing $\mathbf{K}_n^+$ so that \eqref{infinity-parametrix-form} is consistent with \eqref{eq:Phi-asymptotic-infinity} in the sector $\im(Z)>0$ requires that $\mathbf{K}_n^+$ is an off-diagonal matrix, namely,
\begin{equation*}
 \mathbf{K}_n^+:=\begin{bmatrix}
 0 & \ee^{-\ii\pi\kappa_n/2}\\
 -\ii\ee^{-\ii\pi\kappa_n/2}\mu_n\gamma_n & 0
 \end{bmatrix}.
\end{equation*}
Similarly, for $\lambda$ large such that $\im(Z)<0$,
\begin{gather*}
 \lambda_{\lw}^{\sigma_3/2}x^{\kappa_n\sigma_3}\mathbf{D}_n\mathbf{W}(\ii x\lambda;\kappa_n,\mu_n) \\
 \qquad= \mathbf{D}_n\begin{bmatrix}
 \mathcal{O}(\lambda^{-1}) & \ee^{\ii\pi\kappa_n/2}\big(1+\mathcal{O}\big(\lambda^{-1}\big)\big)\\
 \ee^{-3\ii\pi\kappa_n/2}\big(1+\mathcal{O}\big(\lambda^{-1}\big)\big) & \mathcal{O}\big(\lambda^{-1}\big)
 \end{bmatrix}\lambda_{\lw}^{(\kappa_n-\frac{1}{2})\sigma_3}\ee^{-\ii x\lambda\sigma_3/2},
\end{gather*}
so consistency of \eqref{infinity-parametrix-form} with \eqref{eq:Phi-asymptotic-infinity} in the sector $\im(Z)<0$ requires
\begin{equation*}
 \mathbf{K}_n^-:=\begin{bmatrix}
 0 & \ee^{3\ii\pi\kappa_n/2}\\
 -\ii\ee^{-\ii\pi\kappa_n/2}\mu_n\gamma_n & 0
 \end{bmatrix}.
\end{equation*}

Some additional useful information can be gleaned by enforcing on
$\para{\boldsymbol{\Phi}}_n^{(\infty)}(\lambda,x)$ the jump conditions for $|\lambda|>2$. The jump rays are illustrated in the $Z$-plane with their orientations in Figure~\ref{fig1}.
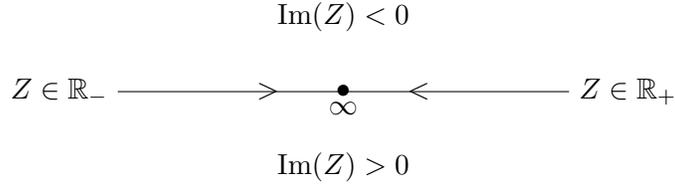
\begin{figure}
 \centering
 \begin{tikzpicture}
 \draw (-8, 0) -- (-2, 0);
 \node at (-5,0){\textbullet};
 \node[below] at (-5,0){$\infty$};

 \node at (-5, 1){$\im(Z) < 0$};
 \node at (-5, -1){$\im(Z) > 0$};

 \node at (-4, 0){$<$};
 \node at (-6, 0){$>$};
 \node[right] at (-2, 0){$Z \in \R_+$};
 \node[left] at (-8, 0){$Z \in \R_-$};
 \end{tikzpicture}
 \caption{Jump contour for $\para{\boldsymbol{\Phi}}_n^{(\infty)}(\lambda,x)$ near $\lambda=\infty$.}
\label{fig1}
\end{figure}
The Whittaker function $W_{\kappa,\mu_n}(Z)$ can be viewed as an analytic function on the cut plane $|{\arg}(Z)|<\pi$, and it follows from the connection formula \cite[equation~(13.14.13)]{DLMF} that the boundary values on the negative real axis are related by
\begin{align}
&W_{\kappa,\mu_n}(-Z+\ii 0)= \ee^{2\pi \ii\kappa}W_{\kappa,\mu_n}(-Z-\ii 0)\nonumber \\
&\hphantom{W_{\kappa,\mu_n}(-Z+\ii 0)=}{} +
 \frac{2\pi \ii\ee^{\ii\pi\kappa}}{\Gamma\big(\frac{1}{2}+\mu_n-\kappa\big)\Gamma\big(\frac{1}{2}-\mu_n-\kappa\big)}W_{-\kappa,\mu_n}(Z),\qquad Z>0.
\label{Whittaker-jump}
\end{align}
Note that the denominators in the second term on the right-hand side of \eqref{Whittaker-jump} are finite due to condition (iii) in the definition of generic data; see the beginning of Section \ref{sec:proof}. Indeed, it follows from~\eqref{eq:e-definitions} and \eqref{kappa-mu} that
\[
e_1^{\pm 2} = e_\infty^2 \ \Leftrightarrow \ \frac{1}{2} - \kappa_n \pm \mu_n \in \Z.
\]

On $Z\in\R_-$ the left ($+$) and right ($-$) boundary values correspond to limits from ${\im(Z)<0}$ and $\im(Z)>0$, respectively. Therefore, the second column of $\mathbf{W}(Z;\kappa,\mu_n)$ is continuous across~$\mathbb{R}_-$, and from \eqref{Whittaker-jump} (replacing $Z$ with $-Z$),
\begin{gather*}
 \mathbf{W}_{-}(Z;\kappa,\mu_n)
 =
 \begin{bmatrix}
 \alpha_{\kappa,\mu_n}Z^{-1}W_{\kappa-1,\mu_n}(Z+\ii 0) & -Z^{-1}W_{1-\kappa,\mu_n}(-Z)\vspace{1mm}\\
 W_{\kappa,\mu_n}(Z+\ii 0) & W_{-\kappa,\mu_n}(-Z)
 \end{bmatrix}\\
 =\begin{bmatrix}
 \displaystyle \ee^{2\pi \ii\kappa}\alpha_{\kappa,\mu_n}Z^{-1}W_{\kappa-1,\mu_n}(Z-\ii 0) - \frac{2\pi \ii\ee^{\ii\pi \kappa}Z^{-1}W_{1-\kappa,\mu_n}(-Z)}{\Gamma\big(\frac{1}{2}+\mu_n-\kappa\big)\Gamma\big(\frac{1}{2}-\mu_n-\kappa\big)}& -Z^{-1}W_{1-\kappa,\mu_n}(-Z)\vspace{1mm}\\
 \displaystyle \ee^{2\pi \ii\kappa}W_{\kappa,\mu_n}(Z-\ii 0)+\frac{2\pi \ii\ee^{\ii\pi\kappa}W_{-\kappa,\mu_n}(-Z)}{\Gamma\big(\frac{1}{2}+\mu_n-\kappa\big)\Gamma\big(\frac{1}{2}-\mu_n-\kappa\big)} & W_{-\kappa,\mu_n}(-Z)
 \end{bmatrix}\\
 =\mathbf{W}_{+}(Z;\kappa,\mu_n)\begin{bmatrix}
 \ee^{2\pi \ii\kappa} & 0\vspace{1mm}\\
 \displaystyle \frac{2\pi \ii\ee^{\ii\pi\kappa}}{\Gamma\big(\frac{1}{2}+\mu_n-\kappa\big)\Gamma\big(\frac{1}{2}-\mu_n-\kappa\big)} & 1
 \end{bmatrix},\qquad Z<0.
\end{gather*}
Here, on the {third} line we used the definition \eqref{M-alpha-define} of $\alpha_{\kappa,\mu_n}$ and the factorial identity $\Gamma(\diamond+1)=\diamond\Gamma(\diamond)$. Since $\mathbf{H}(\lambda,x)$ is analytic across $\ii\mathbb{R}_+$, it follows that \smash{$\para{\mathbf{\Phi}}^{(\infty)}(\lambda)=\para{\mathbf{\Phi}}_n^{(\infty)}(\lambda,x)$} satisfies the jump condition
\begin{align*}
 \para{\boldsymbol{\Phi}}_+^{(\infty)}(\lambda)&=\para{\boldsymbol{\Phi}}_-^{(\infty)}(\lambda)\big[\mathbf{W}_-(\ii x\lambda;\kappa_n,\mu_n)\mathbf{K}_n^{+}\big]^{-1}\mathbf{W}_+(\ii x\lambda;\kappa_n,\mu_n)\mathbf{K}_n^-\\
 &=\para{\boldsymbol{\Phi}}_-^{(\infty)}(\lambda)
 \big(\mathbf{K}^+_n\big)^{-1}
 \begin{bmatrix}
 \ee^{2\pi \ii\kappa_n} & 0\vspace{1mm}\\
 \displaystyle \frac{2\pi \ii\ee^{\ii\pi\kappa_n}}{\Gamma\big(\frac{1}{2}+\mu_n-\kappa_n\big)\Gamma\big(\frac{1}{2}-\mu_n-\kappa_n\big)} & 1
 \end{bmatrix}^{-1}
 \mathbf{K}^-_n\\
 &=\para{\boldsymbol{\Phi}}_-^{(\infty)}(\lambda)\begin{bmatrix}
 1 & \displaystyle\frac{2\pi \ee^{\ii\pi\kappa_n}}{\mu_n\gamma_n\Gamma\big(\frac{1}{2}+\mu_n-\kappa_n\big)\Gamma\big(\frac{1}{2}-\mu_n-\kappa_n\big)}\\
 0 & 1
 \end{bmatrix},\qquad \lambda\in \ii\mathbb{R}_+.
\end{align*}
Requiring that this matches with the corresponding jump condition in Figure \ref{fig:6} gives the condition
\begin{equation*}
 \frac{2\pi \ee^{\ii\pi\kappa_n}}{\mu_n\gamma_n\Gamma\big(\frac{1}{2}+\mu_n-\kappa_n\big)\Gamma\big(\frac{1}{2}-\mu_n-\kappa_n\big)} = \frac{e_\infty^2-e_1^2}{e_1^2e_\infty^4}.
\end{equation*}
For $Z\in\R_+$ the $\pm$ boundary values correspond to the limit from $\im(Z)\gtrless 0$. Therefore, now the first column of $\mathbf{W}(Z;\kappa,\mu_n)$ is continuous across $\mathbb{R}_+$, and from \eqref{Whittaker-jump},
\begin{align*} 
 &\mathbf{W}_{-}(Z;\kappa,\mu_n) =
 \begin{bmatrix}
 \alpha_{\kappa,\mu_n}Z^{-1}W_{\kappa-1,\mu_n}(Z) & -Z^{-1}W_{1-\kappa,\mu_n}(-Z+\ii 0)\vspace{1mm}\\
 W_{\kappa,\mu_n}(Z) & W_{-\kappa,\mu_n}(-Z+\ii 0)
 \end{bmatrix}\\
 & =\begin{bmatrix}
 \! \alpha_{\kappa,\mu_n}Z^{-1}W_{\kappa-1,\mu_n}(Z) \!&\displaystyle\!\! -\ee^{-2\pi \ii\kappa}Z^{-1} W_{1-\kappa,\mu_n}(-Z\!-\ii 0)+\frac{2\pi \ii\ee^{-\ii\pi\kappa}\alpha_{\kappa,\mu_n}Z^{-1}W_{\kappa-1,\mu_n}(Z)}{\Gamma\big(\frac{1}{2} +\mu_n+\kappa\big)\Gamma\big(\frac{1}{2} -\mu_n+\kappa\big)} \! \vspace{1mm}\\
 W_{\kappa,\mu_n}(Z)\!& \!\displaystyle \ee^{-2\pi \ii\kappa}W_{-\kappa,\mu_n}(-Z\!-\ii 0) +
 \frac{2\pi \ii \ee^{-\ii\pi\kappa}W_{\kappa,\mu_n}(Z)}{\Gamma\big(\frac{1}{2} +\mu_n+\kappa\big)\Gamma\big(\frac{1}{2} -\mu_n+\kappa\big)}
 \end{bmatrix} \\
 & =\mathbf{W}_{+}(Z;\kappa,\mu_n)\begin{bmatrix}
 1 & \displaystyle \frac{2\pi \ii\ee^{-\ii\pi\kappa}}{\Gamma\big(\frac{1}{2}+\mu_n+\kappa\big)\Gamma\big(\frac{1}{2}-\mu_n+\kappa\big)} \vspace{1mm} \\
 0 & \ee^{-2\pi \ii\kappa}
 \end{bmatrix},\qquad Z>0.
\end{align*}
Again here, the finiteness of the denominators is guaranteed by condition (iii) at the beginning of Section \ref{sec:proof}. Since $\mathbf{H}(\lambda,x)$ changes sign across $\ii\mathbb{R}_-$, we get that \smash{$\para{\mathbf{\Phi}}^{(\infty)}(\lambda)=\para{\mathbf{\Phi}}^{(\infty)}_n(\lambda,x)$} satisfies the jump condition
\begin{align*}
 \para{\boldsymbol{\Phi}}_+^{(\infty)}(\lambda)&=-\para{\boldsymbol{\Phi}}_-^{(\infty)}(\lambda)[\mathbf{W}_{-}(\ii x\lambda;\kappa_n,\mu_n)\mathbf{K}_n^-]^{-1}\mathbf{W}_{+}(\ii x\lambda;\kappa_n,\mu_n)\mathbf{K}_n^+\\
 &=\para{\boldsymbol{\Phi}}_-^{(\infty)}(\lambda)(\mathbf{K}^-_n)^{-1}\begin{bmatrix}
 -1 & \displaystyle -\frac{2\pi \ii\ee^{-\ii\pi\kappa_n}}{\Gamma\big(\frac{1}{2}+\mu_n+\kappa_n\big)\Gamma\big(\frac{1}{2}-\mu_n+\kappa_n\big)}\vspace{1mm}\\
 0 & -\ee^{-2\pi \ii\kappa_n}
 \end{bmatrix}^{-1}\mathbf{K}_n^+\\
 &=\para{\boldsymbol{\Phi}}_-^{(\infty)}(\lambda)\begin{bmatrix}
 -\ee^{2\pi \ii\kappa_n} & 0\vspace{1mm}\\
 \displaystyle \frac{2\pi \ee^{-\ii\pi\kappa_n}\mu_n\gamma_n }{\Gamma\big(\frac{1}{2}+\mu_n+\kappa_n\big)\Gamma\big(\frac{1}{2}-\mu_n+\kappa_n\big)} &-\ee^{-2\pi \ii\kappa_n}
 \end{bmatrix} \\
 &= \para{\boldsymbol{\Phi}}_-^{(\infty)}(\lambda) \begin{bmatrix}e_{\infty}^{2}&0 \vspace{1mm}\\ e_\infty^2s_2^\infty&e_{\infty}^{-2}\end{bmatrix} = \para{\boldsymbol{\Phi}}_-^{(\infty)}(\lambda) \mathbf{S}_2^\infty e_\infty^{2\sigma_3}, \qquad \lambda\in\mathrm{i}\mathbb{R}_-.
\end{align*}
The last two equalities follow by a direct calculation using the definitions of $s_2^\infty$, $\kappa_n$, $e_1$, and $e_\infty$ in \eqref{eq:stokes-parameters}, \eqref{kappa-mu}, and \eqref{eq:e-definitions}, respectively, along with the expression
\begin{equation} \label{Qinfty-gammainfty}
 \mu_n \gamma_n = \dfrac{2\pi \ee^{\pi \ii \kappa_n}}{\Gamma\big( \frac{1}{2} + \mu_n - \kappa_n\big)\Gamma\big( \frac{1}{2} - \mu_n - \kappa_n\big)} \cdot \frac{e_1^2 e_\infty^4}{e_\infty^2 - e_1^2},
\end{equation}
and the classical identity
\begin{equation}
\label{eq:Gamma-reflection}
\Gamma \left(\frac{1}{2} - z \right) \Gamma \left(\frac{1}{2} + z \right) = \dfrac{\pi }{\cos(\pi z)}.
\end{equation}

\subsubsection[The parametrix Phi\^(infty)\_n(lambda,x) in the region lambda less than 2 in absolute value]{The parametrix $\boldsymbol{\para{\boldsymbol{\Phi}}^{(\infty)}_n(\lambda,x)}$ in the region $\boldsymbol{|\lambda|<2}$}

We use the identity \cite[equation~(13.14.33)]{DLMF} to express the elements of $\mathbf{W}(Z;\kappa,\mu)$ in terms of the alternative basis of solutions $M_{-\kappa,\pm\mu}(-Z)$ of Whittaker's equation with parameters $(\kappa,\mu)$ that form a numerically satisfactory pair in a neighborhood of the origin and that are analytic for $\arg(-Z)\in (-\pi,\pi)$. Moreover, these functions are the Maclaurin series associated with the regular singular point at $Z=0$, so they have the property that
\begin{equation}
M_{\kappa,\mu}(-Z)(-Z)^{-\frac{1}{2}-\mu} = 1+\mathcal{O}(Z) \qquad \text{as} \quad Z\to 0,
\label{eq:WhittakerM-near-origin}
\end{equation}
where the power function denotes the principal branch and where the error term represents an analytic function of $Z$ vanishing at the origin. To deal with the first column of $\mathbf{W}(Z;\kappa,\mu)$ we also use the corresponding identity $M_{\kappa,\mu}(Z)=\ee^{\pm \ii\pi(\frac{1}{2}+\mu)}M_{-\kappa,\mu}(-Z)$ which holds for $\pm\im(Z)>0$ (see also \cite[equation~(13.14.10)]{DLMF}).
Using the above identities, and under the condition $2\mu \not \in \Z$ (which follows from the condition~(i) at the beginning of Section~\ref{sec:proof} in our case), we can write the elements of $\mathbf{W}(Z;\kappa,\mu)$ in the form
\begin{gather*}\nonumber
 W_{11}(Z;\kappa,\mu)= -Z^{-1}\frac{\Gamma(-2\mu)\Gamma\big(\frac{1}{2}-\mu+\kappa\big)}{\Gamma\big(\frac{1}{2}-\mu-\kappa\big)\Gamma\bigl(-\frac{1}{2}-\mu+\kappa\bigr)}\ee^{\pm \ii\pi(\frac{1}{2}+\mu)}M_{1-\kappa,\mu}(-Z)\\
\hphantom{W_{11}(Z;\kappa,\mu)=}{}
-Z^{-1}\frac{\Gamma(2\mu)\Gamma\big(\frac{1}{2}+\mu+\kappa\big)}{\Gamma\big(\frac{1}{2}+\mu-\kappa\big)\Gamma\bigl(-\frac{1}{2}+\mu+\kappa\bigr)}
\ee^{\pm \ii\pi(\frac{1}{2}-\mu)}M_{1-\kappa,-\mu}(-Z),\quad\! \pm\im(Z)>0,\\
 W_{12}(Z;\kappa,\mu)= -Z^{-1}\frac{\Gamma(-2\mu)}{\Gamma\bigl(-\frac{1}{2}-\mu+\kappa\bigr)}M_{1-\kappa,\mu}(-Z)-Z^{-1}\frac{\Gamma(2\mu)}{\Gamma\bigl({-}\frac{1}{2}+\mu+\kappa\bigr)}M_{1-\kappa,-\mu}(-Z),\\
 W_{21}(Z;\kappa,\mu)= \frac{\Gamma(-2\mu)}{\Gamma\big(\frac{1}{2}-\mu-\kappa\big)}\ee^{\pm \ii\pi (\frac{1}{2}+\mu)}M_{-\kappa,\mu}(-Z)\\
 \hphantom{W_{21}(Z;\kappa,\mu)= }{}
 +
 \frac{\Gamma(2\mu)}{\Gamma\big(\frac{1}{2}+\mu-\kappa\big)}\ee^{\pm \ii\pi(\frac{1}{2}-\mu)}M_{-\kappa,-\mu}(-Z),\qquad\pm\im(Z)>0,
\end{gather*}%
and
\begin{equation*}
 W_{22}(Z;\kappa,\mu)=\frac{\Gamma(-2\mu)}{\Gamma\big(\frac{1}{2}-\mu+\kappa\big)}M_{-\kappa,\mu}(-Z) +\frac{\Gamma(2\mu)}{\Gamma\big(\frac{1}{2}+\mu+\kappa\big)}M_{-\kappa,-\mu}(-Z).
\end{equation*}
These expressions can be usefully combined into a matrix identity:
\begin{equation}
 \mathbf{W}(Z;\kappa,\mu)=\mathbf{M}(Z;\kappa,\mu)\mathbf{G}^\pm_{\kappa,\mu},\qquad \pm\im(Z)>0,
 \label{F-to-N}
\end{equation}
where
\begin{gather}
 \mathbf{M}(Z;\kappa,\mu):=
 \begin{bmatrix}
 \big(\frac{1}{2}-\kappa+\mu\big)Z^{-1}M_{1-\kappa,\mu}(-Z) & \big(\frac{1}{2}-\kappa-\mu\big)Z^{-1}M_{1-\kappa,-\mu}(-Z)\vspace{1mm}\\
 M_{-\kappa,\mu}(-Z) & M_{-\kappa,-\mu}(-Z)
 \end{bmatrix},\nonumber\\
 \arg(-Z)\in (-\pi,\pi),
 \label{N-define}
\end{gather}
and
\begin{equation*} 
 \mathbf{G}^\pm_{\kappa,\mu}:=\begin{bmatrix}
 \displaystyle \frac{\Gamma(-2\mu)\ee^{\pm \ii\pi(\frac{1}{2}+\mu)}}{\Gamma\big(\frac{1}{2}-\mu-\kappa\big)} &
 \displaystyle \frac{\Gamma(-2\mu)}{\Gamma\big(\frac{1}{2}-\mu+\kappa\big)}\vspace{1mm}\\
 \displaystyle\frac{\Gamma(2\mu)\ee^{\pm \ii\pi(\frac{1}{2}-\mu)}}{\Gamma\big(\frac{1}{2}+\mu-\kappa\big)} &
 \displaystyle \frac{\Gamma(2\mu)}{\Gamma\big(\frac{1}{2}+\mu+\kappa\big)}
 \end{bmatrix}.
\end{equation*}

To define the parametrix $\para{\boldsymbol{\Phi}}_n^{(\infty)}(\lambda,x)$ for $|\lambda|<2$, we first introduce a constant matrix by
\begin{equation}
 \mathbf{J}_n:=\mathbf{G}_{\kappa_n,\mu_n}^+ \mathbf{K}_n^+ \mathbf{S}_1^{\infty} \mathbf{E}^{\infty}=\mathbf{G}_{\kappa_n,\mu_n}^- \mathbf{K}_n^- \mathbf{E}^{\infty}.
\label{eq:M-infty}
\end{equation}
The equality of these two expressions can be seen as follows. First, combining \eqref{infinity-parametrix-form} and \eqref{F-to-N}, and using the fact that the matrix $\mathbf{M}(\ii x\lambda;\kappa,\mu_n)$ is analytic in a neighborhood of $\lambda=2\ii$, the jump condition for \smash{$\para{\mathbf{\Phi}}_n^{(\infty)}(\lambda,x)$} across the positive imaginary axis for $|\lambda|>2$ shown in Figure~\ref{fig:6} implies the identity $\mathbf{G}_{\kappa_n,\mu_n}^+\mathbf{K}_n^+ \mathbf{S}_1^{\infty}=\mathbf{G}_{\kappa_n,\mu_n}^-\mathbf{K}_n^-$, which yields the desired equality.

Then, we set
\begin{equation}
 \para{\boldsymbol{\Phi}}_n^{(\infty)}(\lambda,x):= \lambda_{\lw}^{\sigma_3/2}x^{\kappa_n\sigma_3}\mathbf{D}_n\mathbf{M}(\ii x\lambda;\kappa_n,\mu_n)\mathbf{J}_n,\qquad\text{for} \quad |\lambda|<2.
 \label{infinity-parametrix-form-origin}
\end{equation}
It is straightforward to then check that, regardless of the choice of $\mu_n$,
the matrix $\mathbf{J}_n$ defined by~\eqref{eq:M-infty} is diagonal. Comparing \eqref{infinity-parametrix-form} and \eqref{infinity-parametrix-form-origin} shows that the jump conditions for~\smash{$\para{\mathbf{\Phi}}_n^{(\infty)}(\lambda,x)$} across the arcs of the circle $|\lambda|=2$ shown in Figure~\ref{fig:6} are satisfied. Using~\eqref{eq:WhittakerM-near-origin} then proves that~\smash{$\para{\mathbf{\Phi}}_n^{(\infty)}(\lambda,x)$} satisfies the simple jump condition across the negative imaginary axis with~${|\lambda|<2}$ shown in Figure~\ref{fig:6} and that an expansion of the form shown in \eqref{eq:Phi-asymptotic-zero} holds.
To check that the matrix $\mathbf{J}_n$ is diagonal and arrive at its final form below, we use the identity~\eqref{eq:Gamma-reflection} to get
\begin{gather}
\Gamma\left( \dfrac{1}{2} + \kappa_n - \mu_n \right) \Gamma\left( \dfrac{1}{2} - \kappa_n + \mu_n \right) = \dfrac{\pi}{\cos \left( \pi (\kappa_n - \mu_n)\right)} = \dfrac{2 \pi \ii e_1e_\infty}{e_1^2 - e_\infty^2}, \nonumber\\
\Gamma\left( \dfrac{1}{2} + \kappa_n + \mu_n \right) \Gamma\left( \dfrac{1}{2} - \kappa_n - \mu_n \right) = \dfrac{\pi}{\cos \left( \pi (\kappa_n + \mu_n)\right)} = \dfrac{2\pi \ii e_1e_\infty}{1 - e_1^2 e_\infty^2}.\label{eq:cosine-identity-1}
\end{gather}
The result is that the diagonal matrix $\mathbf{J}_n$ from \eqref{eq:M-infty} is given by
\begin{equation}
\mathbf{J}_n =
\begin{bmatrix}\dfrac{\ee^{{\ii\pi/4}} e_1 e_\infty^{7/2} \Gamma(-2 \mu_n )}{\Gamma \big(\frac{1}{2}-\kappa_n -\mu_n \big)}&0 \\
0&\dfrac{\ee^{{\ii\pi/4}} \big(e_1^4-1\big) e_\infty^{3/2} \Gamma(2 \mu_n )}{e_1 \big(e_1^2-e_\infty^2\big) \Gamma \big(\frac{1}{2}-\kappa_n +\mu_n \big)}
\label{eq:J-n}
\end{bmatrix}.
\end{equation}

\subsection[Parametrix for Phi\_n(lambda,x) near lambda=0]{Parametrix for $\boldsymbol{{\Phi}_n(\lambda,x)}$ near $\boldsymbol{\lambda=0}$}
By definition, the parametrix $\para{\boldsymbol \Phi}_n^{(0)}(\lambda,x)$ satisfies the following Riemann--Hilbert problem.
\begin{rhp} \label{rhp:zero_parametrix}
Fix generic monodromy parameters $( e_1, e_2 )$ determining the Stokes and connection matrices, $n\in\mathbb{Z}$, and $x>0$. We seek a $2 \times 2$ matrix function $\lambda\mapsto\para{\boldsymbol \Phi}_n^{(0)}(\lambda,x)$ satisfying:
\begin{itemize}\itemsep=0pt
\item Analyticity: $\para{\boldsymbol \Phi}_n^{(0)}(\lambda,x)$ is analytic in $\C \setminus \Gamma^{(0)}$, where $\Gamma^{(0)} = \big\{ | \lambda| = \frac{1}{2}\big\} \cup \big(\ii \R\cap \big\{ \im \lambda < \frac{1}{2}\big\}\big)$ is the jump contour shown in Figure {\rm\ref{fig:8}}.
\item  Jump condition: $\para{\boldsymbol \Phi}_n^{(0)}(\lambda,x)$ has continuous boundary values on $\Gamma^{(0)} \setminus \{0\}$ from each component of $\mathbb{C}\setminus \Gamma^{(0)}$, which satisfy
\[
\para{\boldsymbol \Phi}_{n,+}^{(0)}(\lambda,x) = \para{\boldsymbol \Phi}_{n,-}^{(0)}(\lambda,x)\mathbf J_{\para{\boldsymbol \Phi}_n^{(0)}}(\lambda),
\]
 where \smash{$\mathbf J_{\para{\boldsymbol \Phi}_n^{(0)}}(\lambda)$} is as shown in Figure {\rm\ref{fig:8}} and where the $+$ $($resp., $-)$ subscript denotes a boundary value taken from the left $($resp., right$)$ of an arc of $\Gamma^{(0)}$.
\item  Normalization: $\para{\boldsymbol \Phi}_n^{(0)}(\lambda,x)$ satisfies the asymptotic conditions
\begin{gather}
\label{eq:Phi_z-asymptotic-infinity}
\para{\boldsymbol \Phi}_n^{(0)}(\lambda,x) = \mathcal{O}(1) \lambda_{\lw}^{ {\mu_n \sigma_3}{}} \qquad \text{as} \quad \lambda \to \infty,
\end{gather}
where $\mathcal{O}(1)$ refers to a function analytic and bounded in a neighborhood of $\lambda=\infty$ and
\begin{equation}
\label{eq:Phi_z-asymptotic-zero}
\para{\boldsymbol \Phi}_n^{(0)}(\lambda,x) = ( \mathbb I + \mathcal{O}(\lambda) )\ee^{-\ii x\lambda^{-1}\sigma_3/2} \lambda_{\lw}^{ (n+ \Theta_0 )\sigma_3/2 } \qquad \text{as} \quad \lambda \to 0.
\end{equation}
\end{itemize}
\end{rhp}

 \begin{figure}
 \centering
\includegraphics[width=0.4\columnwidth]{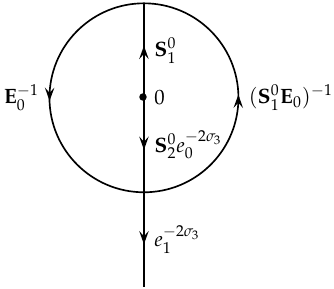}
\caption{The contour $\Gamma^{(0)}$, which includes the circle $|\lambda|=\frac{1}{2}$.}\label{fig:8}
\end{figure}
We can write down the unique solution $\para{\mathbf{\Phi}}_n^{(0)}(\lambda,x)$ explicitly in terms of the parametrix \smash{$\para{\mathbf{\Phi}}_n^{(\infty)}(\lambda,x)$} obtained in Section~\ref{sec:Parametrix-infinity}, but taken with the index $1-n$ instead of $n$ and $\Theta_\infty$ replaced by $\Theta_0$. If we indicate the dependence of \smash{$\para{\mathbf{\Phi}}_n^{(\infty)}(\lambda,x)$} and \smash{$\para{\mathbf{\Phi}}_n^{(0)}(\lambda,x)$} on $\Theta_\infty$ and $\Theta_0$ respectively with the notation \smash{$\para{\mathbf{\Phi}}_n^{(\infty)}(\lambda,x)=\para{\mathbf{\Phi}}_n^{(\infty)}(\lambda,x,\Theta_\infty)$} and \smash{$\para{\mathbf{\Phi}}_n^{(0)}(\lambda,x)=\para{\mathbf{\Phi}}_n^{(0)}(\lambda,x,\Theta_0)$},
then we have the following.
\begin{prop}\label{prop:symmetry}
Fix $\Theta_\infty, \Theta_0 \in \C$ and generic monodromy parameters $(e_1, e_2)$. Then
\begin{equation}
 \para{\mathbf{\Phi}}^{(0)}_n(\lambda,x, \Theta_0)=\begin{cases}
 e_0^{3\sigma_3}\para{\mathbf{\Phi}}^{(\infty)}_{-n}\bigl(-\lambda^{-1},x, \Theta_0\bigr) e_0^{-2\sigma_3}\quad ,& |\lambda|<\frac{1}{2}, \vspace{1mm}\\
 e_0^{3\sigma_3}\para{\mathbf{\Phi}}^{(\infty)}_{-n}\bigl(-\lambda^{-1},x, \Theta_0\bigr)\begin{bmatrix}0 & \beta_n\\-\beta_n^{-1} & 0\end{bmatrix},& |\lambda|>\frac{1}{2},
 \end{cases}
 \label{eq:Phi0-Phiinfty}
\end{equation}
where
\begin{equation}
 \beta_n:= \dfrac{1-e_1^4 }{e_1^2\big(1 - e_0^2 e_1^2\big)}.
 \label{eq:a-def}
\end{equation}
\end{prop}
\begin{proof}
The mapping $\lambda\mapsto -\lambda^{-1}$ takes the contour $\Gamma^{(0)}$ onto the contour $\Gamma^{(\infty)}$ up to the reversal of orientation of certain arcs, and swaps the circles centered at the origin of radii $\frac{1}{2}$ and $2$. Therefore, the domain of analyticity of \smash{$\para{\mathbf{\Phi}}^{(0)}_n(\lambda,x, \Theta_0)$} is as desired.

Under the map $n \mapsto -n$, $\Theta_\infty \mapsto \Theta_0$, the exponentials defined in \eqref{eq:e-definitions} satisfy $e_{0}^2 \mapsto e_{\infty}^2$, whereas $\mu_n = \mu_{-n}$ since this quantity depends only on the parity of $n$. This implies that the Stokes matrices defined in \eqref{eq:Stokes-infty}--\eqref{eq:Stokes-zero} satisfy the corresponding identities
\begin{equation*}
 \mathbf{S}^\infty_1 \mapsto e_0^{-2\sigma_3}\big(\mathbf{S}^0_{1}\big)^{-1}e_0^{2\sigma_3} \qquad\text{and}\qquad
 \mathbf{S}^\infty_{{2}}e_\infty^{2\sigma_3} \mapsto e_0^{-2\sigma_3}\big(\mathbf{S}^0_{2} e_0^{-2\sigma_3}\big)^{-1} e_0^{2\sigma_3}.
\end{equation*}
Comparing Figures~\ref{fig:6} and \ref{fig:8} then shows that the function defined by \eqref{eq:Phi0-Phiinfty} satisfied the required jump conditions across the imaginary axis for $|\lambda|<\frac{1}{2}$. Likewise, the jump condition on the negative imaginary axis for $|\lambda|>\frac{1}{2}$ is easily verified due to the identity valid for any $a\neq 0$:
\begin{equation*}
 e_1^{-2 \sigma_3}\begin{bmatrix}0&a\\-a^{-1} & 0\end{bmatrix} +\begin{bmatrix}0 & a\\-a^{-1} & 0\end{bmatrix}e_1^{2 \sigma_3}=\mathbf{0}.
\end{equation*}
Finally, the fact that $\para{\mathbf{\Phi}}^{(0)}_n(\lambda,x, \Theta_0)$ defined by \eqref{eq:Phi0-Phiinfty} satisfies the required jump conditions across the circle $|\lambda|=\frac{1}{2}$ follows from the corresponding jump conditions for \smash{$\para{\mathbf{\Phi}}^{(\infty)}_{1-n}(\lambda,x, \Theta_0)$} for $|\lambda|=2$ and the identities
\begin{equation*}
 (\mathbf{E}^{\infty})^{-1} \mapsto \begin{bmatrix}0 & \beta_n\\-\beta_n^{-1} & 0\end{bmatrix} \big(\mathbf{S}_1^0 \mathbf{E}^0\big)^{-1} e_0^{2\sigma_3}, \qquad(\mathbf{S}_1^\infty \mathbf{E}^{\infty})^{-1} \mapsto \begin{bmatrix}0 & \beta_n\\-\beta_n^{-1} & 0\end{bmatrix} \big( \mathbf{E}^0\big)^{-1} e_0^{2\sigma_3}
\end{equation*}
among the matrices $\mathbf{E}^\infty$, $\mathbf{E}^0$ defined in \eqref{eq:E-infinity-formula}--\eqref{eq:E-zero-formula}, which hold for the value of $\beta_n$ indicated in~\eqref{eq:a-def}.

It only remains to verify the asymptotics in \eqref{eq:Phi_z-asymptotic-infinity}--\eqref{eq:Phi_z-asymptotic-zero}. However, these follow from the corresponding formul\ae\ in \eqref{eq:Phi-asymptotic-infinity}--\eqref{eq:Phi-asymptotic-zero} with the help of the identity
\begin{equation*}
 \lambda^p_{\lw}=\ee^{\ii \pi p}\zeta^{-p}_{\lw},\qquad\zeta:=-\lambda^{-1},
\end{equation*}
which holds for all $\lambda$ not on the negative imaginary axis.

Since the matrix function $\para{\mathbf{\Phi}}_n^{(0)}(\lambda,x, \Theta_0)$ defined by \eqref{eq:Phi0-Phiinfty}--\eqref{eq:a-def} satisfies all the required Riemann--Hilbert conditions, and there is at most one solution of those conditions, as is easily confirmed by a Liouville argument, the proof is finished.
\end{proof}

\subsection{An equivalent Riemann--Hilbert problem on the unit circle}
The parametrix for $\mathbf{\Phi}_n(\lambda,x)$ is by definition the following matrix function:
\begin{equation*} 
 \para{\mathbf{\Phi}}_n(\lambda,x):=\begin{cases}
 \para{\mathbf{\Phi}}^{(\infty)}_n(\lambda,x, \Theta_\infty),& |\lambda|>1,\\
 \para{\mathbf{\Phi}}_n^{(0)}(\lambda,x, \Theta_0),& |\lambda|<1.
 \end{cases}
\end{equation*}
This matrix function satisfies exactly the same jump conditions in the domains $|\lambda|>1$ and~${|\lambda|<1}$ as does $\mathbf{\Phi}_n(\lambda,x)$ itself, and it is also consistent with the asymptotics given in \eqref{eq:Psi-n-asymptotic-infinity}--\eqref{eq:Psi-n-asymptotic-zero} (note that $\mathbf{\Psi}_n(\lambda,x)=\mathbf{\Phi}_n(\lambda,x)$ for $|\lambda|$ sufficiently large or small). The parametrix has unit determinant, so the matrix quotient
\begin{equation*}
 \mathbf{Q}_n(\lambda,x):=\mathbf{\Phi}_n(\lambda,x)\para{\mathbf{\Phi}}_n(\lambda,x)^{-1}
\end{equation*}
is an analytic function of $\lambda$ except possibly on the jump contour $\Gamma$ shown in Figure~\ref{2} and on the unit circle, where there is a discontinuity in the definition of $\para{\mathbf{\Phi}}_n(\lambda,x)$. However, since the jumps of $\para{\mathbf{\Phi}}_n(\lambda,x)$ and $\mathbf{\Phi}_n(\lambda,x)$ agree on $\Gamma$, a Morera argument shows that $\mathbf{Q}_n(\lambda,x)$ is actually analytic both for $|\lambda|>1$ and for $0<|\lambda|<1$. The asymptotic behavior of the factors in $\mathbf{Q}_n(\lambda,x)$ as $\lambda\to 0$ then shows that any singularity of $\mathbf{Q}_n(\lambda,x)$ at the origin $\lambda=0$ is removable, and the asymptotic behavior of the same factors in the limit $\lambda\to\infty$ shows that $\mathbf{Q}_n(\lambda,x)\to\mathbb{I}$ as $\lambda\to\infty$.

$\mathbf{Q}_n(\lambda,x)$ is therefore characterized by its jump condition across the unit circle $|\lambda|=1$. Taking counterclockwise orientation for the circle, the jump condition for $\mathbf{Q}_n(\lambda,x)$ reads
\begin{equation*}
 \mathbf{Q}_{n,+}(\lambda,x)=\mathbf{Q}_{n,-}(\lambda,x)\para{\mathbf{\Phi}}_n^{(\infty)}(\lambda,x, \Theta_\infty)e_2^{\sigma_3}\para{\mathbf{\Phi}}_n^{(0)}(\lambda,x, \Theta_0)^{-1},\qquad |\lambda|=1.
\end{equation*}
Using Proposition~\ref{prop:symmetry}, the jump matrix can be written as
\begin{gather}
 \para{\mathbf{\Phi}}_n^{(\infty)}(\lambda,x, \Theta_\infty)e_2^{\sigma_3}\para{\mathbf{\Phi}}_n^{(0)}(\lambda,x, \Theta_0)^{-1}\nonumber\\
 \qquad =\para{\mathbf{\Phi}}^{(\infty)}_n(\lambda,x, \Theta_\infty) e_2^{\sigma_3}
 \begin{bmatrix}0 & -\beta_n\\ \beta_n^{-1} & 0\end{bmatrix}
 \para{\mathbf{\Phi}}^{(\infty)}_{-n}(-\lambda^{-1},x, \Theta_0)^{-1}e_0^{-3\sigma_3},\qquad |\lambda|=1.\label{eq:VQ-rewrite}
\end{gather}
We summarize by writing the Riemann--Hilbert problem for $\mathbf{Q}_n(\lambda,x)$.
\begin{rhp}\label{rhp:Q}
Fix generic monodromy parameters $( e_1, e_2)$, $n\in\mathbb{Z}$, and $x\in\mathbb{C}$. Seek a $2\times 2$ matrix function $\lambda\mapsto\mathbf{Q}_n(\lambda,x)$ with the following properties:
\begin{itemize}\itemsep=0pt\samepage
 \item Analyticity: $\mathbf{Q}_n(\lambda,x)$ is an analytic function of $\lambda$ for $|\lambda|\neq 1$.
 \item Jump condition:
 $\mathbf{Q}_n(\lambda,x)$ takes analytic boundary values on the unit circle from the interior and exterior, denoted $\mathbf{Q}_{n,+}(\lambda,x)$ and $\mathbf{Q}_{n,-}(\lambda,x)$ for $|\lambda|=1$ respectively, and they are related by
 \begin{equation*}
 \mathbf{Q}_{n,+}(\lambda,x)=\mathbf{Q}_{n,-}(\lambda,x)\breve{\mathbf{\Phi}}_n^{(\infty)}(\lambda,x,\Theta_\infty)e_2^{\sigma_3}
 \begin{bmatrix}0 & -\beta_n\\\beta_n^{-1} & 0\end{bmatrix}
 \para{\mathbf{\Phi}}^{(\infty)}_{-n}\bigl(-\lambda^{-1},x, \Theta_0\bigr)^{-1}e_0^{-3\sigma_3}.
 \end{equation*}
 \item Normalization:
 $\mathbf{Q}_n(\lambda,x)\to\mathbb{I}$ as $\lambda\to\infty$.
\end{itemize}
\end{rhp}
Henceforth, to avoid the notation becoming unwieldy, we understand that all quantities appearing with subscript $n$ are evaluated at parameter $\Theta_\infty$ while quantities appearing with subscript $-n$ are evaluated at parameter $\Theta_0$.

\subsection[The limit n to infinity]{The limit $\boldsymbol{n\to+\infty}$}
Having succeeded in removing the problematic jump conditions along rays emanating from~$0$,~$\infty$ in the $\lambda$ plane by defining $\mathbf{Q}_n(\lambda,x)$, we would next like to consider the limiting behavior of this problem as $n\to+\infty$ with $x=z/n$ and $z$ fixed. It is convenient to first renormalize~$\mathbf{Q}_n(\lambda,x)$, essentially by a transformation that diagonalizes the coefficient $\mu_n \mathbf{B}_n(x)\sigma_3\mathbf{B}_n(x)^{-1}$ of $\lambda^{-1}$ in the matrix of the Lax equation \eqref{Lax-system-lambda}.
In other words, in the jump condition for $\mathbf{Q}_n(\lambda,x)$ we prefer to replace \smash{$\para{\mathbf{\Phi}}_n^{(\infty)}(\lambda,x)$} with a suitable left-diagonal multiple of~\smash{$\mathbf{B}_n(x)^{-1}\para{\mathbf{\Phi}}_n^{(\infty)}(\lambda,x)$}. Observe that the coefficient $\mathbf{B}_n(x)$ is determined up to right-multiplication by a diagonal matrix by
\eqref{eq:zero-coefficient-matrix}, in which the second row of the matrix on the right-hand side is $(c_n(x),-a_n)=\big(\gamma_n x^{1-2\kappa_n},\mu_n ^{-1}\big(\kappa_n-\frac{1}{2}\big)\big)$, where we used \eqref{kappa-mu} and \eqref{c-of-x}. Indeed, the first column \smash{$\mathbf{b}_n^{(1)}(x)$} satisfies \smash{$\big(\gamma_n x^{1-2\kappa_n},\mu_n ^{-1}\big(\kappa_n-\frac{1}{2}\big)-1\big)\mathbf{b}_n^{(1)}(x)=0$} while the second column \smash{$\mathbf{b}_n^{(2)}(x)$} satisfies \smash{$\big(\gamma_n x^{1-2\kappa_n},\mu_n ^{-1}\big(\kappa_n-\frac{1}{2}\big)+1\big)\mathbf{b}_n^{(2)}(x)=0$}. By selecting specific constant factors for each column, we obtain a matrix $\mathbf{P}_n(x)$ differing from $\mathbf{B}_n(x)$ by right-multiplication by a diagonal matrix, and given explicitly by
\begin{equation*}
 \mathbf{P}_n(x):=\begin{bmatrix}\frac{1}{2}-\kappa_n+\mu_n & \frac{1}{2}-\kappa_n-\mu_n \\
 \mu_n \gamma_n x^{1-2\kappa_n} & \mu_n \gamma_n x^{1-2\kappa_n}\end{bmatrix},
\end{equation*}
in which the dependence on the index $n$ enters via \eqref{kappa-mu} and \eqref{Qinfty-gammainfty}. Then to get the desired modification of the jump matrix we set
\begin{equation*}
 \mathbf{R}_n(\lambda,x):=\begin{cases}
 \mathbf{P}_n(x)^{-1}\mathbf{Q}_n(\lambda,x)\mathbf{P}_n(x),& |\lambda|>1,\\
 \mathbf{P}_n(x)^{-1}\mathbf{Q}_n(\lambda,x)e_0^{3\sigma_3}\mathbf{P}_{-n}(x)d_n(x),& |\lambda|<1,
 \end{cases}
\end{equation*}
where $d_n(x)$ is a scalar satisfying
\begin{equation}
 d_n(x)^2 = \frac{\det(\mathbf{P}_{n}(x))}{\det(\mathbf{P}_{-n}(x))}.
 \label{eq:dn-squared}
\end{equation}
Then, $\mathbf{R}_n(\lambda,x)$ solves the following
Riemann--Hilbert problem.

\begin{rhp}
Fix generic monodromy parameters $( e_1, e_2)$, $n\in\mathbb{Z}$, and $x\in\mathbb{C}$. Seek a $2\times 2$ matrix function $\lambda\mapsto\mathbf{R}_n(\lambda,x)$ with the following properties:
\begin{itemize}\itemsep=0pt
 \item Analyticity: $\mathbf{R}_n(\lambda,x)$ is an analytic function of $\lambda$ for $|\lambda|\neq 1$.
 \item Jump condition:
 $\mathbf{R}_n(\lambda,x)$ takes analytic boundary values on the unit circle from the interior and exterior, denoted $\mathbf{R}_{n,+}(\lambda,x)$ and $\mathbf{R}_{n,-}(\lambda,x)$ for $|\lambda|=1$ respectively, and they are related by
 \begin{align}
 & \mathbf{R}_{n,+}(\lambda,x)= \mathbf{R}_{n,-}(\lambda,x)\mathbf{P}_n(x)^{-1}\para{\mathbf{\Phi}}_n^{(\infty)}(\lambda,x)e_2^{\sigma_3}\begin{bmatrix}0 & -\beta_n\\\beta_n^{-1} & 0\end{bmatrix}\nonumber\\
 &\hphantom{\mathbf{R}_{n,+}(\lambda,x)=}{} \times
 \para{\mathbf{\Phi}}_{-n}^{(\infty)}\bigl(-\lambda^{-1},x\bigr)^{-1}\mathbf{P}_{-n}(x)d_n(x)\label{eq:R-jump}
 \end{align}
 \item Normalization:
 $\mathbf{R}_n(\lambda,x)\to\mathbb{I}$ as $\lambda\to\infty$.
\end{itemize}
\end{rhp}
The matrices $\mathbf{\Xi}_n^{(6)}(x)$ and $\mathbf{\Delta}_n^{(6)}(x)$ defined in \eqref{eq:T-n-definition} and \eqref{eq:Psi-n-asymptotic-zero}, respectively, can be expressed in terms of $\mathbf{R}_n(\lambda,x)$ as follows:
\begin{gather}
 \mathbf{\Xi}_n^{(6)}(x)=\mathbf{A}_n(x)+\mathbf{P}_n(x)\big[\lim_{\lambda\to\infty}\lambda(\mathbf{R}_n(\lambda,x)-\mathbb{I})\big]\mathbf{P}_n(x)^{-1}-\frac{1}{2}\ii x\sigma_3,\nonumber\\
 \mathbf{\Delta}_n^{(6)}(x)=\mathbf{P}_n(x)\mathbf{R}_n(0,x)\mathbf{P}_{-n}(x)^{-1}d_n(x)^{-1}e_0^{-3\sigma_3}.\label{eq:TU-R}
\end{gather}
Here, $\mathbf{A}_n(x)$ is the matrix coefficient defined in \eqref{eq:Phi-asymptotic-infinity}.

We now show that the jump matrix in \eqref{eq:R-jump} has explicit limits as $n\to +\infty$ along even or odd subsequences, with the convergence being uniform for $|\lambda|=1$ and bounded $z$ where $x=z/n$. To this end, we compute the asymptotic behavior of \smash{$\mathbf{P}_n\big(n^{-1}z\big)^{-1}\para{\mathbf{\Phi}}_n^{(\infty)}\big(\lambda,n^{-1}z\big)$} assuming that $|\lambda|=1$. The relevant formula for \smash{$\para{\mathbf{\Phi}}_n^{(\infty)}(\lambda,x)$} in this setting is \eqref{infinity-parametrix-form-origin}.
When $|\lambda|=1$,
\begin{gather*}
 \mathbf{P}_n(x)^{-1}\para{\mathbf{\Phi}}_n^{(\infty)}(\lambda,x)\\
 \qquad=\frac{1}{2\mu_n^2\gamma_nx^{1-2\kappa_n}}\begin{bmatrix}\mu_n\gamma_nx^{1-2\kappa_n} & -\frac{1}{2}+\mu_n+\kappa_n \vspace{1mm}\\-\mu_n\gamma_nx^{1-2\kappa_n} & \frac{1}{2}+\mu_n-\kappa_n\end{bmatrix}\mathbf{D}_nx^{\kappa_n\sigma_3}\lambda_{\lw}^{\sigma_3/2}\mathbf{M}(\ii x\lambda;\kappa_n,\mu_n)\mathbf{J}_n.
\end{gather*}
Using \eqref{H-define-2}, we see that
\begin{equation*}
 \mathbf{P}_n(x)^{-1}\para{\mathbf{\Phi}}_n^{(\infty)}(\lambda,x)=\frac{1}{2\mu_n^2 \gamma_n}x^{\kappa_n-\frac{1}{2}}\begin{bmatrix}
 \ii & -\frac{1}{2} + \mu_n+\kappa_n \vspace{1mm}\\-\ii & \frac{1}{2} + \mu_n-\kappa_n\end{bmatrix}x^{\sigma_3/2}\lambda_{\lw}^{\sigma_3/2}\mathbf{M}(\ii x\lambda;\kappa_n,\mu_n)\mathbf{J}_n.
\end{equation*}
Now, for $x>0$, the principal branch power $(-\ii x\lambda)^{\sigma_3/2}$ has the same domain of analyticity as~\smash{$\lambda_{\lw}^{\sigma_3/2}$}, and these two analytic functions are related by the identity \[\lambda_{\lw}^{\sigma_3/2}=x^{-\sigma_3/2}\ee^{\ii\pi\sigma_3/4}(-\ii x\lambda)^{\sigma_3/2}.\] Therefore,
\begin{gather}
 \mathbf{P}_n(x)^{-1}\para{\mathbf{\Phi}}_n^{(\infty)}(\lambda,x)\nonumber\\
 \qquad =\frac{1}{2\mu_n^2 \gamma_n}x^{\kappa_n-\frac{1}{2}}\begin{bmatrix}
 \ii & -\frac{1}{2} + \mu_n+\kappa_n \\-\ii & \frac{1}{2} + \mu_n-\kappa_n\end{bmatrix}\ee^{\ii\pi\sigma_3/4}(-\ii x\lambda)^{\sigma_3/2}\mathbf{M}(\ii x\lambda;\kappa_n,\mu_n)\mathbf{J}_n\nonumber\\
 \qquad=\frac{\ee^{-\ii\pi/4}}{2 \mu_n^2 \gamma_n}x^{\kappa_n-\frac{1}{2}}\begin{bmatrix}-1 & -\frac{1}{2} + \mu_n +\kappa_n \\1 & \frac{1}{2} + \mu_n-\kappa_n\end{bmatrix}(-Z)^{\sigma_3/2}\mathbf{M}(Z;\kappa_n,\mu_n)\mathbf{J}_n,\qquad Z=\ii x\lambda.
\label{eq:Psi-infty-hat}
\end{gather}
Now using \eqref{N-define}, we have
\begin{gather*}
 (-Z)^{\sigma_3/2}\mathbf{M}(Z;\kappa_n,\mu_n) \\
 \quad=
 \begin{bmatrix}
 \bigl(-\frac{1}{2} - \mu_n+\kappa_n\bigr)(-Z)^{-\frac{1}{2}}M_{1-\kappa_n,\mu_n}(-Z) & \bigl(-\frac{1}{2}+\mu_n+\kappa_n\bigr)(-Z)^{-\frac{1}{2}}M_{1-\kappa_n,-\mu_n}(-Z)\vspace{1mm}\\
 (-Z)^{-\frac{1}{2}}M_{-\kappa_n,\mu_n}(-Z) & (-Z)^{-\frac{1}{2}}M_{-\kappa_n,-\mu_n}(-Z)
 \end{bmatrix}.
\end{gather*}
The diagonal elements in \eqref{eq:Psi-infty-hat} can be simplified using the identity (see \cite[equation~(13.15.3)]{DLMF})
\begin{equation*}
 \big(\kappa-\mu-\tfrac{1}{2}\big)M_{\kappa-\frac{1}{2},\mu+\frac{1}{2}}(\diamond)+(1+2\mu)\diamond^{\frac{1}{2}}M_{\kappa,\mu}(\diamond) -\big(\kappa+\mu+\tfrac{1}{2}\big)M_{\kappa+\frac{1}{2},\mu+\frac{1}{2}}(\diamond)=0
\end{equation*}
replacing $\kappa\to\frac{1}{2}-\kappa_n$ and $\mu\to\mu_n-\frac{1}{2}$ for the $(1,1)$ entry and $\mu\to-\frac{1}{2}-\mu_n$ for the $(2,2)$ entry,
and the off-diagonal elements can be simplified using the identity (see \cite[equation~(13.15.4)]{DLMF})
\begin{equation*}
 2\mu M_{\kappa-\frac{1}{2},\mu-\frac{1}{2}}(\diamond)-2\mu M_{\kappa+\frac{1}{2},\mu-\frac{1}{2}}(\diamond)-\diamond^{\frac{1}{2}}M_{\kappa,\mu}(\diamond)=0
\end{equation*}
replacing $\kappa \to \frac{1}{2}-\kappa_n$ and $\mu\to\frac{1}{2}-\mu_n$ for the $(1,2)$ entry and $\mu\to\frac{1}{2}+\mu_n$ for the $(2,1)$ entry. The result is that
\begin{gather}
 \mathbf{P}_n(x)^{-1}\para{\mathbf{\Phi}}_n^{(\infty)}(\lambda,x)
 =\frac{\ee^{-\ii\pi/4}}{2 \mu_n^2 \gamma_n}x^{\kappa_n-\frac{1}{2}}\nonumber\\
 \qquad {}\times\begin{bmatrix}2\mu_n M_{\frac{1}{2}-\kappa_n,\mu_n-\frac{1}{2}}(-Z) & \left(\dfrac{\kappa_n + \mu_n - \frac{1}{2}}{1 - 2\mu_n}\right) M_{\frac{1}{2}-\kappa_n,\frac{1}{2}-\mu_n}(-Z) \\
 \left(\dfrac{\frac{1}{2} - \kappa_n + \mu_n}{1 + 2\mu_n}\right)M_{\frac{1}{2}-\kappa_n,\frac{1}{2}+\mu_n}(-Z) & 2\mu_n M_{\frac{1}{2}-\kappa_n,-\frac{1}{2}-\mu_n}(-Z)\end{bmatrix}\mathbf{J}_n,\nonumber\\ Z=\ii x\lambda.
 \label{eq:Psi-infty-hat-1}
\end{gather}
We will need the following result for the large $n$ limit of Whittaker functions appearing here, cf.~\cite[equation~(13.21.1)]{DLMF}.

\begin{Lemma}\label{lem:M-expansion}
Assume that $\mu$ is fixed with $2\mu\neq -1,-2,-3,\dots$, and let $f(\zeta;\mu)$ denote the entire function
\begin{equation}
 f(\zeta;\mu):=\sum_{s=0}^\infty \frac{(-\zeta)^s}{\Gamma(1+2\mu+s)s!}.
\label{eq:Bessel-f-define}
\end{equation}
Then the asymptotic formula
\begin{equation*}
 M_{\kappa,\mu}\left(\frac{\zeta}{\kappa}\right) = \Gamma(1+2\mu)\left(\frac{\zeta}{\kappa}\right)^{\mu+1/2}\big[f(\zeta;\mu) + \mathcal{O}\big(\kappa^{-2}\big)\big]
\end{equation*}
holds uniformly in the limit $\kappa\to\infty$ in any $($possibly complex$)$ direction under the assumption~${\zeta=\mathcal{O}(1)}$.
\end{Lemma}
\begin{proof}
We start from the formula \cite[equation~(13.14.6)]{DLMF} which holds under the indicated condition on $\mu$:
\begin{equation*}
 M_{\kappa,\mu}\left(\frac{\zeta}{\kappa}\right)=
 \Gamma(1+2\mu)\left(\frac{\zeta}{\kappa}\right)^{\mu+1/2}\ee^{-\zeta/(2\kappa)}\sum_{s=0}^\infty\frac{(-\zeta)^s}{\Gamma(1+2\mu+s)s!}\prod_{j=0}^{s-1}\left(1-\frac{\mu+1/2+j}{\kappa}\right).
\end{equation*}
Clearly, $\ee^{-\zeta/(2\kappa)}=1-\zeta/(2\kappa)+\mathcal{O}\big(\kappa^{-2}\big)$ as $\kappa\to\infty$ for $\zeta=\mathcal{O}(1)$, and the product in the summand has the expansion
\begin{gather}
 \prod_{j=0}^{s-1}\left(1-\frac{\mu+1/2+j}{\kappa}\right) = 1 -\frac{1}{\kappa}\bigg[\left(\mu+\frac{1}{2}\right)s +\frac{1}{2}s(s-1)\bigg] + \mathcal{O}\big(\kappa^{-2}s^s\big),\qquad\kappa\to\infty\!\!\!
\label{eq:uniform-estimate}
\end{gather}
uniformly for all indices $s$.
This follows from the Fredholm expansion formula
\begin{equation*}
 \prod_{j=0}^{s-1}(1+r_j) = 1+\sum_{j=0}^{s-1}r_j + \sum_{k=2}^{s}
 \mathop{\sum_{S\subset \mathbb{Z}_s}}_{|S|=k}\prod_{l\in S}r_l
\end{equation*}
and the estimate
\begin{equation*}
 \Biggl| \sum_{\substack{S\subset \mathbb{Z}_s \\ |S|=k}}\prod_{l\in S}r_l\Biggr|\le \binom{s}{k}R_s^k,\qquad R_s:=\max_{0\le l\le s-1}\{|r_l|\}.
\end{equation*}
Indeed, with $r_j=-\kappa^{-1}(\mu+\frac{1}{2}+j)$, we have
\begin{equation*}
 \sum_{j=0}^{s-1}r_j=-\frac{1}{\kappa}\sum_{j=0}^{s-1}\left(\mu+\frac{1}{2}+j\right) = -\frac{1}{\kappa}\left[\left(\mu+\frac{1}{2}\right)s + \frac{1}{2}s(s-1)\right],
\end{equation*}
and $R_s\le |\kappa|^{-1}\big(\big|\mu+\frac{1}{2}\big| + (s-1)\big)$. Therefore, $R_s^k\le |\kappa|^{-2}\big(\big|\mu+\frac{1}{2}\big|+(s-1)\big)^k$ holds for all $k\ge 2$ whenever $|\kappa|\ge 1$. Consequently,
\begin{gather*}
 \left|\prod_{j=0}^{s-1}\left(1-\frac{\mu+1/2+j}{\kappa}\right)-1+\right. \left.\frac{1}{\kappa}\left[\left(\mu+\frac{1}{2}\right)s+ \frac{1}{2}s(s-1)\right]\right|\\
 \qquad=\Biggl|\sum_{k=2}^s\mathop{\sum_{S\subset\mathbb{Z}_s}}_{|S|=k}\prod_{l\in S}r_l\Biggr|\le \frac{1}{|\kappa|^2}\sum_{k=2}^s\binom{s}{k}\left(\left|\mu+\frac{1}{2}\right|+(s-1)\right)^k1^{s-k}\\
 \qquad\le \frac{1}{|\kappa|^2}\sum_{k=0}^s\binom{s}{k}\left(\left|\mu+\frac{1}{2}\right|+(s-1)\right)^k1^{s-k}=\frac{\big(\big|\mu+\frac{1}{2}\big|+s\big)^s}{|\kappa|^2},
\end{gather*}
which proves \eqref{eq:uniform-estimate}.
Since the series
\begin{equation*}
 \sum_{s=0}^\infty \frac{(-\zeta)^ss^s}{\Gamma(1+2\mu_n+s)s!}
\end{equation*}
converges uniformly for $|\zeta|$ bounded, it follows that
\begin{equation*}
 M_{\kappa,\mu}\left(\frac{\zeta}{\kappa}\right) = \Gamma(1+2\mu)\left(\frac{\zeta}{\kappa}\right)^{\mu+1/2}\left[f(\zeta;\mu) -\frac{1}{\kappa}g(\zeta;\mu) + \mathcal{O}\big(\kappa^{-2}\big)\right]
\end{equation*}
holds as $\kappa\to\infty$ in $\mathbb{C}$ with $\zeta=\mathcal{O}(1)$, where
\begin{equation*}
 g(\zeta;\mu):=\frac{1}{2}\zeta f(\zeta;\mu)+\left(\mu+\frac{1}{2}\right)\zeta f'(\zeta;\mu) +\frac{1}{2}\zeta^2f''(\zeta;\mu).
\end{equation*}
Now using the series \eqref{eq:Bessel-f-define} one checks that for all indicated values of $\mu$, $g(\zeta;\mu)$ vanishes identically, so the proof is complete.
\end{proof}

The series in \eqref{eq:Bessel-f-define} defines an entire function of $\zeta$ related to Bessel functions (see \cite[Chapter~10]{DLMF}) in the following way:
\begin{equation}
\label{f-bessel-j}
 f(\zeta,\mu)=\zeta^{-{\mu}{}}J_{2\mu}\big(2\sqrt{\zeta}\big).
\end{equation}
We apply Lemma~\ref{lem:M-expansion} to \eqref{eq:Psi-infty-hat-1} by taking $\zeta=-\big(\frac{1}{2}-\kappa_n\big)Z=-\big(\frac{1}{2}-\kappa_n\big)\ii x\lambda$. If $x=z/n$ and $z=\mathcal{O}(1)$, then using~\eqref{kappa-mu}, we see that $\zeta=-\frac{1}{2}\ii z\lambda + \mathcal{O}\big(n^{-1}\big)$ holds for $|\lambda|=1$. So, \eqref{eq:Psi-infty-hat-1}~becomes the statement that
\begin{gather*}
 \mathbf{P}_n(x)^{-1}\para{\mathbf{\Phi}}_n^{(\infty)}(\lambda,x)
\\
\quad=\frac{\ee^{-\ii\pi/4}}{2\mu_n^2\gamma_n}x^{\kappa_n-\frac{1}{2}}\Bigg(\frac{\rho_\infty}{2}\begin{bmatrix}
 \Gamma(2\mu_n+1) J_{2\mu_n-1}(\rho_\infty) & -\Gamma(1-2\mu_n) J_{1-2\mu_n}(\rho_\infty) \vspace{1mm}\\
 \Gamma(2\mu_n+1) J_{2\mu_n+1}(\rho_\infty) & -\Gamma(1-2\mu_n) J_{-1-2\mu_n}(\rho_\infty)
 \end{bmatrix}+\mathcal{O}\big(n^{-1}\big)\Bigg) \\
 \phantom{\quad=}{}\times\left(\dfrac{n}{2} \right)^{-\mu_n \sigma_3}
 \mathbf{J}_n
\end{gather*}
with
\begin{equation}
 \rho_\infty=\rho_\infty(\lambda,z):=(-2\ii z\lambda)^{1/2}\qquad\text{(principal branch)}
 \label{eq:rho-infty-branch}
\end{equation}
holds in the limit $n\to\infty$ with $x=n^{-1}z$ uniformly for $z=\mathcal{O}(1)$ and $|\lambda|=1$. Similarly, to study the parametrix near 0, it will be convenient to rewrite formula \eqref{f-bessel-j} in terms of modified Bessel functions:
\begin{equation*}
 f(\zeta, \mu) = (-\zeta)^{-\mu} I_{2\mu}\big(2 \sqrt{-\zeta} \big).
\end{equation*}

Replacing $n$ with $-n$, $\Theta_\infty$ with $\Theta_0$, and $\lambda$ with $-\lambda^{-1}$ and recalling that $\mu_{-n} = \mu_n$, gives in the same limit,
\begin{gather*}
 \mathbf{P}_{-n}(x)^{-1}\para{\mathbf{\Phi}}_{-n}^{(\infty)}\bigl(-\lambda^{-1},x\bigr)
\\
 \quad =\frac{\ee^{-\ii\pi/4}}{2\mu_n^2\gamma_{-n}}x^{\kappa_{-n}-\frac{1}{2}}\Bigg(\frac{\rho_0}{2}\begin{bmatrix}{\Gamma(2\mu_n+1)}{} I_{2\mu_n-1}(\rho_0) & {\Gamma(1-2\mu_n)}{} I_{1-2\mu_n}(\rho_0) \vspace{1mm}\\
 -{\Gamma(2\mu_n+1)}{}I_{2\mu_n+1}(\rho_0)&-{\Gamma(1-2\mu_n)}{}I_{-2\mu_n-1}(\rho_0)
 \end{bmatrix} +\mathcal{O}\big(n^{-1}\big)\Bigg)\\
\phantom{ \quad =}{} \times \left(\frac{n}{2}\right)^{-\mu_n \sigma_3}
 \mathbf{J}_{-n}
\end{gather*}
with
\begin{equation}
\rho_0=\rho_0(\lambda,z):=\big(2\ii z\lambda^{-1}\big)^{1/2}\qquad\text{(principal branch)}.
\label{eq:rho-zero-branch}
\end{equation}
The jump matrix in \eqref{eq:R-jump} therefore reads
\begin{gather}
 \mathbf{R}_{n,-}\big(\lambda,n^{-1}z\big)^{-1}
 \mathbf{R}_{n,+}\big(\lambda,n^{-1}z\big) \nonumber\\
 = \frac{\rho_\infty}{2}\begin{bmatrix}
 \Gamma(2\mu_n+1) J_{2\mu_n-1}(\rho_\infty)+ \mathcal{O}\big(n^{-1}\big) & -\Gamma(1-2\mu_n) J_{1-2\mu_n}(\rho_\infty)+ \mathcal{O}\big(n^{-1}\big) \vspace{1mm}\\
 \Gamma(2\mu_n+1) J_{2\mu_n+1}(\rho_\infty)+ \mathcal{O}\big(n^{-1}\big) & -\Gamma(1-2\mu_n) J_{-1-2\mu_n}(\rho_\infty)+ \mathcal{O}\big(n^{-1}\big)
 \end{bmatrix}\nonumber\\
\phantom{ =}{} \times{} \frac{\gamma_{-n}}{\gamma_n}x^{\kappa_n-\kappa_{-n}}d_n(x)\left(\frac{n}{2}\right)^{-\mu_n\sigma_3}\mathbf{J}_ne_2^{\sigma_3}\begin{bmatrix}0 & -\beta_n\\\beta_n^{-1} & 0\end{bmatrix}\mathbf{J}_{-n}^{-1}\left(\frac{n}{2}\right)^{\mu_n\sigma_3}\frac{2}{\rho_0}\nonumber
 \\
\phantom{ =}{}\times \begin{bmatrix}{\Gamma(2\mu_n+1)}{} I_{2\mu_n-1}(\rho_0)+ \mathcal{O}\big(n^{-1}\big) & {\Gamma(1-2\mu_n)}{} I_{1-2\mu_n}(\rho_0)+ \mathcal{O}\big(n^{-1}\big) \vspace{1mm}\\
 -{\Gamma(2\mu_n+1)}{}I_{2\mu_n+1}(\rho_0)+ \mathcal{O}\big(n^{-1}\big)&-{\Gamma(1-2\mu_n)}{}I_{-2\mu_n-1}(\rho_0)+ \mathcal{O}\big(n^{-1}\big)
 \end{bmatrix}^{-1}.
 \label{eq:E-hat-jump-1}
\end{gather}
Expanding \eqref{eq:dn-squared} for large $n>0$ gives
\begin{align*}
& d_n(x)^2= \frac{\gamma_n}{\gamma_{-n}}x^{2\kappa_{-n}-2\kappa_n} = \dfrac{e_\infty^5}{e_0^5} \dfrac{e_0^2 - e_1^2}{e_\infty^2 - e_1^2} \dfrac{\Gamma\big( \frac{1}{2} - \kappa_{-n} -\mu_n\big) \Gamma\big( \frac{1}{2} - \kappa_{-n} + \mu_n\big)}{\Gamma\big( \frac{1}{2} - \kappa_{n} -\mu_n\big) \Gamma\big( \frac{1}{2} - \kappa_{n} +\mu_n\big)} x^{2\kappa_{-n}-2\kappa_n}\\
&\hphantom{d_n(x)^2}{} = \dfrac{e_\infty^5 e_1^2}{e_0^3\big(1 - e_1^2 e_0^2\big)\big(e_\infty^2 - e_1^2\big)} \\
&\hphantom{d_n(x)^2=}{} \times\dfrac{ 4\pi^2x^{2\kappa_{-n}-2\kappa_n}}{\Gamma \big( \frac{1}{2} + \kappa_{-n}- \mu_n \big) \Gamma \big( \frac{1}{2} + \kappa_{-n} + \mu_n \big) \Gamma \big( \frac{1}{2} - \kappa_n + \mu_n\big) \Gamma \big( \frac{1}{2} - \kappa_n - \mu_n \big)} \\
&\hphantom{d_n(x)^2}{}=
 4\dfrac{e_\infty^4}{e_0^4}\dfrac{e_0e_\infty e_1^2}{\big(1 - e_1^2 e_0^2\big)\big(e_\infty^2 - e_1^2\big)} x^{2\kappa_{-n}-2\kappa_n}n^{-2n-\Theta_0+\Theta_\infty}\ee^{2n}2^{2n+\Theta_0-\Theta_\infty-2}\big(1+\mathcal{O}\big(n^{-1}\big)\big), \\
& n \to +\infty.
\end{align*}
We now properly define $d_n(x)$ for large $n$ by selecting a definite value for the square root of
\begin{equation}
\label{eq:V-root}
\sqrt{\dfrac{1 - e_1^2 e_0^2}{e_\infty^2 - e_1^2} }
\end{equation}
after which $d_n(x)$ has the asymptotic expansion
\begin{align*}
&d_n(x) = \dfrac{e_1 e_\infty^{5/2}}{e_0^{3/2}} \dfrac{1}{\big(1 - e_1^2 e_0^2\big)}\sqrt{\dfrac{1 - e_1^2 e_0^2}{e_\infty^2 - e_1^2} } x^{\kappa_{-n}-\kappa_n} n^{-n+\frac{1}{2}(-\Theta_0+\Theta_\infty)}\ee^{n}2^{n + \frac{1}{2}(\Theta_0-\Theta_\infty)}\\
&\hphantom{d_n(x) =}{}
\times\big(1+\mathcal{O}\big(n^{-1}\big)\big), \qquad n \to +\infty.
\end{align*}
Then, by definition, we have
\begin{align*}
\dfrac{\gamma_{-n}}{\gamma_n} d_n(x) x^{\kappa_n - \kappa_{-n}} ={}& \dfrac{e_0^{3/2}}{e_1 e_\infty^{5/2}} {\big(1 - e_1^2 e_0^2\big)}\left(\sqrt{\dfrac{1 - e_1^2 e_0^2}{e_\infty^2 - e_1^2} }\right)^{-1}n^{n-\frac{1}{2}(-\Theta_0+\Theta_\infty)}\ee^{-n}\\&
\times2^{-n - \frac{1}{2}(\Theta_0-\Theta_\infty)}\big(1+\mathcal{O}\big(n^{-1}\big)\big),\qquad n \to +\infty.
\end{align*}
Furthermore, using identities \eqref{eq:Gamma-reflection}, \eqref{eq:cosine-identity-1}, and Stirling's formula yields
\begin{align*}
 & -\left(\frac{2}{n}\right)^{2\mu_n}\beta_n\frac{J_{n,11}}{J_{-n,22}} = -\left(\frac{2}{n}\right)^{2\mu_n} \frac{e_\infty^{7/2} \big(e_1^2-e_0^2\big) \Gamma (-2 \mu_n ) \Gamma \big(\frac{1}{2}-\kappa_{-n}+\mu_n \big)}{e_0^{3/2} \big(e_0^2 e_1^2-1\big) \Gamma (2 \mu_n ) \Gamma \big(\frac{1}{2}-\kappa_{n}-\mu_n \big)} \\
 &\hphantom{-\left(\frac{2}{n}\right)^{2\mu_n}\beta_n\frac{J_{n,11}}{J_{-n,22}}}{}
 = \frac{e_\infty^{7/2}}{e_0^{3/2}} \dfrac{\ii e_0e_1}{\big(1- e_0^2 e_1^2\big) } \dfrac{\Gamma (-2 \mu_n ) }{ \Gamma (2 \mu_n ) } n^{-n+\frac{1}{2}(-\Theta_0+\Theta_\infty)}\ee^n \\
 &\hphantom{-\left(\frac{2}{n}\right)^{2\mu_n}\beta_n\frac{J_{n,11}}{J_{-n,22}}=}{}
 \times 2^{n + \frac{1}{2}(\Theta_0-\Theta_\infty)}\big(1+\mathcal{O}\big(n^{-1}\big)\big),\qquad n\to+\infty,
\end{align*}
and similarly,
\begin{align*}
&\left(\frac{n}{2}\right)^{2\mu_n}\frac{J_{n,22}}{\beta_nJ_{-n,11}} = \left(\frac{n}{2}\right)^{2\mu_n} \frac{e_\infty^{3/2}\big(e_0^2 e_1^2-1\big) \Gamma (2 \mu_n ) \Gamma \big(\frac{1}{2}-\kappa_{-n}-\mu_n \big)}{e_0^{7/2} \big(e_1^2-e_\infty^2\big) \Gamma (-2 \mu_n ) \Gamma \big(\frac{1}{2}-\kappa_n+\mu_n \big)}\\
&\hphantom{\left(\frac{n}{2}\right)^{2\mu_n}\frac{J_{n,22}}{\beta_nJ_{-n,11}}}{}
= \frac{e_\infty^{3/2}}{e_0^{7/2}} \dfrac{\ii e_0e_1}{\big(e_\infty^2 - e_1^2\big)}\dfrac{ \Gamma (2 \mu_n ) }{ \Gamma (-2 \mu_n )}n^{-n+\frac{1}{2}(-\Theta_0+\Theta_\infty)}\ee^n \\
&\hphantom{\left(\frac{n}{2}\right)^{2\mu_n}\frac{J_{n,22}}{\beta_nJ_{-n,11}}=}{}
\times 2^{n + \frac{1}{2}(\Theta_0-\Theta_\infty)}\big(1+\mathcal{O}\big(n^{-1}\big)\big),\qquad n\to+\infty.
\end{align*}
Therefore, the central factor on the right-hand side of \eqref{eq:E-hat-jump-1} satisfies
\begin{gather*}
\frac{\gamma_{-n}}{\gamma_n}x^{\kappa_n-\kappa_{-n}}d_n(x)\left(\frac{n}{2}\right)^{-\mu_n\sigma_3}\mathbf{J}_ne_2^{\sigma_3}\begin{bmatrix}0 & -\beta_n\\\beta_n^{-1} & 0\end{bmatrix}\mathbf{J}_{-n}^{-1}\left(\frac{n}{2}\right)^{\mu_n\sigma_3}\\
\qquad = \begin{bmatrix} 0 & -\dfrac{e_0 e_2 e_\infty}{\ii}\dfrac{ \Gamma (-2 \mu_n )}{ \Gamma (2 \mu_n )} \left(\sqrt{\dfrac{1 - e_0^2 e_1^2}{e_\infty^2-e_1^2 } }\right)^{-1} \\
 \dfrac{ \ii }{e_0 e_2 e_\infty} \dfrac{\Gamma (2 \mu_n )}{ \Gamma (-2 \mu_n )} \sqrt{\dfrac{1 - e_0^2 e_1^2}{e_\infty^2-e_1^2 }} & 0
\end{bmatrix} + \mathcal{O}\big(n^{-1}\big).
\end{gather*}
The leading term is independent of $n\pmod{2}$ and has unit determinant. This proves the following.
\begin{prop}
Define the constant matrix which depends only on the even/odd parity of $n$ via $\mu_n$, $e_1$, $e_0^2$, and $e_\infty^2$:
\begin{equation}
 \mathbf{V}^\mathrm{even/odd}:=\begin{bmatrix} 0 & -\dfrac{e_0 e_2 e_\infty}{\ii} \left(\sqrt{\dfrac{1 - e_0^2 e_1^2}{e_\infty^2-e_1^2 } }\right)^{-1} \\
 \dfrac{ \ii }{e_0 e_2 e_\infty} \sqrt{\dfrac{1 - e_0^2 e_1^2}{e_\infty^2-e_1^2 }} & 0
\end{bmatrix}.
 \label{eq:Vm-even}
\end{equation}
Then the following asymptotic formula holds uniformly for $|\lambda|=1$ and $z$ bounded:
\begin{align*}
& \mathbf{R}_{n,-}(\lambda,z/n)^{-1}\mathbf{R}_{n,+}(\lambda,z/n)= {\rho_\infty}\begin{bmatrix}J_{2\mu_n - 1}(\rho_\infty) & -J_{1- 2\mu_n}(\rho_\infty)\\
 J_{2\mu_n +1}(\rho_\infty) & -J_{-1-2\mu_n}(\rho_\infty)\end{bmatrix}\cdot\mathbf{V}^\mathrm{even/odd}\\
&\hphantom{\mathbf{R}_{n,-}(\lambda,z/n)^{-1}\mathbf{R}_{n,+}(\lambda,z/n)=}{}
\times
 \dfrac{1}{\rho_0}\begin{bmatrix}I_{2\mu_n - 1}(\rho_0) & I_{1- 2\mu_n}(\rho_0)\\
 -I_{2\mu_n +1}(\rho_0) & -I_{-1-2\mu_n}(\rho_0)\end{bmatrix}^{-1}+\mathcal{O}\big(n^{-1}\big),
\end{align*}
as $n\to\infty$ along even/odd subsequences.
\label{prop:limiting-jump}
\end{prop}
Proposition~\ref{prop:limiting-jump} suggests defining the following limiting Riemann--Hilbert problem.
\begin{rhp}[limiting problem, even/odd subsequences of $n$]
Fix generic monodromy parameters $( e_1, e_2)$, and $z\in\mathbb{C}$ with $|{\arg}(z)|<\pi$. Seek a $2\times 2$ matrix function $\lambda\mapsto\hat{\mathbf{R}}^\mathrm{even/odd}(\lambda,z)$ with the following properties:
\begin{itemize}\itemsep=0pt
 \item Analyticity: $\hat{\mathbf{R}}^\mathrm{even/odd}(\lambda,z)$ is an analytic function of $\lambda$ for $|\lambda|\neq 1$.
 \item Jump condition:
 $\hat{\mathbf{R}}^\mathrm{even/odd}(\lambda,z)$ takes analytic boundary values on the unit circle from the interior and exterior, denoted $\hat{\mathbf{R}}^\mathrm{even/odd}_+(\lambda,z)$ and $\hat{\mathbf{R}}^\mathrm{even/odd}_-(\lambda,z)$ for $|\lambda|=1$ respectively, and they are related by
 \begin{align}
 &\hat{\mathbf{R}}^\mathrm{even/odd}_+(\lambda,z)=\hat{\mathbf{R}}^\mathrm{even/odd}_-(\lambda,z) {\rho_\infty}\begin{bmatrix}J_{2\mu_n - 1}(\rho_\infty) & -J_{1- 2\mu_n}(\rho_\infty)\\
 J_{2\mu_n +1}(\rho_\infty) & -J_{-1-2\mu_n}(\rho_\infty)\end{bmatrix}\nonumber\\
 & \hphantom{\hat{\mathbf{R}}^\mathrm{even/odd}_+(\lambda,z)=}{}
 \times\mathbf{V}^\mathrm{even/odd}
 \cdot \dfrac{1}{\rho_0}\begin{bmatrix}I_{2\mu_n - 1}(\rho_0) & I_{1- 2\mu_n}(\rho_0)\\
 -I_{2\mu_n +1}(\rho_0) & -I_{-1-2\mu_n}(\rho_0)\end{bmatrix}^{-1}.\label{r-hat-jump}
 \end{align}
 \item Normalization:
 $\hat{\mathbf{R}}^\mathrm{even/odd}(\lambda,z)\to\mathbb{I}$ as $\lambda\to\infty$.
\end{itemize}
\label{rhp:Rhat-even/odd}
\end{rhp}
Note that the Bessel functions $J_\nu(\rho_\infty)$ and $I_\nu(\rho_0)$ appearing in the jump matrix in \eqref{r-hat-jump} are analytic on the unit circle $|\lambda|=1$ except at the point $\lambda=\lambda_\mathrm{c}:=-\ii\ee^{-\ii\arg(z)}$. However, from the identities
\[
\left.J_\nu(\rho_\infty)\right|_{\lambda=\lambda_\mathrm{c}\ee^{-\ii 0}}=\ee^{\ii\pi\nu}\left.J_\nu(\rho_\infty)\right|_{\lambda=\lambda_\mathrm{c}\ee^{\ii 0}} \qquad \text{and} \qquad
\left.I_\nu(\rho_0)\right|_{\lambda=\lambda_\mathrm{c}\ee^{-\ii 0}}=\ee^{-\ii\pi\nu}\left.I_\nu(\rho_0)\right|_{\lambda=\lambda_\mathrm{c}\ee^{\ii 0}}
\]
and the fact that the indices $\nu$ in each column of the Bessel matrix factors in \eqref{r-hat-jump} differ by $2$, combined with the fact that $\mathbf{V}^\mathrm{even/odd}$ is an off-diagonal matrix, one sees easily that
\[
\begin{bmatrix}J_{2\mu_n - 1}(\rho_\infty) & -J_{1- 2\mu_n}(\rho_\infty)\\
 J_{2\mu_n +1}(\rho_\infty) & -J_{-1-2\mu_n}(\rho_\infty)\end{bmatrix}\cdot\mathbf{V}^\mathrm{even/odd}
 \cdot \begin{bmatrix}I_{2\mu_n - 1}(\rho_0) & I_{1- 2\mu_n}(\rho_0)\\
 -I_{2\mu_n +1}(\rho_0) & -I_{-1-2\mu_n}(\rho_0)\end{bmatrix}^{-1}
\]
is continuous at $\lambda=\lambda_\mathrm{c}$ and hence is an analytic function of $\lambda$ on the unit circle. The scalar factor $\rho_\infty/\rho_0$ is also analytic for $|\lambda|=1$, and therefore the jump matrix in \eqref{r-hat-jump} is an analytic function of $\lambda$ when $|\lambda|=1$. At this stage, the existence of a matrix function $\hat{\mathbf{R}}^\mathrm{even/odd}(\lambda, z)$ satisfying Riemann--Hilbert Problem \ref{rhp:Rhat-even/odd} is not clear. However, it turns out that there exists a~discrete set $\Sigma^{\mathrm{even/odd}} \subset \C$ such that for $z \in \C \setminus \Sigma^{\mathrm{even/odd}}$, such a matrix does exist and is in fact a meromorphic function of $z$, see Section \ref{sec:suleimanov-solution-connection} below.

\begin{Lemma}
\label{lemma:limit}
Let $(e_1, e_2)$ be generic monodromy parameters and take $z \in \C \setminus \Sigma^\mathrm{even/odd}$. Then,
\begin{align}
 \mathop{\lim_{n\to\infty}}_{\text{$n$ even/odd}}u_n\big(n^{-1}z;m\big)={}&-\dfrac{8 \mu_n^2}{z}\big[\hat{R}_{11}^\mathrm{even/odd}(0,z)+\hat{R}_{21}^\mathrm{even/odd}(0,z)\nonumber\\
 &-\hat{R}_{12}^\mathrm{even/odd}(0,z)+\hat{R}_{22}^\mathrm{even/odd}(0,z)\big]^{-2}.\label{eq:un-Rhat}
\end{align}
\end{Lemma}
\begin{proof}

Noting that $\hat{\mathbf{R}}^\mathrm{even/odd}(\lambda,z)$ necessarily has unit determinant, we form the matrix quotient
\begin{equation*}
 \mathbf{E}_n(\lambda,z):=\mathbf{R}_n\big(\lambda,n^{-1}z\big)\hat{\mathbf{R}}^\mathrm{even/odd}(\lambda,z)^{-1},\qquad |\lambda|\neq 1.
\end{equation*}
Clearly, $\mathbf{E}_n(\lambda,z)$ is analytic as a function of $\lambda$ in the domain of definition, and for each fixed $n$ it tends to $\mathbb{I}$ as $\lambda\to\infty$ as this is true for both $\mathbf{R}_n\big(\lambda,n^{-1}z\big)$ and $\hat{\mathbf{R}}^\mathrm{even/odd}(\lambda,z)$. Across the unit circle, the boundary values of $\mathbf{E}_n(\lambda,z)$ are related by
\begin{align*}
 &\mathbf{E}_{n,+}(\lambda,z)= \mathbf{E}_{n,-}(\lambda,z)\hat{\mathbf{R}}_-^\mathrm{even/odd}(\lambda,z)\big[\mathbf{R}_{n,-}\big(\lambda,n^{-1}z\big)^{-1}\mathbf{R}_{n,+}\big(\lambda,n^{-1}z\big)\big]\\
 &\hphantom{\mathbf{E}_{n,+}(\lambda,z)=}{}
 \times \big[\hat{\mathbf{R}}^\mathrm{even/odd}_-(\lambda,z)^{-1}\hat{\mathbf{R}}^\mathrm{even/odd}_+(\lambda,z)\big]^{-1}
 \hat{\mathbf{R}}^\mathrm{even/odd}_-(\lambda,z)^{-1},\qquad |\lambda|=1.
\end{align*}
Thus, the jump matrix for $\mathbf{E}_n(\lambda,z)$ is the conjugation, by a unit-determinant matrix function of $\lambda$ independent of $n$, of the matrix ratio of the jump matrices for $\mathbf{R}_n\big(\lambda,n^{-1}z\big)$ and for \smash{$\hat{\mathbf{R}}^\mathrm{even/odd}(\lambda,z)$}. But by Proposition~\ref{prop:limiting-jump}, the latter ratio is $\mathbb{I}+\mathcal{O}\big(n^{-1}\big)$ uniformly on the unit circle as $n\to\infty$ along even or odd subsequences. The conjugating factors exist and are uniformly bounded for $z$ in compact subsets of $\C\setminus \Sigma(y_1, y_2, y_3)$. 
It follows that in this limit, $\mathbf{E}_{n,+}(\lambda,z)=\mathbf{E}_{n,-}(\lambda,z)\big(\mathbb{I}+\mathcal{O}\big(n^{-1}\big)\big)$ uniformly for $|\lambda|=1$ and $z$ in compact subsets of~$\C\setminus \Sigma(y_1, y_2, y_3)$ as $n\to\infty$. By standard small-norm theory, $\mathbf{E}_n(\lambda,z)$ exists for large enough even or odd $n$, and tends to the identity as $n\to\infty$, in particular in the sense that
\begin{equation*}
 \lim_{\lambda\to\infty}\lambda(\mathbf{E}_n(\lambda,z)-\mathbb{I})\to\mathbf{0}\qquad\text{and}\qquad\mathbf{E}_n(0,z)\to\mathbb{I}
\end{equation*}
as $n\to\infty$ along even/odd subsequences. By the definition of $\mathbf{E}_n(\lambda,z)$ it follows that in the same limit
\begin{gather*}
 \lim_{\lambda\to\infty}\lambda\big(\mathbf{R}_n\big(\lambda,n^{-1}z\big)-\mathbb{I}\big)\to\lim_{\lambda\to\infty}\lambda\big(\hat{\mathbf{R}}^\mathrm{even/odd}(\lambda,z)-\mathbb{I}\big), \qquad \text{and}\\
 \mathbf{R}_n\big(0,n^{-1}z\big)\to\hat{\mathbf{R}}^\mathrm{even/odd}(0,z).
\end{gather*}
Combining \eqref{eq:u-n-recover} with \eqref{eq:TU-R} then shows \eqref{eq:un-Rhat}.
Partly, this works because the dominant term in $\Xi_{n,12}^{(6)}\big(n^{-1}z;m\big)$ is $A_{n,12}\big(n^{-1}z\big)$.
\end{proof}

\subsection{Transformations of the limiting Riemann--Hilbert problem} In this section we transform Riemann--Hilbert Problem \ref{rhp:Rhat-even/odd} to match the form of Riemann--Hilbert Problem \ref{rhp:D8}. To this end, using \cite[equations~(10.4.4) and (10.4.6)]{DLMF} to express the Bessel function~$J_\nu(\diamond)$ in terms of the Hankel functions \smash{$H_\nu^{(1)}(\diamond)$}, \smash{$H_\nu^{(2)}(\diamond)$} and the relations \cite[equation~(10.4.4) and~(10.4.6)]{DLMF},
\begin{align*}
 H_{-\nu}^{(1)}(\diamond) &= \ee^{\pi \ii \nu}H_{\nu}^{(1)}(\diamond) , \qquad
 H_{-\nu}^{(2)}(\diamond) = \ee^{-\pi \ii \nu}H_{\nu}^{(2)}(\diamond) ,
\end{align*}
we arrive at the identity:
\begin{gather}
 \rho_\infty \begin{bmatrix}J_{2\mu_n - 1}(\rho_\infty) & -J_{1- 2\mu_n}(\rho_\infty)\\
 J_{2\mu_n +1}(\rho_\infty) & -J_{-1-2\mu_n}(\rho_\infty)\end{bmatrix}\nonumber \\
 \qquad= \dfrac{\rho_\infty}{2} \begin{bmatrix}H_{2\mu_n - 1}^{(1)}(\rho_\infty) & H_{1 -2\mu_n}^{(2)}(\rho_\infty) \\ H^{(1)}_{2\mu_n + 1}(\rho_\infty) & H_{-1 -2\mu_n}^{(2)}(\rho_\infty) \end{bmatrix} \begin{bmatrix} 1 & \ee^{2\pi \ii \mu_n} \\ -\ee^{2\pi \ii \mu_n} & -1 \end{bmatrix}. \label{J-to-H}
\end{gather}
To obtain appropriate asymptotic formul\ae\ for the matrix on the right-hand side of \eqref{J-to-H}, we first apply the identity \cite[equation~(10.6.1)]{DLMF}
\begin{equation*}
 H^{(k)}_{\nu-1}(\diamond) + H^{(k)}_{\nu + 1}(\diamond) = \dfrac{2\nu}{\diamond} H^{(k)}_\nu(\diamond), \qquad k = 1, 2,
\end{equation*}
which gives
\begin{gather}
 \rho_\infty^{-\sigma_3/2} \begin{bmatrix} 1 & 0 \\ 1 & 1\end{bmatrix} \dfrac{\rho_\infty}{2} \begin{bmatrix}H_{2\mu_n - 1}^{(1)}(\rho_\infty) & H_{1 -2\mu_n}^{(2)}(\rho_\infty) \vspace{1mm}\\ H^{(1)}_{2\mu_n + 1}(\rho_\infty) & H_{-1 -2\mu_n}^{(2)}(\rho_\infty) \end{bmatrix}\nonumber\\
 \qquad= \dfrac{\sqrt{\rho_\infty}}{2} \begin{bmatrix} H_{2\mu_n-1}^{(1)}(\rho_\infty) & H_{1-2\mu_n}^{(2)}(\rho_\infty)\vspace{1mm}\\4\mu_n H_{2\mu_n}^{(1)}(\rho_\infty) & -4\mu_n H_{-2\mu_n}^{(2)}(\rho_\infty)\end{bmatrix}.\label{H-inf-identity}
\end{gather}
The matrix on the right-hand side is amenable to asymptotic analysis as $\rho_\infty \to \infty$; using the asymptotics of Hankel functions \cite[equations~(10.17.5) and (10.17.6)]{DLMF} and \eqref{H-inf-identity} yields
\begin{gather}
 \dfrac{\rho_\infty}{2} \begin{bmatrix}H_{2\mu_n - 1}^{(1)}(\rho_\infty) & H_{1 -2\mu_n}^{(2)}(\rho_\infty) \vspace{1mm}\\ H^{(1)}_{2\mu_n + 1}(\rho_\infty) & H_{-1 -2\mu_n}^{(2)}(\rho_\infty) \end{bmatrix}\nonumber\\
 \qquad= \begin{bmatrix} 1 & \dfrac{1}{2}-\dfrac{\mu_n}{2}-\dfrac{3}{32 \mu_n} \vspace{1mm}\\ -1 & \dfrac{1}{2} + \dfrac{\mu_n}{2}+\dfrac{3}{32 \mu_n}
\end{bmatrix} \Biggl( \mathbb{I}+ \dfrac{1}{128 \rho_\infty^2}  \nonumber\\
 \phantom{\qquad=}{}\times \begin{bmatrix} \big(16\mu_n^2 - 9\big)\big(16\mu_n^2 - 1\big) & \dfrac{\big(16 \mu_n^2 - 13\big)\big(16\mu_n^2 - 9\big)\big(16\mu_n^2 - 1\big)}{48\mu_n} \vspace{1mm}\\ 64\mu_n\big(16\mu_n^2 - 1\big) & - \big(16\mu_n^2 - 9\big)\big(16\mu_n^2 - 1\big)\end{bmatrix} + \mathcal{O}\big(\rho_\infty^{-4}\big)\Biggr) \rho_\infty^{\sigma_3/2} \nonumber\\
\phantom{\qquad=}{}\times (2\sqrt{\mu_n})^{\mathbb{I}-\sigma_3}{\dfrac{\ee^{-\pi \ii \mu_n}}{\sqrt{2\pi}}} {}{} \ee^{\pi \ii\sigma_3/4}\begin{bmatrix} 1 & \ii\\ 1 &-\ii \end{bmatrix} \ee^{\ii \rho_\infty \sigma_3}, \qquad \arg(\rho_\infty) \in (-\pi, \pi).\label{h-infty-asymptotics}
\end{gather}
We turn to analogously treating the final factor of the jump of $\hat{\mathbf{R}}^{\mathrm{even/odd}}(\lambda,z)$; using \cite[equation~(10.27.7)]{DLMF} and the above relations, we have
\begin{gather*}
{\rho_0} \begin{bmatrix}I_{2\mu_n - 1}(\rho_0) & I_{1- 2\mu_n}(\rho_0)\vspace{1mm}\\
 -I_{2\mu_n +1}(\rho_0) & -I_{-1-2\mu_n}(\rho_0)\end{bmatrix} \\
 \qquad= \dfrac{\ee^{-\pi \ii/2}\rho_0}{2} \begin{bmatrix} H^{(1)}_{2\mu_n-1}\big(\ee^{-\pi \ii/2} \rho_0\big) & H^{(2)}_{1-2\mu_n}\big(\ee^{-\pi \ii/2} \rho_0\big)\vspace{1mm}\\ H^{(1)}_{1+2\mu_n}\big(\ee^{-\pi \ii/2} \rho_0\big) & H^{(2)}_{-1-2\mu_n}\big(\ee^{-\pi \ii/2} \rho_0\big)\end{bmatrix} \begin{bmatrix} \ee^{\pi \ii \mu_n}& \ee^{\pi \ii \mu_n}\\ - \ee^{3\pi \ii \mu_n} & - \ee^{-\pi \ii \mu_n} \end{bmatrix}.
\end{gather*}
This allows us to find the following large-$\rho_0$ asymptotics:
\begin{gather}
\dfrac{\ee^{-\pi \ii /2}\rho_0}{2} \begin{bmatrix} H^{(1)}_{2\mu_n-1}\big(\ee^{-\pi \ii/2} \rho_0\big) & H^{(2)}_{1 -2\mu_n}\big(\ee^{-\pi \ii/2} \rho_0\big)\vspace{1mm}\\ H^{(1)}_{2\mu_n + 1}\big(\ee^{-\pi \ii/2} \rho_0\big) & H^{(2)}_{-1-2\mu_n}\big(\ee^{-\pi \ii/2} \rho_0\big)\end{bmatrix} \nonumber\\
\qquad= \begin{bmatrix} 1 & \dfrac{1}{2} - \dfrac{\mu_n}{2} - \dfrac{3}{32\mu_n} \vspace{1mm}\\ -1 & \dfrac{1}{2} + \dfrac{\mu_n}{2} + \dfrac{3}{32\mu_n} \end{bmatrix} \Biggl( \mathbb{I} - \dfrac{1}{128 \rho_0^2}  \nonumber\\
\phantom{\qquad=}{}  \times\begin{bmatrix} \big(16\mu_n^2 - 9\big)\big(16\mu_n^2 - 1\big) & \dfrac{\big(16 \mu_n^2 - 13\big)\big(16\mu_n^2 - 9\big)\big(16\mu_n^2 - 1\big)}{48\mu_n} \vspace{1mm}\\ 64\mu_n\big(16\mu_n^2 - 1\big) & - \big(16\mu_n^2 - 9\big)\big(16\mu_n^2 - 1\big)\end{bmatrix} + \mathcal{O}\big(\rho_0^{-4}\big) \Biggr) \rho_0^{\sigma_3/2}\nonumber\\
\phantom{\qquad=}{} \times \big(2\sqrt{\mu_n}\big)^{\mathbb{I}-\sigma_3}{\dfrac{\ee^{-\pi \ii \mu_n}}{\sqrt{2\pi}}} {}{} \begin{bmatrix} 1 & \ii\\ 1 &-\ii \end{bmatrix} \ee^{ \rho_0 \sigma_3}, \qquad \arg(\rho_0) \in \left( -\dfrac{\pi}{2}, \dfrac{3\pi}{2} \right).
\label{eq:H-expand-zero}
\end{gather}
For convenience, we introduce the notation
\begin{equation}
 \mathbf{H}_n(\diamond) := \sqrt{\dfrac{\pi}{4\mu_n}} \ee^{\pi \ii/4 }\ee^{\pi \ii \mu_n} \cdot \dfrac{\diamond}{2} \begin{bmatrix} H^{(1)}_{2\mu_n - 1}(\diamond) & H^{(2)}_{1 - 2\mu_n }(\diamond) \vspace{1mm}\\ H^{(1)}_{1 + 2\mu_n}(\diamond) & H^{(2)}_{-1-2\mu_n}(\diamond) \end{bmatrix},
 \label{eq:Bn-def}
\end{equation}
with a fixed determination of the square root; this choice of prefactor guarantees that we have~$\det (\mathbf{H}_n ) = 1$ identically. Using the identity \cite[equation~(10.11.4)]{DLMF}, we note that $\mathbf{H}_n$ satisfies
\begin{gather}
\label{Bn-jump}
\mathbf{H}_n(\ee^{\pi \ii}\diamond) = \mathbf{H}_n(\diamond) \begin{bmatrix}
0 & -1 \\ 1 & 2\cos(2\pi \mu_n)
\end{bmatrix}.
\end{gather}
We can now rewrite the jump condition \eqref{r-hat-jump} as
\begin{align*}
 &\hat{\mathbf{R}}^\mathrm{even/odd}_+(\lambda,z)= \hat{\mathbf{R}}^\mathrm{even/odd}_-(\lambda,z) \mathbf{H}_n(\rho_\infty) \begin{bmatrix} 1 & \ee^{2\pi \ii \mu_n} \\ -\ee^{2\pi \ii \mu_n} & -1 \end{bmatrix} \mathbf{V}^\mathrm{even/odd}\\
 & \hphantom{\hat{\mathbf{R}}^\mathrm{even/odd}_+(\lambda,z)=}{}
 \times \begin{bmatrix} \ee^{\pi \ii \mu_n}& \ee^{\pi \ii \mu_n}\\ - \ee^{3\pi \ii \mu_n} & - \ee^{-\pi \ii \mu_n} \end{bmatrix}^{-1} \mathbf{H}_n^{-1}\big(\ee^{-\pi \ii/2}\rho_0\big).
\end{align*}
Next, define
\begin{align}
 & \mathbf{\Omega}^\mathrm{even/odd} (\lambda, z) := \begin{bmatrix} 1 & \dfrac{1}{2} - \dfrac{\mu_n}{2} - \dfrac{3}{32\mu_n} \vspace{1.5mm}\\ -1 & \dfrac{1}{2} + \dfrac{\mu_n}{2} + \dfrac{3}{32\mu_n} \end{bmatrix}^{-1} \big(2\sqrt{\mu_n}\big)^{\sigma_3}\nonumber\\
 &\hphantom{\mathbf{\Omega}^\mathrm{even/odd} (\lambda, z) :=}{}
 \times\hat{\mathbf{R}}^\mathrm{even/odd}(\lambda, z) \begin{cases}
 \mathbf{H}_n(\rho_\infty), & |\lambda| >1, \\
 \mathbf{H}_n\big(\ee^{-\pi \ii/2}\rho_0\big), & |\lambda| < 1.
 \end{cases}
 \label{eq:Omega-def}
\end{align}
Then, $\mathbf{\Omega}^\mathrm{even/odd}$ satisfies
\begin{align}
 &\mathbf{\Omega}^\mathrm{even/odd}_+(\lambda, z) = \mathbf{\Omega}_-^\mathrm{even/odd}(\lambda, z) \begin{bmatrix} 1 & \ee^{2\pi \ii \mu_n} \\ -\ee^{2\pi \ii \mu_n} & -1 \end{bmatrix} \cdot \mathbf{V}^\mathrm{even/odd} \nonumber\\
 &\hphantom{\mathbf{\Omega}^\mathrm{even/odd}_+(\lambda, z) =}{}
 \times\begin{bmatrix} \ee^{\pi \ii \mu_n}& \ee^{\pi \ii \mu_n}\\ - \ee^{3\pi \ii \mu_n} & - \ee^{-\pi \ii \mu_n} \end{bmatrix}^{-1}, \qquad |\lambda| = 1,
 \label{eq:Omega-jump-circle}
\end{align}
where the jump depends only on the parity of $n$. Furthermore, since $\rho_\infty$ and $\rho_0$ change signs across the negative imaginary axis, we may use \eqref{Bn-jump} to find
\begin{equation}
 \mathbf{\Omega}_+^\mathrm{even/odd} (\lambda, z) = \mathbf{\Omega}_-^\mathrm{even/odd} (\lambda, z) \begin{bmatrix}
0 & -1 \\ 1 & 2\cos(2\pi \mu_n)
\end{bmatrix},
\end{equation}
for $\lambda$ on the negative imaginary axis with $|\lambda| > 1$, oriented towards the origin and
\begin{equation}
 \mathbf{\Omega}_+^\mathrm{even/odd} (\lambda, z) = \mathbf{\Omega}_-^\mathrm{even/odd} (\lambda, z) \begin{bmatrix}
0 & -1 \\ 1 & 2\cos(2\pi \mu_n)
\end{bmatrix},
\label{eq:Omega-jump-outside-circle}
\end{equation}
for $\lambda$ on the negative imaginary axis with $|\lambda| <1$, oriented away from the origin.

It follows from Riemann--Hilbert Problem \ref{rhp:Rhat-even/odd}, \eqref{eq:Omega-def}, and \eqref{h-infty-asymptotics} that $\mathbf{\Omega}^\mathrm{even/odd}$ has the following asymptotic behavior as $\lambda \to \infty$:
\begin{align}
\mathbf{\Omega}^\mathrm{even/odd}(\lambda, z) = \big( \mathbb{I} + \mathbf{\Xi^{\mathrm{even/odd}}}(z) \lambda^{-1} + \mathcal{O}\big(\lambda^{-2}\big) \big) \rho_\infty^{\sigma_3/2} {\dfrac{1}{\sqrt{2}}} \begin{bmatrix} \ii & -1\\ 1 &-\ii \end{bmatrix} \ee^{\ii \rho_\infty \sigma_3},
 \label{eq:Omega-expand-infty}
\end{align}
where the $\mathcal{O}\big(\lambda^{-2}\big)$ represents an asymptotic series that is differentiable term-by-term with respect to both $\lambda$ and $z$. Analogously, we have
\begin{align}
 &\mathbf{\Omega}^{\mathrm{even/odd}}(\lambda,z)= \mathbf{\Delta}^\mathrm{even/odd}(z)\big(\mathbb{I}+\mathbf{\Pi}^\mathrm{even/odd}(z)\lambda+\mathcal{O}\big(\lambda^2\big)\big)\rho_0^{\sigma_3/2}\ee^{-\pi \ii \sigma_3/4} \nonumber\\
 &\hphantom{\mathbf{\Omega}^{\mathrm{even/odd}}(\lambda,z)=}{}
 \times{\dfrac{1}{\sqrt{2}}} {}{} \begin{bmatrix} \ii & -1 \\ 1 &-\ii \end{bmatrix} \ee^{ \rho_0 \sigma_3}, \qquad \lambda\to 0,
 \label{eq:Omega-expand-zero}
\end{align}
where $\mathcal{O}\big(\lambda^2\big)$ represents an asymptotic series at the origin $\lambda=0$ which is similarly term-by-term differentiable.
Notice that we can now relate the limiting formula from Lemma \ref{lemma:limit} to \smash{$\mathbf{\Omega}^{\mathrm{even/odd}}$} using definitions \eqref{eq:Omega-def} and \eqref{eq:Bn-def} to find that
\begin{gather*}
 \hat{R}^\mathrm{even/odd}_{11}(0,z)+\hat{R}^\mathrm{even/odd}_{21}(0,z)-\hat{R}^\mathrm{even/odd}_{12}(0,z)-\hat{R}^\mathrm{even/odd}_{22}(0,z)\\
 \qquad =\sqrt{\frac{\pi}{4\mu_n}} \ee^{\pi \ii/4} \ee^{\pi \ii \mu_n} \frac{\ee^{-\pi \ii/2}\rho_0}{2} \big[ \big(H^{(2)}_{-1-2\mu_n}\big(\ee^{-\pi \ii/2}\rho_0\big) + H^{(2)}_{1-2\mu_n}\big(\ee^{-\pi \ii/2}\rho_0\big)\big)\Omega^{\mathrm{even/odd}}_{21}(\lambda, z) \\
 \phantom{\qquad =}{}- \big(H^{(1)}_{2\mu_n-1}\big(\ee^{-\pi \ii/2}\rho_0\big) + H^{(1)}_{2\mu_n +1}\big(\ee^{-\pi \ii/2}\rho_0\big)\big)\Omega^{\mathrm{even/odd}}_{22}(\lambda, z) \big]_{\lambda=0}.
\end{gather*}
Then, using \eqref{eq:Omega-expand-zero} and \eqref{eq:H-expand-zero} yields
\begin{gather*}
 \mathop{\lim_{n\to\infty}}_{\text{$n$ even/odd}}u_n\big(n^{-1}z\big)=U^\mathrm{even/odd}(z):=-\frac{1}{2z\Delta^\mathrm{even/odd}_{21}(z)^2}.
\end{gather*}

To extract the monodromy parameters of $U(z)$ from $\mathbf{\Omega}(\lambda, z)$, we notice that it solves Riemann--Hilbert Problem \ref{rhp:D8} with
\begin{equation}\label{eq:limiting-stokes-data}
t_1^\infty = t_0^0 = -2\cos(2\pi \mu_n)= -\left(e_1^2+\frac{1}{e_1^2}\right),
\end{equation}
and
\begin{equation}
\mathbf{C}_{0\infty} = \begin{bmatrix} 1 & \ee^{2\pi \ii \mu_n} \\ -\ee^{2\pi \ii \mu_n} & -1 \end{bmatrix} \cdot \mathbf{V}^\mathrm{even/odd} \cdot \begin{bmatrix} \ee^{\pi \ii \mu_n}& \ee^{\pi \ii \mu_n}\\ - \ee^{3\pi \ii \mu_n} & - \ee^{-\pi \ii \mu_n} \end{bmatrix}^{-1}.
\label{eq:limiting-connection-matrix}
\end{equation}
Since $\mu_n$ and $\mathbf{V}^\mathrm{even/odd}$ depend only on the parity of $n$ (see \eqref{eq:mu-n}, \eqref{eq:Vm-even}, respectively), and ${e_1 = e_{1, n} = \ee^{\pi \ii \mu_n}}$ (see Remark \ref{rem:n-dependent-notation}), it follows that \eqref{eq:limiting-stokes-data}--\eqref{eq:limiting-connection-matrix} depend only on the parity of $n$. Recalling the formul\ae\ for $y_i$ in Section \ref{sec:monodromy-rep-$D_8$-manifold}, one immediately arrives at formul\ae\ \eqref{eq:y1}--\eqref{eq:y3}.

\section[Small x asymptotics and proof of Proposition 1.5]{Small $\boldsymbol{x}$ asymptotics and proof of Proposition \ref{prop:un-zero-asymptotics}} \label{sec:prop-1-proof}
Inspired by \cite{N}, see also \cite[Theorem 3.2]{Jimbo}, the goal of this section is to compute the asymptotics as $x\to 0$ of the B\"acklund iterates $u_n(x)$ for fixed $n$ and, by evaluating at $n = 0$, arriving at the asymptotic behavior of a generic solution of PIII($D_6$) in this limit. Observe that the matrices~\smash{$\mathbf{\Xi}^{(6)}_n(x)$} and \smash{$\mathbf{\Delta}^{(6)}_n(x)$} defined in \eqref{eq:T-n-definition} and \eqref{eq:Psi-n-asymptotic-zero}, respectively, can be expressed in terms of $\mathbf{Q}_n(\lambda,x)$ as follows:
\begin{equation*}
 \mathbf{\Xi}^{(6)}_n(x)=\mathbf{A}_n(x) + \lim_{\lambda\to\infty}\lambda(\mathbf{Q}_n(\lambda,x)-\mathbb{I})-\frac{1}{2}\ii x\sigma_3,\qquad
 \mathbf{\Delta}^{(6)}_n(x)=\mathbf{Q}_n(0,x).
\end{equation*}
Using \eqref{eq:u-n-recover} then gives
\begin{equation}
 u_n(x)=\frac{\displaystyle -\ii A_{n,12}(x)-\ii\lim_{\lambda\to\infty}\lambda Q_{n,12}(\lambda,x)}{Q_{n,11}(0,x)Q_{n,12}(0,x)}.
 \label{eq:un-Q}
\end{equation}
Also, combining \eqref{eq:Ainfty-identity} and \eqref{c-of-x}, we have
\begin{equation}
 A_{n,12}(x)=\frac{\ii \mu_n \big(1-a_n^2\big)}{\gamma_n x^{2-2\kappa_n}}
 \label{eq:An12}
\end{equation}
in which all $x$-dependence is explicit. To analyze $u_n(x)$ for small $x$ it therefore remains to obtain asymptotics of $\mathbf{Q}_n(\lambda,x)$ as $x\to 0$.
To do this, let $\mathbf{V}_{\mathbf{Q}}$ denote the jump matrix \eqref{eq:VQ-rewrite}. Then, denoting the off-diagonal constant matrix
\begin{equation}
\label{eq:def-O-n}
\mathbf{O}_n:=\mathbf{J}_ne_2^{\sigma_3}\begin{bmatrix}0 & -\beta_n\\\beta_n^{-1} & 0\end{bmatrix}\mathbf{J}_{-n}^{-1},
\end{equation}
we arrive at
\begin{align*}
 & \mathbf{V}_\mathbf{Q}(\lambda,x)= \lambda_{\lw}^{\sigma_3/2}x^{\kappa_n\sigma_3}\mathbf{D}_n
 \mathbf{M}(\ii x\lambda;\kappa_n,\mu_n)\mathbf{O}_n\mathbf{M}\bigl(-\ii x\lambda^{-1};\kappa_{-n},\mu_n\bigr)^{-1}\\
 &\hphantom{\mathbf{V}_\mathbf{Q}(\lambda,x)=}{} \times\mathbf{D}_{-n}^{-1}x^{-\kappa_{-n}\sigma_3}\left(-\frac{1}{\lambda}\right)_{\lw}^{-\sigma_3/2} e_0^{-3\sigma_3},\qquad |\lambda|=1,
\end{align*}
where matrices $\mathbf{D}_n$, $\mathbf{M}$ are defined in \eqref{H-define-2} and \eqref{N-define}, respectively. To consider the limit $x\to 0$, we start with the Whittaker matrix $\mathbf{M}(\ii x\lambda;\kappa_n,\mu_n)$. Using \cite[equation~(13.14.6)]{DLMF}, we obtain
\begin{equation*}
 M_{\kappa,\mu}(z)=z^{\frac{1}{2}+\mu}\left(1-\frac{\kappa}{1+2\mu}z + \mathcal{O}\big(z^2\big)\right),\qquad z\to 0.
\end{equation*}
Therefore, using the definition \eqref{N-define} of $\mathbf{M}(Z;\kappa,\mu)$ we have
\begin{gather*}
 (-\ii x\lambda)^{\sigma_3/2}\mathbf{M}(\ii x\lambda;\kappa_n,\mu_n)(-\ii x\lambda)^{-\mu_n\sigma_3} \\ \qquad=
 \begin{bmatrix}\kappa_n-\frac{1}{2}-\mu_n &
 \kappa_n-\frac{1}{2}+\mu_n\\[0.2cm]
 1 & 1\end{bmatrix}+\mathcal{O}(\lambda x)=:\mathbf{M}_\infty(\lambda x),
\end{gather*}
in the limit $x\to 0$ uniformly for $|\lambda|=1$. Similarly, in the same limit,
\begin{gather*}
 \big(\ii x\lambda^{-1}\big)^{\sigma_3/2}\mathbf{M}\bigl(-\ii x\lambda^{-1};\kappa_{-n},\mu_{-n}\bigr)\big(\ii x\lambda^{-1}\big)^{-\mu_{-n}\sigma_3}
\\
 \qquad =\begin{bmatrix}\kappa_{-n}-\frac{1}{2}-\mu_{-n} &
 \kappa_{-n}-\frac{1}{2}+\mu_{-n} \\
 1 &
 1\end{bmatrix}+\mathcal{O}\big(x\lambda^{-1}\big)\\
 \qquad=:\mathbf{M}_0\big(x\lambda^{-1}\big)=\mathbf{M}_0(0)+\mathcal{O}\big(x\lambda^{-1}\big).
\end{gather*}
For $x>0$, the functions \smash{$\lambda_{\lw}^p$}, \smash{$\bigl(-\lambda^{-1}\bigr)_{\lw}^p$}, $(-\ii x\lambda)^p$, and $\big(\ii x\lambda^{-1}\big)^p$ (the latter two being principal branches) all have the same branch cut, namely $\ii\mathbb{R}_-$. One has the following identities:
\begin{equation*}
 (-\ii x\lambda)^p=\ee^{-\ii\pi p/2}x^p\lambda_{\lw}^p, \qquad
 \big(\ii x\lambda^{-1}\big)^p=\ee^{\ii\pi p/2}x^p\lambda_{\lw}^{-p}, \qquad
 \big(-\lambda^{-1}\big)^p_{\lw}=\ee^{\ii\pi p}\lambda_{\lw}^{-p}.
\end{equation*}
It follows that
\begin{equation*}
 \lambda_{\lw}^{\sigma_3/2}x^{\kappa_n\sigma_3}\mathbf{D}_n\mathbf{M}(\ii x\lambda;\kappa_n,\mu_n)=
 x^{(\kappa_n-\frac{1}{2})\sigma_3}\ee^{\ii\pi\sigma_3/4}\mathbf{D}_n
 \mathbf{M}_\infty(\lambda x)
 x^{\mu_n\sigma_3}\ee^{-\ii\pi\mu_n\sigma_3/2}\lambda_{\lw}^{\mu_n\sigma_3},
\end{equation*}
and
\begin{gather*}
 \mathbf{M}\bigl(-\ii x\lambda^{-1};\kappa_{-n},\mu_{-n}\bigr)^{-1}\mathbf{D}_{-n}^{-1}x^{-\kappa_{-n}\sigma_3}\left(-\frac{1}{\lambda}\right)_{\lw}^{-\sigma_3/2}\\
 \qquad=\ee^{-\ii\pi\mu_{-n}\sigma_3/2}x^{-\mu_{-n}\sigma_3}\lambda_{\lw}^{\mu_{-n}\sigma_3}\mathbf{M}_0\big( x\lambda^{-1}\big)^{-1}\ee^{-\ii\pi\sigma_3/4}\mathbf{D}_{-n}^{-1}x^{(\frac{1}{2}-\kappa_{-n})\sigma_3}.
\end{gather*}
Because $\mathbf{O}_n$ is off-diagonal, the central factors in $\mathbf{V}_\mathbf{Q}(\lambda,x)$ simplify as follows:
\begin{equation*}
 x^{\mu_n\sigma_3}\ee^{-\frac{1}{2}\ii\pi\mu_n\sigma_3}\lambda_{\lw}^{\mu_n\sigma_3}\mathbf{O}_n\ee^{-\frac{1}{2}\ii\pi\mu_n\sigma_3}x^{-\mu_n\sigma_3}\lambda_{\lw}^{\mu_n\sigma_3}=
 \mathbf{O}_nx^{-2\mu_n\sigma_3}.
\end{equation*}
Consequently, we have
\begin{align}
 & \mathbf{V}_\mathbf{Q}(\lambda,x)= x^{(\kappa_n-\frac{1}{2})\sigma_3}\ee^{\ii\pi\sigma_3/4}\mathbf{D}_n
 {} \cdot \mathbf{M}_\infty(\lambda x) \cdot \mathbf{O}_n x^{-2\mu_n\sigma_3}\nonumber\\
 &\hphantom{\mathbf{V}_\mathbf{Q}(\lambda,x)=}{}
 \times\mathbf{M}_0\big( x\lambda^{-1}\big)^{-1}\ee^{-\ii\pi\sigma_3/4}\mathbf{D}_{-n}^{-1}x^{(\frac{1}{2}-\kappa_{-n})\sigma_3}e_0^{-3\sigma_3}.
 \label{eq:V_Q-x-asymptotics}
\end{align}
The matrix $\mathbf{V}_{\mathbf{Q}}(\lambda, x)$ does not possess a finite limit as $x \to 0$ due to the factors \smash{$x^{(\kappa_n - \frac{1}{2})\sigma_3}$}, \smash{$x^{( \frac{1}{2} - \kappa_{-n})\sigma_3}$}; this can be handled by introducing the following transformation. Let $\varsigma := \lambda x$ and
\begin{gather}
 \widetilde{\mathbf{Q}}_n(\varsigma,x):=\mathbf{D}_n^{-1}\ee^{-\frac{1}{4}\ii\pi\sigma_3}x^{(\frac{1}{2}-\kappa_n)\sigma_3}\nonumber\\
 \quad\times\begin{cases}
 \mathbf{Q}_n(\frac{\varsigma}{x},x)x^{(-\frac{1}{2}+\kappa_n)\sigma_3}\ee^{\frac{1}{4}\ii\pi\sigma_3}\mathbf{D}_n
 ,& |\varsigma|>1, \\
 \mathbf{Q}_n(\frac{\varsigma}{x},x) \mathbf{V}_\mathbf{Q}(\frac{\varsigma}{x},x)e_0^{3\sigma_3}x^{(\kappa_{-n}-\frac{1}{2})\sigma_3}\mathbf{D}_{-n}\ee^{\frac{1}{4}\ii\pi\sigma_3} \\ \quad\times \mathbf{M}_0(0)x^{2\mu_n\sigma_3}\mathbf{O}_n^{-1}\mathbf{M}_\infty(\varsigma)^{-1},& |x|<|\varsigma|<1, \\
 \mathbf{Q}_n(\frac{\varsigma}{x},x)e_0^{3\sigma_3}x^{(\kappa_{-n}-\frac{1}{2})\sigma_3}\mathbf{D}_{-n}\ee^{\frac{1}{4}\ii\pi\sigma_3}\mathbf{M}_0(0)x^{2\mu_n\sigma_3}\mathbf{O}_n^{-1}\mathbf{M}_\infty(\varsigma)^{-1},&|\varsigma|<|x|.
 \end{cases}
 \label{eq:Q-tilde}
\end{gather}
It follows that $\widetilde{\mathbf{Q}}_n$ is analytic as a function of $\varsigma$ for $\varsigma \in \C \setminus \{|\varsigma| = 1 \}$ and satisfies
\[
\lim_{\varsigma \to \infty} \widetilde{\mathbf{Q}}_n(\varsigma, x) = \mathbb{I}.
\]
Furthermore, on the circle $|\varsigma| = 1$, the jump condition $\widetilde{\mathbf{Q}}_{n,+}(\varsigma,x)=\widetilde{\mathbf{Q}}_{n,-}(\varsigma,x)\mathbf{V}_{\widetilde{\mathbf{Q}}}(\varsigma,x)$ holds, where the jump contour has counterclockwise orientation, and
\begin{align*}
&\mathbf{V}_{\widetilde{\mathbf{Q}}}(\varsigma,x) := \mathbf{D}_n^{-1} \ee^{-\pi \ii \sigma_3/4} x^{(\frac{1}{2} - \kappa_n)\sigma_3} \mathbf{V}_{\mathbf{Q}}\left(\frac{\varsigma}{x},x\right) e_0^{3\sigma_3} x^{(\kappa_{-n} - \frac{1}{2})\sigma_3} \ee^{\pi \ii \sigma_3 /4}\\
&\hphantom{\mathbf{V}_{\widetilde{\mathbf{Q}}}(\varsigma,x) :=}{} \times \mathbf{D}_{-n} \mathbf{M}_0(0)x^{2\mu_n \sigma_3} \mathbf{O}_n^{-1} \mathbf{M}_\infty^{-1}(\varsigma).
\end{align*}
Using \eqref{eq:V_Q-x-asymptotics} immediately yields
\begin{equation*}
 \mathbf{V}_{\widetilde{\mathbf{Q}}}(\varsigma,x)= \mathbf{M}_\infty(\varsigma)\mathbf{O}_n x^{-2\mu_n\sigma_3}
 \mathbf{M}_0\big( {x^2}\varsigma^{-1}\big)^{-1}\mathbf{M}_0( 0)x^{2\mu_n\sigma_3}\mathbf{O}_n^{-1} \mathbf{M}_\infty(\varsigma)^{-1}.
\end{equation*}
Therefore, $\widetilde{\mathbf{Q}}_n(\varsigma,x)$ solves the following Riemann--Hilbert problem.
\begin{rhp}\label{rhp:Q-tilde}
Fix generic monodromy parameters $( e_1, e_2)$, and $x\in\mathbb{C}$. Seek a $2\times 2$ matrix function $\varsigma\mapsto \widetilde{\mathbf{Q}}_n(\varsigma,x)$ with the following properties:
\begin{itemize}\itemsep=0pt \samepage
 \item Analyticity: $\widetilde{\mathbf{Q}}_n(\varsigma,x)$ is an analytic function of $\varsigma$ for $|\varsigma|\neq 1$.
 \item Jump condition:
 $\widetilde{\mathbf{Q}}_n(\varsigma,x)$ takes analytic boundary values on the unit circle from the interior and exterior, denoted $\widetilde{\mathbf{Q}}_{n,+}(\varsigma,x)$ and $\widetilde{\mathbf{Q}}_{n,-}(\varsigma,x)$ for $|\varsigma|=1$ respectively, and they are related by
 \begin{equation*}
 \widetilde{\mathbf{Q}}_{n,+}(\varsigma,x)=\widetilde{\mathbf{Q}}_{n,-}(\varsigma,x)\mathbf{M}_\infty(\varsigma)\mathbf{O}_n x^{-2\mu_n\sigma_3}
 \mathbf{M}_0\big( {x^2}\varsigma^{-1}\big)^{-1}\mathbf{M}_0( 0)x^{2\mu_n\sigma_3}\mathbf{O}_n^{-1} \mathbf{M}_\infty(\varsigma)^{-1}.
 \end{equation*}
 \item Normalization:
 $\widetilde{\mathbf{Q}}_n(\varsigma,x)\to\mathbb{I}$ as $\varsigma\to\infty$.
\end{itemize}
\end{rhp}

The jump $\mathbf{V}_{\widetilde{\mathbf{Q}}}$ has a limit as $x\to 0$, uniformly for $|\lambda|=1$ for $|{\re}\,  \mu_n| <\frac{1}{2}$, and satisfies the estimate
\begin{equation*}
 \mathbf{V}_{\widetilde{\mathbf{Q}}}(\varsigma,x)=\mathbb{I}+\mathcal{O} \big(x^{2-|4\re\mu_n|} \big).
\end{equation*}
By the standard theory of small-norm Riemann--Hilbert problems, we arrive at
\begin{equation*}
 \widetilde{\mathbf{Q}}_n(\varsigma,x)=\mathbb{I} + \mathcal{O} \big(x^{2-|4\re\mu_n|} \big) \qquad \text{as} \quad x\to 0
\end{equation*}
uniformly for $\varsigma$ sufficiently small, and
\begin{equation*}
 \lim_{\varsigma\to\infty}\varsigma\big(\widetilde{\mathbf{Q}}_n(\varsigma,x)-\mathbb{I}\big)=\lim_{\varsigma\to\infty}\varsigma\big(\widetilde{\mathbf{Q}}_n(\varsigma,0)-\mathbb{I}\big)+\mathcal{O}\big(x^{2-|4\re\mu_n|} \big)=\mathcal{O}\big(x^{2-|4\re\mu_n|} \big).
\end{equation*}
We can now use the above estimate and expressions \eqref{eq:un-Q} and \eqref{eq:An12} to compute the asymptotic behavior of $u_n(x)$ as $x\to 0$. To this end, note that by \eqref{eq:Q-tilde} and the definition of $\mathbf{D}_n$ and $\varsigma$,
\begin{equation*}
 \lim_{\lambda\to\infty}\lambda Q_{n,12}(\lambda,x)=-\frac{1}{\mu_n\gamma_n} x^{2\kappa_n-2} \lim_{\varsigma\to\infty}\varsigma \widetilde{Q}_{n,12}(\varsigma,x),\qquad |\varsigma|>1,
\end{equation*}
and so
\begin{equation}
 \lim_{\lambda\to\infty}\lambda Q_{n,12}(\lambda,x)=\mathcal{O} \big(x^{2\kappa_n-|4\re\mu_n|} \big),\qquad x\to 0.
 \label{eq:Q-at-infty}
\end{equation}
Likewise, \eqref{eq:Q-tilde} gives
\begin{gather}
 Q_{n,11}(0,x)Q_{n,12}(0,x)\nonumber\\
 \qquad=-\ii\frac{D_{n,11}^2}{4\mu_n^2D_{-n,11}D_{-n,22}}
 \begin{cases} {O_{n,21}^2\big(\kappa_n-\tfrac{1}{2}+\mu_n\big)^2\big(\kappa_{-n}-\tfrac{1}{2}+\mu_n\big)}{}x^{2\kappa_n-1-4\mu_n}\\
 \quad{}+ \mathcal{O}\big(x^{2\kappa_n-|4\re\mu_n|}\big),\quad \re \mu_n>0,
 \\{O_{n,12
 }^2\big(\kappa_n-\tfrac{1}{2}-\mu_n\big)^2\big(\kappa_{-n}-\tfrac{1}{2}-\mu_n\big)}{}x^{2\kappa_n-1+4\mu_n}\\
 \quad{}+ \mathcal{O}\big(x^{2\kappa_n-|4\re\mu_n|}\big),\quad \re \mu_n<0.
 \end{cases}
 \label{eq:Q-at-zero}
\end{gather}
Using \eqref{eq:Q-at-infty}, \eqref{eq:Q-at-zero} in \eqref{eq:un-Q} yields
\begin{equation*}
u_n(x) = \dfrac{4\ii \mu_n^3 \big(1-a_n^2\big)D_{-n, 11}D_{-n, 22}}{\gamma_n D_{n, 11}^2 O_{n, 21}^2 \big(\kappa_n-\tfrac{1}{2}+\mu_n\big)^2\big(\kappa_{-n}-\tfrac{1}{2}+\mu_n\big)} x^{4\mu_n - 1} \big(1 + \mathcal{O}\big(x^\delta\big) \big) \qquad \text{as} \quad x \to 0,
\end{equation*}
when $\re \mu_n > 0$ and
\begin{align*}
&u_n(x) = \dfrac{4\ii \mu_n^3 \big(1-a_n^2\big)D_{-n, 11}D_{-n, 22}}{\gamma_n D_{n, 11}^2 O_{n, 12}^2 \big(\kappa_n-\tfrac{1}{2}-\mu_n\big)^2\big(\kappa_{-n}-\tfrac{1}{2}-\mu_n\big)}x^{-4\mu_n - 1} \big(1 + \mathcal{O}\big(x^\delta\big) \big) \qquad \text{as} \quad x \to 0,
\end{align*}
when $\re \mu_n < 0$, where $\delta = \min(1, 2 - |4\re(\mu_n)| )$ in both cases. Using \eqref{eq:def-O-n}, \eqref{eq:a-def}, \eqref{eq:J-n}, \eqref{Qinfty-gammainfty}, \eqref{H-define-2}, and \eqref{eq:little-an} gives the expression\footnote{The case $\re(\mu_n) = 0$ can be treated similarly, and produces a leading term that is a combination of both leading terms, which we omit for brevity.}
 \begin{align}
& u_n(x) =
 -\frac{ \Gamma (1 - 2\epsilon_n \mu_n)^2 \Gamma \bigl(-\frac{n}{2} + \epsilon_n \mu_n -\frac{\Theta_0}{2}\bigr) \Gamma \big(\frac n2 + \epsilon_n \mu_n - \frac{\Theta_\infty}{2} + 1 \big)}{\Gamma (2\epsilon_n \mu_n )^2 \Gamma \bigl(-\frac{n}{2} - \epsilon_n \mu_n -\frac{\Theta_0}{2}+1\bigr) \Gamma \big(\frac{n}{2} - \epsilon_n \mu_n -\frac{\Theta_\infty }{2}+1\big)} x^{4 \epsilon_n \mu_n - 1} \nonumber\\
 &\hphantom{u_n(x) =}{}
 \times\big(1 + \mathcal{O}(x^\delta) \big)
 \begin{cases} \dfrac{e_0^2 e_2^2 e_\infty^2\big(e_0^2-e_1^2\big) \big(e_1^2-e_\infty^2\big)}{\big(e_0^2 e_1^2-1\big)^2} ,& \re \mu_n >0, \\ \dfrac{\big(e_0^2 e_1^2-1\big)^2}{e_0^2 e_2^2 e_\infty^2\big(e_0^2-e_1^2\big) \big(e_1^2-e_\infty^2\big)},& \re \mu_n < 0, \end{cases}
 \label{eq:u-n-leading}
 \end{align}
 where $\epsilon_n = \sgn(\re \mu_n)$.
The concerned reader may note that the leading coefficient in \eqref{eq:u-n-leading} is finite due to the genericity conditions on $(e_1, e_2)$ (see the beginning of Section \ref{sec:proof}). Indeed, assumption (i) guarantees that $2\mu_n \not \in \Z$, condition (ii) requires $e_1 e_2 \neq 0$, and condition (iii) guarantees that
\[
\dfrac{n}{2} \pm \mu_n + \dfrac{\Theta_0}{2} \not \in \Z \qquad \text{and} \qquad \dfrac{n}{2} \pm \mu_n + \dfrac{\Theta_\infty}{2} \not \in \Z.
\]
Evaluating the above at $n =0$ yields \eqref{eq:u-0-leading} and finishes the proof of Proposition \ref{prop:un-zero-asymptotics}.

One notable application of this is to the family of rational solutions of Painlev\'e-III already discussed at the end of Section \ref{sec:initial-rhp}. This corresponds to the choice $m = \Theta_0 = \Theta_\infty - 1$ and~${\mu_0 = 1/4}$. It follows from \eqref{eq:u-n-leading} that $u_n(x;m)$ has a well-defined value at $x = 0$ which is given by \eqref{eq:u-at-zero-even}, \eqref{eq:u-at-zero-odd} in the case where $n$ is even or odd, respectively. We can verify that these values are consistent with \eqref{eq:u-n-leading} by noting that $e_1$, $e_2$, $e_0^2$, $e_\infty^2$ are invariant under an even increment $n \mapsto n+2$, and so we have the general formul\ae
\begin{align*}
\dfrac{u_{2k+2}(0)}{u_{2k}(0)} &= \frac{(2 k + 2\mu_{2k+2} + \Theta_0 ) (2 k+2\mu_{2k+2} +1-\Theta_\infty )}{(2 +2 k-2 \mu_{2k} +\Theta_0) (2+2 k-2 \mu_{2k} -\Theta_\infty )}, \\
\dfrac{u_{2k+1}(0)}{u_{2k-1}(0)} &= \frac{(2 k - 1+2 \mu_{2k+1} +\Theta_0) (2 k + 1+2 \mu_{2k+1} -\Theta_\infty )}{(2 k + 1-2 \mu_{2k-1} +\Theta_0) (+2 k + 1-2 \mu_{2k-1} -\Theta_\infty )}.
\end{align*}
Plugging in the specialized values of the parameters and using the known values of $u_0(0;m)$, $u_1(0;m)$ yields the equality of the expression in \eqref{eq:u-n-leading} with the product formul\ae\ \eqref{eq:u-at-zero-even}--\eqref{eq:u-at-zero-odd}.

\section[Alternative Riemann-Hilbert problem for Painlev\'e-III(D\_8)]{Alternative Riemann--Hilbert problem for Painlev\'e-III($\boldsymbol{D_8}$)}

\subsection[Fabry-type transformation and existence of Rhateven/odd(lambda,z)]{Fabry-type transformation and existence of $\boldsymbol{\widehat{R}^\mathrm{even/odd}(\lambda,z)}$} \label{sec:suleimanov-solution-connection}
The Lax pair \eqref{eq:Lax-system} is unusual in that its coefficient matrices have non-diagonalizable leading terms at both of its singular points $\lambda = 0$ and $\lambda = \infty$, i.e., the coefficients of $\lambda^0$ and $\lambda^{-2}$ in~\eqref{eq:Lambda-exact} are not diagonalizable. To deduce the existence of the matrix functions ${\mathbf{\Omega}}^{\mathrm{even/odd}}(\lambda, z)$ and \smash{$\widehat{\mathbf{R}}^{\mathrm{even/odd}}(\lambda, z)$}, we identify this Lax pair with ones appearing in the literature by considering the following Fabry-type transformation
\begin{align}
&\mathbf{S}(\xi, z) := {\dfrac{1}{\sqrt{2}}} \begin{bmatrix} -\ii & 1\\ -1 &\ii \end{bmatrix} (2z)^{-\sigma_3/4} \xi^{-\sigma_3/2}\nonumber
\\
&\hphantom{\mathbf{S}(\xi, z) :=}{}
\times \begin{cases}
\mathbf \Omega\big(\xi^2 \ee^{\frac{\ii\pi}{2}}, z\big),& -\frac{\pi}{2}<\arg(\xi)<\frac{\pi}{2},\\
\mathbf \Omega\big(\xi^2 \ee^{-\frac{3\pi \ii}{2}}, z\big)(-\ii\sigma_2),& \frac{\pi}{2}<\arg(\xi)<\pi, \\
\mathbf \Omega\big(\xi^2 \ee^{\frac{5\pi \ii}{2}}, z\big)\ii\sigma_2,& -\pi<\arg(\xi)<-\frac{\pi}{2},
 \end{cases}
 \label{eq:fabry-1}
\end{align}
when $|\xi|>1$, and
\begin{gather}
\mathbf{S}(\xi, z) := {\dfrac{1}{\sqrt{2}}} \begin{bmatrix} -\ii & 1\\ -1 &\ii \end{bmatrix} (2z)^{-\sigma_3/4} \xi^{-\sigma_3/2}
\cdot \begin{cases}
\mathbf \Omega\big(\xi^2 \ee^{\frac{\ii\pi}{2}}, z\big),& -\frac{\pi}{2}<\arg(\xi)<\frac{\pi}{2},\\
\mathbf \Omega\big(\xi^2 \ee^{-\frac{3\pi \ii}{2}}, z\big)\ii\sigma_2,& \frac{\pi}{2}<\arg(\xi)<\pi, \\
\mathbf \Omega\big(\xi^2 \ee^{\frac{5\pi \ii}{2}}, z\big)(-\ii\sigma_2),& -\pi<\arg(\xi)<-\frac{\pi}{2},
 \end{cases}\!\!\!\!
 \label{eq:fabry-2}
\end{gather}
when $|\xi|<1$. Denoting
\begin{equation}
\mathbf{K} := {\dfrac{1}{\sqrt{2}}} \begin{bmatrix} \ii & -1\\ 1 & -\ii \end{bmatrix}
\label{eq:K-def}
\end{equation}
and using expansions \eqref{eq:Omega-expand-infty}, \eqref{eq:Omega-expand-zero} (note the branch choices in \eqref{eq:rho-infty-branch}, \eqref{eq:rho-zero-branch}) one can directly check that
\begin{align*}
 \mathbf{S}(\xi, z) = \mathbf{K}^{-1} \Bigg( \mathbb{I} \!+\! \begin{bmatrix} 0 & 0 \\ {(2z)^{1/2}}\Xi^{(8)}_{21}(z) & 0\end{bmatrix} \dfrac{1}{\ii \xi}\! + \!\begin{bmatrix} \Xi^{(8)}_{11}(z) & 0 \\ 0 & {\Xi^{(8)}_{22}(z)} \end{bmatrix} \dfrac{1}{\ii \xi^2}\!+\! \mathcal{O}\big(\xi^{-3}\big)\Bigg)\mathbf{K}\ee^{\ii (2z)^{1/2} \xi \sigma_3},
\end{align*}
as $\xi \to \infty$, and
\begin{align*}
& \mathbf{S}(\xi, z) = \mathbf{K}^{-1} \left( \begin{bmatrix} \Delta^{(8)}_{11}(z) & 0 \\ 0 & 0\end{bmatrix} \dfrac{1}{\xi} + \begin{bmatrix} 0 & \dfrac{\Delta^{(8)}_{12}(z)}{(2z)^{1/2}} \\ (2z)^{1/2} \Delta^{(8)}_{21}(z) & 0 \end{bmatrix} \right.\\
 &\left.\hphantom{\mathbf{S}(\xi, z) =}{}
 + \begin{bmatrix} f(z)& 0 \\ 0 & {\Delta^{(8)}_{22}(z)} \end{bmatrix}\xi + \mathcal{O}\big(\xi^2\big)\right)\ee^{-\frac{\ii\pi}{4}\sigma_3}\mathbf{K} \ee^{ (2z)^{1/2}{ \xi^{-1}} \sigma_3},
\end{align*}
as $\xi\to 0$, where $f(z):= \ii\left((\Delta_{11} \Pi_{11})(z) + (\Delta_{12}\Pi_{21})(z) \right)$; this partly works due to the identity
\begin{equation}
\mathbf{K}\sigma_2 = - \sigma_3 \mathbf K.
\label{eq:K-symmetry}
\end{equation}
Furthermore, one can directly verify that the jump relations \eqref{eq:Omega-jump-circle}--\eqref{eq:Omega-jump-outside-circle} translate to the jumps shown in Figure \ref{fig:D8-limit}, where $\mathbf{C}_{0 \infty}$ is as in \eqref{eq:limiting-connection-matrix}. \begin{figure}
 \centering
 \includegraphics{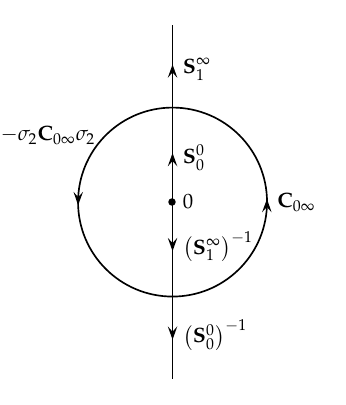}
 \caption{The jump contour and matrices for $\widetilde{\mathbf{S}}$, where $\mathbf{S}_1^\infty$, $\mathbf{S}_0^0$ are as in \eqref{eq:Stokes-zero} and $\mathbf{C}_{0\infty}$ is in~\eqref{eq:limiting-connection-matrix}.}
 \label{fig:D8-limit}
\end{figure}
The jump matrices satisfy two cyclic relations about the nonsingular self-intersection points of the jump contour, namely,
 \begin{alignat*}{3}
 &\text{about}\ \xi = +\ii\colon\ &&\mathbf{C}_{0\infty}^{-1} \mathbf{S}_1^\infty (\ii \sigma_2) \mathbf{C}_{0\infty} (\ii \sigma_2) \big(\mathbf{S}_0^0\big)^{-1} = \mathbb{I}, &\\
 &\text{about}\ \xi = -\ii\colon\ &&\left[(\ii \sigma_2) \mathbf{C}_{0\infty} (\ii \sigma_2)\right]^{-1} \big(\mathbf{S}_{0}^0\big)^{-1}\mathbf{C}_{0\infty}\mathbf S_1^\infty = \mathbb{I}.&
\end{alignat*}
Observe that the matrix ${\mathbf{S}}(\xi, z)$ possesses the following useful symmetry:
\begin{equation}
{\mathbf{S}}(\xi, z) = \sigma_2 \begin{cases}
 -{\mathbf{S}}(-\xi, z) \sigma_2, & |\xi|<1, \\
{\mathbf{S}}(-\xi, z) \sigma_2, & |\xi|>1.
\label{eq:S-symmetry}
\end{cases}
\end{equation}
This result also uses the identity \eqref{eq:K-symmetry}. Using this symmetry, it can be checked that the Fabry transformation \eqref{eq:fabry-1}--\eqref{eq:fabry-2} is invertible with{\samepage
\begin{equation}
\mathbf \Omega(\lambda ) = \rho_\infty^{\sigma_3/2}(\lambda, z) \mathbf{K} \mathbf S\big(\sqrt{-\ii \lambda}\big),
\label{eq:inverse-fabry}
\end{equation}
where all roots are principal branches.}

While the singular behavior of $\mathbf{S}(\xi, z)$ at $\xi = 0$ is concerning, the fact that the leading coefficient is a singular matrix allows us to handle this problem by letting
\begin{equation}\label{eq:triangular-transformation}
\widetilde{\mathbf{S}}(\xi, z) = \left( \mathbb{I} -\dfrac{1}{\xi} \mathbf{T}(z) \right) \mathbf{S}(\xi, z),
\end{equation}
where
\begin{equation}
\mathbf{T}(z) = \mathbf{K}^{-1}\begin{bmatrix} 0 & \dfrac{\Delta^{(8)}_{11}(z)}{(2z)^{1/2}\Delta^{(8)}_{21}(z)} \\ 0 & 0\end{bmatrix}\mathbf{K}.
\label{eq:L-def}
\end{equation}
Since the prefactor is analytic in $\C \setminus \{0\}$, the jumps of $\widetilde{\mathbf{S}}$ are unchanged. As for the asymptotic behavior, as $\xi\to\infty$
\begin{align*}
\widetilde{\mathbf S}(\xi, z) = \left( \mathbb{I} + \mathbf K^{-1} \begin{bmatrix}
 0 & -\dfrac{\Delta^{(8)}_{11}(z)}{(2z)^{1/2}\Delta^{(8)}_{21}(z)} \\ -\ii(2z)^{1/2} \Xi^{(8)}_{21}(z) & 0
\end{bmatrix} \mathbf K \dfrac{1}{\xi}+ \mathcal{O}\big(\xi^{-2}\big) \right)\ee^{\ii (2z)^{1/2} \xi \sigma_3},
\end{align*}
and as $\xi\to 0$
\begin{align*}
\widetilde{\mathbf S}(\xi, z) = \left( \mathbf{K}^{-1}\begin{bmatrix} 0 & -\dfrac{1}{(2z)^{1/2}\Delta^{(8)}_{21}(z)} \\ (2z)^{1/2}\Delta^{(8)}_{21}(z) & 0\end{bmatrix} \ee^{-\frac{\ii\pi}{4}\sigma_3}\mathbf{K} + \mathcal{O}(\xi) \right)\ee^{{ (2z)^{1/2}}{ \xi^{-1}} \sigma_3}.
\end{align*}
\begin{Remark}
Noting that
\[
\det \left( \mathbf{K}^{-1}\begin{bmatrix} 0 & -\dfrac{1}{(2z)^{1/2}\Delta^{(8)}_{21}(z)} \\ (2z)^{1/2}\Delta^{(8)}_{21}(z) & 0\end{bmatrix} \ee^{-\frac{\ii\pi}{4}\sigma_3}\mathbf{K} \right) = 1,
\]
one can carry out a computation similar to the one in Section \ref{sec:Lax-pair} to arrive at a pair of differential equations analogous to \eqref{eq:Lax-system}, but with diagonalizable leading matrices at the two singular points at $\xi = 0, \infty$; this system appears in \cite[Chapter 2]{N} and \cite{GL}, for example. Since we do not make use of this Lax pair, we omit the calculation.
\end{Remark}

Using \eqref{eq:K-symmetry}, it follows that
\[
\mathbb I - \dfrac{1}{\xi}\mathbf T(z) = \sigma_2\left( \mathbb I + \dfrac{1}{\xi}\mathbf T(z) \right) \sigma_2,
\]
which implies that matrix $\widetilde{\mathbf{S}}(\xi, z)$ also satisfies the symmetry \eqref{eq:S-symmetry}. To simplify this symmetry, let
\begin{equation}
\widehat{\mathbf{S}}(\xi, z) := \ee^{-\pi \ii \sigma_3/4 } \begin{cases} \widetilde{\mathbf{S}}(\xi, z) \ee^{\pi \ii \sigma_3/4}, & |\xi| >1, \\
\widetilde{\mathbf{S}}(\xi, z) \ee^{-\pi \ii \sigma_3/4}, & |\xi| <1.
\end{cases}
 \label{eq:rotation-transformation}
\end{equation}
Then, $\widehat{\mathbf{S}}(\xi, z)$ solves the following Riemann--Hilbert problem.
\begin{rhp}
Let $(y_1,y_2, y_3) \in \C^3$ be the monodromy data corresponding to~$U(z)$ given in \eqref{eq:y1}--\eqref{eq:y3}, and fix $z\in\mathbb{C}$. Seek a $2\times 2$ matrix function~${\xi\mapsto\widehat{\mathbf{S}}(\xi, z)}$ satisfying the following properties:
\begin{itemize}\itemsep=0pt
 \item Analyticity: $\widehat{\mathbf{S}}(\xi, z)$ is an analytic function of $\xi$ for $|\xi|\neq 1$.
 \item Jump condition:
 $\widehat{\mathbf{S}}(\xi, z)$ takes analytic boundary values on the unit circle from the interior and exterior, denoted $\widehat{\mathbf{S}}_+(\xi,z)$ and $\widehat{\mathbf{S}}_-(\xi,z)$ for $|\xi|=1$ respectively, and they are related by
 \begin{equation*}
 \widehat{\mathbf{S}}_+(\xi,z)=\widehat{\mathbf{S}}_-(\xi,z) \mathbf J_{\widehat{\mathbf{S}}}(\xi),
 \end{equation*}
 where $\mathbf J_{\widehat{\mathbf{S}}}(\xi)$ is shown in Figure {\rm\ref{fig:D8-limit-rotated}} and
 \begin{align}
 &\widehat{\mathbf C}_{0\infty} := \ee^{\pi \ii \sigma_3/4} \mathbf C_{0\infty}\ee^{\pi \ii \sigma_3/4}, \qquad
 \widehat{\mathbf S}_1^\infty := \ee^{-\pi \ii \sigma_3/4} {\mathbf S}_1^\infty \ee^{\pi \ii \sigma_3/4} = \ee^{\pi \ii \sigma_3/4} ({\mathbf S}_1^\infty)^{-1} \ee^{-\pi \ii \sigma_3/4}, \nonumber\\
 & \widehat{\mathbf{S}}_0^0:= \ee^{\pi \ii \sigma_3/4} {\mathbf S}_0^0 \ee^{-\pi \ii \sigma_3/4} = \ee^{-\pi \ii \sigma_3/4} \big({\mathbf S}_0^0\big)^{-1} \ee^{\pi \ii \sigma_3/4}. \label{eq:hat-jumps}
 \end{align}
 \item Normalization:
\begin{gather}\label{eq:s-hat-norm-inf}\widehat{\mathbf{S}}(\xi,z) = \big(\mathbb{I} +\widehat{\mathbf{\Xi}}^{(8)}(z)\xi^{-1}+ \mathcal{O}\big(\xi^{-2}\big) \big) \ee^{\ii (2z)^{1/2} \xi \sigma_3 } \qquad \text{as} \quad \xi \to \infty,
\end{gather}
 and
\begin{gather}
\widehat{\mathbf{S}}(\xi,z) = \widehat{\mathbf{\Delta}}^{(8)}(z) \big(\mathbb{I} + \widehat{\mathbf{\Pi}}(z)\xi+\mathcal{O}\big(\xi^2\big) \big) \ee^{ (2z)^{1/2} \xi^{-1} \sigma_3 } \qquad \text{as} \quad \xi \to 0,\label{eq:s-hat-norm-zero}
\end{gather}
 where $\widehat{\mathbf{\Delta}}^{(8)}(z)$ may be written in terms of entries of $\mathbf{\Delta}^{(8)}(z)$ and $\mathbf{K}$.
\end{itemize}
\label{rhp:S-hat}
\end{rhp}
\begin{figure}
 \centering
 \includegraphics{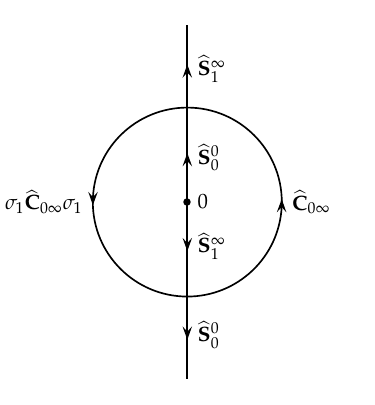}
 \caption{The jump contour and matrices for $\hat{\mathbf S}$, where the jump matrices are as in \eqref{eq:hat-jumps}.}
 \label{fig:D8-limit-rotated}
\end{figure}
Now, the matrix $\widehat{\mathbf{S}}(\xi,z)$ satisfies the symmetry
\[
\sigma_1 \widehat{\mathbf{S}}(-\xi,z) \sigma_1 = \widehat{\mathbf{S}}(\xi,z) .
\]
Furthermore, it was shown in \cite[Theorem 4]{N} that matrix $\widehat{\mathbf{S}}(\xi, z)$ exists for all $z$ outside of a~discrete set $\Sigma$ and is a meromorphic function of $z$ in $\C \setminus \Sigma$. Since the transformations used to arrive to $\widehat{\mathbf{S}}(\xi,z) $ from $\widehat{\mathbf{R}}(\lambda,z) $ and ${\mathbf{\Omega}}(\lambda,z) $ are invertible, we deduce the existence of matrix functions satisfying Riemann--Hilbert Problems \ref{rhp:Rhat-even/odd} and \ref{rhp:D8}.

It was shown in \cite{Palmer} that $\Sigma$ coincides with the set of zeros of the $\tau$-function associated to the Riemann--Hilbert problem. According to \cite{JMU} the expression for the logarithmic derivative of the $\tau$-function associated to Riemann--Hilbert Problem \ref{rhp:S-hat} is given by
\begin{gather} \label{eq:jmu-tau}
\dfrac{\mathrm{d}}{\mathrm{d}z}\ln(\tau(z))=-\frac{1}{\sqrt{2z}}\big(\Tr\big(\widehat{\mathbf{\Pi}}(z)\sigma_3\big)+\ii\Tr\big(\widehat{\mathbf{\Xi}}^{(8)}(z)\sigma_3\big)\big).
\end{gather}
 After a long symbolic computation using the transformations \eqref{eq:fabry-1}, \eqref{eq:fabry-2}, \eqref{eq:triangular-transformation}, and \eqref{eq:rotation-transformation}, we get the following result. The Lax pair \eqref{eq:lax_pair_D8}--\eqref{eq:lax_pair_D8_2} is gauge equivalent to
\[
\frac{\partial\widehat{\mathbf{S}}}{\partial\xi} (\xi, z) = {\widehat{\mathbf{\Lambda}}}^{(8)}(\xi, z) \widehat{\mathbf{S}}(\xi, z), \qquad
\frac{\partial\widehat{\mathbf{S}}}{\partial z} (\xi, z) = {\widehat{\mathbf{Z}}}(\xi, z) \widehat{\mathbf{S}}(\xi, z),
\]
where
\begin{align*}
&\widehat{\mathbf{\Lambda}}^{(8)}(\xi, z) = \ii(2z)^{1/2}\sigma_3 +\frac{1}{\xi}\frac{zU'(z)}{2U(z)}\sigma_1+\frac{\ii}{\xi^2}\left({\frac{z}{2}}\right)^{1/2}\left(U(z)-\frac{1}{U(z)}\right)\sigma_3\\
&\hphantom{\widehat{\mathbf{\Lambda}}^{(8)}(\xi, z) =}{}
+\frac{1}{\xi^2}\left({\frac{z}{2}}\right)^{1/2}\left(U(z)+\frac{1}{U(z)}\right)\sigma_2,
\end{align*}
and
\begin{align*}
&\widehat{\mathbf{Z}}(\xi, z) = \frac{\ii\xi}{(2z)^{1/2}}\sigma_3 +\frac{U'(z)}{4U(z)}\sigma_1-\frac{\ii}{2\xi}\frac{1}{(2z)^{1/2}}\left(U(z)-\frac{1}{U(z)}\right)\sigma_3\\
&\hphantom{\widehat{\mathbf{Z}}(\xi, z) =}{}
-\frac{1}{2\xi}\frac{1}{(2z)^{1/2}}\left(U(z)+\frac{1}{U(z)}\right)\sigma_2.
\end{align*}
Similarly, the coefficients in \eqref{eq:s-hat-norm-inf}--\eqref{eq:s-hat-norm-zero} have the following expressions
\begin{gather}
\widehat{\mathbf{\Delta}}^{(8)}(z)=\big(\ee^{-\ii\pi/4}U(z)^{1/2}\big)^{\sigma_1},\nonumber\\
\widehat{\mathbf{\Xi}}^{(8)}(z)=\ii\left(\frac{z}{2}\right)^{1/2}\left(\frac{zU'(z)^2}{8U(z)^2}-
 U(z)+\frac1{U(z)}\right)\sigma_3-\left(\frac{z}{2}\right)^{1/2}\frac{U'(z)}{4U(z)}\sigma_2,\nonumber\\
 \widehat{\mathbf{\Pi}}(z)=-\left(\frac{z}{2}\right)^{1/2}\left(\frac{zU'(z)^2}{8U(z)^2}-
 U(z)+\frac1{U(z)}\right)\sigma_3+\ii\left(\frac{z}{2}\right)^{1/2}\frac{U'(z)}{4U(z)}\sigma_2.\label{eq:s-hat-residues}
\end{gather}
In our computation we expressed $W(z)$, $X(z)$, and $V(z)$ in terms of $U(z)$ and $U'(z)$ using the identities \eqref{eq:UVWXidentity1}--\eqref{eq:UVWidentity1} and the first equation in \eqref{eq:PossiblyD8system}. Plugging \eqref{eq:s-hat-residues} into \eqref{eq:jmu-tau} we get
\begin{gather*}
\dfrac{\mathrm{d}}{\mathrm{d}z}\ln(\tau(z))=\frac{zU'(z)^2}{4U(z)^2}-
 2 U(z)+\frac2{U(z)}.
\end{gather*}
Differentiating once again, we have
\begin{gather} \label{eq:d8-solvability}
\dfrac{\mathrm{d}^2}{\mathrm{d}z^2}\ln(\tau(z))=-\frac{1}{4}\left(\dfrac{\mathrm{d}}{\mathrm{d}z}\ln(U(z))\right)^2.
\end{gather}
 Now we see from \eqref{eq:d8-solvability} that the set $\Sigma$ of zeros of the $\tau$-function coincides precisely with the union of poles and zeros of the function $U(z)$.

\subsection[Relationship between U\^{}\{even\}, U\^{}\{odd\}]{Relationship between $\boldsymbol{U^{\mathrm{even}}}$, $\boldsymbol{U^{\mathrm{odd}}}$}

To complete the proof of Theorem \ref{thm:general}, we must show that $U^{\mathrm{odd}}(z) = -1/U^{\mathrm{even}}(z)$. One can already observe that this should be the case by checking that the leading behavior predicted in Theorem \ref{thm:D8-asymptotics-zero} satisfies the involution, but we now present a proof on the level of Riemann--Hilbert problems. First, note that if one chooses the square root in \eqref{eq:Vm-even} in such a way that~$\mathbf{V}^\mathrm{even} = \ii \sigma_3 \mathbf{V}^\mathrm{odd}$, it follows from \eqref{eq:limiting-connection-matrix} that\footnote{One can check that making the other choice of the square root yields the same connection matrix but with the opposite sign, and so it follows from Remark \ref{rem:minus-y1-y2} that this choice is immaterial.}
\[
 \mathbf{C}_{0\infty}^{\mathrm{odd}} = \sigma_3 \mathbf{C}^{\mathrm{even}}_{0 \infty} \sigma_3.
\]
This, in particular, implies the symmetry
\[
\widetilde{\mathbf{S}}^{\mathrm{odd}}(\lambda, z) = \sigma_3 \widetilde{\mathbf{S}}^{\mathrm{even}}(\lambda, z) \sigma_3,
\]
and, in view of \eqref{eq:inverse-fabry}, we have
\begin{equation}
 \mathbf{\Omega}^{\mathrm{odd}}(\lambda, z) = \rho_\infty^{\sigma_3/2}(\lambda, z) \mathbf{K}\left( \mathbb{I} + \dfrac{1}{\sqrt{-\ii \lambda}} \mathbf{T}^{\mathrm{even}}(z) \right) \sigma_3 \widetilde{\mathbf{S}}^{\mathrm{even}}\big(\sqrt{-\ii \lambda}, z\big) \sigma_3.
 \label{eq:Omega-even-odd}
\end{equation}
Recalling \eqref{eq:rho-infty-branch}, \eqref{eq:K-def}, \eqref{eq:L-def}, and the identity $2\ii U(z)X(z) = \Delta^{(8)}_{11}(z)/\Delta^{(8)}_{21}(z)$, see \eqref{eq:Delta-U-X-identity}, we have~$\mathbf{\Omega}^{\mathrm{odd}}(\lambda, z) = \mathbf{G}(\lambda, z) \mathbf{\Omega}^{\mathrm{even}}(\lambda, z)$, where
\begin{align}
\mathbf{G}(\lambda, z):={}& \rho_\infty^{\sigma_3/2}(\lambda, z) \mathbf{K}\left( \mathbb{I} + \dfrac{1}{\sqrt{-\ii \lambda}} \mathbf{T}^{\mathrm{even}}(z) \right) \sigma_3 \left( \mathbb{I} + \dfrac{1}{\sqrt{-\ii \lambda}} \mathbf{T}^{\mathrm{even}}(z) \right) \rho_\infty^{-\sigma_3/2}(\lambda, z) \nonumber\\
={}& \frac{1}{\rho_\infty(\lambda, z)}\begin{bmatrix} 2U^{\mathrm{even}}(z)X^{\mathrm{even}}(z) & -4\ii (U^{\mathrm{even}}(z) X^{\mathrm{even}}(z))^2 \\ -\ii & -2U^{\mathrm{even}}(z)X^{\mathrm{even}}(z) \end{bmatrix}\nonumber\\
& + \rho_\infty(\lambda, z) \begin{bmatrix}
 0 & \ii \\ 0 & 0
\end{bmatrix}.\label{eq:G-def}
\end{align}
To deduce the relationship between $U^{\mathrm{even}}$, $U^{\mathrm{odd}}$, we now recall that $\mathbf{\Omega}^{\mathrm{even/odd}}$ satisfy the Lax pair~\eqref{eq:Lax-system}. Transforming $\mathbf{\Omega}^{\mathrm{even}}$ as in the right-hand side of \eqref{eq:Omega-even-odd} induces a gauge transformation of the $\lambda$-equation and we have that $\mathbf{\Omega}^{\mathrm{odd}}$ satisfies two equations; the first is the one in \eqref{eq:Lax-system} and the second is
\[
\dpd{\mathbf{\Omega}^{\mathrm{odd}}}{\lambda}(\lambda, z) = \widetilde{\mathbf{\Lambda}}(\lambda, z) \mathbf{\Omega}^{\mathrm{odd}}(\lambda, z),
\]
where
\[
\widetilde{\mathbf{\Lambda}}(\lambda, z) = \dpd{\mathbf{G}}{\lambda}(\lambda, z) \mathbf{G}^{-1}(\lambda, z) + \mathbf{G}(\lambda, z) \mathbf{\Lambda}^{\mathrm{even}}(\lambda, z)\mathbf{G}^{-1}(\lambda, z).
\]
Using \eqref{eq:G-def} and \eqref{eq:UVWXidentity1}, we see that
\begin{align*}
& \widetilde{\mathbf{\Lambda}}(\lambda, z) = \begin{bmatrix} 0 & \ii z \\ 0 & 0 \end{bmatrix}\\
 & \hphantom{\widetilde{\mathbf{\Lambda}}(\lambda, z) =}{}
 + \dfrac{1}{4 \!\lambda}\begin{bmatrix} 2 - V^{\mathrm{even}}(z) + 8\ii U^{\mathrm{even}}(z)X^{\mathrm{even}}(z) & F^\mathrm{even}(z) \\ 2 & -2 + V^{\mathrm{even}}(z) - 8\ii U^{\mathrm{even}}(z)X^{\mathrm{even}}(z) \end{bmatrix} \\
 & \hphantom{\widetilde{\mathbf{\Lambda}}(\lambda, z) =}{} -\dfrac{1}{\lambda^2} \!\begin{bmatrix} (U^{\mathrm{even}}(z))^2 X^{\mathrm{even}}(z) & 2\ii (U^{\mathrm{even}}(z))^3(X^{\mathrm{even}}(z))^2 \\ \ii U^{\mathrm{even}}(z)/2 & - (U^{\mathrm{even}}(z))^2 X^{\mathrm{even}}(z)\end{bmatrix},
\end{align*}
where
\[
F^\mathrm{even}(z) := 4\ii U^{\mathrm{even}}(z) X^{\mathrm{even}}(z) \left(V^{\mathrm{even}}(z) + 6 U^{\mathrm{even}}(z) X^{\mathrm{even}}(z)-2 \ii U^{\mathrm{even}}(z) \right)- \dfrac{4z}{ U^{\mathrm{even}}(z)}.
\]
Since $\det(\mathbf{\Omega}^\mathrm{odd}) = 1$, it follows that $\widetilde{\mathbf{\Lambda}}(\lambda, z) = \mathbf{\Lambda}^{\mathrm{odd}}(\lambda, z)$ and we arrive at identities relating all the potentials $U^{\mathrm{even/odd}}(z)$, $V^{\mathrm{even/odd}}(z)$, $W^{\mathrm{even/odd}}(z)$, $X^{\mathrm{even/odd}}(z)$; comparing the (2,1) entries of the coefficient of $\lambda^{-2}$ yields the desired relation
\[
U^{\mathrm{even/odd}}(z) = -1/U^{\mathrm{odd/even}}(z).
\]

\subsection{Solutions of Suleimanov} Considering the limit of even B\"acklund iterates when $\mu = 1/4$ yields a particularly symmetric solution of Painlev\'e-III($D_8$). The corresponding monodromy data are the following:
\[
\mathbf{S}_0^0 = \mathbf{S}_1^\infty = \mathbb{I},
\qquad \text{and} \qquad
\mathbf{C}_{0\infty} = (\ii \sigma_2) \mathbf{C}_{0\infty} (\ii \sigma_2) = \begin{bmatrix}-y_2 & y_1 \\ y_1 & y_2 \end{bmatrix}.
\]
In this case, we have that $y_3 = 0$. This is, for example, the situation when considering rational solutions of Painlev\'e-III as outlined in Section \ref{sec:rational-solutions-parameters}.

\begin{Remark}
In this setting and up to a rescaling of the $z$ and $\xi$ variables, Riemann--Hilbert Problem \ref{rhp:S-hat} is the same Riemann--Hilbert problem as in \cite[Section 13.1]{FIKN}, which corresponds to solutions of the sine-Gordon reduction of Painlev\'e-III
\begin{equation}
\dod[2]{w}{t} + \dfrac{1}{t} \dod{w}{t} + \sin w(t) = 0.
\label{eq:sine-gordon}
\end{equation}
This is partly due to the parameters $y_1$, $y_2$ satisfying the condition $y_1^2 + y_2^2 + 1 = 0$ (i.e., ${\det \mathbf{C}_{0\infty} = 1}$), and is to be expected since equation \eqref{eq:sine-gordon} is equivalent to \eqref{eq:PIII-$D_8$} and their solutions are related via the formula
\begin{gather}
U(z) = \ii \ee^{-\ii w\left(\pm\ee^{3\pi \ii/4}4\sqrt{2z}\right)}.\label{eq:U-w-correspondence}
\end{gather}
We can also mention that real-valued (for real $z$) solutions of \eqref{eq:sine-gordon} are singled out by condition~\eqref{eq:suleimanov condition} below. We use identity \eqref{eq:U-w-correspondence} taking the minus sign to formulate Theorem \ref{thm:D8-asymptotics-zero}. Since this connection will not be used further, we do not elaborate on it.
\end{Remark}

We end this discussion by noting yet another interesting connection to certain highly symmetric solutions of PIII($D_8$) which appear in the work of Suleimanov \cite{suleimanov} on nonlinear optics, and later were found in the context of the focusing nonlinear Schr\"odinger equation \cite{MR4007631, BLM20}. More precisely, note that Riemann--Hilbert Problem~\ref{rhp:S-hat} (and the Riemann--Hilbert problem satisfied by $\widetilde{\mathbf{S}}(\xi, z)$) agrees with \cite[Riemann--Hilbert Problem 4]{MR4007631} up to an appropriate rescaling of $z$, $\xi$ in the special case when $y_1$, $y_2$ are chosen such that
\begin{equation} \label{eq:suleimanov condition}
\sigma_2 \begin{bmatrix} -y_2 & y_1 \\ y_1 & y_2 \end{bmatrix} \sigma_2 = \overline{\begin{bmatrix} -y_2 & y_1 \\ y_1 & y_2 \end{bmatrix}} \Leftrightarrow y_1 = -\overline{y_1} \qquad \text{and} \qquad y_2 = -\overline{y_2}.
\end{equation}
This imposes conditions on $e_0$, $e_2$, $e_\infty$, which can be written out explicitly in the case of the rational solutions of Painlev\'e-III. Namely, in this case
\[
y_1 = \frac{\ii \ee^{\ii\pi m}}{\sqrt{1+ \ee^{2\pi \ii m}}}, \qquad y_2 = \frac{\ii }{\sqrt{1+ \ee^{2\pi \ii m}}}, \qquad y_3 = 0.
\]
These satisfy the symmetry conditions above exactly when $m \in \ii\R +\Z$.

\subsection*{Acknowledgements}

We would like to thank Roozbeh Gharakhloo and Deniz Bilman for bringing to our attention applications of our work to $2j-k$ determinants and the Suleimanov solutions respectively. We would like to thank Marco Fasondini for providing us his program that we used to numerically confirm our results and produce Figures~\ref{fig:rational} and \ref{fig:random}. The work of Andrei Prokhorov was supported by NSF MSPRF grant DMS-2103354, NSF grant DMS-1928930, and RSF grant 22-11-00070. Part of the work was done while Prokhorov was in residence at the Mathematical Sciences Research Institute in Berkeley, California, during the Fall 2021 semester. Ahmad Barhoumi was partially supported by the NSF under grant DMS-1812625. Peter Miller was partially supported by the NSF under grants DMS-1812625 and DMS-2204896.

\pdfbookmark[1]{References}{ref}
\LastPageEnding

\end{document}